\documentclass[11pt]{article}
\usepackage{amsmath, amstext, amsthm, amssymb, dsfont, mathrsfs}

\usepackage{hyperref}
\usepackage[totalheight=23 true cm, totalwidth=15 true cm]{geometry}

\usepackage[all]{xy}

\usepackage[shortcuts]{extdash}

\title{Differential Galois groups, specializations and Matzat's conjecture}

\author{Ruyong Feng\thanks{Supported by the NSFC grants
		11771433, 11688101, Beijing Natural Science Foundation under
		Grant Z190004 and National Key Research and Development Project 2020YFA0712300.} \ and Michael Wibmer\thanks{Supported by the NSF grants DMS-1760212, DMS-1760413, DMS-1760448 and the Lise Meitner grant M-2582-N32 of the Austrian Science Fund FWF.}}

\newtheorem{theo}{Theorem}[section]
\newtheorem{lemma}[theo]{Lemma}
\newtheorem{prop}[theo]{Proposition}
\newtheorem{cor}[theo]{Corollary}
\newtheorem{defi}[theo]{Definition}
\newtheorem{rem}[theo]{Remark}

\newtheorem*{conjmatzat}{Matzat's Conjecture}
\newtheorem*{theoremA}{Theorem A}
\newtheorem*{theoremB}{Theorem B}
\newtheorem*{corB1}{Corollary B1}
\newtheorem*{corB2}{Corollary B2}

\theoremstyle{definition}
\newtheorem{ex}[theo]{Example}
\newtheorem{notation}[theo]{Notation}

\def\spec{\operatorname{Spec}}

\def\s{\sigma}
\def\A{\mathbb{A}}

\def\bQ{\mathbb{Q}}
\def\bE{{\mathbb E}}

\def\bZ{\mathbb{Z}}

\def\Hom{\operatorname{Hom}}
\def\Gl{\operatorname{GL}}

\def\Sl{\operatorname{SL}}
\def\mpp{\operatorname{pp}}

\def\pp{\operatorname{PP}}

\def\Div{\operatorname{Div}}

\def\G{\mathcal{G}}
\def\H{\mathcal{H}}

\def\Aut{\operatorname{Aut}}

\def\m{\mathfrak{m}}
\def\km{\mathfrak{m}}

\def\cJ{{\mathcal J}}

\def\cE{{\mathcal E}}

\def\cP{{\mathcal P}}

\def\<{\langle}
\def\>{\rangle}

\def\de{\delta}

\def\ord{\operatorname{ord}}

\newcommand{\f}{\phi}

\newcommand{\p}{\mathfrak{p}}
\newcommand{\z}{\mathfrak{z}}

\newcommand{\Z}{\mathcal{Z}}
\newcommand{\gal}{\operatorname{Gal}}

\newcommand{\im}{\operatorname{Im}}
\newcommand{\ram}{\operatorname{ram}}
\newcommand{\res}{\operatorname{res}}
\newcommand{\Q}{\mathcal{Q}}

\newcommand{\C}{\mathbb{C}}

\newcommand{\I}{\mathbb{I}}

\newcommand{\Ga}{\mathbb{G}_a}

\newcommand{\X}{\mathcal{X}}
\newcommand{\Y}{\mathcal{Y}}

\newcommand{\id}{\operatorname{id}}

\newcommand{\VV}{\mathcal{V}}

\newcommand{\Autt}{\underline{\operatorname{Aut}}}

\newcommand{\B}{\mathcal{B}}
\newcommand{\SSS}{\mathcal{S}}
\newcommand{\T}{\mathcal{T}}
\newcommand{\R}{\mathcal{R}}
\newcommand{\U}{\mathcal{U}}

\newcommand{\nn}{\mathbb{N}}

\def\bfd{{\mathbf d}}
\def\bff{{\mathbf f}}

\def\bfm{{\mathbf m}}

\newcommand{\gen}{{\operatorname{gen}}}
\newcommand{\sing}{\operatorname{Sing}}

\renewcommand{\L}{\mathcal{L}}

\renewcommand{\labelenumi}{{\rm (\roman{enumi})}}

\newcommand{\cA}{\mathcal{A}}
\newcommand{\algrel}{\operatorname{AlgRel}}
\newcommand{\linrel}{\operatorname{LinRel}}

\usepackage{todonotes} % for comments

\begin{document}

%\setcounter{tocdepth}{1}
%\tableofcontents

\maketitle

%14L15 Group Schemes
%14L17 Affine alg groups
%34M50 Inverse Problems (Galois)
%12H05 Differential algebra

\begin{abstract}
	We study families of linear differential equations parametrized by an algebraic variety $\X$ and show that the set of all points $x\in \X$, such that the differential Galois group at the generic fibre specializes to the differential Galois group at the fibre over $x$, is Zariski dense in $\X$. As an application, we prove Matzat's conjecture in full generality: The absolute differential Galois group of a one-variable function field over an algebraically closed field of characteristic zero is a free proalgebraic group.\let\thefootnote\relax\footnotetext{
		{\em Key words and phrases}:
		Differential Galois theory, Matzat's conjecture, specializations of linear differential equations, N\'{e}ron's specialization theorem		
%		Michael Wibmer, Institute of Analysis and Number Theory, Graz University of Technology, Kopernikusgasse 24, 8010 Graz, Austria, \texttt{wibmer@math.tugraz.at}
	}
\end{abstract}	

\newpage

%\enlargethispage{40mm}

\setcounter{tocdepth}{2}
\tableofcontents

\section{Introduction}

The absolute differential Galois group of a differential field $F$ governs the algebraic properties of solutions of linear differential equations over $F$. Determining the structure of this group for interesting differential fields is a central problem in differential Galois theory (see e.g., \cite[Section 10]{SingerPut:differential}). Matzat's conjecture addresses the case of one-variable function fields.

\begin{conjmatzat}
		Let $k$ be an algebraically closed field of characteristic zero and let $F$ be a one-variable function field over $k$, equipped with a non-trivial $k$-derivation. Then the absolute differential Galois group of $F$ is the free proalgebraic group on a set of cardinality $|F|$.
\end{conjmatzat}	

The main goal of this article is to prove Matzat's conjecture (Theorem \ref{theo: Matzat's conjecture}) in the form stated above. Several special cases of Matzat's conjecture are already known:  The first case, when the differential field $F$ equals $(k(x),\frac{d}{dx})$, with $k$ countable and of infinite transcendence degree (over $\mathbb{Q}$) was established in \cite{BachmayrHarbaterHartmannWibmer:FreeDifferentialGaloisGroups}. The case when $F$ is $(k(x),\frac{d}{dx})$, with $k$ of infinite transcendence degree was treated in \cite{BachmayrHarbaterHartmannWibmer:TheDifferentialGaloisGroupOfRationalFunctionField}. The case when $k$ has infinite transcendence degree was established in \cite{Wibmer:SubgroupsOfFreeProalgebraicGroupsAndMatzatsConjecture}.
Moreover, in \cite{Wibmer:SubgroupsOfFreeProalgebraicGroupsAndMatzatsConjecture} it was shown that for an arbitrary algebraically closed field $k$ of characteristic zero, Matzat's conjecture for $(k(x), \frac{d}{dx})$ implies Matzat's conjecture for one-variable function fields over $k$. On the other hand, thanks to the work in \cite{Wibmer:FreeProalgebraicGroups} and \cite{BachmayrHarbaterHartmannWibmer:FreeDifferentialGaloisGroups}, it is known that Matzat's conjecture holds for $F$ countable if and only if every differential embedding problem of finite type over $F$ is solvable. Thus, to prove Matzat's conjecture in full generality, it suffices to prove the following:

\begin{theoremA}[Theorem \ref{theo: solve embedding problems}]
	Let $k$ be an algebraically closed field of characteristic zero. Then every differential embedding problem of finite type over $(k(x),\frac{d}{dx})$ is solvable.
\end{theoremA}

Theorem A was proved in \cite{BachmayrHartmannHarbaterPop:Large} under the assumption that $k$ has infinite transcendence degree. To remove this unnecessary assumption, we establish a specialization result for Picard-Vessiot rings and differential Galois groups that we deem of independent interest.
The most conceptual and useful formulation of our specialization result is in terms of differential torsors (Theorem \ref{theo: main specialization}). For clarity, we state here a simpler, more concrete, version of our specialization result.

Let $k\subseteq k'$ be an inclusion of algebraically closed fields of characteristic zero. A Picard-Vessiot ring $R/k'(x)$ (the differential analog of the splitting field of a polynomial) for a linear differential equation $\de(y)=Ay$ with $A\in k'(x)^{n\times n}$, is as a quotient $R= k'(x)[X,\frac{1}{\det(X)}]/\m$, %of $\K(x)[X,\frac{1}{\det(X)}]$ 
where $X$ is an $n\times n$ matrix of indeterminates, $\de(X)=AX$ and $\m$ is a maximal differential ideal of $k'(x)[X,\frac{1}{\det(X)}]$.

Assume the ideal $\m$ is generated by $p_1,\ldots,p_m$. Choose a finitely generated $k$-subalgebra $\B$ of $k'$ and a monic polynomial $f\in \B[x]$ such that $A\in \B[x]_f^{n\times n}$ and $p_1,\ldots,p_m\in \B[x]_f[X,\frac{1}{\det(X)}]$. For $\X=\spec(\B)$, a specialization $c\in \X(k)=\Hom_k(\B,k)$ extends to a map $c\colon \B[x]_f\to k(x)$ by $c(x)=x$. In particular, applying $c$ to the coefficients of $A$ yields a matrix $A^c\in k(x)^{n\times n}$ and applying $c$ to the coefficients of a $p_i$ yields a $p_i^c\in k(x)[X,\frac{1}{\det(X)}]$.

\begin{theoremB}[Corollary \ref{cor: algebraic relations preserved}] \label{theo: B}
The set of all $c\in \X(k)$ such that $\m^c=(p_1^c,\ldots,p_m^c)$ is a maximal differential ideal of $k(x)[X,\frac{1}{\det(X)}]$, i.e., $R^c=k(x)[X,\frac{1}{\det(X)}]/\m^c$ is a Picard-Vessiot ring for $\de(y)=A^cy$, is Zariski dense in $\X(k)$.
\end{theoremB}

So, roughly speaking, Theorem B asserts that, for many specializations, the algebraic relations among the solutions of the specialized equation $\de(y)=A^cy$ are exactly the specializations of the algebraic relations among the solutions of the generic equation $\de(y)=Ay$. We note that already the existence of a single ``good'' specialization in Theorem B is a nontrivial problem.

The behavior of differential Galois groups and Picard-Vessiot rings under specialization has been studied by various authors in various different settings (\cite{Goldman:SpecializationAndPicardVessiotTheory}, \cite[Section~2.4]{Katz:ExponentialSumsAndDifferentialEquations}, \cite{Singer:ModuliOfLinearDifferentialEquations}, \cite[Section V]{Hrushovski:ComputingTheGaloisGroupOfaLinearDifferentialEequation}, \cite{Andre:SurLaConjectureDesPCourburesDeGrothendieckKatzEtUnProblemeDeDwork} \cite{BerkenboschVanDerPut:FamiliesOfLinearDifferentialEquationsOnTheProjectiveLine}, \cite{DosSantos:TheBehaviourOfTheDifferentialGaloisGroupOnTheGenericAndSpecialFibres}, 
\cite[Section 1.6]{Davy:Specialization}, \cite{Seiss:RootParametrizedDifferentialEquationsForFheClassicalGroups}, \cite{RobertzSeiss:NormalFormsInDifferentialGaloisTheoryForTheClassicalGroups}). Unfortunately, these results are not general or precise enough to be applicable to a proof of Matzat's conjecture.

It is well-known that the set of good specializations in Theorem B may not contain a Zariski open subset, in particular it is in general not Zariski constructible. The easiest example illustrating this point is the differential equation $\de(y)=\frac{\alpha}{x}y$ with solution $x^\alpha$, which is transcendental for $\alpha\in k\smallsetminus\mathbb{Q}$ but algebraic for $\alpha\in\mathbb{Q}$.

As in this example, despite the lack of generic behavior with respect to the Zariski topology, the set of good specializations is expected to be ``large'' and not merely Zariski-dense. To confirm this expectation, we use, following \cite{Hrushovski:ComputingTheGaloisGroupOfaLinearDifferentialEequation}, commutative group schemes of finite type over $\B$ to define ``open'' subsets of $\X(k)$.
An ad-open subset (the ad refers to the additive group $\Ga$) of $\X(k)$ is of the form $\{c\in \X(k)|\ c \text{ is injective on } \Gamma\}$, where $\Gamma$ is a finitely generated subgroup of $\Ga(\B)=(\B,+)$ and $\X(k)$ is identified with $\Hom_k(\B,k)$. For example, $k$ minus a finitely generated $\mathbb{Q}$\=/subspace of $k$, is an ad-open subset of $k=\A^1(k)$. Note that in case $\Gamma$ is generated by a single element $h$, the corresponding ad-open subset is the Zariski open subset of $\X(k)$ where $h$ does not vanish. We can thus think of ad-open subsets as a generalization of Zariski open subsets. Ad$\times$Jac-open subsets of $\X(k)$ are defined roughly similar but with the additive group replaced by a product of the additive group with a Jacobian variety. (For a precise formulation see Definition~\ref{defi: ab-open subset} below.)
For example, 
\begin{equation}
\label{eq: adxjac}
\{\alpha\in k|\ 256\alpha^3-27\neq 0 \text{ and } (0,1) \text{ is not torsion on the elliptic curve } y^2=x^3-4\alpha x+1\}
\end{equation}
is an ad$\times$Jac-open subset of $k=\mathbb{A}^1(k)$.

A more detailed version of Theorem B, confirming the expectation that the set of good specializations is large, states that the set of good specializations contains an ad$\times$Jac-open subset of $\X(k)$ (Theorem \ref{theo: main specialization}). Moreover, if the differential Galois group of $
\de(y)=A y$ (over $k'(x)$) is connected, then the Jacobian variety is not needed, i.e., the set of good specializations contains an ad-open subset (Corollary \ref{cor: no Jac if connected}).

We note that a result similar to Theorem B but for finite difference equations, i.e., for the operator $k(x)\to k(x),\ h(x)\mapsto h(x+1)$ instead of the derivation $\frac{d}{dx}$, was established in \cite{Feng:DifferenceGaloisGrousUnderSpecialization}. There Jacobian varieties are not needed, essentially because $k(x)$ does not have nontrivial finite difference field extensions.
 (In our context, the Jacobian varieties arise via Jacobians of curves associated to finite field extensions of $k(x)$, which are of course differential fields.)  However, in \cite{Feng:DifferenceGaloisGrousUnderSpecialization} one also requires subsets that are ``open'' with respect to the multiplicative group. These occur in the step when one has to determine the multiplicative independence of rational functions. In our context, this step is replaced by determining the logarithmic independence of algebraic functions which leads to ad$\times$Jac\=/open subsets.

 Our proof of Theorem B owes a lot to \cite{Hrushovski:ComputingTheGaloisGroupOfaLinearDifferentialEequation}. In particular, it relies on the ideas underlying Hrushovski's algorithm for computing the differential Galois group of a linear differential equation. 
 Compared to the specialization results in \cite[Section V]{Hrushovski:ComputingTheGaloisGroupOfaLinearDifferentialEequation}, our result is more general and more precise: In \cite{Hrushovski:ComputingTheGaloisGroupOfaLinearDifferentialEequation}  open subsets with respect to arbitrary commutative algebraic groups are required, whereas for us, the algebraic group can be chosen to be a direct product of the additive group with a Jacobian variety. If the generic differential Galois group is connected, even the additive group alone is sufficient for us. Moreover, many of the statements in \cite[Section V.A]{Hrushovski:ComputingTheGaloisGroupOfaLinearDifferentialEequation} are restricted to the case that $\B$ has Krull dimension one and $k=\overline{\mathbb{Q}}$.

 \medskip

 Historically, a prime force for the development of differential algebra and differential Galois theory, was the desire to understand when an indefinite integral or differential equation can be solved ``by quadratures''. Of course, making the notion of ``solving'' rigorous was part of the problem.
As explained in the beautiful survey \cite{Zannier:ICM2014}, when studied in families, this somewhat old-fashioned topic of solving, is closely connected to some modern problems in the realm of unlikely intersections.  

Given an indefinite integral of an algebraic function depending on a parameter, or a linear differential equation depending on a parameter, an important question (asked in \cite[p. 550]{Zannier:ICM2014}) is the following:

\medskip

\begin{centering}
\emph{
For which values of the parameter can the solutions be expressed within some
prescribed class, assuming this can't be done for the generic solution?
}
\end{centering}

\medskip

The expectation is that this set of ``exceptional'' parameter values is small (albeit not necessarily finite). For example, in \cite{MasserZannier:TorsionPointsPellEquationAndIntegrationInElementaryTerms}, it was shown that for a parametric family of differentials on an algebraic curve, which cannot be integrated in elementary terms at the generic fibre, the set of specializations such that the specialized differential can be integrated in elementary terms is small, in fact, finite in many circumstances.

Based on our specialization result (Theorem \ref{theo: main specialization}) we can answer the above question for classes of functions that are amenable to differential Galois theory in the sense that solvability of a linear differential equation in the class of functions can be characterized through a property of the differential Galois group. For example, for the widely used class of Liouvillian functions we obtain:   

\begin{corB1}[Corollary \ref{cor: Liouvillian solutions preserved}] \label{cor: B1}
	Assume that the differential equation $\de(y)=A y$ (over $k'(x)$) does not have a basis of solutions consisting of Liouvillian functions, then the set of specializations $c\in\X(k)$ such that the differential equation $\de(y)=A^cy$ (over $k(x)$) has a basis of solutions consisting of Liouvillian functions, is contained in an ad$\times$Jac-closed subset of $\X(k)$. 
\end{corB1}
Moreover, if the differential Galois group of $\de(y)=Ay$ (over $k'(x)$) is connected, then already an ad-closed subset suffices. Here a subset of $\X(k)$ is called ad$\times$Jac-closed (or ad-closed) if its complement is ad$\times$Jac-open (or ad-open). For example, a finitely generated $\mathbb{Q}$-subspace of $k$ is an ad-closed subset of $k=\mathbb{A}^1(k)$.

As an illustration of Corollary B1, consider Bessel's differential equation
$$\de^2(y)+\tfrac{1}{x}\de(y)+(1-(\tfrac{\alpha}{x})^2)y=0.$$
It has a basis of solutions consisting of Liouvillian functions if and only if $\alpha\notin \frac{1}{2}+\mathbb{Z}$. The exceptional set $\frac{1}{2}+\mathbb{Z}$ is contained in $\mathbb{Q}$, an ad-closed subset of $k$. % the complement of the ad-open subset $k\smallsetminus \mathbb{Q}$ of $\X(k)=k$.
As predicted by the general theory, in this case, the Jacobian variety is not needed because the differential Galois group at the generic fibre (i.e., the differential Galois group of Bessel's equation over $k'(x)$, where $k'$ is the algebraic closure of the rational function field $k(\alpha)$) is $\Sl_2$, which is connected.

If we choose as class of functions, the class of all algebraic functions, we obtain an equicharacteristic zero version of the Grothendieck-Katz conjecture (cf. \cite[Cor. to Prop. 5.1]{Hrushovski:ComputingTheGaloisGroupOfaLinearDifferentialEequation} and \cite[Prop. 7.1.1]{Andre:SurLaConjectureDesPCourburesDeGrothendieckKatzEtUnProblemeDeDwork}): 

\begin{corB2}[Corollary \ref{cor: algebraic solutions preserved}]
	Assume that the differential equation $\de(y)=A y$ (over $k'(x)$) does not have a basis of solutions consisting of algebraic functions, then the set of all specializations $c\in\X(k)$ such that the differential equation $\de(y)=A^cy$ (over $k(x)$) has a basis of solutions consisting of algebraic functions, is contained in an ad$\times$Jac-closed subset of $\X(k)$.
\end{corB2}
Again, if the differential Galois group of $\de(y)=Ay$ (over $k'(x)$) is connected, then already an ad-closed subset suffices.
As the intersection of two ad$\times$Jac-open subsets of $\X(k)$ is nonempty (Remark \ref{rem: finite intersection of ad and ab is dense}), Corollary B2 can also be reformulated in the maybe more familiar spirit of a local to global principle: If the set of all $c\in \X(k)$ such that $\de(y)=A^cy$ has a basis of solutions consisting of algebraic functions, contains an ad$\times$Jac-open subset (e.g., a nonempty Zariski open subset), then  $\de(y)=A y$ has a basis of solutions consisting of algebraic functions.

\medskip

Theorem B (and its more detailed variant Theorem \ref{theo: main specialization}) can be used to transfer results from a single algebraically closed field of constants (e.g., the field $\C$ of complex numbers) to an arbitrary algebraically closed field of constants. Our prime illustration of this fact is of course our proof of Theorem A. However, we also show how one can deduce a new and very short proof of the \emph{solution of the inverse problem} in differential Galois theory from Theorem \ref{theo: main specialization}. Recall that the solution of the inverse problem states that every linear algebraic group over $k$ is a differential Galois group over $k(x)$. Using analytic methods, namely Plemelj's (weak) solution of the Riemann-Hilbert problem, this was proved for $k=\mathbb{C}$ already in 1979 (\cite{TretkoffTretkoff:SolutionOfTheInverseProblem}). However, it took more than 25 years and contributions of many authors to finally solve the inverse problem for an arbitrary algebraically closed field $k$ of characteristic zero (\cite{Hartmann:OnTheInverseProblemInDifferentialGaloisTheory}).
Using Theorem \ref{theo: main specialization}, we can deduce the solution of the inverse problem from the solution over $\mathbb{C}$ rather directly. To the best of our knowledge, so far only two types of proofs of the solution of the inverse problem were known. Firstly, the proof from \cite{Hartmann:OnTheInverseProblemInDifferentialGaloisTheory}, secondly proofs that rely on patching. Of course, also Matzat's conjecture and Theorem A imply the solution of the inverse problem. However, these results ultimately rely on patching.

\medskip

We conclude the introduction with a more detailed outline of the paper. 
The larger part of the article (Sections~\ref{sec: Some topics in differential Galois theory} and \ref{sec: Specialization of differential torsors}) is concerned with the proof of the specialization theorem,  which states that the set of ``good'' specializations contains an ad$\times$Jac-open subset. This result would be vacuous if we do not know that ad$\times$Jac-open subsets are nonempty. While the Zariski denseness of ad-open subsets is readily available in the literature, the Zariski denseness of ad$\times$Jac-open subsets does not seem to be available in the generality required for the proof of Matzat's conjecture. In Section \ref{sec: adopen and abopen} we therefore provide a proof of the Zariski denseness of ad$\times$Jac-open subsets based on a variant of N\'{e}ron's specialization theorem and a geometric version of abstract Hilbert sets.

In Section \ref{sec: Some topics in differential Galois theory} we collect various constructions and results from differential Galois theory that are required for the proof of the specialization theorem, including the concept of differential torsors. Differential torsors provide a convenient framework to simultaneously study Picard-Vessiot rings and differential Galois groups under specialization in a compatible fashion. This compatibility is vital for the proof of Theorem A, where, in effect, we consider differential embedding problems under specialization. Roughly, a differential torsor is obtained by spreading out a Picard-Vessiot ring $R$ for $\de(y)=Ay$ (over $k'(x)$) with the action of the differential Galois group $G$, into a nice family $\R$ with an action of an affine group scheme $\G$. Our main specialization result, proved at the end of Section~\ref{sec: Specialization of differential torsors}, then states that there exists an ad$\times$Jac-open subset $\U$ of the parameter space $\X(k)$ such that $\R_c$ is a Picard-Vessiot ring with differential Galois group $\G_c$ for every $c\in \U$.

In the earlier parts of Section \ref{sec: Specialization of differential torsors} the various constructions and criteria from Section~\ref{sec: Some topics in differential Galois theory} are shown to be well-behaved under specialization. Roughly, the proof of the main specialization theorem relies on two intermediate specialization results corresponding to the two main steps of Hrushovski's algorithm. In the first main step of Hrushovski's algorithm one computes the $k'(x)$-algebraic relations among the entries of a fundamental solution matrix for $\de(y)=Ay$, $A\in k'(x)^{n\times n}$ up to a fixed predetermined degree $d=d(n)$. The corresponding specialization result (Theorem \ref{theo:basisunderspecialization}) is that, if $p_1,\ldots,p_m$ is a basis of the $k'(x)$-vector space of all algebraic relations of degree at most $d$ among the entries of a generic solution matrix, then there exists an ad-open subset $\U$ of $\X(k)$ such that $p_1^c,\ldots,p_m^c$ is a $k(x)$-basis of the vector space of all $k(x)$-algebraic relations of degree at most $d$ among the entries of a fundamental solution matrix for the specialized equation $\de(y)=A^cy$ for all $c\in \U$.

Roughly, the first step in Hrushovski's algorithm allows one to compute the differential Galois group up to a torus (and a finite group, which is however irrelevant for our purpose). In the second main step of Hrushovski's algorithm (``the toric part'') the central question is to decide the logarithmic independence of algebraic functions. Here, elements $f_1,\ldots,f_m\in F$ belonging to a finite field extension $F$ of $k'(x)$ are called \emph{logarithmically independent} over $F$ if a relation of the form
$$d_1f_1+\ldots+d_mf_m=\tfrac{\de(f)}{f}\quad  \text{ with } d_1,\ldots,d_m\in\mathbb{Z} \text{ and } f\in F^\times$$
implies $d_1,\ldots,d_m=0$.

 For $F=k'(x)$ this question can be settled as follows. An element of $k'(x)$ is a logarithmic derivative of an element of $k'(x)$ if and only if all its poles are simple with integer residues. Thus an inspection of the poles and residues of $f_1,\ldots,f_m$ will yield all possible values of $d_1,\ldots,d_m$. For $F/k'(x)$ finite, there exists a similar criterion, however, in this case, it involves the degree zero divisor class group of $F/k'$. This is exactly how the Jacobian varieties enter into the picture. The degree zero divisor class group of $F/k'$ can be identified with the group $J(k')$, where $J$ is the Jacobian of a smooth projective curve with function field $F/k'$.
The criterion for the logarithmic independence, thus relates the computation of the differential Galois group to the question if certain elements $\gamma_1,\ldots,\gamma_\ell\in J(K)$ are $\mathbb{Z}$-linearly independent. For example, if $a=\frac{x^2}{x^4+x+\alpha}$, the differential Galois group $G_\alpha$ of
$$\de^2(y)-\big(2a+\tfrac{\de(a)}{2a}\big)\de(y)+\big(a^2-a-\tfrac{\de(a)}{2}\big)y=0$$
is a subgroup of the group $G$ of $2\times 2$ monomial matrices. A sufficient criterion for $G_\alpha=G$, derived from the criterion for logarithmic independence, is that
the point $(0,1)$ is not a torsion point on the elliptic curve $y^2=x^3-4\alpha x+1$. This leads to the ad$\times$Jac-open subset in (\ref{eq: adxjac}).
The specialization result corresponding to the second step of Hrushovski's algorithm (Theorem \ref{theo: main Appendix B}) states that logarithmic independence of algebraic functions is preserved on an ad$\times$Jac-open subset.

Finally, in Section \ref{sec: Applications of the specialization theorem}, with the specialization theorem in hand, we prove all the results outlined in the introduction above, including Matzat's conjecture.

% Appendix A corresponds to the first main step in Hrushovski's algorithm while Appendix B corresponds to the second main step in Hrushovski's algorithm. 

%
%The two main results from the appendices are combined in the later parts of Section \ref{sec: Specialization of differential torsors} to yield a proof of the specialization theorem. 
%Before tackling the proof, in Section \ref{sec: Specialization of differential torsors} we introduce differential torsors and show that ad$\times$Jac-open subsets are Zariski dense. 

\medskip

 We are grateful to Sebastian Petersen for helpful comments on N\'{e}ron's specialization theorem. We are also thankful to the anonymous referee for helpful suggestions.

 \medskip
 
 \medskip
 
 \noindent {\bf \large Notation and conventions:} All rings are assumed to be commutative and unital. All fields are assumed to be of characteristic zero.
 The group of multiplicative units of a ring $\B$ is denoted by $\B^\times$.
 
 If $\X$ is a scheme over $\B$ and $\B\to \B'$ is a morphism of rings, then $\X_{\B'}$ denotes the scheme over $\B'$ obtained by base change via $\B\to \B'$. A similar notation applies to morphisms of schemes.
 For an affine scheme $\X$ over $\B$ we write $\B[\X]$ for the ring of global sections of $\X$ so that $\X=\spec(\B[\X])$. We sometimes abbreviate $\B'[\X_{\B'}]$ to $\B'[\X]$.
 
 For a ring $\B$, the polynomial ring $\B[x]$ over $\B$ in the variable $x$ is always considered as a differential ring with respect to the derivation $\frac{d}{dx}$. Similarly, the field of rational functions in $x$ (over some field of constants) is always considered as a differential field  with respect to  $\frac{d}{dx}$.
 We use $R^\de=\{r\in R|\ \de(r)=0\}$ to denote the constants of a differential ring $(R,\de)$.
  With $X=(X_{ij})_{1\leq i,j\leq n}$ we always denote an $n\times n$ matrix of indeterminates (over practically any ring that is around). It is the matrix of coordinate functions on $\Gl_n$.
 
 For an algebra $\B$ over a field $k$ and a prime ideal $\p$ of $\B$ we denote the residue field at $\p$ with $k(\p)$. If $\B$ is an integral domain, we write $k(\B)$ for the field of fractions of $\B$. If $\X$ is an integral scheme over $k$, we write $k(\X)$ for the function field of $\X$. With $\overline{K}$ we denote the algebraic closure of a field $K$.
 As a general rule, we use calligraphic letters (like $\cA, \mathcal{B},\mathcal{E},\mathcal{X}$,$\ldots)$ when thinking about objects that vary in a family.

 {\bf Throughout this article $k$ is an algebraically closed field of characteristic zero.} By ``algebraic group over $k$'' we mean a (not necessarily linear) group scheme of finite type over $k$.
 Algebraic groups, or more generally group schemes, are often identified with their functor of points.
 By a ``closed subgroup'' of a group scheme we mean a closed subgroup scheme. A ``variety'' is a geometrically integral separated scheme of finite type over a field.

\section{Ad$\times$Jac-open sets} \label{sec: adopen and abopen}

\medskip

As explained in the introduction, our main goal is to show that, for a linear differential equation depending on parameters, the set of parameter values under which the algebraic properties of the solutions are preserved is ``large''. In this section, we make precise the meaning of ``large'' by introducing ad$\times$Jac-open sets following \cite{Hrushovski:ComputingTheGaloisGroupOfaLinearDifferentialEequation} and \cite{Feng:DifferenceGaloisGrousUnderSpecialization}.
Our main specialization theorem (Theorem \ref{theo: main specialization}) states that the set of good specializations contains an ad$\times$Jac-open subset. This result would of course be useless if it is not known that an ad$\times$Jac-open subset is nonempty.

The main result of this section is that an ad$\times$Jac-open subset is not only nonempty but in fact Zariski dense (Theorem \ref{theo: ad meets ab is dense}). It was already shown in \cite{Feng:DifferenceGaloisGrousUnderSpecialization} that the intersection of an ad-open set with an open set with respect to the multiplicative group is Zariski dense. 
%\rf{Two ``set" are added. }
So most of our work here is focused on the case of abelian varieties. For related but weaker results see \cite[Section V.A]{Hrushovski:ComputingTheGaloisGroupOfaLinearDifferentialEequation}.

\medskip

Throughout Section \ref{sec: adopen and abopen} we assume that
\begin{itemize}
	\item $k$ is an algebraically closed field of characteristic zero;
	\item $\B$ is a finitely generated $k$\=/algebra that is an integral domain and
	\item $\X=\spec(\B)$.
\end{itemize}

\subsection{Open sets defined by commutative group schemes}
\label{subsec: Open subsets defined by commutative group schemes}

Let $\cE$ be a commutative group scheme of finite type over $\B$. For every $c\in \X(k)$ we have a specialization morphism 
$$\s_c\colon\cE(\B)\to \cE(k),$$
given by applying $\cE$, considered as a functor on the category of $\B$-algebras, to the morphism $c\colon\B\to k$. For a finitely generated subgroup $\Gamma$ of $\cE(\B)$ we set
$$W_\X(\cE,\Gamma)=\{c\in\X(k)|\ \s_c \text{ is injective on } \Gamma\}.$$  
The subsets of $\X(k)$ of the form $W_\X(\cE,\Gamma)$ satisfy the axioms for a basis of a topology: $\X(k)=W_\X(\cE,1)$ and 
%if $c\in W_\X(\cE_1,\Gamma_1)\cap W_\X(\cE_2,\Gamma_2)$, then 
\begin{equation} \label{eq: intersection of ads}
	W_\X(\cE_1,\Gamma_1)\cap W_\X(\cE_2,\Gamma_2)=W_\X(\cE_1\times \cE_2, \Gamma_1\times\Gamma_2).
\end{equation}
This statement remains true if we restrict the $\cE$'s to belong to a class of group schemes that is closed under taking products, such as the class of all abelian schemes.

We can thus consider the topology of $\X(k)$ generated by such a basis. The point of this construction is that, while families of linear differential equations parametrized by $\X$ do not exhibit generic behaviour with respect to the Zariski topology, they may exhibit generic behaviour with respect to such a finer topology. If we allow more $\cE$'s in our class, the corresponding topology gets finer and so the genericity statement we hope to prove gets weaker. It is therefore vital to restrict the shape of the allowed $\cE$'s as much as possible. In this context, it is then irrelevant whether or not the sets of the form $W_\X(\cE,\Gamma)$ with $\cE$ an allowed group scheme still form a basis of a topology. We note however, that already a genericity result with respect to the discrete topology is not easily obtained, i.e., in the context of Theorem~B from the introduction, it is nontrivial to prove the existence of a single good specialization. 
In fact, we are not aware of a proof of the existence of a single good specialization that does not go through Theorem \ref{theo: ad meets ab is dense}.

We next discuss the shape that we will allow for the possible $\cE$'s.

\begin{defi}
	A subset $\U$ of $\X(k)$ is \emph{ad-open} if it is of the form $\U=W_\X(\Ga,\Gamma)$, for some finitely generated subgroup $\Gamma$ of $\Ga(\B)=(\B,+)$. 
\end{defi}

In other words, a subset $\U$ of $\X(k)$ is ad-open if it is open with respect to the additive group scheme $\Ga$ over $\B$.
As we are in characteristic zero, the group $\Gamma\leq (\B,+)$ is torsion free and thus a finite free $\mathbb{Z}$-module. If follows that the ad-open subsets of $\X(k)$ consist of exactly those specializations that preserve the $\mathbb{Z}$-linear independence of a finite family of elements of $\B$.

Note that if $\Gamma$ is generated by a single nonzero element $b\in \B$, then $$W_\X(\Ga,\Gamma)=\{c\in \X(k)|\ c(b)\neq 0\}$$ is the basic Zariski open subset of $\X(k)$ where $b$ does not vanish. In this sense, the ``topology'' generated by the ad-open sets is finer than the Zariski topology.

\begin{ex}
	Let $\B=k[\alpha]$ be a univariate polynomial ring over $k$ so that $\X=\A^1_k$ and $\X(k)=k$. For $\mathbb{Q}$-linearly independent elements $\lambda_1,\ldots,\lambda_n\in k$, let $\Gamma$ be the subgroup of $(\B,+)$ generated by $\lambda_1,\ldots,\lambda_n$ and $\alpha$. Then $W_\X(\Ga,\Gamma)=k\smallsetminus V$ consists of all elements of $k$ that are not contained in the $\mathbb{Q}$-vector space $V$ generated by $\lambda_1,\ldots,\lambda_n$.
\end{ex}

\begin{rem} \label{rem: intersection of ad-opens}
	If $\Gamma_1,\ldots,\Gamma_n$ are finitely generated subgroups of $(\B,+)$ and $\Gamma$ is the subgroup generated by $\Gamma_1,\ldots,\Gamma_n$, then $W_\X(\Ga,\Gamma)\subseteq W_\X(\Ga,\Gamma_1)\cap\ldots\cap W_\X(\Ga,\Gamma_n)$. Thus a finite intersection of ad-open subsets of $\X(k)$ contains an ad-open subset.
\end{rem}

For the specialization arguments to be carried out in Section \ref{sec: Specialization of differential torsors}, it is often important to be able to enlarge the $k$\=/algebra $\B$ by adjoining finitely many elements from $\overline{k(\B)}$, because certain properties of the differential equation at the generic fibre may only manifest themselves over the algebraic closure $\overline{k(\B)}$ of the field of fractions of $\B$. We therefore also need to consider subsets of $\X(k)$ that are the form $\f(W_{\X'}(\cE,\Gamma))$, where $\X'=\spec(\B')$ with $\B\subseteq\B'\subseteq \overline{k(\B)}$ and $\B'$ is finitely generated over $\B$, $\cE$ is a commutative group scheme of finite type over $\B'$, $\Gamma$ is a finitely generated subgroup of $\cE(\B')$ and $\f\colon \X'(k)\to \X(k)$ is the morphism induced by the inclusion $\B\subseteq \B'$.

We call an inclusion of integral domains $\B\subseteq\B'$ \emph{algebraic} if the field extension $k(\B')/k(\B)$ is algebraic.

The following lemma is a slight variation of
\cite[Lemma 5A.1]{Hrushovski:ComputingTheGaloisGroupOfaLinearDifferentialEequation}. It shows that for ad-open subsets it is not necessary to consider these more general subsets.

\begin{lemma}
	\label{LM:extensionsofbasicopensubsets}
	Let $\B\subseteq\B'$ be an inclusion of integral domains such that $\B'$ is finitely generated and algebraic over $\B$. Set $\X'=\spec(\B')$ and let $\f\colon  \X'(k)\to\X(k)$ be the morphism induced by the inclusion $\B\subseteq\B'$. If $\U'$ is an ad-open subset $\X'(k)$, then there exists an ad-open subset $\U$ of $\X(k)$ such that $\U\subseteq \f(\U')$ and $\f^{-1}(\U)\subseteq \U'$. 
\end{lemma}
\begin{proof}
	Let $\U'=W_{\X'}(\Ga, \Gamma')$, where $\Gamma'\leq (\B',+)$ is generated by $\gamma'_1,\dots,\gamma'_n$. Let $M$ be the Galois closure of $k(\B')/k(\B)$ and let $T, y_1,\dots,y_n$ be indeterminates. Set
	$$
	p(T,y_1,\dots,y_n)=\prod_{g\in \gal(M/k(\B))}\left(T-\sum_{i=1}^n g(\gamma'_i)y_i\right)=\sum_{i=0}^{|\gal(M/k(\B))|}\left(\sum_{j=1}^{\ell_i} b_{i,j} \bfm_{i,j}\right)T^i,$$
	where $\bfm_{i,1},\dots,\bfm_{i,\ell_i}$ are all monomials in $y_1,\dots,y_n$ of degree $|\gal(M/k(\B))|-i$.
	Since $g(p)=p$ for every $g\in \gal(M/k(\B))$, all $b_{i,j}$ are in $k(\B)$. Write $b_{i,j}=\gamma_{i,j}/a$ where $\gamma_{i,j},a\in \B$ and $a\neq 0$. Let $b$ be a nonzero element of $\B$ such that any $c\in \X(k)$ with $c(b)\neq 0$ can be lifted to a $c'\in \X'(k)$. (Such a $b$ exists by Chevalley's theorem.) Set $\U=W_{\X}(\Ga, \Gamma)$, where $\Gamma$ is the subgroup of $\B$ generated by $b$ and all $\gamma_{i,j}$'s.
	
	Assume that $c\in \U$, $c'\in \X'(k)$ and $\f(c')=c$. We claim that $c'\in \U'$. For a contradiction, suppose $c'\notin \U'$. Then there exists a $\gamma'\in \Gamma'\setminus\{0\}$ such that $c'(\gamma')=0$. Write $\gamma'=\sum_{i=1}^n d_i\gamma'_i$ with $d_1,\dots,d_n\in \bZ$. Since $\gamma'\neq 0$, 
	$$
	a p(0,d_1,\dots,d_n)=a(-1)^{|\gal(M/k(\B))|}\prod_{g\in \gal(M/k(\B))} g(\gamma')\neq 0.
	$$
	Furthermore, $a p(\gamma',d_1,\dots,d_n)=0$ and
	$$
	a p(0,d_1,\dots,d_n)=a\sum_{j=1}^{\ell_0} b_{0,j}\bfm_{0,j}(d_1,\dots,d_n)=\sum_{j=1}^{\ell_0} \gamma_{0,j}\bfm_{0,j}(d_1,\dots,d_n)\in \Gamma.
	$$
	On the other hand, since $c'(\gamma')=0$, one has that
	\begin{align*}
		0&=c'\left( ap(\gamma',d_1,\dots,d_n)\right)=\sum_{i=0}^{|\gal(M/k(\B))|} c'\left(\sum_{j=1}^{\ell_i} \gamma_{i,j}\bfm_{i,j}(d_1,\dots,d_n)\right) c'(\gamma')^i \\
		&=c'\left(\sum_{j=1}^{\ell_0} \gamma_{0,j}\bfm_{0,j}(d_1,\dots,d_n)\right)=c'(ap(0,d_1,\dots,d_n))=c(ap(0,d_1,\dots,d_n)).
	\end{align*}
	This implies that $c$ is not injective on $\Gamma$, because $0\neq ap(0,d_1,\dots,d_n)\in \Gamma$, a contradiction. Therefore $\f^{-1}(\U)\subseteq  \U'$. 
	
	To see that $\U\subseteq \f(\U')$, it suffices to note that every $c\in \U$ lifts to a $c'\in \X'(k)$ because $c(b)\neq 0$ as $b\in\Gamma$.
\end{proof}

For group schemes, other than the additive group $\Ga$, we do not have a statement analogous to Lemma \ref{LM:extensionsofbasicopensubsets} at hand. Therefore, in this case, we need to consider algebraic extensions $\B'\supseteq\B$.

For simplicity, we call a group scheme $\cE$ over some integral domain $\B_0$ of \emph{abelian type} (or of \emph{Jacobian type}) if it is commutative, separated and of finite type over $\B_0$ such that the generic fibre $\cE\times_{\spec(\B_0)}\spec(K_0)$ is an abelian variety (or a Jacobian variety) over $K_0$, where $K_0$ is the field of fractions of $\B_0$.

\begin{lemma} \label{lemma: equivalence of topologies}
	Consider the following three kinds of subsets of $\X(k)$:
	\begin{enumerate}
		\item subsets of the form $\f(W_{\X'}(\Ga\times \cJ,\Gamma))$ with $\cJ$ of Jacobian type;
		\item subsets of the form $\f(W_{\X'}(\Ga,\Gamma_1)\cap W_{\X'}(\cJ,\Gamma_2))$ with $\cJ$ of Jacobian type;
		\item subsets of the form $W_\X(\Ga,\Gamma_1)\cap \f(W_{\X'}(\cJ,\Gamma_2))$ with $\cJ$ of Jacobian type.
	\end{enumerate}
	Then a subset of one kind contains a subset of any other kind. Here, in all of the above cases, $\X'=\spec(\B')$, with $\B'$ finitely generated and algebraic over $\B$ and $\f\colon \X'(k)\to \X(k)$ is induced by the inclusion $\B\subseteq\B'$.  
\end{lemma}
\begin{proof}
	As $W_{\X'}(\Ga,\Gamma_1)\cap W_{\X'}(\cJ,\Gamma_2)=W_{\X'}(\Ga\times\cJ,\Gamma_1\times\Gamma_2)$, a set of the second kind is in fact a set of the first kind. Conversely, given $\Gamma\leq \Ga(\B')\times \cJ(\B')$ let $\Gamma_1\leq\Ga(\B')$ denote the image of $\Gamma$ under $\Ga(\B')\times \cJ(\B')\to \Ga(\B')$ and $\Gamma_2\leq\cJ(\B')$ the image of $\Gamma$ under $\Ga(\B')\times \cJ(\B')\to \cJ(\B')$. Then $\Gamma\leq \Gamma_1\times \Gamma_2$ and so 
	$$W_{\X'}(\Ga,\Gamma_1)\cap W_{\X'}(\cJ,\Gamma_2)=W_{\X'}(\Ga\times\cJ,\Gamma_1\times\Gamma_2)\subseteq W_{\X'}(\Ga\times \cJ,\Gamma). $$
	Thus a subset of the first kind contains a subset of the second kind.
	
	If $W_\X(\Ga,\Gamma_1)\cap \f(W_{\X'}(\cJ,\Gamma_2))$ is a subset of the third kind, we may consider $\Gamma_1$ as a subgroup of $(\B',+)$ and then  $W_\X(\Ga,\Gamma_1)\cap \f(W_{\X'}(\cJ,\Gamma_2))= \f(W_{\X'}(\Ga,\Gamma_1)\cap W_{\X'}(\cJ,\Gamma_2))$. Thus a subset of the third kind is in fact a subset of the second kind.
	
	Finally, given a subset $\f(W_{\X'}(\Ga,\Gamma_1)\cap W_{\X'}(\cJ,\Gamma_2))$ of the second kind, there exists, by Lemma \ref{LM:extensionsofbasicopensubsets}, a finitely generated subgroup $\Gamma_0$ of $(\B,+)$ such that $W_\X(\Ga,\Gamma_0)\subseteq \f(W_{\X'}(\Ga,\Gamma_1))$ and $\f^{-1}(W_\X(\Ga,\Gamma_0))\subseteq (W_{\X'}(\Ga,\Gamma_1)$. We then have $$W_\X(\Ga,\Gamma_0)\cap \f(W_{\X'}(\cJ,\Gamma_2))\subseteq \f(W_{\X'}(\Ga,\Gamma_1)\cap W_{\X'}(\cJ,\Gamma_2)).$$	
	So a subset of the second kind contains a subset of the third kind.
\end{proof}

We now define the ``open'' subsets of $\X(k)$ that are the most relevant for us.
%\rf{The ``the" before ``define" was removed.}

\begin{defi} \label{defi: ab-open subset}
	A subset $\U$ of $\X(k)$ is \emph{ab-open} (or \emph{Jac-open}) if
	it is of the form $\U=\f(W_{\X'}(\cE,\Gamma))$, where
	$\X'=\spec(\B')$ with $\B'$ containing $\B$ an integral domain finitely generated and algebraic over $\B$, $\cE$ is a group scheme over $\B'$ of abelian (or of Jacobian) type, $\Gamma$ is a finitely generated subgroup of $\cE(\B')$ and $\f\colon \X'(k)\to \X(k)$ is the morphism induced by the inclusion $\B\subseteq \B'$.
	
	A subset of $\X(k)$ is \emph{ad$\times$Jac-open} if it is the intersection of an ad-open subset with a Jac-open subset of $\X(k)$.
	
	A subset of $\X(k)$ is \emph{ad$\times$Jac-closed} (or \emph{ad-closed}) if its complement is ad$\times$Jac-open (or \emph{ad-open}). 
\end{defi}
In particular, a Jac-open subset is ab-open. We have used the third kind of subsets from Lemma \ref{lemma: equivalence of topologies} to define ad$\times$Jac-open subset. The lemma shows that we could have, more or less equivalently, also used the first or second kind of subsets in the definition.

\begin{lemma} \label{lemma: lift opens}
	Let $\B'$ be an integral domain containing $\B$ such that $\B'$ is finitely generated and algebraic over $\B$. Let $\f\colon \X'(k)\to \X(k)$ denote the morphism induced by the inclusion $\B\subseteq \B'$ and let $\U'$ be a subset of $\X'(k)$ that is either ad-open, ab-open, Jac-open or ad$\times$Jac\=/open, then $\f(\U')$ contains a subset $\U$ of $\X(k)$ that is of the same type as $\U'$.
\end{lemma}
\begin{proof}
	The case of ad-open subsets follows from Lemma \ref{LM:extensionsofbasicopensubsets}. For the case of ab-open and Jac-open subsets it suffices to note that if $\B\subseteq \B'\subseteq \B''$ are such that $\B'$ is finitely generated and algebraic over $\B$ and $\B''$ is finitely generated and algebraic over $\B'$, then $\B''$ is finitely generated and algebraic over $\B$. The case of ad$\times$Jac-open subsets is similar but also uses Lemma \ref{lemma: equivalence of topologies}.
\end{proof}

\subsection{Hilbert sets} \label{subsec: Hilbert sets}

To show that ad$\times$Jac-open subsets are Zarisiki dense we will use Hilbert sets. We therefore, in this section, recall their definition and discuss the properties that are relevant for our proof.

\medskip

Throughout Section \ref{subsec: Hilbert sets} let $k_0$ be a field of characteristic zero and let $\B_0$ be a finitely generated $k_0$-algebra that is an integral domain.

\medskip

We begin by recalling the abstract definition of \emph{Hilbert sets} from \cite[Chapter 9, Section~5]{Lang:FundamentalsofDiophantineGeometry}. This is closely related to the notion of \emph{thin sets} (see e.g., \cite[Section 13.5]{FriedJarden:FieldArithmetic}) but restricted to the affine setting.

For an inclusion $\B_0\subseteq \B_0'$ of integral domains, an element $b'\in \B_0'$ is called \emph{algebraic} over $\B_0$ if $b'$ (considered as an element of the field of fractions of $\B_0'$) is algebraic over the field of fractions of $\B_0$, i.e., $b'$ satisfies a nonzero (not necessarily monic) univariate polynomial over $\B_0$. Moreover, $\B_0'$ is \emph{algebraic} over $\B_0$ if every element of $\B'$ is algebraic over $\B$, i.e., the field extension $k_0(\B_0')/k_0(\B_0)$ of the corresponding fields of fractions is algebraic.

Let $\B_0'$ be an integral domain containing $\B_0$ such that $\B_0'$ is algebraic over $\B_0$ and a finitely generated $\B_0$-algebra. The corresponding fields of fractions then form a finite field extension $k_0(\B_0')/k_0(\B_0)$.

Let $H(\B_0'/\B_0)$ denote the subset of $\spec(\B_0)$ consisting of all $\p\in\spec(\B_0)$ such that there exists a unique $\p'\in\spec(\B_0')$ lying above $\p$ and for this $\p'$ one has $[k_0(\p'):k_0(\p)]=[k_0(\B_0'):k_0(\B_0)]$.
A subset of $\spec(\B_0)$ of the form $H(\B'_0/\B_0)$ is called a \emph{basic Hilbert set}. 

Recall that a basic Zariski open subset of $\spec(\B_0)$ has the form $D(b)=\{\p\in\spec(\B_0)|\ b\notin  \p\}$ for some nonzero $b\in\B_0$. 
%\rf{It seems that $\B_0$ should be $\p$.}
A \emph{Hilbert subset} of $\spec(\B_0)$ is a finite intersection of basic Hilbert sets with a basic Zariski open set.
%\rf{A ``set" is added.}

\begin{ex}
	Let $p\in \B_0[y]$ be a univariate polynomial of positive degree such that $(p)\subseteq \B_0[y]$ is a prime ideal. Set $\B_0'=\B_0[y]/(p)$.
	Then $H(\B_0'/\B_0)$ consists of all $\p\in\spec(\B_0)$ such that the leading coefficient of $p$ does not lie in $\p$ and the image $\overline{p}$ of $p$ in $k_0(\p)[y]$ is an irreducible polynomial. To see this, note that the set of primes of $\B_0'$ lying above $\p$ can be identified with the spectrum of $\B_0'\otimes_{\B_0}k_0(\p)=k_0(\p)[y]/(\overline{p})$.
\end{ex}

For $k_0$ algebraically closed, the above notion of Hilbert subset is not really useful: If $\p$ is a closed point of $\spec(\B_0)$ (i.e., a maximal ideal) then $k_0(\p)=k_0$ has no finite extensions. Thus, if $\p$ belongs to some basic Hilbert set $H(\B_0'/\B_0)$, we must have $[k_0(\B_0'):k_0(\B_0)]=1$. So $k_0(\B_0')=k_0(\B_0)$ and  $\B_0'$ is contained in a localization of $\B_0$.

The following definition introduces a notion of Hilbert set that is useful when working over an algebraically closed field. Recall that (throughout Section \ref{sec: adopen and abopen}) $k$ is an algebraically closed field of characteristic zero, $\B$ is a finitely generated $k$-algebra that is an integral domain and $\X=\spec(\B)$.

\begin{defi}
%	Let $k$ be an algebraically closed field of characteristic zero and let $\B$ be a finitely generated $k$-algebra that is an integral domain.
	 A subset $\H$ of $\X(k)$ is a \emph{geometric Hilbert set} if there exists a subfield $k_0$ of $k$ finitely generated over $\mathbb{Q}$, a finitely generated $k_0$-subalgebra $\B_0$ of $\B$ such that the canonical map $\B_0\otimes_{k_0}k\to \B$ is an isomorphism and a Hilbert subset $H$ of $\X_0=\spec(\B_0)$ such that $\H$ is the inverse image of $H$ under $\X(k)\to \X\to \X_0$. 
	In case we need to be more specific, we call such an $\H$ a geometric \emph{$\B_0/k_0$-Hilbert set}.
\end{defi}

We will see below that geometric Hilbert sets are always Zariski dense in $\X(k)$. The following two lemmas will be needed to establish further good behavior of (geometric) Hilbert sets. 

\begin{lemma} \label{lemma: good generic behaviour}
	Let $\B_0'$ be an integral domain containing $\B_0$ such that $\B_0'$ is algebraic over $\B_0$ and a finitely generated $\B_0$-algebra. Then there exists a nonzero $b\in\B_0$ such that for every $\p\in D(b)$ 
	\begin{itemize}
		\item there exists a $\p'\in\spec(\B_0')$ lying above $\p$;
		\item the fibre $\B'_0\otimes_{\B_0}k_0(\p)$ is a $k_0(\p)$-vector space of dimension $[k_0(\B_0'):k_0(\B_0)]$;
		\item for every prime $\p'\in\spec(\B_0')$ lying above $\p$ one has $[k_0(\p'):k_0(\p)]\leq [k_0(\B_0'):k_0(\B_0)]$.
	\end{itemize}
\end{lemma}
\begin{proof}
	As $\B_0'$ is finitely generated algebraic over $\B_0$, there exists a nonzero $b_1\in \B_0$ such that $(\B'_0)_{b_1}$ is integral over $(\B_0)_{b_1}$. Since $(\B'_0)_{b_1}$ is finitely generated over $(\B_0)_{b_1}$, in fact  $(\B'_0)_{b_1}$ is finite over $(\B_0)_{b_1}$. By generic freeness (see e.g., \cite[\href{https://stacks.math.columbia.edu/tag/051S}{Tag 051S}]{stacks-project}) there exists a nonzero $b_2\in \B_0$ such that $(\B_0')_{b_1b_2}$ is not only finite over $(\B_0)_{b_1b_2}$ but even free. The rank $d$ of $(\B_0')_{b_1b_2}$ as a $(\B_0)_{b_1b_2}$-module is then necessarily equal to $[k_0(\B_0'):k_0(\B_0)]$.
	
	We claim that $b=b_1b_2$ has the required property.
	First note that because $(\B_0')_b$ is integral over $(\B_0)_b$, for every $\p\in D(b)$ there exists a $\p'\in\spec(\B_0')$ lying above $\p$.
	
	Regarding the second point, note that for $\p\in D(b)$ the fibre $$\B_0'\otimes_{\B_0}k_0(\p)=(\B_0')_b\otimes_{(\B_0)_b}k_0(\p)$$ is a $k_0(\p)$-vector space of dimension $d$. Thirdly, a prime $\p'\in\spec(\B_0')$ lying above $\p$ corresponds to a prime $\tilde{\p}'$ in $\spec(\B_0'\otimes_{\B_0}k_0(\p))$. The residue field of $\p'$ can be identified with the residue field of $\tilde{\p}'$, which is an extension of degree at most $d$ of $k_0(\p)$. 
\end{proof}

\begin{lemma} \label{lemma: Hilbert set under projection}
	Let $\B_0'$ be an integral domain containing $\B_0$ such that $\B_0'$ is algebraic and finitely generated over $\B_0$. Let $\pi\colon \spec(\B_0')\to \spec(\B_0)$ denote the morphism corresponding to the inclusion $\B_0\subseteq \B_0'$ and let $H'\subseteq \spec(\B_0')$ be a Hilbert set. Then there exists a Hilbert set $H$ in $\spec(\B_0)$ satisfying the following two conditions:
	\begin{itemize}
		\item $H\subseteq \pi(H')$ and 
		\item for every $\p\in H$ there exists a unique $\p'\in \spec(\B_0')$ lying above $\p$. (Of course then $\p'\in H'$.)
	\end{itemize}
\end{lemma}
\begin{proof}
	Write $H'=H(\B''_1/\B_0')\cap\ldots\cap H(\B''_n/\B_0')\cap D(b')$. 
	By Lemma \ref{lemma: good generic behaviour} there exists a nonzero $b'_1\in \B_0'$ such that, for $i=1,\ldots,n$, every prime in $D(b'_1)$ lifts to a prime in $\B_i''$ and for every prime $\p_i''\in\spec(\B_i'')$ lying above $\p'\in D(b_1')$ one has $[k_0(\p_i''):k_0(\p')]\leq[k_0(\B_i''):k_0(\B_0')]$. By Chevalley's theorem there exists a nonzero $b_1\in \B_0$ such that $D(b_1)\subseteq \pi(D(b'_1b'))$.
	
	By Lemma \ref{lemma: good generic behaviour} again, there exists a nonzero $b_2\in\B_0$ such 
	that %every prime in $D(b_2)$ lifts to a prime of $\B_0'$ and 
	for every $\p\in D(b_2)$ the $k_0(\p)$-vector space $\B_0'\otimes_{\B_0}k_0(\p)$ has dimension $[k_0(\B_0'):k_0(\B_0)]$ and for every prime $\p'$ of $\B_0'$ lying above $\p$ one has $[k_0(\p'):k_0(\p)]\leq[k_0(\B_0'):k_0(\B_0)]$.
	
	We claim that $H=H(\B''_1/\B_0)\cap\ldots\cap H(\B''_n/\B_0)\cap D(b_1b_2)$ has the required property. 	
	Let $\p\in H$. As $\p\in D(b_1)\subseteq \pi(D(b_1'b'))$, there exists a prime $\p'\in D(b_1'b')$ lying above $\p$. Since $\p\in D(b_2)$ we have  $[k_0(\p'):k_0(\p)]\leq[k_0(\B_0'):k_0(\B_0)]$. 
	
	Because $\p'\in D(b_1')$ there exists, for every $i=1,\ldots,n$, a prime $\p''_i\in\spec(\B_i'')$ lying above $\p'$ and satisfying $[k_0(\p_i''):k_0(\p')]\leq[k_0(\B_i''):k_0(\B_0')]$. As $\p_i''\in\spec(B_i'')$ is a prime above $\p$ and $\p\in H(\B_i''/\B_0)$, we see that $\p_i''$ is the unique prime of $\B_i''$ above $\p$ and that $[k_0(\p_i''):k_0(\p)]=[k_0(\B_i''):k_0(\B_0)]$. Therefore
	
	\begin{align*}
		[k_0(\B_i''):k_0(\B_0)] & =[k_0(\B_i''):k_0(\B'_0)]\cdot [k_0(\B'):k_0(\B_0)]\geq [k_0(\p_i''):k_0(\p')]\cdot [k_0(\p'):k_0(\p)]= \\
		& = [k_0(\p_i''):k_0(\p)]=[k_0(\B_i''):k_0(\B_0)]
	\end{align*}
	and it follows that 
	\begin{equation} \label{eq: degree equality}
		[k_0(\p'):k_0(\p)]=[k_0(\B'):k_0(\B_0)]\quad  \text{ and  } \quad [k_0(\p_i''):k_0(\p')]=[k_0(\B_i''):k_0(\B'_0)].
	\end{equation}

	We next show that $\p'\in H'$. As $\p'\in D(b')$ we only have to show that $\p'\in H(\B_i''/\B')$ for $i=1,\ldots,n$. If there was a prime of $\B_i''$ other than $\p_i''$ lying above $\p'$, this would also be a prime other than $\p_i''$ lying above $\p$ which is impossible. Thus $\p_i''$ is the unique prime of $\B_i''$ lying above $\p'$. By (\ref{eq: degree equality}) we see that $\p'\in H(\B_i''/\B')$. So $\p'\in H'$ and $\p\in\pi(H')$, i.e., $H\subseteq \pi(H')$. 
	
	Regarding the second item, note that by the choice of $b_2$, the fibre $\B_0'\otimes_{\B_0}k_0(\p)$ is a $k_0(\p)$-vector space of dimension $[k_0(\B'):k_0(\B_0)]$. On the other hand, there exists a prime in $\B_0'\otimes_{\B_0}k_0(\p)$ whose residue field is an extension of $k_0(\p)$ of degree $[k_0(\B'):k_0(\B_0)]$. Thus $\B_0'\otimes_{\B_0}k_0(\p)$ is a field and there is only one prime in $\spec(\B_0')$ lying above $\p$.
\end{proof}

We next relate the abstract notion of Hilbert set with the more classical notion concerned with polynomials that remain irreducible under specialization (see \cite[Chapter~12]{FriedJarden:FieldArithmetic}).

Let $p_1,\ldots,p_n\in \B_0[y]$ be monic univariate polynomials that are irreducible as elements of $k_0(\B_0)[y]$. For a prime ideal $\p$ of $\B_0$ let $\overline{p_i}$ denote the image of $p_i$ in $k_0(\p)[y]$, i.e., the coefficients of $p_i$ are reduced modulo $\p$. 
%\rf{$f_i$ is changed into $p_i$}
We set
$$H_{\B_0}(p_1,\ldots,p_n)=\{\p\in\spec(\B_0)|\ \overline{p_1},\ldots,\overline{p_n}\in k_0(\p)[y] \text{ are irreducible }\}. $$

\begin{lemma} \label{lemma: relate Hilbert sets}
	Let $\B_0'$ be an integral domain containing $\B_0$, finitely generated and algebraic over $\B_0$. Then there exists a monic polynomial $p\in \B_0[y]$ such that $p$ is irreducible in $k_0(\B_0)[y]$ and a nonzero $b\in\B_0$ such that $$H_{\B_0}(p)\cap D(b)\subseteq H(\B_0'/\B_0).$$
\end{lemma}
\begin{proof}
	As $k_0(\B_0')/k_0(\B_0)$ is a finite field extension, there exists, by the primitive element theorem, a $b'\in \B_0'$ such that $k_0(\B')=k_0(\B_0)[b']$. Replacing $b'$ with a $\B_0$-multiple of $b'$ if necessary, we can assume that the minimal polynomial $p$ of $b'$ over $k_0(\B_0)$ has coefficients in $\B_0$. Write $\B_0'=\B_0[b'_1,\ldots,b'_n]$ and $b'_i=p_i(b')$  with $p_i\in k_0(\B_0)[y]$  $(1\leq i\leq n)$. Let $b\in \B_0\smallsetminus\{0\}$ be such that $bp_i$ has coefficients in $\B_0$ for $i=1,\ldots,n$. Then $b'_i\in (\B_0)_b[b']$ and therefore $(\B_0')_b=(\B_0)_b[b']=(\B_0)_b[y]/(p)$.
	
	Let $\p\in H_{\B_0}(p)\cap D(b)$. Then $\p_b$ is a prime ideal of $(\B_0)_b$ and for the fibre of $\B_0'/\B_0$ over $\p$ we have
	$$\B_0'\otimes_{\B_0}k_0(\p)=(\B_0')_b\otimes_{(\B_0)_b}k_0(\p_b)=(\B_0)_b[y]/(p)\otimes_{(\B_0)_b}k_0(\p_b)=k_0(\p)[y]/(\overline{p}).$$
	As $\p\in H_{\B_0}(p)$, we see that $\B_0'\otimes_{\B_0}k_0(\p)$ is a field extension of $k_0(\p)$ of degree $\deg(p)=[k_0(\B_0'):k_0(\B_0)]$. Thus $\p\in H(\B_0'/\B_0)$ as desired.
\end{proof}
As an immediate corollary to Lemma \ref{lemma: relate Hilbert sets} we obtain:

\begin{cor} \label{cor: relate Hilbert sets}
	Let $H$ be a Hilbert subset of $\spec(\B_0)$. Then there exist $p_1\ldots,p_n\in \B_0[y]$ monic and irreducible in $k_0(\B_0)[y]$ and $b\in \B_0\smallsetminus\{0\}$ such that
	$$H_{\B_0}(p_1,\ldots,p_n)\cap D(b)\subseteq H.$$ \qed
\end{cor}

A different, less intrinsic, notion of geometric Hilbert set was used in \cite{Feng:DifferenceGaloisGrousUnderSpecialization}. In order to be able to use the results from \cite{Feng:DifferenceGaloisGrousUnderSpecialization}, we have to relate the two notions. Let us first recall the definition from \cite[Notation 2.6]{Feng:DifferenceGaloisGrousUnderSpecialization}.

Recall that $\B$ is a finitely generated $k$-algebra that is an integral domain. By Noether's normalization lemma, we can find $\eta_1,\ldots,\eta_\ell,\ldots,\eta_m\in\B$ such that $\B=k[\eta_1,\ldots,\eta_m]$, $\eta_1,\ldots,\eta_\ell$ are algebraically independent over $k$ and $\eta_{\ell+1},\ldots,\ldots,\eta_m$ are integral over $k[\eta_1,\ldots,\eta_\ell]$. We abbreviate $\boldsymbol{\eta}=(\eta_1,\ldots,\eta_m)$ and $\boldsymbol{\eta}_\ell=(\eta_1,\ldots,\eta_l)$.

Note that for $i=1,\ldots,m-\ell$, the minimal polynomial of $\eta_{\ell+i}$ over $k(\boldsymbol{\eta}_l)$ has coefficients in $k[\boldsymbol{\eta}_\ell]$ since $k[\boldsymbol{\eta}_\ell]$ is integrally closed (\cite[Prop. 5.15]{AtiyahMacdonald:Introductiontocommutativealgebra}). We can thus find a subfield $k_0$ of $k$ finitely generated over $\mathbb{Q}$ such that, for $i=1,\ldots,m-\ell$, the minimal polynomial of $\eta_{\ell+i}$ over $k(\boldsymbol{\eta}_l)$ has coefficients in $k_0[\boldsymbol{\eta}_\ell]$.

Let $\boldsymbol{p}$ be a finite tuple of monic elements of $k_0[\boldsymbol{\eta}][y]$ that are irreducible in $k_0(\boldsymbol{\eta})[y]$. Furthermore, let $b\in k_0[\boldsymbol{\eta}]\smallsetminus\{0\}$ and let $\boldsymbol{d}=(d_1,\ldots,d_\ell)\in\mathbb{Z}^\ell$ be an $\ell$-tuple of positive integers. Identifying $\X(k)$ with $\Hom_k(\B,k)$ we define 
$$\H_{k_0,\X(k)}^{\boldsymbol{\eta}}(\boldsymbol{d},\boldsymbol{p},b)$$
as the set of all $c\in \X(k)$ such that 
\begin{itemize}
	\item $[k_0(c(\eta_1),\ldots, c(\eta_i)): k_0(c(\eta_1),\ldots,c(\eta_{i-1}))]\geq d_i$ for $i=1,\ldots,\ell$;
	\item for every entry $p$ of $\boldsymbol{p}$ the polynomial $p^c$ obtained from $p$ by applying $c$ to the coefficients is irreducible in $k_0(c(\boldsymbol{\eta}))[y]$;
	\item $c(b)\neq 0$.
\end{itemize}

The following lemma shows that every geometric Hilbert set contains a set of the form $\H_{k_0,\X(k)}^{\boldsymbol{\eta}}(\boldsymbol{d},\boldsymbol{p},b)$.

\begin{lemma} \label{lemma: containmemt for Hilber sets}
%	Let $k$ be an algebraically closed field of characteristic zero and $\B$ a finitely generated $k$-algebra that is an integral domain. 
	Let $k_0\subseteq k$ be a finitely generated field extension of $\mathbb{Q}$ and $\B_0\subseteq \B$ a finitely generated $k_0$-algebra such that $\B_0\otimes_{k_0}k\to \B$ is an isomorphism. Furthermore, let $H$ be a Hilbert subset of $\X_0=\spec(\B_0)$ and let $\H\subseteq \X(k)$ be the corresponding geometric Hilbert set.
	
	Then there exists $\boldsymbol{\eta}\in \B_0^m$ with $\B_0=k_0[\boldsymbol{\eta}]$ and appropriate $\boldsymbol{p}$, $b$ such that
	$$\H_{k_0,\X(k)}^{\boldsymbol{\eta}}((1,\ldots,1),\boldsymbol{p},b)\subseteq \H.$$
\end{lemma}
\begin{proof}
	Applying Noether's normalization lemma to $\B_0$, we find $\boldsymbol{\eta}=(\eta_1,\ldots,\eta_\ell,\ldots,\eta_m)\in \B_0^m$ with $\B_0=k_0[\boldsymbol{\eta}]$ such that $\eta_1,\ldots,\eta_\ell$ are algebraically independent over $k_0$ and $\eta_{\ell+1},\ldots,\eta_m$ are integral over $k_0[\eta_1,\ldots,\eta_\ell]$. Note that $\eta_1,\ldots,\eta_\ell$ are also algebraically independent over $k$ and $\eta_{\ell+1},\ldots,\eta_m$ are integral over $k[\eta_1,\ldots,\eta_\ell]$. Moreover, for $i=1,\ldots,m-\ell$, the minimal polynomial of $\eta_{\ell+i}$ over $k(\eta_1,\ldots,\eta_\ell)$ has coefficients in $k_0[\eta_1,\ldots,\eta_\ell]$.
	
	By Corollary \ref{cor: relate Hilbert sets}, there exist monic polynomials $p_1,\ldots,p_n\in \B_0[y]$ irreducible in $k_0(\B_0)[y]$ and a $b\in\B_0\smallsetminus\{0\}$ such that 
	$H_{\B_0}(p_1,\ldots,p_n)\cap D(b)\subseteq H$. With $\boldsymbol{p}=(p_1,\ldots,p_n)$ we thus have $\H_{k_0,\X(k)}^{\boldsymbol{\eta}}((1,\ldots,1),\boldsymbol{p},b)\subseteq \H.$
\end{proof}

The following proposition, going back to \cite{Hrushovski:ComputingTheGaloisGroupOfaLinearDifferentialEequation}, explains why the sets of the form   $\H_{k_0,\X(k)}^{\boldsymbol{\eta}}(\boldsymbol{d},\boldsymbol{p},b)$ are useful for us. We assume that $k_0$ and $\boldsymbol{\eta}$ are as described after Corollary~\ref{cor: relate Hilbert sets}. 

\begin{prop} \label{prop: Hilbert set in ad-open}
	Let $\Gamma$ be a finitely generated subgroup of $(k_0[\boldsymbol{\eta}],+)$. Then there exists a $\bf{d}$ such that $\H_{k_0,\X(k)}^{\boldsymbol{\eta}}(\boldsymbol{d},\emptyset,1)\subseteq W_\X(\Ga,\Gamma)$. 
\end{prop}
%\rf{It seems that $W_\X(\Ga,\Gamma)$ is defined in the next subsection.}
\begin{proof}
	This is really just a restatement of \cite[Prop. 2.10]{Feng:DifferenceGaloisGrousUnderSpecialization}.
\end{proof}

\subsection{N\'{e}ron's specialization theorem} \label{subsection Nerons specialization theorem}

N\'{e}ron's specialization theorem is a crucial ingredient in our proof that ad$\times$Jac-open sets are Zariski dense (Theorem \ref{theo: ad meets ab is dense}). Roughly speaking, N\'{e}ron's specialization theorem states that in an algebraic family of abelian varieties, the specialization map defined by a point of the parameter space is injective for all points in a Hilbert subset of the parameter space. This theorem can be used to construct abelian varieties with groups of rational points of large rank. Several versions of N\'{e}ron's specialization theorem are available in the literature. The original reference is \cite{Neron:ProblemesArithmetiquesEtGeometriques}. Other presentations are in \cite[Section~11.1]{Serre:LecturesOnTheMordellWeilTheorem}, \cite[Chapter 9, Section 6]{Lang:FundamentalsofDiophantineGeometry} and \cite[Section~3]{Petersen:OnAQuestionOfFreyAndJardenAboutTheRankOfAbelianVarieties}. Unfortunately, none of these versions seems to directly yield the result we need. We therefore include a self-contained proof of our variant of N\'{e}ron's specialization theorem, mostly following \cite{Serre:LecturesOnTheMordellWeilTheorem} in the argument (cf. \cite[Section 2]{Feng:DifferenceGaloisGrousUnderSpecialization}).

\medskip

Throughout Section \ref{subsection Nerons specialization theorem} we make the following assumptions:

\begin{itemize}
	\item $k_0$ is a field of characteristic zero;
	\item $\X_0=\spec(\B_0)$ is an affine variety over $k_0$ (i.e., a geometrically integral affine scheme of finite type over $k_0$);
	\item $K_0=k_0(\B_0)$ is the function field of $\X_0$, i.e., the field of fractions of $\B_0$;
	\item $\B_0'\subseteq \overline{K_0}$ is a finitely generated $\B_0$-algebra;
	
	\item $\X'_0=\spec(\B_0')$;
	
	\item $\cE_0$ is a commutative separated group scheme of finite type over $\B_0$ such that $\cE_0\times_{\X_0}\spec(K_0)$ is connected.
\end{itemize}

We also make the following assumption on $\cE_0$.

\begin{enumerate}
	\item [(A)] For every finitely generated subgroup $\Gamma$ of $\cE_0(K_0)$, the group 
	$$
	\{\varepsilon\in \cE_0(K_0) \mid \varepsilon^n\in \Gamma \mbox{ for some $n\geq 1$} \}
	$$
	is also finitely generated.  
\end{enumerate}

We will need a few lemmas of a preparatory nature.

\begin{lemma} \label{lemma: multiplication by m}
	Let $G$ be a connected commutative algebraic group over $k_0$ and let $n\geq 1$ be an integer. Then the morphism $[n]\colon G\to G,\ g\mapsto g^n$ has finite kernel and $[n]\colon G(\overline{k_0})\to G(\overline{k_0})$ is surjective.
\end{lemma}
\begin{proof}
	Note that $[n]$ is a morphism of algebraic groups because $G$ is commutative. The induced morphism
	$\operatorname{Lie}([n])\colon \operatorname{Lie}(G)\to \operatorname{Lie}(G)$ on the Lie algebra of $G$ is multiplication by $n$. As we are in characteristic zero, this implies that $\operatorname{Lie}([n])$ is an isomorphism. In particular, $\operatorname{Lie}(\ker([n]))=\ker(\operatorname{Lie}([n]))=0$. Thus $\ker([n])$ is finite.
	Then, for dimension reasons, the image of $[n]\colon G(\overline{k_0})\to G(\overline{k_0})$ must have the same dimension as $G$. As the image is closed and $G$ is connected, it follows that $[n]\colon G(\overline{k_0})\to G(\overline{k_0})$ is surjective.
\end{proof}
%Let $\B'_0\subseteq \overline{K_0}$ be a finitely generated $\B_0$-algebra, e.g., $\B'_0=\B_0$. Then 
For every $x'\in \X'_0$ we have a specialization morphism 
$$\sigma_{x'}\colon \cE_0(\B'_0)\to \cE_0(k_0(x')),$$
obtained by applying $\cE_0$, considered as a functor on the category of $\B_0$-algebras, to the morphism $\B'_0\to k_0(x')$.
%\rf{$\cE$ is changed into $\cE_0$.}

\begin{lemma} \label{lemma: assumption A1}
	Let %$\B'_0\subseteq \overline{K_0}$ be a finitely generated $\B_0$-algebra,
	$\gamma\in \cE_0(\B'_0)$ and $n\geq 1$.  Then the set $\{\varepsilon\in \cE_0(\overline{K_0})\mid \varepsilon^n=\gamma\}$ is finite and nonempty and there exist a nonempty Zariski open subset $\U'$ of $\X'_0$ such that
	$$
	|\{ \varepsilon\in \cE_0(\overline{k_0(x')})\mid \varepsilon^n=\sigma_{x'}(\gamma)\}|=|\{\varepsilon\in \cE_0(\overline{K_0})\mid \varepsilon^n=\gamma\}|
	$$
	for every $x'\in \U'$.
\end{lemma}
\begin{proof}
	The set $\{\varepsilon\in \cE_0(\overline{K_0})\mid \varepsilon^n=\gamma\}$ is finite and nonempty by Lemma \ref{lemma: multiplication by m}. Set $\cE'_0=\cE_{0,\B'_0}$ and consider the morphism $[n]\colon \cE'_0\to \cE'_0, \ g\mapsto g^n$ and its kernel $\ker([n])\leq \cE'_0$. Then $\ker([n])_{x'}$ is the kernel of the $n$-th power map on $\cE'_{0,x'}$ for every $x'\in \X'_0$. 
	
	It follows from Lemma \ref{lemma: multiplication by m} that $\ker([n])_{\xi'}$ is finite, where $\xi'$ is the generic point of $\X'_0$. Being finite spreads out from the generic fibre and also the number of geometric points is constant over a nonempty Zariski open subset (\cite[Cor. 9.7.9]{Grothendieck:EGAIV3}). Thus, there exists a nonempty Zariski open subset $\U'$ of $\X'_0$ such that $$|\{\varepsilon\in \cE'_{0,x'}(\overline{k_0(x')})|\ \varepsilon^n=1\}|=|\{\varepsilon\in \cE'_{0,\xi'}(\overline{K_0})|\ \varepsilon^n=1\}|$$
	for all $x'\in \U'$. Shrinking $\U'$ if necessary, we can also assume that $\cE'_{0,x'}$ is connected for all $x'\in\U'$. (Recall that an algebraic group is connected if and only if it is geometrically connected). Using Lemma \ref{lemma: multiplication by m} we then find
	\begin{align*}
		|\{ \varepsilon\in \cE_0(\overline{k_0(x')})\mid\ & \varepsilon^n=\sigma_{x'}(\gamma)\}|=|\{ \varepsilon\in \cE'_{0,x'}(\overline{k_0(x')})\mid \varepsilon^n=\sigma_{x'}(\gamma)\}|= \\
		& =|\{\varepsilon\in \cE'_{0,x'}(\overline{k_0(x')})|\ \varepsilon^n=1\}|=|\{\varepsilon\in \cE'_{0,\xi'}(\overline{K_0})|\ \varepsilon^n=1\}| = \\
		&=|\{\varepsilon\in \cE_0(\overline{K_0})\mid \varepsilon^n=\gamma\}|
	\end{align*}
	for all $x'\in\U'$.	
\end{proof}

\begin{lemma}
	\label{LM:residucefields}
	Let $\SSS$ be a scheme of finite type over $\B_0$ and let $s\in\SSS(\B_0')$ be such that $k_0(s(\xi'))\to k_0(\xi')=k_0(\B'_0)$ is an isomorphism, where $\xi'$ is the generic point of $\X_0'$ and we think of $s$ as a morphism $s\colon \X'_0\to \SSS$. Then there exists a nonempty Zariski open subset $\U'$ of $\X'_0$ such that the morphism $k_0(s(x'))\to k_0(x')$ induced by $s$ is an isomorphism for every $x'\in \U'$.
	% the residue field of $\SSS$ at $b(x_1)$ is $k_0(x_0)$-isomorphic to $k_0(x_1)$, where $x_0=x_1\cap \B$.
\end{lemma}
\begin{proof}
	Let $\U$ be an affine open subset of $\SSS$ containing $s(\xi')$ and let $D(b')$ be a basic Zariski open subset of $\X'_0$ contained in $s^{-1}(\U)$. Then the restriction $\tilde{s}\colon D(b')\to \U$ of $s\colon \X'_0\to \SSS$ satisfies the assumption of the lemma, i.e., the induced map $k_0(\tilde{s}(\xi'))\to k_0(\xi')$ is an isomorphism. We can thus assume without loss of generality that $\SSS$ is affine. 
	
	So we can identify $\SSS$ with the spectrum of a finitely generated $\B_0$-algebra $\R$. The point $s\in\SSS(\B'_0)$ corresponds to a morphism $\psi\colon\R\to \B'_0$ of $\B_0$-algebras. The assumption that $k_0(s(\xi'))\to k_0(\xi')$ is an isomorphism, means that the field of fractions of $\psi(\R)$ (formed inside $k_0(\B'_0)$) equals all of $k_0(\B'_0)$.	
	
	Fix $b'_1,\ldots,b'_n\in\B'_0$ such that $\B'_0=\B_0[b'_1,\ldots,b'_n]$. For $i=1,\ldots,n$ we can then write $b'_i=\frac{\psi(r_i)}{\psi(\widetilde{r_i})}$ with $r_1,\ldots,r_n,\widetilde{r_1},\ldots,\widetilde{r_n}\in \R$ and $\psi(\widetilde{r_1}),\ldots,\psi(\widetilde{r_m})\in \B_0'$ nonzero. We claim that $\U'=D(\psi(\widetilde{r_1})\ldots\psi(\widetilde{r_n}))\subseteq \X'_0$ has the desired property.
	
	Let $x'\in \U'$ correspond to the prime ideal $\p'$ of $\B'_0$. For $b'\in \B'_0$ let $\overline{b'}\in \B'_0/\p'\subseteq k_0(\B'_0/\p')=k_0(x')$ denote the image of $b'$ in $\B'_0/\p'$. As $\psi(\widetilde{r_i})b'_i=\psi(r_i)$ in $\B_0'$, we have $\overline{\psi(\widetilde{r_i})}\overline{b'_i}=\overline{\psi(r_i)}$ for $i=1,\ldots,n$. As $x'\in \U'$, we can divide by $\overline{\psi(\widetilde{r_i})}$ in $k_0(x')$ to obtain $\overline{b'_i}=\frac{\overline{\psi(r_i)}}{\overline{\psi(\widetilde{r_i})}}$ in $k_0(x')$. 
	Note that the morphism $k_0(s(x'))\to k_0(x')$ is obtained from the injection $\R/\psi^{-1}(\p')\to \B'_0/\p'$ by passing to the field of fractions. In particular, $k_0(s(x'))\to k_0(x')$ is an isomorphism if and only if the field of fractions of $\overline{\psi(\R)}\subseteq \B'_0/\p'\subseteq k_0(\B'_0/\p')$ equals all of $k_0(\B'_0/\p')$. As $k_0(\B'_0/\p')=k_0(\B_0)(\overline{b'_1},\ldots,\overline{b'_n})$, the latter follows from $\overline{b_i'}=\frac{\overline{\psi(r_i)}}{\overline{\psi(\widetilde{r_i})}}$.
\end{proof}

\begin{lemma}
	\label{LM:distinctpointsunderspecialiation}
	Let $S$ be a finite subset of $\cE_0(\B'_0)$. Then there exists a nonempty Zariski open subset $\U'$ of $\X'_0$ such that $\sigma_{x'}$ is injective on $S$ for every $x'\in \U'$.
\end{lemma}
\begin{proof}
	It suffices to treat the case when $S$ consists of two distinct elements $s_1,s_2$.
	As $\cE_0\to \X_0$ is separated, the equalizer $\Z\to \X'_0$ of $s_1, s_2\colon \X'_0\to \cE_0$ (in the category of schemes over $\X_0$) is a closed subscheme of $\X'_0$ (\cite[\href{https://stacks.math.columbia.edu/tag/01KM}{Tag 01KM}]{stacks-project}). In particular, for a morphism $\f\colon \widetilde{\Z}\to \X'_0$ of $\X_0$-schemes one has $s_1\f=s_2\f$ if and only if $\f$ factors through the closed immersion $\Z\hookrightarrow \X'_0$.
	
	We will show that $\Z$ defines a proper closed subset of $\X'_0$. Suppose, for a contradiction, that the underlying set of $\Z$ is all of $\X'_0$. Since $\X'_0$ is reduced, the only closed subscheme supported on all of $\X'_0$ is $\X'_0$ itself. So $\Z=\X'_0$ (as schemes). But then $s_1=s_2$; a contradiction.
	
	Thus $\Z$ defines a proper closed subset of $\X'_0$ and its complement $\U'\subseteq \X'_0$ is a nonempty Zariski open subset. For $x'\in \U'$, the morphism $\f_{x'}\colon \spec(k_0(x'))\to \X'_0$ does not factor through $\Z\to \X'_0$ and so $s\f_{x'}\neq s'\f_{x'}$, i.e., $\s_{x'}(s_1)\neq \s_{x'}(s_2)$.	
\end{proof}

The following lemma explains how Hilbert subsets enter into N\'{e}ron's specialization theorem.

\begin{lemma}
	\label{LM:solutionsunderspecialization}
	Let $n\geq 1$, $\tilde{\B}\subseteq K_0$ a finitely generated $\B_0$-algebra and $\gamma\in \cE_0(\tilde{\B})$ such that $\{\varepsilon\in \cE_0(K_0)\mid \varepsilon^n=\gamma\}\subseteq \cE_0(\tilde{\B})$. Then there exists a Hilbert subset $\tilde{H}$ of $\tilde{\X}=\spec(\tilde{\B})$ such that
	\begin{equation} \label{eq: torsion identity}
		\{\varepsilon\in \cE_0(k_0(\tilde{x})) \mid \varepsilon^n=\sigma_{\tilde{x}}(\gamma)\}=\sigma_{\tilde{x}}(\{\varepsilon\in \cE_0(K_0)\mid \varepsilon^n=\gamma\}).
	\end{equation}
	for all $\tilde{x}\in\tilde{H}$.
\end{lemma}
\begin{proof}
	Note that the inclusion ``$\supseteq$'' in (\ref{eq: torsion identity}) is trivial.
	Set $S=\{\varepsilon\in \cE_0(\overline{K_0})\mid \varepsilon^n=\gamma\}$.  Let $\B'_0\subseteq \overline{K_0}$ be a finitely generated $\tilde{\B}$-algebra such that $S\subseteq \cE_0(\B'_0)$. By Lemma~\ref{LM:distinctpointsunderspecialiation} there exists a nonempty Zariski open subset $\U'$ of $\X'_0=\spec(\B'_0)$ such that $\sigma_{x'}$ is injective on $S$ for all $x'\in \U'$. By Lemma \ref{lemma: assumption A1}, there exists a nonempty Zariski open subset $\mathcal{V}'$ of $\X'_0$ such that 
	\begin{equation} \label{eq: cardinal equality}
		|\{ \varepsilon\in \cE_0(\overline{k_0(x')})\mid \varepsilon^n=\sigma_{x'}(\gamma)\}|=|S|
	\end{equation}
	for all $x'\in \mathcal{V}'$. 
	Let $\f\colon \X'_0\to \tilde{\X}$ be the morphism corresponding to the inclusion $\tilde{\B}\subseteq \B'_0$. Since $\B'_0$ is algebraic over $\tilde{\B}$, there exists a nonempty Zariski open subset $\mathcal{W}'$ of $\X'_0$ 
	such that $k_0(x')$ is algebraic over $k_0(\f(x'))$ for all $x'\in \mathcal{W}'$. (This follows, for example, from Lemma \ref{lemma: good generic behaviour}.) By Chevalley's theorem, there exists a nonempty basic Zariski open subset $\tilde{\U}$ of $\tilde{\X}$ contained in $\f(\U'\cap \mathcal{V}'\cap \mathcal{W}')$.
	
	Let $\tilde{x}\in \tilde{\U}$. Then there exists an $x'\in \U'\cap \mathcal{V}'\cap \mathcal{W}'$ with $\f(x')=\tilde{x}$. We claim that
	\begin{equation} \label{eq: surjective on algebraic closure}
		\{\varepsilon\in \cE_0(\overline{k_0(\tilde{x})}) \mid \varepsilon^n=\sigma_{\tilde{x}}(\gamma)\}=\sigma_{x'}(S).
	\end{equation}
	As $\s_{x'}$ is injective on $S$, we find 
	$$
	|S|=|\sigma_{x'}(S)|\leq |\{\varepsilon\in \cE_0(k_0(x')) \mid \varepsilon^n=\sigma_{x'}(\gamma)\}| \leq |\{\varepsilon\in \cE_0(\overline{k_0(x')}) \mid \varepsilon^n=\sigma_{x'}(\gamma)\}|.
	$$
	From (\ref{eq: cardinal equality}), we thus deduce that $\sigma_{x'}(S)=\{\varepsilon\in \cE_0(\overline{k_0(x')}) \mid \varepsilon^n=\sigma_{x'}(\gamma)\}$. This implies (\ref{eq: surjective on algebraic closure}) because $\overline{k_0(x')}=\overline{k_0(\tilde{x})}$ since $k_0(x')$ is algebraic over $k_0(\tilde{x})=k_0(\f(x'))$.

	\medskip
	
	We first assume that $S\subseteq \cE_0(K_0)$. Then $S=\{\varepsilon\in\cE_0(K_0)|\ \varepsilon^n=\gamma\}\subseteq\cE_0(\tilde{\B})$ by assumption. 
	The commutativity of	
	$$
	\xymatrix{
		\cE_0(\tilde{\B}) \ar@{^{(}->}[r] \ar_-{\s_{\tilde{x}}}[d] & \cE_0(\B'_0) \ar^-{\s_{x'}}[d] \\
		\cE_0(k_0(\tilde{x})) \ar@{^{(}->}[r] & \cE_0(k_0(x'))
	}
	$$
	paired with (\ref{eq: surjective on algebraic closure}) then yields $\s_{\tilde{x}}(S)=\{\varepsilon\in \cE_0(\overline{k_0(\tilde{x})}) \mid \varepsilon^n=\sigma_{\tilde{x}}(\gamma)\}$. In particular, $\{\varepsilon\in \cE_0(k_0(\tilde{x})) \mid \varepsilon^n=\sigma_{\tilde{x}}(\gamma)\}=\s_{\tilde{x}}(S)$ as desired. 
	So, in case $S\subseteq \cE_0(K_0)$, we can choose as $\tilde{H}=\tilde{\U}$. 
	
	\medskip
	
	Now assume that $S\setminus \cE_0(K_0)\neq \emptyset$ and let $\varepsilon\in S\setminus \cE_0(K_0)$.
	Then $k_0(\varepsilon(\overline{\xi}))\hookrightarrow \overline{K_0}$ is a finite extension of $K_0$ with $[k_0(\varepsilon(\overline{\xi})):K_0]>1$, where we think of $\varepsilon$ as a morphism $\varepsilon\colon \spec(\overline{K_0})\to \cE_0$ and $\overline{\xi}$ is the unique point of $\spec(\overline{K_0})$. Let $\B_\varepsilon\subseteq k_0(\varepsilon(\overline{\xi}))$ be a finitely generated $\tilde{\B}$\=/algebra with $k_0(\varepsilon(\overline{\xi}))$ as field of fractions and such that $\varepsilon\in \cE_0(\B_\varepsilon)$. Furthermore, we may assume $\B_\varepsilon\subseteq \B'_0$. 
	
	Then Lemma~\ref{LM:residucefields}, applied with $\SSS=\cE_0$, $\B'_0=\B_\varepsilon$ and $s=\varepsilon\in\cE_0(\B_\varepsilon)$, yields a nonempty Zariski open subset $\mathcal{V}_\varepsilon$ of $\spec(\B_\varepsilon)$ such that
	$k_0(\varepsilon(x))\to k_0(x)$ is an isomorphism for every $x\in \mathcal{V}_\varepsilon$.
	By Chevalley's theorem, there exists a nonempty basic Zariski open subset $\mathcal{U}_\varepsilon$ of $\tilde{\X}=\spec(\tilde{\B})$ such that $\U_\varepsilon$ lies in the image of $\mathcal{V}_\varepsilon$ under $\spec(\B_\varepsilon)\to \spec(\tilde\B)$.
		
	Set $\tilde{H}_\varepsilon=\U_\varepsilon\cap H(\B_\varepsilon/\tilde{\B})$ and let $\tilde{x}\in \tilde{H}_\varepsilon$. Then $[k_0(x_\varepsilon):k_0(\tilde{x})]=[k_0(\varepsilon(\overline{\xi})):K_0]>1$, where $x_\varepsilon$ is the unique element of $\spec(\B_\varepsilon)$ lying above $\tilde{x}$. As $x_\varepsilon\in\mathcal{V}_\varepsilon$, the field extensions $k_0(\varepsilon(x_\varepsilon))$ and $k_0(x_\varepsilon)$ of $k_0(\tilde{x})$ are isomorphic. Therefore, $k_0(\varepsilon(x_\varepsilon))$ is a non-trivial field extension of $k_0(\tilde{x})$.
	
	We will show that $\tilde{H}=\tilde{\U}\cap (\cap_{\varepsilon\in S\setminus \cE_0(K_0)} \tilde{H}_\varepsilon)$ has the desired property. Let $\tilde{x}\in\tilde{H}$. 
	Then there exist an $x'\in \U'\cap\mathcal{V}'\cap\mathcal{W}'\subseteq \X'_0$ with $\f(x')=\tilde{x}$.
	Given (\ref{eq: surjective on algebraic closure}), to verify (\ref{eq: torsion identity}), it suffices to show that $\sigma_{x'}(\varepsilon)\notin \cE_0(k_0(\tilde{x}))$ for every $\varepsilon\in S\smallsetminus\cE_0(K_0)$.
	Note that $\sigma_{x'}(\varepsilon)\notin \cE_0(k_0(\tilde{x}))$ if and only if $k_0(\varepsilon(x'))$ is a non-trivial extension of $k_0(\tilde{x})$.
	
	For $\varepsilon\in S\setminus \cE_0(K_0)$, since $x_\varepsilon$ is the unique element of $\spec(\B_\varepsilon)$ lying above $\tilde{x}$ and $\B_\varepsilon\subseteq \B'_0$, we see that $x_\varepsilon$ is the image of $x'$ under $\spec(\B'_0)\to \spec(\B_\varepsilon)$. So $\varepsilon(x')=\varepsilon(x_\varepsilon)$ and $k_0(\varepsilon(x'))=k_0(\varepsilon(x_\varepsilon))$ is a non-trivial extension of $k_0(\tilde{x})$ as desired.	
\end{proof}
We are now prepared to prove our variant of N\'{e}ron's specialization theorem.

\begin{theo}
	\label{thm:specializationtheorem}
	Let $\Gamma$ be a finitely generated subgroup of $\cE_0(\B_0)$. Then there exists a Hilbert subset $H_0$ of $\X_0$ such that $\sigma_{x_0}\colon \cE_0(\B_0)\to \cE_0(k_0(x_0))$ is injective on $\Gamma$ for every $x_0\in H_0$.
\end{theo}
\begin{proof}
	Set $\tilde{\Gamma}=\{\varepsilon\in \cE_0(K_0) \mid \varepsilon^n\in \Gamma \mbox{ for some $n\geq 1$} \}$. Then $\tilde{\Gamma}$ is a subgroup of $\cE_0(K_0)$ and by assumption (A) it is finitely generated. Let $\tilde{\B}\subseteq K_0$ be a finitely generated $\B_0$-algebra such that $\tilde{\Gamma}\subseteq \cE_0(\tilde{\B})$. We shall show that there exists a Hilbert subset $\tilde{H}$ of $\tilde{\X}=\spec(\tilde{\B})$ such that the specialization map $\sigma_{\tilde{x}}\colon \cE_0(\tilde{\B})\to\cE(k_0(\tilde{x}))$ is injective on $\tilde{\Gamma}$ for every $\tilde{x}\in \tilde{H}$. Note that this implies the theorem since, by Lemma \ref{lemma: Hilbert set under projection}, there exists a Hilbert subset $H_0$ of $\X_0$ contained in the image of $\tilde{H}$ under $\tilde{\X}\to\X_0$ and if $\tilde{x}\in\tilde{H}$ maps to $x_0\in H_0$, then the diagram 
	$$
	\xymatrix{
		\cE_0(\B) \ar@{^{(}->}[r] \ar_-{\s_{x_0}}[d] & \cE_0(\tilde{\B}) \ar^-{\s_{\tilde{x}}}[d] \\
		\cE_0(k_0(x_0)) \ar@{^{(}->}[r] & \cE_0(k_0(\tilde{x}))
	}
	$$
	commutes.
	
	Let $n>1$ be an integer such that the order of the torsion subgroup of $\tilde{\Gamma}$ divides $n$. 
	Let $\{\gamma_1,\ldots,\gamma_m\}$ be a set of representatives of $\tilde{\Gamma}/\tilde{\Gamma}^n$. Without loss of generality we may assume that $\gamma_1=1$. For $i=1,\dots,m$ set $\Delta_i=\{ \varepsilon\in \tilde{\Gamma}\mid \varepsilon^n=\gamma_i\}$. Note that $\Delta_i=\{ \varepsilon\in \cE_0(K_0)\mid \varepsilon^n=\gamma_i\}$ by construction of $\tilde{\Gamma}$. 
	
	Lemma~\ref{LM:solutionsunderspecialization} applied with $\gamma=\gamma_i$, yields a Hilbert subset $\tilde{H}_i$ of $\tilde{\X}$ such that 
	$$
	\{\varepsilon\in \cE_0(k_0(\tilde{x})) \mid \varepsilon^n=\sigma_{\tilde{x}}(\gamma_i)\}=	\sigma_{\tilde{x}}(\Delta_i)
	$$
	for every $\tilde{x}\in \tilde{H}_i$.
	Set $\tilde{H}=\cap_{i=1}^m \tilde{H}_i$ and let $\tilde{x}\in \tilde{H}$.
	We will show that $\sigma_{\tilde{x}}$ is injective on $\tilde{\Gamma}$.
	It suffices to show that $\Gamma_0=\ker(\sigma_{\tilde{x}})\cap \tilde{\Gamma}$ is trivial. Note that $\Gamma_0$ is a finitely generated group because it is a subgroup of the finitely generated group $\tilde{\Gamma}$. 
	
	We claim that $\Gamma_0=\Gamma_0^n$. Let $\gamma\in \Gamma_0$ and write $\gamma=\gamma_i \tilde{\gamma}^n$ for some $i$ and $\tilde{\gamma}\in \tilde{\Gamma}$. Then $1=\sigma_{\tilde{x}}(\gamma)=\sigma_{\tilde{x}}(\gamma_i)\sigma_{\tilde{x}}(\tilde{\gamma})^n$ and so $\sigma_{\tilde{x}}(\tilde{\gamma}^{-1})^n=\sigma_{\tilde{x}}(\gamma_i)$. Because $\tilde{x}\in\tilde{H}_i$, we have $\s_{\tilde{x}}(\tilde{\gamma}^{-1})=\sigma_{\tilde{x}}(\varepsilon)$ for some $\varepsilon\in\Delta_i$, i.e., $\varepsilon^n=\gamma_i$. This implies that $\gamma_i=1$, i.e., $i=1$.
		
	Hence $\sigma_{\tilde{x}}(\tilde{\gamma}^{-1})=\sigma_{\tilde{x}}(\varepsilon)$ for $\varepsilon\in\tilde{\Gamma}$ and $\varepsilon^n=1$. So $\s_{\tilde{x}}(\varepsilon\tilde{\gamma})=1$ and $\varepsilon\tilde{\gamma}\in \Gamma_0$. Moreover, $$\gamma=\gamma_i\tilde{\gamma}^n=\tilde{\gamma}^n=\varepsilon^n\tilde{\gamma}^n=(\varepsilon\tilde{\gamma})^n\in\Gamma_0^n.$$
	Thus $\Gamma_0=\Gamma_0^n$ as claimed.
	As the order of the torsion subgroup of $\tilde{\Gamma}$ divides $n$, the group $\tilde{\Gamma}^n$ is torsion free. Therefore, also $\Gamma_0=\Gamma_0^n\subseteq\tilde{\Gamma}^n$ is torsion free.
	But a finitely generated torsion-free group is a free $\mathbb{Z}$-module and so can satisfy $\Gamma_0^n=\Gamma_0$ only if $\Gamma_0=1$ is trivial. Thus $\Gamma_0=1$ and $\sigma_{\tilde{x}}$ is injective on $\tilde{\Gamma}$ for every $\tilde{x}\in\tilde{\Gamma}$ as desired.
\end{proof}

The following corollary contains the special case of Theorem \ref{thm:specializationtheorem} relevant for our purpose.

\begin{cor} \label{cor: Nerons specialization theorem}
	Let $k_0$ be a finitely generated field extension of $\mathbb{Q}$, $\X_0=\spec(\B_0)$ an affine variety over $k_0$ and $\cE_0$ a group scheme over $\B_0$ of abelian type. Then there exists a Hilbert subset $H_0$ of $\X_0$ such that the specialization morphism $\s_{x_0}\colon \cE_0(\B_0)\to \cE_0(k_0(x_0))$ is injective for every $x_0\in H_0$.
\end{cor}
\begin{proof}
	Since $K_0$ is finitely generated over $\mathbb{Q}$, by the generalized Mordell-Weil theorem (\cite[Chapter 6, Thm. 1]{Lang:FundamentalsofDiophantineGeometry}) the group $\cE_0(K_0)$ is finitely generated. Therefore, assumption (A) is satisfied and also $\Gamma=\cE_0(\B_0)\leq \cE_0(K_0)$ is finitely generated. Thus Theorem~\ref{thm:specializationtheorem} applies.
\end{proof}

In addition to the case of group schemes of abelian type explained in the above corollary, Theorem~\ref{thm:specializationtheorem} also applies in the case where $k_0$ is a finitely generated field extension of $\mathbb{Q}$ and $\cE_0=\mathbb{G}_{m,\B_0}$ is the multiplicative group scheme over $\B_0$. In this case assumption 
assumption (A) is verified by \cite[Prop. 2.12]{Feng:DifferenceGaloisGrousUnderSpecialization}. 

Theorem \ref{thm:specializationtheorem} does not apply to the additive group scheme $\mathbb{G}_{a,\B_0}$ over $\B_0$ because assumption (A) is violated. For example, for $\Gamma=(\mathbb{Z},+)$, the group
in (A) equals $(\mathbb{Q},+)$.

\subsection{Denseness of ad$\times$Jac-open sets}
\label{subsec: Denseness of ad-open and ab-open sets}

In this section we combine the results of the previous sections to show that ad$\times$Jac-open subsets are Zariski dense.
%
%Throughout Section \ref{subsec: Denseness of ad-open and ab-open sets} we make the following assumptions:
%\begin{itemize}
%	\item $k$ is an algebraically closed field of characteristic zero;
%	\item $\B$ is a finitely generated $k$-algebra that is an integral domain and
%	\item $\X=\spec(\B)$.
%\end{itemize}
%We continue to work towards showing that ad$\times$Jac-open subsets are Zariski dense.
 The following proposition is an important step in this direction. It can be seen as a ``geometric'' version of Corollary \ref{cor: Nerons specialization theorem}.

\begin{prop} \label{prop: ab contains Hilbert}
	Let $\U$ be an ab-open subset of $\X(k)$ and $S\subseteq \B$ a finite subset. Then there exists a geometric $\B_0/k_0$-Hilbert set $\H\subseteq \X(k)$ such that $S\subseteq \B_0$ and $\H\subseteq \U$. In particular,  
	every ab-open subset of $\X(k)$ contains a geometric Hilbert set.
\end{prop}
\begin{proof}
	The ab-open set $\U$ is of the form $\U=\f(W_{\X'}(\cE,\Gamma))$,
	where $\X'=\spec(\B')$ with $\B'$ an integral domain containing $\B$, finitely generated and algebraic over $\B$, $\cE$ is a group scheme of abelian type over $\B'$, $\Gamma$ is a finitely generated subgroup of $\cE(\B')$ and $\f\colon\X'(k)\to\X(k)$ is the morphism induced by the inclusion $\B\subseteq \B'$.
	The idea of the proof is to descend the whole situation from $k$ to a finitely generated field extension of $\mathbb{Q}$ so that we can apply N\'{e}ron's specialization theorem.	
	We claim that we can find
	\begin{itemize}
		\item a nonzero $b\in \B$ such that every $\p\in D(b)$ lifts to a prime ideal of $\B'$;
		\item a subfield $k_0$ of $k$ that is finitely generated over $\mathbb{Q}$;
		\item a finitely generated $k_0$-subalgebra $\B_0$ of $\B$ containing $b$ and $S$ such that the canonical map $\B_0\otimes_{k_0}k\to \B$ is an isomorphism;
		\item a finitely generated $k_0$-subalgebra $\B'_0$ of $\B'$ containing $\B_0$ such that the canonical map $\B_0'\otimes_{k_0}k\to \B'$ is an isomorphism and such that $\B_0'$ is algebraic over $\B_0$;
		\item a group scheme $\cE_0$ over $\B'_0$ of abelian type such that the base change from $k_0$ to $k$ of $\cE_0/\B'_0$ equals $\cE/\B'$ and $\Gamma\subseteq \cE_0(\B'_0)$ (where we identify $\cE_0(\B_0')$ with a subset of $\cE(\B')$).
	\end{itemize}
	To see that we can find the above items we argue as follows. First, by Chevalley's theorem, we can find an appropriate $b\in \B$. 
	Let $\gamma_1,\ldots,\gamma_n$ be a finite generating set of $\Gamma$.
	Writing $\B'$ as the directed union of all of its finitely generated $\mathbb{Z}$-algebras and considering the morphisms $\gamma_1\colon \spec(\B')\to \cE,\ldots,\gamma_n\colon\spec(\B')\to \cE$ and the separated commutative group scheme $\cE$ over $\B'$, we can combine Theorem 8.8.2 and Theorem 8.10.5 (v) of \cite{Grothendieck:EGAIV3}, to obtain a finitely generated $\mathbb{Z}$-subalgebra $\tilde{\B}$ of $\B'$, a commutative separated group scheme $\tilde{\cE}$ of finite type over $\tilde{\B}$ and morphisms $\tilde{\gamma_i}\colon \spec(\tilde{\B})\to \tilde{\cE}$ $(i=1,\ldots,n)$ such that the base change of $\tilde{\cE}$ from $\tilde{\B}$ to $\B'$ is $\cE$ and the base change of $\tilde{\gamma_i}$ from $\tilde{\B}$ to $\B'$ is $\gamma_i$ for $i=1,\ldots,n$. So, if we identify $\tilde{\cE}(\tilde{\B})$ with a subset of $\cE(\B')$, then $\Gamma\subseteq  \tilde{\cE}(\tilde{\B})$.
	
	Next, choose a finite generating set $S'$ of $\B'$ as a $k$-algebra such that $S'$ contains a finite generating set of $\tilde{\B}$ as a $\mathbb{Z}$-algebra. Then choose a finite generating set $T$ of $\B$ as a $k$-algebra such that $b\in T$, $S\subseteq T$ and all elements of $S'$ satisfy a nonzero univariate (not necessarily monic) polynomial over $\mathbb{Z}[T]$. Set $T'=S'\cup T$. Then $T\subseteq T'$, $T$ generates $\B$ as a $k$-algebra, $T'$ generates $\B'$ as a $k$-algebra and every element of $T'$ is algebraic over $\mathbb{Z}[T]$. Next, let $k_0\subseteq k$ be a finitely generated field extension of $\mathbb{Q}$ such that the ideal of all $k$-algebraic relations among the elements of $T$ as well as the ideal of all $k$-algebraic relations among the elements of $T'$ is generated by polynomials with coefficients in $k_0$. Set $\B_0=k_0[T]$ and $\B_0'=k_0[T']$. Then $\B_0\subseteq \B'_0$ and $\B'_0$ is finitely generated and algebraic over $\B_0$. Moreover, $b\in \B_0$, $S\subseteq\B_0$, $\tilde{\B}\subseteq \B_0'$ and the morphisms $\B_0\otimes_{k_0}k\to \B$ and $\B_0'\otimes_{k_0}k\to \B'$ are isomorphism. 
	
	Let $\cE_0$ be the base change of $\tilde{\cE}$ from $\tilde{\B}$ to $\B_0'$.
	Then $\cE_0$ is a commutative separated group scheme of finite type over $\B_0'$ with $\Gamma\subseteq\cE_0(\B'_0)$ and the base change of $\cE_0$ from $\B_0'$ to $\B'$ is $\cE$. It follows that the base change of $\cE_0\times_{\spec(\B'_0)}\spec(k_0(\B'_0))$ from $k_0(\B'_0))$ to $k(\B')$ is $\cE\times_{\spec(\B')}k(\B')$. Since the latter is an abelian variety, also $\cE_0\times_{\spec(\B'_0)}\spec(k_0(\B'_0))$ must be an abelian variety. So $\cE_0$ is of abelian type. In summary, we have successfully descended everything to $k_0$: The base change of $\B'_0/\B_0$ from $k_0$ to $k$ is $\B'/\B$ and the base change of $\cE_0/\B'_0$ from $k_0$ to $k$ is $\cE/\B'$.
	
	Set $\X_0=\spec(\B_0)$ and $\X'_0=\spec(\B'_0)$ so that $\cE\to\X'\to\X$ is the base change from $k_0$ to $k$ of $\cE_0\to\X'_0\to\X_0$.
	
	By Neron's specialization theorem (Corollary \ref{cor: Nerons specialization theorem}), there exists a Hilbert subset $H_0'$ of $\X'_0$ such that the specialization map $\sigma_{x_0'}\colon \cE_0(\B_0')\to \cE_0(k_0(x_0'))$ is injective for every $x'_0\in H'_0$.
	By Lemma \ref{lemma: Hilbert set under projection} there exists a Hilbert subset $H_0$ of $\X_0$ such that the image of $H_0'$ in $\X_0$ contains $H_0$ and every element of $H_0$ has a unique lift to an element of $\X_0'$. Let $\H$ be the inverse image of $H_0\cap D(b)$ under $\X(k)\to \X\to \X_0$. Then $\H$ is a geometric $\B_0/k_0$-Hilbert set.
	
	We claim that the geometric Hilbert set $\H$ is contained in $\U$. Let $c\in \H$ and let $x,x_0$ denote the image of $c$ in $\X,\X_0$ respectively. By construction of $H_0$, there exists an $x_0'\in H_0'$ mapping to $x_0$.
	On the other hand, if $\p$ is the prime ideal of $\B$ corresponding to $x$, then $b\notin\p$ and so $x$ lifts to a point $x'$ of $\X'$ by the choice of $b$. The commutative diagram
	$$
	\xymatrix{
		\X' \ar[r] \ar[d] & \X \ar[d] \\
		\X_0' \ar[r] & \X_0
	}
	$$
	then shows that the image of $x'$ in $\X_0'$ equals $x_0'$ because $x_0'$ is the unique lift of $x_0$ to an element of $\X_0'$.
	We then have a commutative diagram
	$$
	\xymatrix{
		\cE_0(\B_0') \ar[r] \ar_-{\sigma_{x_0'}}[d] & \cE(\B') \ar^-{\sigma_{x'}}[d] \\
		\cE_0(k_0(x_0')) \ar[r] & \cE(k(x'))	
	}
	$$
	where the horizontal maps are injections. As $\sigma_{x_0'}$ is injective, it is in particular injective on $\Gamma$. Thus also $\sigma_{x'}=\sigma_c$ is injective on $\Gamma$. So $c\in W_{\X'}(\cE,\Gamma)=\U$ as desired. 
\end{proof}

We are now prepared to prove the main theorem of Section \ref{sec: adopen and abopen}.

\begin{theo} \label{theo: ad meets ab is dense}
	The intersection of an ad-open and an ab-open subset of $\X(k)$ is Zariski dense in $\X(k)$. In particular, an ad$\times$Jac-open subset of $\X(k)$ is Zariski dense in $\X(k)$.
\end{theo}
\begin{proof}
	Let $\U=W_\X(\Ga,\Gamma)$ be an ad-open subset of $\X(k)$ and let $\mathcal{V}$ be an ab-open subset of $\X(k)$. We first show that $\U\cap\mathcal{V}$ is nonempty.
		
	Choose a finite generating set $S\subseteq \B$ of the finitely generated subgroup $\Gamma$ of $(\B,+)$. By Proposition \ref{prop: ab contains Hilbert} there exists a geometric $\B_0/k_0$-Hilbert set $\H$ of $\X(k)$ with $S\subseteq \B_0$ and $\H\subseteq \mathcal{V}$. Now Lemma \ref{lemma: containmemt for Hilber sets} yields an $\boldsymbol{\eta}\in \B_0^m$ with $\B_0=k_0[\boldsymbol{\eta}]$ and appropriate $\boldsymbol{p}$, $b$ such that
	$\H_{k_0,\X(k)}^{\boldsymbol{\eta}}((1,\ldots,1),\boldsymbol{p},b)\subseteq \H.$ In particular, $\H_{k_0,\X(k)}^{\boldsymbol{\eta}}((1,\ldots,1),\boldsymbol{p},b)\subseteq \mathcal{V}.$
	
	On the other hand, $\Gamma\subseteq \B_0$ as $S\subseteq \B_0$ and so by Proposition \ref{prop: Hilbert set in ad-open}, there exists a $\boldsymbol{d}$ such that $\H^{\boldsymbol{\eta}}_{k_0,\X(k)}(\boldsymbol{d},\emptyset,1)\subseteq \U$. 
	Thus 
	$$\H_{k_0,\X(k)}^{\boldsymbol{\eta}}(\boldsymbol{d},\boldsymbol{p},b)=\H^{\boldsymbol{\eta}}_{k_0,\X(k)}(\boldsymbol{d},\emptyset,1)\cap \H_{k_0,\X(k)}^{\boldsymbol{\eta}}((1,\ldots,1),\boldsymbol{p},b)\subseteq \U\cap\mathcal{V}. $$
	Since $\H_{k_0,\X(k)}^{\boldsymbol{\eta}}(\boldsymbol{d},\boldsymbol{p},b)\neq\emptyset$ by \cite[Prop. 2.8]{Feng:DifferenceGaloisGrousUnderSpecialization}, we can conclude that $\U\cap\mathcal{V}\neq\emptyset$.
	
	\medskip
	
	We next show that $\U\cap \mathcal{V}$ is Zariski dense in $\X(k)$. Suppose, for a contradiction, that $\U\cap \mathcal{V}$ is not Zariski dense. Then there exists a nonempty Zariski open subset $\U'$ of $\X(k)$ such that $\U'\cap\U\cap \mathcal{V}=\emptyset$. So, if $\U''$ is any ad-open subset of $\X(k)$ contained in $\U'\cap\U$, then $\U''\cap \mathcal{V}=\emptyset$. This contradicts the first part of the proof.
\end{proof}

We can do slightly better:
\begin{rem} \label{rem: finite intersection of ad and ab is dense}
	The intersection of finitely many ad-open and finitely many ab-open subsets of $\X(k)$ is Zariski dense in $\X(k)$.
\end{rem}
\begin{proof}
	The intersection of finitely many ad-open subsets contains an ad-open subset (Remark \ref{rem: intersection of ad-opens}) and the intersection of finitely many ab-open subsets is an ab-open subset (argue as in (\ref{eq: intersection of ads})). Thus the remark follows from Theorem \ref{theo: ad meets ab is dense}.
\end{proof}

\section{Some topics in differential Galois theory}
\label{sec: Some topics in differential Galois theory} 

This section is a collection of definitions, results and constructions from differential Galois theory, mostly related to the computation of the differential Galois group via Hrushovski's algorithm. In Section \ref{sec: Specialization of differential torsors}, in the course of the proof of our main specialization theorem (Theorem \ref{theo: main specialization}), we will then show that many of the properties and constructions discussed here are preserved on an ad$\times$Jac-open subset of the parameter space.

Throughout Section \ref{sec: Some topics in differential Galois theory} we assume that
\begin{itemize}
\item $(F,\delta)$ is a differential field with field of constants 
%\rf{Do we need to add $F^\delta$ here ?}
%\mw{I am not sure what you mean. Would you prefer if it was not added in the next line?}
\item $k=F^\de$ algebraically closed and of characteristic zero.
\end{itemize}

In the later subsections we will specialize to the case $(F,\de)=(k(x),\frac{d}{dx})$ or $F$ a finite field extension of $k(x)$. The reader is invited to recall our notational conventions from the end of the introduction.

 \subsection{Basics of differential Galois theory} \label{subsec:differential Galois theory}
 
 We begin by recalling the basic definitions and results from the Galois theory of linear differential equations. Proofs and more background can be found in any of the introductory textbooks \cite{Magid:LecturesOnDifferentialGaloisTheory}, \cite{SingerPut:differential}, \cite{CrespoHajto:AlgebraicGroupsAndDifferentialGaloisTheory}, \cite{Sauloy:DifferentialGaloisTheoryThroughRiemannHilbertCorrespondence}.
 
% 
% Let $(F,\de)$ be a differential field with field of constants $k=F^\de=\{a\in F|\ \de(a)=0\}$.
 Recall that a differential ring $R$ is \emph{$\de$-simple} if it is not the zero ring and the only $\de$-ideals of $R$ are $R$ and the zero ideal.
 
 \begin{defi} \label{defi: PV ring}
 	A \emph{Picard-Vessiot ring} for $\de(y)=Ay$, ($A\in F^{n\times n}$) is a $\de$-simple $F$\=/$\de$\=/algebra $R$ such that there exists a $Y\in \Gl_n(R)$ with $\de(Y)=AY$ and $R=F[Y,\frac{1}{\det(Y)}]$.
 \end{defi}
 
 For any given $A\in F^{n\times n}$, there exist a Picard-Vessiot ring $R/F$ for $\de(y)=Ay$ and it is unique up to an $F$-$\de$-isomorphism. Moreover, $R^\de=k$. The \emph{differential Galois group} $G=G(R/F)$ of $R/F$ (or of $\de(y)=Ay$) is the functor from the category of $k$-algebras to the category of groups given by
 $G(T)=\Aut(R\otimes_k T/F\otimes_k T)$, where $T$ is considered as a constant differential ring. It is a linear algebraic group (over $k$). 
 Its coordinate ring $k[G]$ can explicitly be described as $k[G]=(R\otimes_F R)^\de=k[Z,\frac{1}{\det(Z)}]$, with $Z=Y^{-1}\otimes Y\in \Gl_n((R\otimes_K R)^\de)$ and $Y$ as in Definition \ref{defi: PV ring}. The canonical map
 $$R\otimes_k k[G]\to R\otimes_F R$$
 is an isomorphism and so we may make the identification $R\otimes_F R=R\otimes_k k[G]$. For a $k$\=/algebra $T$, any $g\in G(T)$, i.e., any $K\otimes_k T$-$\de$-isomorphism $g\colon R\otimes_k T\to R\otimes_k T$, is determined by $g(Y\otimes 1)$. Moreover, $g(Y\otimes 1)=Y\otimes [g]$ for a unique $[g]\in \Gl_n(T)$. The assignment $g\mapsto [g]$ determines a morphism $G\to \Gl_{n,k}$ of algebraic groups over $k$ which is a closed embedding. The dual of this morphism is the map $k[\Gl_n]=k[X,\frac{1}{\det(X)}]\to k[G],\ X\mapsto Z$.
  
 The (functorial) action of $G$ on $R/F$ is determined by the map $\rho\colon R\to R\otimes_k k[G]$, which, under the identification $R\otimes_F R=R\otimes_k k[G]$ corresponds to the inclusion into the second factor. Explicitly, we have $\rho(Y)=Y\otimes Z$. For $g\in G(T)=\Hom(k[G],T)$, the corresponding automorphism $g\colon R\otimes_k T\to R\otimes_k T$ is the $T$-linear extension of $R\xrightarrow{\rho} R\otimes_k k[G] \to R\otimes_k T$.

 The following lemma shows that Picard-Vessiot rings and differential Galois group are well-behaved under extension of the constants. For a proof see, e.g. \cite[Lemma~5.2]{BachmayrHarbaterHartmannWibmer:TheDifferentialGaloisGroupOfRationalFunctionField}.
 
 \begin{lemma} \label{lemma: base change of PV ring over constants}
 	Let $k\subseteq k'$ be an inclusion of algebraically closed fields and let $R/k(x)$ be a Picard-Vessiot ring with differential Galois group $G$. Then $R\otimes_{k(x)}k'(x)/k'(x)$ is a Picard-Vessiot ring with differential Galois group $G_{k'}$. \qed
 \end{lemma}

\subsection{Differential torsors} \label{subsec: differential torsors}

 Differential torsors were introduced in \cite{BachmayrHarbaterHartmannWibmer:DifferentialEmbeddingProblems} for the purpose of solving differential embedding problems. To study the behaviour of Picard-Vessiot rings and differential embedding problems under specialization, we will need a version of differential torsors that works over an arbitrary base and not just a differential field, as in \cite{BachmayrHarbaterHartmannWibmer:DifferentialEmbeddingProblems}.
 
 As in the introduction, consider an inclusion $k\subseteq k'$ of algebraically closed fields of characteristic zero and a differential equation $\de(y)=Ay$ over $k'(x)$ with differential Galois group $G$ and Picard-Vessiot ring $R/k'(x)$. To discuss specializations of $G$ and $R$ down to $k$ one first needs to spread out $G$ and $R$ into families, i.e., we would like to have a finitely generated $k$-subalgebra $\B$ of $k'$, an affine group scheme $\G$ and a differential ring $\R$ over $\B$ such that the base change from $\B$ to $k'$ yields $G$ and $R$. It is then natural to also ask for an action of $\G$ on $\R$ compatible with the action of $G$ on $R$. Moreover, since $R$ defines a $G$-torsor, also $\R$ should define a $\G$-torsor. This idea is formalized through the notion of differential torsor.

 \medskip
 
Let $\mathcal{B}$ be a $k$-algebra (considered as a constant $\de$-ring) and let $\mathcal{Q}$ be a $\mathcal{B}$-$\de$-algebra.
 The two most relevant cases for us are, firstly $\B=k$ and $\Q=k(x)$ and, secondly $\Q=\B[x]_f$, where $f\in \B[x]$ is a monic polynomial.

 For a $\mathcal{Q}$-$\de$-algebra $\mathcal{R}$, we define $\Autt(\mathcal{R}/\mathcal{Q})$ to be the functor from the category of $\mathcal{B}$\=/algebras to the category of groups given by $\Autt(\mathcal{R}/\mathcal{Q})(\mathcal{T})=\Aut(\mathcal{R}\otimes_{\mathcal{B}}\mathcal{T}/\mathcal{Q}\otimes_{\mathcal{B}}\mathcal{T})$ for any $\mathcal{B}$-algebra $\mathcal{T}$. Here $\mathcal{T}$ is considered as a constant $\de$-ring and the automorphisms are understood to be differential automorphisms. An element $g$ of $\Autt(\mathcal{R}/\mathcal{Q})(\mathcal{T})$ is thus a $\mathcal{Q}\otimes_\mathcal{B}\mathcal{T}$\=/$\de$\=/isomorphism $g\colon \mathcal{R}\otimes_\mathcal{B}\mathcal{T}\to \mathcal{R}\otimes_\mathcal{B}\mathcal{T}$. On morphisms $\Autt(\mathcal{R}/\mathcal{Q})$ is given by base change.
 
 Let $\G$ be an affine group scheme over $\mathcal{B}$. An \emph{action of $\G$ on $\mathcal{R}/\mathcal{Q}$} is a morphism $\G\to \Autt(\mathcal{R}/\mathcal{Q})$ of group functors. In particular, for a $\mathcal{B}$-algebra $\mathcal{T}$, every $g\in \G(\mathcal{T})$ defines a $\mathcal{Q}\otimes_\mathcal{B}\mathcal{T}$\=/$\de$\=/automorphism $g\colon \mathcal{R}\otimes_\mathcal{B}\mathcal{T}\to \mathcal{R}\otimes_\mathcal{B}\mathcal{T}$.

 \begin{lemma} \label{lemma: equivalence with coaction}
 	To specify an action of $\G$ on $\R/\mathcal{Q}$ is equivalent to specifying a morphism $\rho\colon\R\to\R\otimes_\B \B[\G]$ of $\Q$-$\de$-algebras such that the diagrams
 	$$
 	\xymatrix{
 		\R \ar^-\rho[r] \ar_-\rho[d] & \R\otimes_\B\B[\G] \ar^-{\rho\otimes \B[\G]}[d] \\
 		\R\otimes_\B\B[\G] \ar^-{\R\otimes\Delta}[r] & \R\otimes_\B\B[\G]\otimes_\B \B[\G]		
 	}
 	$$
 	and
 	$$
 	\xymatrix{
 		\R \ar^-\rho[rr] \ar_-{\id}[rd] & & \R\otimes_\B \B[\G] \ar^-{\id\cdot\varepsilon}[ld] \\
 		& \R	&
 	}
 	$$
 	commute, where $\Delta\colon \B[\G]\to \B[\G]\otimes_\B\B[\G]$ is the comultiplication and $\varepsilon\colon \B[\G]\to\B$ the counit of the Hopf algebra $\B[\G]$ over $\B$.	
 \end{lemma}
 \begin{proof}
 	As this is a fairly standard argument, we only sketch the proof. First assume that an action of $\G$ on $\R/\Q$ is given. 
	%\rf{``on" is changed into ``an".} 
	For $\T=\B[\G]$ and $g=\id\in\G(\T)=\Hom(\B[\G],\B[\G])$ we thus have a $\Q\otimes_\B \B[\G]$-$\de$-automorphism $g\colon \R\otimes_\B\B[\G]\to \R\otimes_\B\B[\G]$. Then we can define $\rho$ as the composition $\rho\colon \R\to \R\otimes_\B \B[\G]\xrightarrow{g}\R\otimes_\B \B[\G]$.
 	
 	Conversely, given $\rho$, we define an action of $\G$ on $\R/\Q$ as follows. For a $\B$-algebra $\T$ and $g\in \G(\T)=\Hom(\B[\G],\T)$ we define $g\colon \R\otimes_\B\T\to \R\otimes_\B\T$ as the $\T$-linear extension of $\R\xrightarrow{\rho}\R\otimes_\B\B[\G]\xrightarrow{\R\otimes g} \R\otimes_\B \T$. The first diagram corresponds to the associativity of this action while the second diagram shows that the identity acts trivial.
 \end{proof}
 
 As above, let $\G$ act on $\R/\Q$ and let $\rho\colon \R\to \R\otimes_\B \B[\G]$ denote the \emph{coaction} (as in Lemma \ref{lemma: equivalence with coaction}). We call $\R/\Q$ a \emph{differential $\G$-torsor} if $\R$ is not the zero ring and the map $\R\otimes_\Q\R\to \R\otimes_\B\B[\G],\ a\otimes b \mapsto (a\otimes 1)\cdot \rho(b)$ is an isomorphism. 
 The geometric interpretation of this condition is the usual torsor condition, i.e., for $\mathcal{Z}=\spec(\R)$, the morphism
 $$\Z\times_\Q \G_\Q\to \Z\times_\Q\Z,\ (z,g)\mapsto (z, zg)$$
 of affine schemes over $\Q$ is an isomorphism.

 \begin{defi} \label{defi: differential Gtorsor for de(y)=Ay}
 	Let $\mathcal{A}\in\Q^{n\times n}$. A differential $\G$-torsor $\R/\Q$ is a \emph{differential $\G$-torsor for $\de(y)=\mathcal{A}y$} if there exists a $\Y\in\Gl_n(\R)$ such that
 	\begin{enumerate}
 		\item $\de(\Y)=\mathcal{A}\Y$,
 		\item $\R=\Q[\Y,\frac{1}{\det(\Y)}]$ and
 		\item for every $\B$-algebra $\T$ and every $g\in \G(\T)$ there exists a $[g]\in \Gl_n(\T)$ such that $g(\Y\otimes 1)=(\Y\otimes 1)(1\otimes [g])$.
 	\end{enumerate}
 	A matrix $\Y\in\Gl_n(\R)$ such that the above three conditions are satisfied is called a \emph{fundamental matrix}.
 \end{defi}
 
Assuming Definition \ref{defi: differential Gtorsor for de(y)=Ay}, the assignment $g\mapsto [g]$ defines a morphism $\G\to \Gl_{n,\B}$ of group schemes over $\B$. 
  
 \begin{ex}
 	The most basic example of the above definitions is when $\B=k$, $\Q=F$ is a differential field with $F^\de=k$ and $\mathcal{R}=R$ is a Picard-Vessiot ring over $F$ for $\de(y)=Ay$ ($A\in F^{n\times n}$) with differential Galois group $\G=G$. As explained in Section \ref{subsec:differential Galois theory}, $G$ acts on $R/F$ and $R/F$ is a differential $G$-torsor for $\de(y)=Ay$. Note that if $\Y=Y\in\Gl_n(R)$ is such that $\de(Y)=AY$ and $R=K[Y,\frac{1}{\det{Y}}]$, then condition (iii) of Definition \ref{defi: differential Gtorsor for de(y)=Ay} is automatically satisfied (because $(R\otimes_k T)^\de=T$ for any $k$-algebra $T$).
 \end{ex}

 In the following example we construct a differential torsor for the group scheme of $2\times 2$ monomial matrices.
 
 \begin{ex} \label{ex: torsor for monomial matrices}
 	Let $\B\subseteq \Q\subseteq \Q[\eta]$ be an inclusion of differential $k$-algebras such that $\B$ is constant and $\Q[\eta]$ is an integral domain. We assume that $\eta$ has minimal polynomial $y^2-a$ over the field of fractions of $\Q$, where $a$ is some element of $\Q^\times$. We also fix $b\in \Q$. Let $y_1,y_2$ be indeterminates over $\Q[\eta]$ and define a derivation $\de$ on $\R=\Q[\eta][y_1,y_2,y^{-1}_1,y^{-1}_2]$ by $\de(y_1)=(b+\eta)y_1,\ \de(y_2)=(b-\eta)y_2$.
 	Set
 	$$\Y=\begin{pmatrix}
 		y_1 & y_2 \\
 		\de(y_1) & \de(y_2)
 	\end{pmatrix}=\begin{pmatrix}
 		y_1 & y_2 \\
 		(b+\eta)y_1 & (b-\eta)y_2
 	\end{pmatrix}.
 	$$ 
 	Then $\det(\Y)=-2\eta y_1y_2$. Therefore $\R=\Q[\Y,\frac{1}{\det(\Y)}]$.
 	A direct computation shows that $\de(\Y)=\cA\Y$, where 
 	$$\cA=\begin{pmatrix}  0 & 1 \\ a+\de(b)-b^2-\frac{\delta(a)b}{2a} & 2b+\frac{\delta(a)}{2a}\end{pmatrix}\in\Q^{2\times 2}.$$
 We would like to describe $\Z=\spec(\R)$ explicitly as a closed subscheme of $\Gl_{2,\Q}$. To this end, let $\mathcal{I}$ denote the kernel of the morphism $\Q[X,\frac{1}{\det(X)}]\to \R,\ X\mapsto \Y$. As $\R$ is an integral domain, $\mathcal{I}$ is a prime differential ideal of $\Q[X,\frac{1}{\det(X)}]$, where $\de(X)=\cA X$. We have 
 $$(b+\eta)y_1\cdot(b-\eta)y_2=(b^2-a)y_1y_2 \quad \text{ and } \quad (b+\eta)y_1\cdot y_2+(b-\eta)y_2\cdot y_1=2by_1y_2. $$ 
 Therefore, the polynomials
 $$
 p_1=X_{21}X_{22}-(b^2-a)X_{11}X_{12} \quad \text{ and } \quad  p_2=X_{21}X_{12}+X_{22}X_{11}-2bX_{11}X_{12} $$
 lie in $\mathcal{I}$. We will show that $\mathcal{I}=(p_1,p_2)$. For this, it suffices to show that the induced map $\Q[X,\frac{1}{\det(X)}]/(p_1,p_2)\to \R$ is injective. 
 Let $z_1, z_2, z_1', z_2'$ denote the image in $\Q[X,\frac{1}{\det(X)}]/(p_1,p_2)$ of $X_{11},X_{12},X_{21},X_{22}$ respectively. Then
 \begin{equation} \label{eq: for ex 1}
 	z_1'z_2'=(b^2-a)z_1z_2,
 \end{equation}
 and
 \begin{equation} \label{eq: for ex 2}
 	z_1'z_2+z_2'z_1=2bz_1z_2.
 \end{equation}
 Subtracting $2z_1'z_2$ from (\ref{eq: for ex 2}) yields $z_1z_2'-z_1'z_2=2z_2(bz_1-z_1')$. Since $z_1z_2'-z_1'z_2$ is invertible, we see that also $z_2$ is invertible. Similarly, subtracting $2z_2'z_1$ from (\ref{eq: for ex 2}), we see that $z_1$ is invertible. So (\ref{eq: for ex 2}) can be rewritten as 
 \begin{equation}
 	\label{eq: for ex 3}
 	\tfrac{z'_1}{z_1}+\tfrac{z_2'}{z_2}=2b
 \end{equation} whereas (\ref{eq: for ex 1}) becomes $\frac{z'_1}{z_1}\cdot\frac{z_2'}{z_2}=b^2-a$. Plugging $\frac{z'_2}{z_2}=2b-\frac{z_1'}{z_1}$ into the latter equation yields $(\frac{z'_1}{z_1}-b)^2=a$. Set $\eta'=\frac{z'_1}{z_1}-b$. Then $\eta'^2=a$,  $z_1'=(b+\eta')z_1$ and from (\ref{eq: for ex 3}) we obtain $z_2'=(b-\eta')z_2$.
 
 Set $\mathcal{S}=\Q[\eta',z_1,z_2,z_1^{-1},z_2^{-1}]$. We claim that $\mathcal{S}=\Q[X,\frac{1}{\det(X)}]/(p_1,p_2)$. It suffices to show that $z_1',z_2'$ and $\frac{1}{z_1z_2'-z_1'z_2}$ lie in $\mathcal{S}$. But $z_1'=(b+\eta')z_1\in\mathcal{S}$, $z_2'=(b-\eta')z_2\in\mathcal{S}$ and 
 $z_1z_2'-z_1'z_2=2\eta'z_1z_2$. So $\frac{1}{z_1z_2'-z_1'z_2}\in \SSS$ and $\mathcal{S}=\Q[X,\frac{1}{\det(X)}]/(p_1,p_2)$ as claimed.
 
 The map
 \begin{equation} \label{eq: map}
 	\Q[\eta'][z_1,z_2,z_1^{-1},z_2^{-1}]=\Q[X,\tfrac{1}{\det(X)}]/(p_1,p_2)\longrightarrow \R=\Q[\eta][y_1,y_2,y^{-1}_1,y^{-1}_2]
 \end{equation}
 sends $\eta'$ to $\eta$, $z_1$ to $y_1$ and $z_2$ to $y_2$. 
 As $\eta'^2=a$, the restriction $\Q[\eta']\to \Q[\eta]$ is injective. But then, since $y_1$ and $y_2$ are algebraically independent over $\Q[\eta]$, the map (\ref{eq: map}) is injective. Thus $\mathcal{I}=(p_1,p_2)$ as claimed. In other words, $\Z$ is the closed subscheme of $\Gl_{2,\Q}$ defined by $p_1$ and $p_2$.
 	
 	Let $\G$ be the group scheme of monomial $2\times 2$ matrices over $\B$, i.e., 
 	\begin{equation} \label{eq: group of monomial matrices}
 		\G(\T)=\left\{\begin{pmatrix} g_{11} & g_{12} \\ g_{21} & g_{22} \end{pmatrix}\in\Gl_2(\T) \mid g_{11}g_{12}=g_{21}g_{22}=g_{11}g_{21}=g_{12}g_{22}=0\right\}\leq\Gl_2(\T)
 	\end{equation}
 	for any $\B$-algebra $\T$. In particular, if $\T$ is an integral domain then
 	$$
 	\G(\T)=\left\{\begin{pmatrix} s & 0 \\ 0 & t \end{pmatrix}\mid s,t\in \T^\times\right\}\cup \left\{\begin{pmatrix} 0 & s \\ t & 0  \end{pmatrix}\mid s,t\in \T^\times\right\}\leq\Gl_2(\T).
 	$$ 
 	Note that the equations in (\ref{eq: group of monomial matrices}) are redundant, in fact,
 	$$\G(\T)=\left\{\begin{pmatrix} g_{11} & g_{12} \\ g_{21} & g_{22} \end{pmatrix}\in\Gl_2(\T) \mid g_{11}g_{12}=g_{21}g_{22}=0\right\}$$
 	for any $\B$-algebra $\T$. For example, if $g_{11}g_{12}=g_{21}g_{22}=0$, then $(g_{11}g_{22}-g_{12}g_{21})g_{11}g_{21}=0$ and so also $g_{11}g_{21}=0$ since $g_{11}g_{22}-g_{12}g_{21}$ is invertible. 
 	 	
 	We will define an action of $\G$ on $\R/\Q$. Roughly, the idea is that a diagonal matrix $\begin{pmatrix} s & 0 \\ 0 & t \end{pmatrix}$ acts by $y_1\mapsto sy_1$, $y_2\mapsto ty_2$ and fixing $\eta$, whereas the permutation matrix % $\begin{pmatrix} s & 0 \\ 0 & t \end{pmatrix}$ 
 	%\rf{Here, should the permutation matrix be} 
 	$\begin{pmatrix} 0 & 1 \\ 1 & 0 \end{pmatrix}$
 	acts by interchanging $y_1$ with $y_2$ and $\eta$ with $-\eta$. In general, for $\T$ a $\B$-algebra, $g=\begin{pmatrix} g_{11} & g_{12} \\ g_{21} & g_{22} \end{pmatrix}\in\G(\T)$ acts on $\R\otimes_\B\T$ by $g(\Y)=\Y\otimes g$. To make sure that this is well-defined, we check that 
 	for any $\Q$-algebra $\SSS$, $Z=\begin{pmatrix} Z_{11} & Z_{12} \\ Z_{21} & Z_{22} \end{pmatrix}\in \Z(\SSS)$ and $g=\begin{pmatrix} g_{11} & g_{12} \\ g_{21} & g_{22} \end{pmatrix}\in\G(\SSS)$, we have $Zg=\begin{pmatrix} Z_{11}g_{11}+Z_{12}g_{21} & Z_{11}g_{12}+Z_{12}g_{22} \\ Z_{21}g_{11}+Z_{22}g_{21} & Z_{21}g_{12}+Z_{22}g_{22} \end{pmatrix}\in \Z(\SSS)$.
 	Using (\ref{eq: group of monomial matrices}), we find
 	\begin{align*}
 		p_1(Zg)&=(Z_{21}g_{11}+Z_{22}g_{21})(Z_{21}g_{12}+Z_{22}g_{22})-(b^2-a)(Z_{11}g_{11}+Z_{12}g_{21})(Z_{11}g_{12}+Z_{12}g_{22})\\
 		&=Z_{21}Z_{22}g_{11}g_{22}+Z_{22}Z_{21}g_{21}g_{12}-(b^2-a)(Z_{11}Z_{12}g_{11}g_{22}+Z_{12}Z_{11}g_{21}g_{12})\\
 		&=p_1(Z)g_{11}g_{22}+p_1(Z)g_{21}g_{12}=0
 	\end{align*}
 	and 
 	\begin{align*}
 		p_2(Zg)&=(Z_{21}g_{11}+Z_{22}g_{21})(Z_{11}g_{12}+Z_{12}g_{22})+( Z_{21}g_{12}+Z_{22}g_{22})(Z_{11}g_{11}+Z_{12}g_{21})-\\ 
 		& \qquad \qquad \qquad \qquad \qquad \qquad \qquad \qquad \qquad 2b(Z_{11}g_{11}+Z_{12}g_{21})( Z_{11}g_{12}+Z_{12}g_{22})\\
 		&=Z_{21}Z_{12}g_{11}g_{22}+Z_{22}Z_{11}g_{21}g_{12}+Z_{21}Z_{12}g_{12}g_{21}+Z_{22}Z_{11}g_{11}g_{22}-\\  & \qquad \qquad \qquad \qquad \qquad \qquad \qquad \qquad \qquad2b(Z_{11}Z_{12}g_{11}g_{22}+Z_{12}Z_{11}g_{21}g_{12})\\
 		&=p_2(Z)g_{11}g_{22}+p_2(Z)g_{21}g_{12}=0.
 	\end{align*}
 	To see that $\R/\Q$ is indeed a differential $\G$-torsor, it suffices to check that for $Z,Z'\in \Z(\SSS)$, the matrix $Z^{-1}Z'$ lies in $\G(\SSS)$.
  	As 	$$Z^{-1}Z'=\frac{1}{\det(Z)}\begin{pmatrix} Z_{22} & -Z_{12} \\
 		-Z_{21} & Z_{11}\end{pmatrix}\begin{pmatrix} Z'_{11} & Z'_{12} \\
 		Z'_{21} & Z'_{22}\end{pmatrix}=\frac{1}{\det(Z)}\begin{pmatrix} Z_{22}Z'_{11}-Z_{12}Z'_{21}  & Z_{22}Z'_{12}-Z_{12}Z'_{22} \\
 		Z_{11}Z'_{21}-Z_{21}Z'_{11} & Z_{11}Z'_{22}-Z_{21}Z'_{12}\end{pmatrix},$$
 	it suffices to show that 
 	\begin{equation} \label{eq: torsor 1}
 		(Z_{22}Z'_{11}-Z_{12}Z'_{21})(Z_{22}Z'_{12}-Z_{12}Z'_{22})=0
 	\end{equation} and 
 	\begin{equation} \label{eq: torsor 2}
 		(Z_{11}Z'_{21}-Z_{21}Z'_{11})(Z_{11}Z'_{22}-Z_{21}Z'_{12})=0.
 	\end{equation}
 	To verify these identities, it is helpful to rewrite $p_1(Z)=0$ and $p_2(Z)=0$ as
 	\begin{equation} \label{eq: for fractions}
 		\tfrac{Z_{21}}{Z_{11}}\cdot \tfrac{Z_{22}}{Z_{12}}=b^2-a\quad \text{ and } \tfrac{Z_{21}}{Z_{11}}+ \tfrac{Z_{22}}{Z_{12}}=2b.
 	\end{equation}
 	(Note that subtracting $2Z_{21}Z_{12}$ from $p_2(Z)=0$ and using the invertibility of $\det(Z)$ shows that $Z_{12}$ is invertible. By a similar argument, $Z_{11}$ is invertible.) From (\ref{eq: for fractions}) we deduce $(\frac{Z_{22}}{Z_{12}}-b)^2=a$, which leads to
 	\begin{align*}
 		0&=(\tfrac{Z_{22}}{Z_{12}})^2-2b\tfrac{Z_{22}}{Z_{12}}+(b^2-a)=(\tfrac{Z_{22}}{Z_{12}})^2-\left(\tfrac{Z'_{21}}{Z'_{11}}+ \tfrac{Z'_{22}}{Z'_{12}}\right)\tfrac{Z_{22}}{Z_{12}}-\tfrac{Z'_{21}}{Z'_{11}}\cdot \tfrac{Z'_{22}}{Z'_{12}}=\\
 		&=\left(\tfrac{Z_{22}}{Z_{12}}-\tfrac{Z'_{21}}{Z'_{11}}\right)\left(\tfrac{Z_{22}}{Z_{12}}-\tfrac{Z'_{22}}{Z'_{12}}\right). 
 	\end{align*}
 	Clearing denominators, we find (\ref{eq: torsor 1}). A similar computation verifies (\ref{eq: torsor 2}). In summary, we see that $\R/\Q$ is a differential $\G$-torsor for $\de(y)=\cA y$ with fundamental matrix $\Y$. 
 \end{ex}

 In the next lemma we collect three basic observations that will be used repeatedly in what follows.
 
 \begin{lemma} \label{lemma: actions}
 	Let $\G$ act on $\mathcal{R}/\mathcal{Q}$.
 	\begin{enumerate}
 		\item If $\mathcal{B}\to\mathcal{B'}$ is a morphism of $k$-algebras, then $\G_{\mathcal{B}'}$ acts on $\mathcal{R}\otimes_{\mathcal{B}}\mathcal{B'}/\mathcal{Q}\otimes_{\mathcal{B}}\mathcal{B'}$.
 		\item If $\mathcal{S}$ is a multiplicatively closed subset of $\mathcal{Q}$, then the action of $\G$ on $\mathcal{R}/\mathcal{Q}$ extends canonically to an action of $\G$ on $\mathcal{S}^{-1}\mathcal{R}/\mathcal{S}^{-1}\mathcal{Q}$.
 	\item If $\Q\to\Q'$ is a morphism of $\de$-rings, then $\G$ acts on $\R\otimes_\Q \Q'/\Q'$.
 	\end{enumerate}
 	Moreover, if $\R/\Q$ is a differential $\G$-torsor, then also the actions in {\rm (i)}, {\rm  (ii)} and  {\rm  (iii)} define differential torsors, provided that $\R\otimes_\B \B'$, $\mathcal{S}^{-1}\R$  and $\R\otimes_\Q \Q'$ respectively is not the zero ring. \qed
 \end{lemma}
 
 For later use we record a lemma on differential torsors in the case $\B=k$ and $\Q=k(x)$.
 
 \begin{lemma} \label{lemma: max de ideals dense}
 	Let $G$ be a linear algebraic group over $k$ and let $R/k(x)$ be a differential $G$-torsor. Then the intersection of all maximal $\de$-ideals of $R$ is zero.
 \end{lemma}
 \begin{proof}
 	Let $\m$ be a maximal $\de$-ideal of $R$. For every $g\in G(k)$, also $g(\m)$ is a maximal $\de$-ideal of $R$. It thus suffices to show that the ideal $I=\bigcap_{g\in G(k)}g(\m)$ of $R$ is zero.
 	Clearly $g(I)\subseteq I$ for every $g\in G(k)$. In terms of the coaction $\rho\colon R\to R\otimes_k k[G]$ (as in Lemma~\ref{lemma: equivalence with coaction}) this property translates to $\rho(I)\subseteq I\otimes_k k[G]$. But then the closed subscheme $W$ of $\spec(R)$ defined by $I$ is stable under the action of $G_{k(x)}$. Since $\spec(R)$ is a $G_{k(x)}$-torsor, this is only possible if $W=\spec(R)$, i.e, $I=0$.
 \end{proof}

 To explain how differential torsors are obtained by ``spreading out'' a Picard-Vessiot ring into a family of differential rings we will need a basic lemma on linear algebraic groups.
 
 \begin{lemma} \label{lemma: spread out closed embedding}
 	Let $G$ be a linear algebraic group over $k$ and let $k'/k$ be a field extension. If $\f'\colon G_{k'}\to \Gl_{n,k'}$ is a closed embedding of algebraic groups over $k'$, then there exists a finitely generated $k$-subalgebra $\B$ of $k'$ and a closed embedding $\f\colon G_\B\to \Gl_{n,\B}$ of group schemes over $\B$ such that $\f_{k'}=\f'$.
 \end{lemma}
 \begin{proof}
 	Writing $k'$ as the directed union of its finitely generated $k$-subalgebras, this follows from Theorem 8.8.2 and Theorem 8.10.5 (iv) in \cite{Grothendieck:EGAIV3}. 
  \end{proof}
 
 The following simple example illustrates Lemma \ref{lemma: spread out closed embedding}.
 
 \begin{ex}
 	Let $k'/k$ be a field extension and $b\in k'\smallsetminus\{0\}$. Consider the closed embedding $\f'\colon\mathbb{G}_{a,k'}\to \Gl_{2,k'}$ of algebraic groups over $k'$ (where $\mathbb{G}_{a,k'}$ denotes the additive group over $k'$) given by
 	\begin{equation} \label{eq: embed Ga}
 		g\mapsto \begin{pmatrix}
 			1 & bg \\
 			0 & 1
 		\end{pmatrix}.
 	\end{equation}
 	For $\B=k[b]\subseteq k'$, formula (\ref{eq: embed Ga}) defines a morphism $\mathbb{G}_{a,\B}\to \Gl_{2,\B}$ of group schemes over $\B$ that need not be a closed embedding, e.g., if $b$ is transcendental over $k$, the fibre over $b=0$ is not a closed embedding. However, for $\B=k[b,b^{-1}]\subseteq k'$, formula (\ref{eq: embed Ga}) defines a closed embedding $\mathbb{G}_{a,\B}\to \Gl_{2,\B}$ of group schemes over $\B$.
 \end{ex}
 
 One can think of differential torsors as the kind of objects one obtains  when ``spreading out'' a Picard-Vessiot ring into a nice family. This is formalized in the following lemma. 
 \begin{lemma} \label{lemma: spread out PVring}
 	Let $k\subseteq k'$ be an inclusion of algebraically closed fields and let $R/k'(x)$ be a Picard-Vessiot ring for $\de(y)=Ay$ with $A\in k'(x)^{n\times n}$ and $Y\in\Gl_n(R)$ such that $\de(Y)=AY$. Then there exist
 	\begin{itemize}
 		\item a finitely generated $k$-subalgebra $\B$ of $k'$,
 		\item a monic polynomial $f\in \B[x]$,
 		\item an affine group scheme $\G$ of finite type over $\B$ and
 		\item a differential $\G$-torsor $\R/\B[x]_f$		
 	\end{itemize}
 	such that 
 	\begin{itemize}
 		\item  $A\in\B[x]_f^{n\times n}$,
 		\item $\R$ is flat over $\B[x]_f$ and $\R/\B[x]_f$ is a differential $\G$-torsor for $\de(y)=Ay$ with fundamental solution matrix $\Y\in\Gl_n(\R)$ and $R^\gen=\R\otimes_{\B[x]_f}K(x)$ is $\de$-simple, where $K\subseteq k'$ is the algebraic closure of the field of fractions of $\B$ and
 		\item $R^\gen\otimes_{K(x)}k'(x)\simeq R$ via an isomorphism that identifies $\Y$ with $Y$.
 	\end{itemize}	
 	Moreover, if $\B_0$ is a finitely generated $k$-subalgebra of $k'$ and $f_0\in \B_0[x]$ is monic,
 	then $\B$ can be chosen such that $\B_0\subseteq \B$ and $f$ can be chosen such that $f_0$ divides $f$ in $\B[x]$. If, in addition, $f_0\in \B_0[x]$ is such that $A\in\B_0[x]_{f_0}^{n\times n}$, then $\B$ can, in addition, be chosen such that $\B$ is contained in the algebraic closure of the field of fractions of $\B_0$.

 	Furthermore, if the differential Galois group $G$ of $R/k'(x)$ is of the form $G=G_{0,k'}$ for some algebraic group $G_0$ over $k$, then $\G$ can be chosen to be equal to $G_{0,\B}$.

 \end{lemma}
 \begin{proof}
 	We consider the differential Galois group $G$ of $R/k'(x)$ as a closed subgroup of $\Gl_{n,k'}$ via the choice of $Y$. The contributions to $\B$ and $f$ came from three different sources. We go through them one by one:
 	
 	\renewcommand{\labelenumi}{{\rm (\alph{enumi})}}
 	
 	\begin{enumerate}
 		\item  Let $\B_1$ be a finitely generated $k$\=/subalgebra of $k'$ and $f_1\in \B[x]$ a monic polynomial such that $A\in \B_1[x]_{f_1}^{n\times n}$. Then $\B_1[x]_{f_1}[Y,\frac{1}{\det(Y)}]$ is a $\de$-subring of $R$.
 		\item Consider the ideal $I$ of $k'(x)[X,\frac{1}{\det(X)}]$ given by $I=\{p\in k'(x)[X,\frac{1}{\det(X)}]|\ p(Y)=0\}$. Choose $p_1,\ldots,p_r\in I$ such that $I=(p_1,\ldots,p_r)$. Let $\B_2$ be a finitely generated $k$\=/subalgebra of $k'$ and let $f_2\in \B_2[x]$ be monic such that $p_1,\ldots,p_r\in \B_2[x]_{f_2}[X,\frac{1}{\det(X)}]\subseteq k'(x)[X,\frac{1}{\det(X)}].$ %\rf{``that" is added.}
 		\item Let $\B_3$ be a finitely generated $k$-subalgebra of $k'$ and $\G_3$ a closed subgroup scheme of $\Gl_{n,\B_3}$ such that $\G_{3,k'}=G$ as closed subgroups of $\Gl_{n,k'}$. By generic flatness (\cite[Theorem~14.4]{Eisenbud:view}), we may assume that $\G$ is flat over $\B_3$.
 	\end{enumerate}

 	Let $\B_4$ be the $k$-subalgebra of $k'$ generated by $\B_1,\B_2$ and $\B_3$. Furthermore, set $f_4=f_1f_2$ and consider the finitely generated $\B_4[x]_{f_4}$-algebra $\B_4[x]_{f_4}[Y,\frac{1}{\det(Y)}]\subseteq R$.	By generic flatness (\cite[Theorem~14.4]{Eisenbud:view}), there exists a nonzero (not necessarily monic) $h\in\B_4[x]_{f_4}$ such that $(\B_4[x]_{f_4})_{h}[Y,\frac{1}{\det(Y)}]$ is flat over $(\B_4[x]_{f_4})_{h}$. Write $h=\frac{h_1}{f_4^m}$ with $h_1\in\B_4[x]$ and $m\geq 1$. 
 	With $h'=f_4h_1\in\B_4[x]$ we then have $(\B_4[x]_{f_4})_{h}=\B_4[x]_{h'}$.
 %	\rf{Should we assume $f_4$ divide $h'$ here?}
 %	\mw{That seems to automatic from $h=\frac{h'}{f_4^m}$.}\rf{Sorry that I did not express my confusion clearly. I actually mean that perhaps $h'$ needs to be replaced with $f_4 h'$. For example, if $h=\frac{1}{f_4^m}$, i.e., $h'=1$, then the equality $(\B_4[x]_{f_4})_{h}=\B_4[x]_{h'}$ may no longer be true.}
 %	\mw{Okay, thanks! Yes, you are right. I have corrected this now.}
 
 	Let $\B$ be the $k$-subalgebra of $k'$ generated by $\B_4$ and the inverse (in $k'$) of the leading coefficient $b$ of $h'$.
 	In general, if $R\to S$ is a flat ring map and $f\in R$, then $S_f=S\otimes_R R_f$ is a flat $R_f$-module (\cite[\href{https://stacks.math.columbia.edu/tag/00HI}{Tag 00HI}]{stacks-project}). Therefore, since $\B_4[x]_{h'}[Y,\frac{1}{\det(Y)}]$ is a flat $\B_4[x]_{h'}$\=/module, $(\B_4[x]_{h'})_b[Y,\frac{1}{\det(Y)}]$ is a flat $(\B_4[x]_{h'})_{b}$\=/module. For the monic polynomial $f=\frac{1}{b}h'\in \B[x]$ 
 	% \rf{$\B_5$ was changed into $\B$.}
 	we have $(\B_4[x]_{h'})_{b}=\B[x]_{f}$ and so $\B[x]_{f}[Y,\frac{1}{\det(Y)}]$ is a flat $\B[x]_{f}$-module. 
 	
 	We set $\R=\B[x]_{f}[Y,\frac{1}{\det(Y)}]$ and $\G=\G_{3,\B}$. Then $\R$ is a $\B[x]_f$-$\de$-algebra and flat over $\B[x]_f$. Moreover, the canonical surjection $\R\otimes_{\B[x]_f}k'(x)\to R$ is an isomorphism by item (b) above. Since $\G_3$ is flat over $\B_3$, we see that $\G$ is flat over $\B$.
 	Furthermore, $\G_{k'}=(\G_{3,\B})_{k'}=\G_{3,k'}=G$. In particular, $k'(x)\otimes_\B \B[\G]=k'(x)\otimes_{k'}k'[G]$.
 	 	
 	The above flatness properties allow us to identify $\R\otimes_\B \B[\G]$ with a $\de$-subring of $R\otimes_{k'}k'[\G]$. In detail, since $\B[\G]$ is a flat $\B$-algebra, $\B[x]_f\otimes_\B \B[\G]$ is a flat $\B[x]_f$-algebra and therefore $\R\otimes_\B\B[\G]=\R\otimes_{\B[x]_f}(\B[x]_f\otimes_\B\B[\G])$ is a flat $\B[x]_f$-algebra. This entails that the map $\R\otimes_\B\B[\G]\to (\R\otimes_\B\B[\G])\otimes_{\B[x]_f}k'(x)$ is injective. But 
 	\begin{align*}
 		(\R\otimes_\B\B[\G])\otimes_{\B[x]_f}k'(x)= & (\R\otimes_{\B[x]_f}k'(x))\otimes_{k'(x)}(k'(x)\otimes_\B\B[\G])= \\ = & R\otimes_{k'(x)}(k'(x)\otimes_{k'}k'[G])=R\otimes_{k'}k'[G].
 	\end{align*}
 	So we can consider $\R\otimes_\B\B[\G]$ as a subring of $R\otimes_{k'}k'[G]$.
 	
 	As explained in Section \ref{subsec:differential Galois theory}, the closed embedding $G\to \Gl_{n,k'}$ corresponds to the morphism $k'[X, \frac{1}{\det(X)}]\to k'[G]=(R\otimes_{k'(x)}R)^\de,\ X\mapsto Z=Y^{-1}\otimes Y$. Moreover, the coaction $R\to R\otimes_{k'}k'[G]$ is determined by $Y\mapsto Y\otimes Z$. In summary, we see that the coaction  $R\to R\otimes_{k'}k'[G]$ restricts to a map $\R\to \R\otimes_\B \B[\G]$. It is then clear that $\G$ acts on $\R/\B[x]_f$. Indeed, for a $\B$-algebra $\T$, every $g\in\G(\T)\leq\Gl_n(\T)$ acts on $\R\otimes_\B\T$ by $Y\otimes 1\mapsto Y\otimes g$.
 	 	
 	To see that $\R/\B[x]_f$ is a differential $\G$-torsor, let us first show that $\R\otimes_{\B[x]_f}\R\to R\otimes_{k'(x)}R$ is injective. Note that
 	$$(\R\otimes_{\B[x]_f}\R)\otimes_{\B[x]_f}k'(x)=(\R\otimes_{\B[x]_f}k'(x))\otimes_{k'(x)}(\R\otimes_{\B[x]_f}k'(x))=R\otimes_{k'(x)}R.$$
 	It thus suffices to show that the map $\R\otimes_{\B[x]_f}\R\to (\R\otimes_{\B[x]_f}\R)\otimes_{\B[x]_f}k'(x)$ is injective. 
 	But this follows from the flatness of $\B[x]_f\to \R\otimes_{\B[x]_f}\R$, which in turn follows from the flatness of $\B[x]_f\to \R$.
 	 	
 	We claim that the isomorphism
 	$R\otimes_{k'(x)}R\to R\otimes_{k'}k'[G]$ restricts to an isomorphism $\R\otimes_{\B[x]_f}\R\to \R\otimes_\B\B[\G]$. Injectivity being obvious, the surjectivity of $\R\otimes_{\B[x]_f}\R\to \R\otimes_\B\B[\G]$ follows from $\B[\G]=\B[Z,\frac{1}{\det(Z)}]$, where $Z=Y^{-1}\otimes Y\in \Gl_n(R\otimes_{k'(x)}R)$, and $\R\otimes_{\B[x]_f}\R\to \R\otimes_\B\B[\G]$ maps $Y^{-1}\otimes Y$ to $1\otimes Z$. Thus $\R/\B[x]_f$ is a differential $\G$-torsor. As $\de(Y)=AY$ and $\R=\B[x]_f[Y,\frac{1}{\det(\Y)}]$, we see that, in fact, $\R/\B[x]_f$ is a differential $\G$-torsor for $\de(y)=Ay$ with fundamental matrix $\Y=Y$.
 	
 	As $(\R\otimes_{\B[x]_f}K(x))\otimes_{K(x)}k'(x)=\R\otimes_{\B[x]_f}k'(x)=R$ is $\de$-simple, we see that also $R^\gen=\R\otimes_{\B[x]_f}K(x)$ must be $\de$-simple. We have thus established the first claim of the lemma.

 	To address the second claim, let $K_0\subseteq k'$ be the algebraic closure of the field of fractions of $\B_0$. Note that there is no difficulty in arranging $\B$ such that $\B_0\subseteq \B$. For example, in the above construction of $\B$, we could choose $\B_1$ such that $\B_0\subseteq \B_1$. Similarly, we can choose $f_1$ so that $f_0$ divides $f_1$. 
 	
 	Assume that $A\in \B_0[x]_{f_0}^{n\times n}$.
 	As $A\in K_0(x)^{n\times n}$, we may consider a Picard-Vessiot ring $R_0/K_0(x)$ for $\de(y)=Ay$. It then follows from Lemma~\ref{lemma: base change of PV ring over constants} that $R_0\otimes_{K_0(x)}k'(x)$ is a Picard-Vessiot ring for $\de(y)=Ay$ over $k'(x)$. By the uniqueness of Picard-Vessiot rings, we thus have $R\simeq R_0\otimes_{K_0(x)}k'(x)$.
 	We can now apply the first claim of the lemma to the inclusion $k\subseteq K_0$ and the Picard-Vessiot ring $R_0/K_0(x)$. Then automatically $\B\subseteq K_0$.
 	
 	Concerning the last claim of the lemma, note that, according to Lemma \ref{lemma: spread out closed embedding}, the closed embedding $G=G_{k'}\to \Gl_{n,k'}$ determined by $Y$ spreads out to a closed embedding $G_{\B'}\to \Gl_{n,\B'}$, where $\B'$ is a finitely generated $k$-subalgebra of $k'$. In point (c) above, we can thus choose $\B_3=\B'$ and $\G_3=G_{\B_3}$, so that $\G=G_{0,\B}$.
 
 	\renewcommand{\labelenumi}{{\rm (\roman{enumi})}}
 \end{proof}

\subsection{Proto-Picard-Vessiot rings}
\label{subsec:Proto-Picard-Vessiot rings}
 
For a linear algebraic group $G$, let $G^t$ denote the intersection of all kernels of all characters of $G^\circ$, the identity component of $G$. The algebraic group $G^t$ plays a fundamental role in the study of moduli of linear differential equations (\cite{Singer:ModuliOfLinearDifferentialEquations}) and in Hrushovski's algorithm (\cite{Hrushovski:ComputingTheGaloisGroupOfaLinearDifferentialEequation}). In regards to Hrushovski's algorithm, the key point concerning $G^t$ is the following. While the degree of defining equations of a closed subgroup of $\Gl_n$ cannot be bounded uniformly in $n$ (consider, e.g., the groups $\mu_m$ of $m$-th roots of unity in $\Gl_1$), there exists a function $\bfd$ of $n$ such that for every closed subgroup $H$ of $\Gl_n$ there exists a closed subgroup $G$ of $\Gl_n$ defined by equations of degree at most $\bfd(n)$ satisfying
\begin{equation}\label{eq: proto Galois group}
	G^t\leq H\leq G.
\end{equation}
Given $H$, an algebraic group $G$ satisfying (\ref{eq: proto Galois group}) is called a \emph{proto-Galois group} for $H$ (\cite[Def. 4.1]{AmzallagMinchenkoPogudin:DegreeBoundforToricEnvelope}).
In the first main step of Hrushovski's algorithm a proto-Galois group for the differential Galois group is computed.

For our purpose, the mere existence of the function $\bfd(n)$ is sufficient. To make Hrushovski's algorithm practical (a goal that seems rather far out of reach), it is important to find small bounds. In \cite{AmzallagMinchenkoPogudin:DegreeBoundforToricEnvelope} it is shown that one can take $\bfd(n)=(4n)^{3n^2}$.

One can think of $G$ as an ``approximation'' of $H$. This approximation is good in the sense that $G/G^t$ is small. Indeed, $G^t$ is normal in $G$ and the identity component of $G/G^t$ is a torus.
An alternative description of $G^t$ is, as the closed subgroup of $G$ generated by all unipotent elements of $G$.
In this section we introduce an analog of this notion for differential torsors: Proto-Picard-Vessiot ring. We think of them as ``approximations'' of Picard-Vessiot rings. We also establish a criterion that enables us to decide whether or not a proto-Picard-Vessiot ring is a Picard-Vessiot ring.

%%??  We now begin with the specialization arguments for families of linear differential equations.
% In this section we introduce proto-Picard-Vessiot rings, a certain ``approximation'' of Picard-Vessiot rings, and, in the context of Notation \ref{notation: setup}, we show that there exists an ad\=/open subset $\U$ of $\X(k)$ such that $R^c$ is a proto-Picard-Vessiot ring for all $c\in \U$. Roughly, this corresponds to the first step in Hrushovski's algorithm in which a proto-Galois group is computed.
% 
% 
% 

\medskip 

In addition to the standing assumption of Section \ref{sec: Some topics in differential Galois theory} that $F$ is a differential field with field of constants $k=F^\de$ algebraically closed and of characteristic zero, throughout Section~\ref{subsec:Proto-Picard-Vessiot rings} we make the following assumptions:
 \begin{itemize}
 %	\item $F$ is a differential field of characteristic zero with $F^\de=k$ algebraically closed and of characteristic zero;
 	\item $G$ is a linear algebraic group over $k$;
 	\item $A\in F^{n\times n}$;
 	\item $R/F$ is a differential $G$-torsor for $\de(y)=Ay$ with fundamental matrix $Y\in \Gl_n(R)$ such that $R$ is an integral domain;
% 	\rf{$K$ was changed into $F$.}
\item $F'$ is the integral closure of $F$ in $R$;
 	\item $E$ is the field compositum of $\overline{F}$ and the field of fractions of $R$ (inside the algebraic closure of the field of fractions of $R$). 
 \end{itemize}

% 
% \subsubsection{Proto-Picard-Vessiot rings}
% \label{subsubsec:Proto-Picard-Vessiot rings}
% 
% ??
% 
% Throughout Section \ref{subsubsec:Proto-Picard-Vessiot rings} we make the following assumptions:
% \begin{itemize}
% 	\item $F$ is a differential field of characteristic zero with field of constants $F^\de=k$ algebraically closed and of characteristic zero;
% 	\item $G$ is a linear algebraic group over $k$;
% 	\item $A\in F^{n\times n}$;
% 	\item $R/F$ is a differential $G$-torsor for $\de(y)=Ay$ with fundamental matrix $Y\in \Gl_n(R)$ such that $R$ is an integral domain;
% \end{itemize}
 Note that if $\m$ is a maximal $\de$-ideal of $R$, then $R/\m$ is a Picard-Vessiot ring for $\de(y)=Ay$. As we will now explain, the differential Galois group $H$ of $R/\m$ is canonically a closed subgroup of $G$ (cf. the proof of Proposition 1.15 in \cite{BachmayrHarbaterHartmannWibmer:DifferentialEmbeddingProblems}).
 The composition
 \begin{equation} \label{eq: composite map}
 	R\otimes_k k[G]\xrightarrow{\simeq }R\otimes_{F}R\to R/\m\otimes_{F} R/\m\xrightarrow{\simeq} R/\m\otimes_k k[H]
 \end{equation}
 maps $k[G]$ into $k[H]$, because the elements of $k[G]$ are constant and the constants of $R/\m\otimes_k k[H]$ are just $k[H]$. Thus the surjection (\ref{eq: composite map}) is of the form $\pi\otimes\psi$, where $\pi\colon R\to R/\m$ is the canonical map and $\psi\colon k[G]\to k[H]$ is a surjective map of $k$-algebras.
 So $\psi$ defines a closed embedding $H\to G$ and we may identify $H$ with a closed subscheme of $G$. We claim that
 \begin{equation} \label{eq: H subgroup G}
 	H(T)=\{g\in G(T)|\ g(\m\otimes_k T)=\m\otimes_k T\}
 \end{equation}
 for every $k$-algebra $T$. In particular, $H$ is a closed subgroup of $G$.
 
 The commutative diagram
 $$
 \xymatrix{
 	R \ar^-\rho[r] \ar_-\pi[d]& R\otimes_k k[G] \ar^-{\pi\otimes\psi}[d] \\
 	R/\m \ar[r] & R/\m\otimes_k k[H]	
 }
 $$
 shows that $\rho(\m)\subseteq \m\otimes_k k[G]+R\otimes_k \I(H)$, where $\I(H)=\ker(\psi)\subseteq k[G]$ is the defining ideal of $H$. Therefore $g(\m\otimes_k T)\subseteq \m\otimes_k T$ for $g\in H(T)$. Since also $g^{-1}\in H(T)$, we deduce that indeed $g(\m\otimes_k T)=\m\otimes_k T$.
 
 Conversely, if $g\in G(T)$ satisfies $g(\m\otimes_k T)=\m\otimes_k T$, then $g$ induces a differential automorphism of $(R\otimes_k T)/(\m\otimes_k T)=R/\m\otimes_k T$ and thus belongs to $H(T)$, by definition of $H$.
 This proves (\ref{eq: H subgroup G}).
 
 Now let $\m'$ be another maximal $\de$-ideal of $R$. Then $R/\m'$ is a Picard-Vessiot ring over $F$ with differential Galois group $H'$ canonically contained in $G$. We will show that $G^t\leq H$ if and only if $G^t\leq H'$.
 
 As a first step, let us show that there exists a $g\in G(k)$ such that $\m'=g(\m)$. Consider the morphism
 \begin{equation}\label{eq: morphism for one to all}
 	k[G]\to R\otimes_k k[G]\xrightarrow{\simeq}R\otimes_{F} R\to R/\m'\otimes_{F} R/\m\to S,
 \end{equation}
 where $S$ is a quotient of $R/\m'\otimes_{F} R/\m$ by a maximal $\de$-ideal. As $S$ is $\de$-simple and finitely generated over $F$, we have $S^\de=k$ (\cite[Lemma 1.17]{SingerPut:differential}). Thus (\ref{eq: morphism for one to all}) maps $k[G]$ into $k$ and therefore defines a point $g\in G(k)$. By construction, $g(\m)\subseteq \m'$.
 Since $g(\m)$ is a maximal $\de$-ideal of $R$, we obtain $g(\m)=\m'$ as claimed.
 
 From (\ref{eq: H subgroup G}), we thus deduce that $H'=gHg^{-1}$. As $G^t$ is normal $G$, we see that $G^t\leq H$ if and only if $G^t\leq H'$.
 
 In summary, if $\m$ and $\m'$ are maximal $\de$-ideals in $R$, yielding differential Galois groups $H\leq G$ and $H'\leq G$ respectively, then $G^t\leq H$ if and only if $G^t\leq H'$. So we can use this property to define proto-Picard-Vessiot rings.
 
 \begin{defi}
 	The differential torsor $R/F$ is \emph{proto-Picard-Vessiot} if  
 	for one (equivalently every) maximal $\de$-ideal $\m$ of $R$ we have $G^t\leq H$, where $H$ is the differential Galois group of $R/\m$.
 \end{defi}
 Clearly, $R/F$ is proto-Picard-Vessiot if $R/F$ is Picard-Vessiot. Our next goal is to find a criterion to test if a differential torsor is proto-Picard-Vessiot. In the next defintion and lemma $k'$ is a field of characteristic zero.
 % The main goal of this subsection is to show that, in the context of Notation \ref{notation: setup}, there exists an ad-open subset $\U$ of $\X(k)$ such that $R^c/k(x)$ is proto-Picard-Vessiot for all $c\in \U$ (Proposition \ref{PROP:protoPV}). To this end, we first derive a criterion to test if $R/F$ is Picard-Vessiot.
  Following \cite{Feng:HrushovskisAlgorithmForComputingTheGaloisGroupOfaLinearDifferentialEquation} and \cite[Def. 2.3]{AmzallagMinchenkoPogudin:DegreeBoundforToricEnvelope} we make the below definition.
 
 %\mw{In what follows the $f$'s and $h_ij$ should be $p$'s...}
 
 %\mw{According to the introduction, varieties are geometrically integral, so instead of subvariety, we have to say geometrically reduced closed subscheme here.}

 \begin{defi}
 	\label{DEF:boundvarieties}
 	A geometrically reduced closed subscheme $Z$ of $\Gl_{n,k'}$ is \emph{bounded} by a positive integer $d$, if there exist polynomials $p_1,\dots,p_m\in k'[X]$ of degree at most $d$ such that
 	\begin{equation} \label{eq: bounded}
 		Z(\overline{k'})=\{g\in\Gl_n(\overline{k'})|\ p_1(g)=\ldots=p_m(g)=0\}.
 	\end{equation}
 	Note that, by Hilbert's Nullstellensatz, (\ref{eq: bounded}) is equivalent to $\I(Z)=\sqrt{(p_1,\ldots,p_m)}$, where $\I(Z)\subseteq k'[X,\frac{1}{\det{X}}]$ is the defining ideal of $Z$ and $(p_1,\ldots,p_m)$ is the ideal of $k'[X,\frac{1}{\det{X}}]$ generated by $p_1,\ldots,p_m$.
 \end{defi}
 
 \begin{lemma} \label{lemma: bound for torsor}
 	Let $G$ be a closed subgroup of $\Gl_{n,k'}$ and $Z$ a geometrically reduced closed subscheme of $\Gl_{n,k'}$ such that $Z$ is a $G$-torsor under right multiplication. If $G$ is bounded by $d$, then also $Z$ is bounded by $d$.
 \end{lemma}
 \begin{proof}
 	A point $z\in Z(\overline{k'})$ defines an isomorphism $G_{\overline{k'}}\to Z_{\overline{k'}},\ g\mapsto zg$ of schemes over $\overline{k'}$.
 	If $p_1,\ldots,p_m\in k'[X]$ are such that $$G(\overline{k'})=\{g\in\Gl_n(\overline{k'})|\ p_1(g)=\ldots=p_m(g)=0 \},$$ then $$Z(\overline{k'})=\{g\in\Gl_n(\overline{k'})|\ p_1(z^{-1}g)=\ldots=p_m(z^{-1}g)=0\}.$$ So $\I(Z_{\overline{k'}})=\sqrt{(p_1(z^{-1}X),\ldots,p_m(z^{-1}X))}\subseteq \overline{k'}[X,\frac{1}{\det(X)}]$.
 	
 	We have $\I(Z_{\overline{k'}})=\I(Z)\otimes_{k'}\overline{k'}$ and so, if $p_i(z^{-1}X)\in \overline{k'}[X]$ has degree at most $d$, i.e, $p_i$ has degree at most $d$, then $p_i(z^{-1}X)$ can be written as a $\overline{k'}$-linear combination of polynomials $q_{i,j}\in k'[X]\cap \I(Z)$ of degree at most $d$. The set of solutions of all $q_{ij}$'s in $\Gl_n(\overline{k'})$ is $Z(\overline{k'})$. Therefore $Z$ is bounded by $d$.
 \end{proof}
 
 As above, set
$
\bfd(n)=(4n)^{3n^2}.
$
The following lemma allows us to test if a differential torsor is proto-Picard-Vesssiot. Roughly speaking, it shows that $R/F$ is proto-Picard-Vessiot if $R$ encodes the algebraic relations among the entries of a fundamental solution matrix up to degree $\bfd(n)$.
 \begin{lemma}
 	\label{LM:criterionforprotoPV}
 	Let $\m$ be a maximal $\de$-ideal of $R$. If $p(Y)=0$ for every $p\in F[X]$ with $\deg(p)\leq \bfd(n)$ and $p(Y)\in\m$, then $R/F$ is proto-Picard-Vessiot.
 \end{lemma}
 \begin{proof}
 	Let $H\leq G$ be the differential Galois group of $R/\m$ and consider $H$ as a closed subgroup of $\Gl_{n,k}$ via the fundamental solution matrix $\overline{Y}\in\Gl_n(R/\m)$, the image of $Y$ in $R/\m$. Due to Corollary 4.1 of \cite{AmzallagMinchenkoPogudin:DegreeBoundforToricEnvelope}, there exists a closed subgroup $\tilde{G}$ of $\Gl_{n,k}$ bounded by $\bfd(n)$ such that
 	$
 	\tilde{G}^t \leq H \leq \tilde{G}.
 	$
 	
 	Let $\m'$ be the maximal $\de$-ideal of $F[X,\frac{1}{\det(X)}]$ whose image in $R$ is $\m$ and let $Z'$ be the closed subscheme of $\Gl_{n,F}$ defined by $\m'$. The image $\tilde{Z}$ of $Z'\times \tilde{G}_{F}\to \Gl_{n,F}, (z,g)\mapsto zg$ is a $\tilde{G}_{F}$-torsor. Indeed, for a point $z\in Z'(\overline{F})$ we have $$\tilde{Z}(\overline{F})=Z'(\overline{F})\tilde{G}(\overline{F})=zH(\overline{F})\tilde{G}(\overline{F})=z\tilde{G}(\overline{F}).$$
 	We know from Lemma \ref{lemma: bound for torsor} that $\tilde{Z}$ is bounded by $\bfd(n)$. Let $p_1,\ldots,p_m\in F[X]$ be polynomials of degree at most $\bfd(n)$ such that $\I(\tilde{Z})=\sqrt{(p_1,\ldots,p_m)}$. As
 	$Z'\subseteq \tilde{Z}$, we have $\I(\tilde{Z})\subseteq \m'$. Therefore $p_1(Y),\ldots,p_m(Y)\in \m$. By assumption, the image of $p_1,\ldots,p_m$ in $R$ is zero. Since $R$ is reduced, we see that in fact the image of $\I(\tilde{Z})$ in $R$ is zero. Geometrically, this means that $Z\subseteq \tilde{Z}$, where $Z\subseteq \Gl_{n,F}$ is the closed subscheme defined by $\I(Z)=\{p\in F[X,\frac{1}{\det(X)}]|\ p(Y)=0\}$. As $Z$ is a $G_{F}$-torsor and $\tilde{Z}$ is a $\tilde{G}_{F}$-torsor (both under right multiplication), we must have $G\leq \tilde{G}$. Therefore, $G^t\leq \tilde{G}^t\leq H$ and $R/F$ is proto-Picard-Vessiot.
 \end{proof}

We now work towards a criterion to decided if a proto-Picard-Vessiot ring is Picard-Vessiot.
 Recall that $G^\circ$ denotes the identity component of the algebraic group $G$ and that $F'$ is the integral closure of $F$ in $R$.
 
 \begin{lemma} \label{lemma: Is G0 torsor}
  The $F'$-$\de$-algebra $R/F'$ is a differential $G^\circ$-torsor for $\de(y)=Ay$.
 \end{lemma}
 \begin{proof}
 	As in \cite[Section 6.5]{Waterhouse:IntroductiontoAffineGroupSchemes}, for a $k'$-algebra $S$ over a field $k'$, let $\pi_0(S/k')$ denote the union of all \'{e}tale $k'$-subalgebras of $S$. In particular, $\pi_0(R/F)=F'$ and $\pi_0(k[G]/k)=k[G/G^\circ]$. In fact, the defining ideal $\I(G^\circ)\subseteq k[G]$ of $G^\circ$ is the ideal of $k[G]$ generated by $\ker(\varepsilon)\cap \pi_0(k[G]/k)$, where $\varepsilon\colon k[G]\to k$ is the counit (\cite[Section 6.7]{Waterhouse:IntroductiontoAffineGroupSchemes}).   
 	
% 	We first show that $G^\circ$ acts on $R/F'$. I.e., we have to show that 
% 	\begin{equation} \label{eq: composition with rho}
% 		R\xrightarrow{\rho} R\otimes_k k[G]\to R\otimes_k k[G^\circ]
% 	\end{equation}
% 	is the identity on $F'$. Since the functor $\pi_0$ is compatible with tensor products and base change (\cite[Section 6.5]{Waterhouse:IntroductiontoAffineGroupSchemes}) we have 
% 	\begin{align*}
% 		\pi_0(R\otimes_k k[G]/F) & =\pi_0(R\otimes_F (k[G]\otimes_k F)/F)=\pi_0(R/F)\otimes_F \pi_0(k[G]\otimes_k F/F)= \\ 
% 		& = F'\otimes_F(\pi_0(k[G]/k)\otimes_k F)=F'\otimes_k\pi_0(k[G]/k).
% 	\end{align*}
% 	Let $f\in F'$. Applying $\pi_0$ to $\rho$ yields a map $F'\to F'\otimes_k \pi_0(k[G]/k)$. Thus $\rho(f)$ can be written as $\rho(f)=\sum_{i=1}^mf_i\otimes h_i$ with $f_i\in F'$ and $h_i\in \pi_0(k[G]/k)$. Since $h_i-\varepsilon(h_i)$ lies in the kernel of $k[G]\to k[G^\circ]$, we see that the image of $\rho(f)=\sum_{i=1}^mf_i\otimes h_i$ in $R\otimes_k k[G^\circ]$
% 	is $\sum_{i=1}^m f_i\otimes\varepsilon(h_i)=f\otimes 1\in R\otimes_k k[G^\circ]$. Thus $G^\circ$ acts on $R/F'$.
 	
 	%It suffices to show that the induced map $R\otimes_{F'}R\to R\otimes_k k[G^\circ]$ is an isomorphism. 
 In the diagram
 	\begin{equation} \label{eq: diagram kernels correspond}
 	\xymatrix{
 		R\otimes_FR \ar^-\simeq[r] \ar[d] & R\otimes_k k[G] \ar[d] \\
 		R\otimes_{F'} R  & R\otimes_k k[G^\circ] 	
 	}
 	\end{equation}
 	the kernel of the left vertical map is $(f\otimes 1-1\otimes f|\ f\in F')\subseteq R\otimes_F R$. The kernel of the right vertical map is $R\otimes_k \I(G^\circ)$, the ideal of $R\otimes_kk[G]$ generated by $\ker(\varepsilon)\cap \pi_0(k[G]/k)$. We claim that these two kernels correspond to each other under the isomorphism.
 	
 	The diagram
 	$$
 	\xymatrix{
 		F'\otimes_F F' \ar^-{\simeq}[rr] \ar[rd] & & F'\otimes_k\pi_0(k[G]/k) \ar^-{\id\cdot \varepsilon}[ld] \\
 		& F'&	
 	}
 	$$
 	where the top isomorphism is obtained by applying $\pi_0$ to $R\otimes_F R\simeq R\otimes_k k[G]$, commutes. The kernel of the left map is $(f\otimes 1-1\otimes f|\ f\in F')\subseteq F'\otimes_F F'$. The kernel of the right map is $F'\otimes_k(\pi_0(k[G]/k)\cap \ker(\varepsilon))$. This shows that the two kernels in (\ref{eq: diagram kernels correspond}) correspond to each other. We thus have an induced isomorphism $R\otimes_{F'}R\simeq R\otimes_kk[G^\circ]$, showing that $R/F'$ is a differential $G^\circ$-torsor.
 \end{proof}

Let $\overline{F}R$ denote the ring compositum of $\overline{F}$ and $R$ (inside the algebraic closure of the field of fractions of $R$). In other words, $\overline{F}R$ is the $\overline{F}$-algebra generated by $R$.

 	\begin{cor}\label{cor: move to algebraic closure}
 		We have $\overline{F}R=R\otimes_{F'}\overline{F}$ and $\overline{F}R/\overline{F}$ is a differential $G^\circ$-torsor.
 	\end{cor}	 
 	\begin{proof}
 		Since $R/F'$ is a differential $G^\circ$-torsor (Lemma \ref{lemma: Is G0 torsor}), it follows from Lemma \ref{lemma: actions} that $R\otimes_{F'}\overline{F}/\overline{F}$ is a differential $G^\circ$-torsor. As $G^\circ$ is connected, this implies that $R\otimes_{F'}\overline{F}$ is an integral domain, i.e., $R$ is a regular $F'$-algebra. This implies (\cite[Chapter V, \S 17, Sections 4 and 5 ]{Bourbaki:Algebra2} that $F'$ is relatively algebraically closed in the field of fractions $F'(R)$ of $R$ and that $F'(R)\otimes_F \overline{F}$ is a field. Thus the canonical map $F'(R)\otimes_{F'}\overline{F}\to E$ is injective (in fact bijective) and so also $R\otimes_{F'}\overline{F}\to\overline{F}R$ is injective.
 	\end{proof}	 
 
 Following \cite[Def. 4.5.2]{Grothendieck:EGAIV4}, for a field $k'$ and a $k'$-algebra $R'$, the number of irreducible components of $\spec(R'\otimes_{k'}\overline{k'})$ is called the \emph{geometric number of irreducible components of $R'/k'$}.

 \begin{cor} \label{cor: geometric number}
 	The geometric number of irreducible components of $R/F$ is $[F':F]$.
 \end{cor}
 \begin{proof}
 	Because $R/F$ is a differential $G$-torsor, $R\otimes_F{\overline{F}}$ is isomorphic to $\overline{F}[G_{\overline{F}}]$ (as an $\overline{F}$-algebra). The number of irreducible components of $G_{\overline{F}}$ is the same as the number of connected components of $G_{\overline{F}}$ and this is the same as the number of connected components of $G$. The latter is the same as the dimension $\dim_k \pi_0(k[G]/k)$ of $\pi_0(k[G]/k)$ as a $k$-vector space. As observed in the proof of Lemma \ref{lemma: Is G0 torsor}, we have $F'\otimes_F F'=F'\otimes_k \pi_0(k[G]/k)$. So $\dim_k \pi_0(k[G]/k)$ must equal $[F':F]$.
 \end{proof}

 Recall that $E$ denotes the field compositum of $\overline{F}$ and the field of fractions of $R$ (inside the algebraic closure of the field of fractions of $R$). The idea for the criterion to decide if a proto-Picard-Vessiot ring $R/F$ is Picard-Vessiot is to first Gauge transform the equation $\de(y)=Ay$ with fundamental sulution matrix $Y$ to an equation $\de(y)=A'y$ (where $A'=BAB^{-1}+\de(B)B^{-1}$) with fundamental matrix $BY$ via a matrix $B\in\Gl_n(\overline{F})$ such that $BY\in G^\circ(E)$ and then apply characters of $G^\circ$ to $BY$. The following corollary shows that we can always find such a transformation matrix $B$.
 
 \begin{cor} \label{cor: existence of h}
 	For $\vartheta\in\Hom_{F'}(R,\overline{F})$ and $B=\vartheta(Y)^{-1}\in\Gl_n(\overline{F})$ we have $BY\in G^\circ(E)$.
 \end{cor}
 %\rf{As $h$ is used to denote a polynomial or an element in a ring, I changed $h$ into $\h$ here.}
 \begin{proof}
 	%Let $F'$ denote the relative algebraic closure of $F$ in $R$.
 	Set $Z=\spec(R)$, considered as a scheme over $F'$. As $R/F'$ is a $G^\circ$-torsor by Lemma \ref{lemma: Is G0 torsor}, the morphism
 	\begin{equation}
 		\label{eq: torsor isom}
 		Z\times_{F'} G^\circ_{F'}\to Z\times_{F'} Z, \ (z,g)\mapsto (z,zg)
 	\end{equation}
 	is an isomorphism. If we consider $Z$ and $G^\circ_{F'}$ as closed subschemes of $\Gl_{n,F'}$ via $Y$, the inverse of (\ref{eq: torsor isom}) is given by $(z_1,z_2)\mapsto (z_1,z_1^{-1}z_2)$. In particular, for any field extension $\widetilde{F}$ of $F'$ and $z_1,z_2\in Z(\widetilde{F})$, we have $z_1^{-1}z_2\in G^\circ(\widetilde{F})$.
 	With $\widetilde{F}=E$, $z_1=\vartheta(Y)$ and $z_2=Y$, this implies the claim.
 %	Thus, any $\vartheta\in Z(\overline{F})=\Hom_{F'}(R,\overline{F})\subseteq \Hom_F(R,\overline{F})$ does the job.
 \end{proof}
 
We note that there is a fixed embedding of $F'$ into $\overline{F}$ derived from forming the algebraic closure of $F$ inside the the algebraic closure of the field of fractions of $R$ and $\Hom_{F'}(R,\overline{F})$ has to be interpreted accordingly.
 
 We denote with $I_n$ the $n\times n$-identity matrix.
 
% \mw{I changed $B^{-1}$ to $B$ as then the formulas look nicer.}
 \begin{lemma}
 	\label{LM:characters}
 	Let $\chi$ be a character of $G^\circ$, $\vartheta\in \Hom_{F'}(R,\overline{F})$ and $B=\vartheta(Y)^{-1}\in\Gl_n(\overline{F})$. Then
 	$$
 	\delta(\chi(BY))=\left(\sum_{i,j=1}^n\frac{\partial \chi}{\partial X_{ij}}(I_n)(BAB^{-1}+\delta(B)B^{-1})_{ij}\right)\chi(BY).
 	$$
 \end{lemma}
 \begin{proof}
 	We will work with the ring $E[\varepsilon]$ of dual numbers, i.e., $\varepsilon^2=0$. For $p\in F[X,\frac{1}{\det(X)}]$, $C\in\Gl_n(E)$ and $D\in E^{n\times n}$ (i.e., $C+\varepsilon D\in\Gl_n(E[\varepsilon])$,
	%\rf{$\det(x)$ is changed into $\det(X)$.}
%	\mw{We previously had here $D\in\Gl_n(E)$ but it seems unclear if $\de(B)$ below lies in $\Gl_n(E)$ so better and maybe more natural to have $D\in E^{n\times n}$. }
 	 we then have $p(C+\varepsilon D)=p(C)+\varepsilon\sum_{i,j=1}^n\frac{\partial p}{\partial X_{ij}}(C)D_{ij}$. Therefore
 	\begin{equation}
 		\label{eq: dual numbers}
 		p(C+\varepsilon\de(C))=p(C)+\varepsilon\sum_{i,j=1}^n\frac{\partial p}{\partial X_{ij}}(C)\de(C_{ij})=p(C)+\varepsilon(\de(p(C))).
 	\end{equation}
 	For $Y'=BY\in\Gl_n(E)$ and $A'=BAB^{-1}+\delta(B)B^{-1}\in \overline{F}^{n\times n}$ we have $\de(Y')=A'Y'$.
 	As $Y'\in G^\circ(E)$ (Corollary \ref{cor: existence of h}) it follows from (\ref{eq: dual numbers}) that $Y'+\varepsilon\de(Y')\in G^\circ(E[\varepsilon])$. We have
 	
 	$$\chi(Y'+\varepsilon\de(Y'))\chi(Y'^{-1})=(\chi(Y')+\varepsilon\de(\chi(Y')))\chi(Y'^{-1})=1+\varepsilon\de(\chi(Y'))\chi(Y')^{-1}.$$
 	On the other hand,
 	$$ \chi(Y'+\varepsilon\de(Y'))\chi(Y'^{-1})=\chi((Y'+\varepsilon\de(Y'))Y'^{-1})=\chi(I_n+\varepsilon A')= 1+\varepsilon\sum_{i,j=1}^n\frac{\partial \chi}{\partial X_{ij}}(I_n)A'_{ij}.$$
% 	\rf{I added $\varepsilon$.}
 	Thus 
 	$$\de(\chi(Y'))=\left(\sum_{i,j=1}^n\frac{\partial \chi}{\partial X_{ij}}(I_n)A'_{ij}\right) \chi(Y')$$
 	as claimed.
 \end{proof}
 
 The proof of Lemma \ref{LM:characters} shows that $\sum_{i,j=1}^n\frac{\partial \chi}{\partial X_{ij}}(I_n)(BAB^{-1}+\delta(B)B^{-1})_{ij}\in \overline{F}$ does not depend on the lift of $\chi\in k[G^\circ]$ to an element of $k[X,\frac{1}{\det(X)}]$.

 Our criterion that characterizes Picard-Vessiot rings among proto-Picard-Vessiot rings is based on the concept of \emph{logarithmic independence} that we now define.
 
 \begin{defi} \label{defi: logarithmically independent}
 	Elements $f_1,\dots,f_m$ of a differential field $F$ are \emph{logarithmically independent} (over $F$) if $\sum_{i=1}^m d_i f_i =\delta(f)/f$ for $d_1,\ldots,d_m\in\mathbb{Z}$ and $f\in F^\times$ implies $d_1=\dots=d_m=0$.
 \end{defi}
 
 The significance of the above definition is the following. Let $h_1,\ldots,h_m$ be nonzero elements of some $\de$-field extension of $F$ with no new constants such that $\de(h_i)=f_i h_i$ for $i=1,\ldots,m$. By the Kolchin-Ostrowski theorem (\cite{Kolchin:AlgebraicGroupsAndAlgebraicDependence}), if $h_1,\ldots,h_m$ are algebraically dependent over $F$, then there exist $d_1,\ldots,d_m\in\mathbb{Z}$ not all equal to zero such that $h_1^{d_1}\ldots h_m^{d_m}=f\in F$. 
 Applying the logarithmic derivative $h\mapsto \frac{\de(h)}{h}$ to this equation yields  $\sum_{i=1}^m d_i f_i =\delta(f)/f$.
 %\rf{$a$ is changed into $f$.}
 
 On the other hand, if $\sum_{i=1}^m d_i f_i =\delta(f)/f$ holds, the logarithmic derivative of $\frac{h_1^{d_1}\ldots h_m^{d_m}}{f}$ is zero. So $\frac{h_1^{d_1}\ldots h_m^{d_m}}{f}$ is a constant $\lambda$ and $h_1^{d_1}\ldots h_m^{d_m}=f\lambda\in F$. In summary, we see that $h_1,\ldots,h_m$ are algebraically independent over $F$ if and only if $f_1,\ldots,f_m$ are logarithmically independent.

 Note that if $F\subseteq\widetilde{F}$ is an inclusion of differential fields, then $f_1,\ldots,f_m\in F$ may be logarithmically independent over $F$ but logarithmically dependent over $\widetilde{F}$. However, from the above discussion we deduce the following.
 
 \begin{rem} \label{rem: logarithmic independence preserved under algebraic extension}
 	If $\widetilde{F}/F$ is an algebraic extension of differential fields and $f_1,\ldots,f_m\in F$, then $f_1,\ldots,f_m$ are logarithmically independent over $F$ if and only if $f_1,\ldots,f_m$ are logarithmically independent over $\widetilde{F}$. 
 \end{rem}

 The following lemma provides a criterion to test if a proto-Picard-Vessiot ring is Picard-Vessiot. We denote with $X(G^\circ)$ the group of characters of $G^\circ$.
 
 \begin{lemma}\label{LM:criterion}
 	Assume that $R/F$ is proto-Picard-Vessiot and
 	let $\chi_1,\dots,\chi_m$ be a basis of $X(G^\circ)$ (as a $\mathbb{Z}$-module). Choose $\vartheta \in\Hom_{F'}(R,\overline{F})$ and set $B=\vartheta(Y)^{-1}$ and
 	%\rf{I have not changed the symbol $g$ yet, because so far I did not find a good choice for it.}
 	$$
 	f_\ell=\sum_{i,j=1}^n\frac{\partial \chi_\ell}{\partial X_{ij}}(I_n)(BAB^{-1}+\delta(B)B^{-1})_{ij}\in \overline{F}.
 	$$
 	for $\ell=1,\ldots,m$. Then $R/F$ is Picard-Vessiot if and only if $f_1,\dots,f_m$ are logarithmically independent over $\overline{F}$.
 \end{lemma}
 \begin{proof}
 	%\rf{I removed the first sentence ``We first assume that $f_1,\dots,f_m$ are logarithmically independent over $\overline{F}$."} 
 	 Suppose $R/F$ is not Picard-Vessiot, i.e., $R$ is not $\delta$-simple. We will show that $f_1,\ldots,f_m$ are logarithmically dependent. As a non-trivial $\de$-ideal of $R$, gives rise to a non-trivial $\de$-ideal of $R\otimes_{F'}\overline{F}$, 
 	this implies that also $R\otimes_{F'}\overline{F}$ is not $\de$-simple. 
 	
 	From Corollary \ref{cor: move to algebraic closure} we know that $\overline{F}R=R\otimes_{F'}\overline{F}$ is a differential $G^\circ$-torsor over $\overline{F}$. Let $\overline{\mathfrak{m}}$ be a maximal $\de$-ideal of  $\overline{F}R=R\otimes_{F'}\overline{F}$ and set $\m=\overline{\m}\cap R$. Then $(\overline{F}R)/\overline{\m}$ is a Picard-Vessiot ring for $\de(y)=Ay$ over $\overline{F}$ and $R/\m$ embeds into $(\overline{F}R)/\overline{\m}$. This implies that the field of fractions of $R/\m$ does not have new constants. So $R/\m$ is a Picard-Vessiot ring (over $F$) for $\de(y)=Ay$ and $\m$ is a maximal $\de$-ideal of $R$. Let $H$ be the differential Galois group of $R/\m$. 
 	As $R/F$ is proto-Picard-Vessiot, we have $G^t\leq H\leq G$.
 	
 	 The ring compositum $\overline{F}(R/\m)$ of $\overline{F}$ and $R/\m$ (inside the algebraic closure of the field of fractions of $R/\m)$ can be identified with $(\overline{F}R)/\overline{\m}$. In fact, using Corollary \ref{cor: move to algebraic closure} (applied to $R/\m$), we get a map $\overline{F}(R/\m)\to (\overline{F}R)/\overline{\m}$ which then necessarily is an isomorphism. By Corollary \ref{cor: move to algebraic closure}, the differential Galois group of $\overline{F}(R/\m)=(\overline{F}R)/\overline{\m}$ is $H^\circ$. 	
 	
 	As $\km$ is a nonzero prime ideal of $R$ and $R$ is an integral domain, we see that $$\dim(H)=\dim(R/\km)<\dim(R)=\dim(G).$$
 	Therefore $H^\circ$ is a proper subgroup of $G^\circ$. 
 	As the identity component $G^\circ/G^t$ of $G/G^t$ is a torus and $H^\circ/G^t$ is a proper subgroup of $G^\circ/G^t$, there exists a non-trivial character of $G^\circ/G^t$ that is trivial on $H^\circ/G^t$, 	
 	i.e., there exists a non-trivial character $\chi$ of $G^\circ$ such that $H^\circ\subseteq \ker(\chi)$.

 Let $\overline{Y}\in\Gl_n(R/\m)\subseteq\Gl_n((\overline{F}R)/\overline{\m})$ denote the image of $Y\in\Gl_n(R)$ and set $Y'=B\overline{Y}\in\Gl_n(\overline{F}(R/\m))=\Gl_n((\overline{F}R)/\overline{\m})$. As $BY\in G^\circ(\overline{F}R)$ (Corollary \ref{cor: existence of h}), we have $Y'\in G^\circ((\overline{F}R)/\overline{\m})$.
% 	
% 	
% 	Note that $F'$ is also the relative algebraic closure of $F$ in $R/\km$, $\tilde{\h}\in \Hom_{F'}(R/\km,\overline{F})$ and $\tilde{g}=\tilde{\h}(Y)=\tilde{\h}(\overline{Y})$. The proof of Corollary~\ref{cor: existence of h} (applied with $R/\km$ and $H^\circ$) implies that $Y'\in H^\circ(E')$.
 	For every $h\in H^\circ(k)$ we can write $h(Y')=Y'[h]$ with $[h]\in \Gl_n(k)$. Indeed, as explained in Section \ref{subsec:differential Galois theory}, we may identify $H^\circ$ with a closed subgroup of $\Gl_{n,k}$. Then
 	$$h(\chi(Y'))=\chi(Y'[h])=\chi(Y')\chi([h])=\chi(Y')$$
 	for every $h\in H^\circ(k)$.

 	Thus $\chi(Y')=f\in \overline{F}\setminus\{0\}$ by the differential Galois correspondence. Write $\chi=\prod_{i=1}^m \chi_i^{d_i}$ with $d_i\in \bZ$. Then not all $d_i$ are zero, because $\chi$ is non-trivial.
By Lemma \ref{LM:characters}, we have $\de(\chi_i(BY))=f_i\chi_i(BY)$ for $i=1,\ldots,m$. This is an identity in $\overline{F}R$. Taking the quotient mod $\overline{\mathfrak{m}}$ yields $\de(\chi_i(Y'))=f_i\chi_i(Y')$ in $(\overline{F}R)/\overline{\mathfrak{m}}=\overline{F}(R/\mathfrak{m})$. So
 	$$
 	\frac{\de(f)}{f}=\frac{\delta(\chi(Y'))}{\chi(Y')}=\sum_{i=1}^m d_i\frac{\delta(\chi_i(Y'))}{\chi_i(Y')}=\sum_{i=1}^m d_if_i.
 	$$
 	Thus $f_1,\dots,f_m$ are logarithmically dependent over $\overline{F}$. %\rf{I removed ``; a contradiction."}

 	Conversely, assume that $R/F$ is Picard-Vessiot. Then $G$ is the differential Galois group of $R/F$ and $E/\overline{F}$ is a Picard-Vessiot extension with differential Galois group $G^\circ$. Suppose
 	$f_1,\dots,f_m$ are logarithmically dependent, i.e., there are $d_1,\dots,d_m\in \bZ$ not all zero and $f\in \overline{F}\setminus \{0\}$ such that $\sum_{i=1}^m d_if_i=\delta(f)/f$. Write $\chi=\prod_{i=1}^m \chi_i^{d_i}$. Then $\chi$ is a non-trivial character of $G^\circ$. 
 	For $Y'=BY$, using Lemma \ref{LM:characters}, we obtain
 	$$\frac{\delta(\chi(Y'))}{\chi(Y')}=\sum_{i=1}^m d_i\frac{\delta(\chi_i(Y'))}{\chi_i(Y')}=\sum_{i=1}^m d_i f_i=\frac{\de(f)}{f}.$$
 	Thus $\de(\frac{\chi(Y')}{f})=0$, i.e., $\frac{\chi(Y')}{f}=\lambda\in k$. So $\chi(Y')=\lambda f\in \overline{F}$.
 	
 	For every $g\in G^\circ(k)$ we can write $g(Y')=Y'[g]$ for some $[g]\in \Gl_n(k)$, thereby identifying $G^\circ$ with a closed subgroup of $\Gl_{n,k}$. Then
 	$$\chi(Y')=g(\chi(Y'))=\chi(g(Y'))=\chi(Y'[g])=\chi(Y')\chi([g]),$$ 
 	for all $g\in G^\circ(k)$. This implies that $\chi$ is trivial; a contradiction. Therefore, $f_1,\dots,f_m$ are logarithmically independent over $\overline{F}$. 
 \end{proof}

 Note that in case $X(G^\circ)$ is trivial, Lemma \ref{LM:criterion} asserts that a proto-Picard-Vessiot ring is automatically Picard-Vessiot.

 \begin{rem} \label{rem: C1}
 	If $F=k(x)$ and $R/F$ is Picard-Vessiot with connected differential Galois group $G$, then the $\vartheta\in\Hom_{F'}(R,\overline{F})$ in Lemma \ref{LM:criterion} can be chosen such that $B=\vartheta(Y)^{-1}\in \Gl_n(F)$ and therefore also $f_1,\ldots,f_m\in F$.
 \end{rem}
 \begin{proof}
 	By Tsen's theorem the field $F=k(x)$ is a $C_1$-field and so by Steinberg's theorem every torsor for a connected linear algebraic group over $F$ is trivial (\cite[Chapter III, Section 2.3]{Serre:GaloisCohomology}). As $R$ defines a torsor for $G_F$, we see that there exists a morphisms $\vartheta\in\Hom_F(R,F)$. Because $G$ is connected, the differential Galois correspondence implies that the integral closure $F'$ of $F$ in $R$ is equal to $F$.  So $\vartheta\in\Hom_{F'}(R,\overline{F})$ and $\vartheta(Y)^{-1}\in\Gl_n(F)$ as desired.
% 	 the proof of Corollary~\ref{cor: existence of h} shows that ${\color{red}\h}$ has the desired property ${\color{red}\h}(Y)^{-1}Y\in G^\circ(E)=G(E)$.
 \end{proof}
 
 The following example illustrates Lemma \ref{LM:criterion}
 
 \begin{ex} \label{ex: elements for logarithmic test}
 	Let $a,b\in F$ such that $a\in F^\times $ is not a square in $F$. Then $F(\eta)=F[y]/(y^2-a)$ is a differential field extension of $F$. 
 	As in Example \ref{ex: torsor for monomial matrices} we set $$R=F(\eta)[y_1,y_2,y_1^{-1},y_2^{-1}]=F[X,\tfrac{1}{\det(X)}]/(p_1,p_2)=F[Y,\tfrac{1}{\det(Y)}],$$ 
 	where $\de(y_1)=(b+\eta)y_1$, $\de(y_2)=(b-\eta)y_2$,	
 	$$
 	p_1=X_{21}X_{22}-(b^2-a)X_{11}X_{12}, \quad  \quad  p_2=X_{21}X_{12}+X_{22}X_{11}-2bX_{11}X_{12,} $$
 	
 	$$A=\begin{pmatrix}  0 & 1 \\ a+\de(b)-b^2-\frac{\delta(a)b}{2a} & 2b+\frac{\delta(a)}{2a}\end{pmatrix}\in F^{2\times 2},$$
 	and $\de(X)=AX$. Then $R$ is an integral domain and $R/F$ is a differential $G$-torsor for $\de(y)=A y$, where $G\leq\Gl_{2,k}$ is the group of $2\times 2$ monomial matrices (as seen in Example~\ref{ex: torsor for monomial matrices}). As $G^t=1$, it is clear that $R/F$ is proto-Picard Vessiot.
 	
 	The algebraic group $G^\circ$ is the diagonal torus in $\Gl_{2,k}$ and so a basis $\chi_1,\chi_2$ of $X(G^\circ)$ is represented by $X_{11}$ and $X_{22}$.
 	The integral closure $F'$ of $F$ in $R$ equals $F'=F(\eta)$. Define $\vartheta\in\Hom_{F'}(R,\overline{F})$ by $\vartheta(y_1)=\vartheta(y_2)=1$. Then
 	$$
 	\vartheta(Y)=\begin{pmatrix} 1 & 1 \\ b+\eta & b-\eta \end{pmatrix}\in\Gl_2(\overline{F}) \quad \text{ and } \quad \vartheta(Y)\begin{pmatrix}
 		y_1 & 0 \\
 		0 & y_2
 	\end{pmatrix}=Y.
 	$$
% 	Then $P_1(g)=P_2(g)=0$ and so $g$ defines a morphism ${\color{red}\h}\colon R\to \overline{F},\ Y\mapsto g$ of $F$-algebras.
% 	
% 	
% 	As $g\begin{pmatrix}
% 		y_1 & 0 \\
% 		0 & y_2
% 	\end{pmatrix}=Y$, we see that ${\color{red}\h}(Y)^{-1}Y=g^{-1}Y\in G^\circ(E)$.
 	 We have
 	
 	\begin{equation} \label{eq: systems}
 		\de\begin{pmatrix}
 			y_1 & 0 \\
 			0 & y_2
 		\end{pmatrix}=\begin{pmatrix}
 			b+\eta & 0 \\
 			0 & b-\eta
 		\end{pmatrix}\begin{pmatrix}
 			y_1 & 0 \\
 			0 & y_2
 		\end{pmatrix} \text{ and } \de(Y)=AY.	
 	\end{equation}
 Because  $\begin{pmatrix}
 		y_1 & 0 \\
 		0 & y_2
 	\end{pmatrix}=BY$ for $B=\vartheta(Y)^{-1}$, the matrices defining the differential equations in (\ref{eq: systems}) are related via the Gauge transformation defined by $B$, i.e., 	
 	$BAB^{-1}+\delta(B)B^{-1}=\begin{pmatrix}
 		b+\eta & 0 \\
 		0 & b-\eta
 	\end{pmatrix}$.
 	Therefore,
 	\begin{align*}
 		f_1=&\sum_{i,j=1}^2 \frac{\partial \chi_1}{\partial X_{ij}}(I_2)(BAB^{-1}+\delta(B)B^{-1})_{ij}=b+\eta \quad \text{ and }\\
 		f_2=&\sum_{i,j=1}^2 \frac{\partial \chi_2}{\partial X_{ij}}(I_2)(BAB^{-1}+\delta(B)B^{-1})_{ij}=b-\eta.
 	\end{align*}
 	So Lemma \ref{LM:criterion} shows that $R/F$ is Picard-Vessiot if and only if $b+\eta$ and $b-\eta$ are logarithmically independent over $F(\eta)$.   
 \end{ex}

  \subsection{Logarithmic independence via residues}
 
 \label{subsec: logarithmic independence via residues}
 
In the previous section (Lemma \ref{LM:criterion}), we have seen that whether or not a proto-Picard-Vessiot ring is Picard-Vessiot, depends on the logarithmic (in)dependence of certain algebraic functions.
The point of this short section is to present a criterion from \cite{CompointSinger:ComputingGaloisGroupsOfCompletelyReducibleDifferentialEquations} for the logarithmic independence of algebraic functions via differential forms and their residues. 
In Section \ref{sec: Specialization of differential torsors} we will then show that this criterion is preserved under many specializations.
 %\mw{I changed $K$ to $k$ in this section.}
 \medskip
 
Let $F$ be a finite field extension of $k(x)$, i.e., $F/k$ is a function field of one variable.
% Throughout Section \ref{subsec: logarithmic independence via residues} we make the following assumptions:
% \begin{itemize}
% 	\item $K$ is an algebraically closed field of characteristic zero; 
% 	\item $F$ is a finite field extension of $K(x)$, i.e., $F/K$ is a function field of one variable;
% 	\item $f_1,\ldots,f_s\in F$.
% \end{itemize}	
% 
% \medskip
% 
 Recall (\cite[Def. 4.2.10]{Stichtenoth:AlgebraicFunctionFieldsAndCodes}) that the \emph{residue} $\res_P(\omega)$ of a differential form $\omega$ of $F/k$ at a place $P$ of $F/k$ is defined as
 $\res_P(\omega)=\lambda_{-1}\in k$, where $\omega=f\,dt$ with $t\in F$ a uniformizing variable at $P$ and $f=\sum_{i=\ell}^\infty \lambda_it^i$ the $P$-adic expansion of $f\in F$. In other words, $\res_{P}(\omega)=\res_{P,t}(f)$.
 %\rf{$P$ was replaced with $\p$.}
 
 Let $f_1,\ldots,f_m\in F$ and set $\bff=(f_1,\dots,f_m)$. Let $\cP$ be the (finite) set of places of $F/k$ consisting of all poles of the differential forms $f_1dx,\ldots,f_mdx$. 
 Define
 \begin{align*}
 	Z_1(\bff,\cP)&=\left\{(d_1,\dots,d_m)\in \bZ^m\ \bigg|\ \sum_{i=1}^m d_i \res_{P}(f_idx)\in \bZ\ \forall\, P\in \cP \,\mbox{and}\,
 	\sum_{P\in\cP}\sum_{i=1}^m d_i \res_{P}(f_idx)=0\right\},\\
 	Z_2(\bff,\cP)&=\left\{(d_1,\dots,d_m)\in Z_1(\bff,\cP)\  \big| \ \mbox{$\sum_{{P}\in \cP} \left(\sum_{i=1}^m d_i \res_{P}(f_idx)\right) {P}$ is a principal divisor of $F/k$}\right\}.
 \end{align*}
 Note that $Z_1(\bff,\cP)$ and $Z_2(\bff,\cP)$ are submodules of $\bZ^m$ and thus are free $\bZ$-modules. Let $\{(e_{1,\ell},\dots,e_{m,\ell})\mid \ell=1,\dots,n\}$ be a basis of $Z_2(\bff,\cP)$ and for $\ell=1,\ldots,n$ write
 $$\sum_{P\in \cP} \left(\sum_{i=1}^m e_{i,\ell}\res_{P}(f_idx)\right) {P}=(h_\ell) \ \text{ with } \ h_\ell\in F^\times.
 $$
 For $\ell=1,\dots,n$, set
 $
 \omega_\ell=\left(\frac{\delta(h_\ell)}{h_\ell}-\sum_{i=1}^m e_{i,\ell} f_i\right) dx.
 $
 
 \begin{lemma}
 	\label{lemma:logarithmicderivative}
 	The elements $f_1,\dots,f_m$ are logarithmically independent over $F$ if and only if $\omega_1,\dots,\omega_n$ are $\mathbb{Z}$-linearly independent.
 	In particular, if $Z_2(\bff,\cP)=0$, then $f_1,\dots,f_m$ are logarithmically independent over $F$.
 \end{lemma}
 \begin{proof}
 	The proof of this lemma is contained in the proof of \cite[Prop. 2.4]{CompointSinger:ComputingGaloisGroupsOfCompletelyReducibleDifferentialEquations}. For the convenience of the reader we include the details.
 	
 	First assume that $f_1,\ldots,f_m$ are logarithmically independent over $F$. Let $d_1,\ldots,d_n\in\bZ$ be such that $\sum_{\ell=1}^nd_\ell\omega_\ell=0$. By definition of $\omega_\ell$, this implies $\sum_{\ell=1}^nd_\ell\frac{\de(h_\ell)}{h_\ell}=\sum_{\ell=1}^n\sum_{i=1}^md_\ell e_{i,\ell}f_i$. For $h=\prod_{\ell=1}^nh_\ell^{d_\ell}\in F^\times$, we thus have $$\tfrac{\de(h)}{h}=\sum_{\ell=1}^nd_\ell\tfrac{\de(h_\ell)}{h_\ell}=\sum_{i=1}^m\left(\sum_{\ell=1}^nd_\ell e_{i,\ell}\right)f_i.$$
 	The logarithmic independence of $f_1,\ldots,f_m$ implies that $\sum_{\ell=1}^nd_\ell e_{i,\ell}=0$ for $i=1,\ldots,m$.
 	The $\bZ$-linear independence of the $(e_{1,\ell},\dots,e_{m,\ell})$'s thus yields $d_1=\ldots=d_n=0$ as desired.
 	
 	\medskip
 	
 	Conversely, assume that $\omega_1,\ldots,\omega_n$ are $\bZ$-linearly independent. Let $d_1,\ldots,d_m\in\bZ$ and $h\in F^\times$ be such that $\frac{\de(h)}{h}=\sum_{i=1}^md_if_i$. We will show that $(d_1,\ldots,d_m)\in Z_2(\bff,\cP)$. First note that the differential form $\frac{\de(h)}{h}dx=\frac{dh}{h}$ has only simple poles and the residues are all integers. In fact, the residue at a place $P$ of $F/K$ is the $P$-valuation of $h$. (This can, for example, be deduced from \cite[Prop. 4.2.7]{Stichtenoth:AlgebraicFunctionFieldsAndCodes}.)
 	So $\sum_{i=1}^md_i\res_{P}(f_idx)=\res_{P}(\sum_{i=1}^md_if_idx)=\res_{P}(\frac{dh}{h})\in\mathbb{Z}$ for every place $P$ of $F/K$ and 
 	$$(h)=\sum_{P}\res_{P}(\tfrac{dh}{h}){P}=\sum_{P}\sum_{i=1}^md_i\res_{P}(f_idx){P},$$
 	where the sum is taken over all places of $F/K$. However, since $\res_{P}(\sum_{i=1}^md_if_idx)=0$, unless $P$ is a pole of one of the forms $f_idx$, the sum can be restricted to $P\in\cP$. As principal divisors have degree zero (\cite[Theorem 1.4.11]{Stichtenoth:AlgebraicFunctionFieldsAndCodes}), we see that $(d_1,\ldots,d_m)\in Z_1(\bff,\cP)$ and also $(d_1,\ldots,d_m)\in Z_2(\bff,\cP)$. Therefore, there exist $c_1,\ldots,c_n\in\mathbb{Z}$ such that $d_i=\sum_{\ell=1}^nc_\ell e_{i,\ell}$ for $i=1,\ldots,m$.
 	Then
 	\begin{align*}
 		\sum_{\ell=1}^nc_\ell\omega_\ell &=\sum_{\ell=1}^nc_\ell\omega_\ell-\Big(\tfrac{\delta(h)}{h}-\sum_{i=1}^m d_if_i\Big)dx= \\
 		&=\sum_{\ell=1}^nc_\ell\Big(\tfrac{\de(h_\ell)}{h_\ell}-\sum_{i=1}^me_{i,\ell}f_i\Big)dx-\tfrac{\de(h)}{h}dx+\sum_{i=1}^m\sum_{\ell=1}^nc_\ell e_{i,\ell}f_idx=\\
 		&=\frac{\de(h^{-1}\prod_{\ell=1}^nh_\ell^{c_\ell})}{h^{-1}\prod_{\ell=1}^nh_\ell^{c_\ell}}dx=\frac{d(h^{-1}\prod_{\ell=1}^nh_\ell^{c_\ell})}{h^{-1}\prod_{\ell=1}^nh_\ell^{c_\ell}}.
 	\end{align*}
 	By construction, the differential forms $\omega_\ell$ have no simple poles; so $\sum_{\ell=1}^nc_\ell\omega_\ell$ has no simple pole. On the other hand, $\frac{d(h^{-1}\prod_{\ell=1}^nh_\ell^{c_\ell})}{h^{-1}\prod_{\ell=1}^nh_\ell^{c_\ell}}$ has only simple poles. This is only possible if $\sum_{\ell=1}^nc_\ell\omega_\ell=0$.
 	
 	Now the $\mathbb{Z}$-linear independence of the $\omega_\ell$'s yields that all $c_\ell$'s are zero. Hence also all $d_i$'s are zero and the $f_i$'s are logarithmically independent as desired.
 \end{proof}
 
 The following example illustrates Lemma \ref{lemma:logarithmicderivative}
 
 \begin{ex} \label{ex: get elleptic curve}
 	Let $F=k(x,z)$ with $z^2=x^4+x+\alpha$ and $\alpha\in k$.
 	We assume that the discriminant $256\alpha^3-27$ of $x^4+x+\alpha$ is nonzero so that $x^4+x+\alpha$ has four distinct roots $\alpha_3,\alpha_4,\alpha_5,\alpha_6$ in $k$.
 	Let $\beta\in k$ with $\beta^4+\beta+\alpha\neq 0$. Set $b=\frac{(x-\beta)^2}{x^4+x+\alpha}=(\frac{x-\beta}{z})^2$ and $\eta=\frac{x-\beta}{z}$. We will investigate the logarithmic independence of $f_1=b+\eta$ and $f_2=b-\eta$.
 	(This choice of $f_1,f_2$ is motivated by Example \ref{ex: elements for logarithmic test}.)
 	
 	We first need to find the set $\cP$ of poles of $f_1dx$ and $f_2dx$. Note that the projective curve $C'\subseteq \mathbb{P}^2_k$ defined by the affine equation $z^2=x^4+x+\alpha$ has a singularity at infinity.
 	
 	The transformation 
 	$$ x=-\tfrac{v+1}{2u},\quad  z=(\tfrac{v+1}{2u})^2-\tfrac{u}{2}$$
 	yields a smooth model $C\subseteq \mathbb{P}_k^2$ defined by the affine equation $v^2=u^3-4\alpha u+1$. The inverse transformation is given by 
 	$$ u=2(x^2-z), \quad v=-4(x^2-z)x-1.$$
 	These transformations yield an isomorphism between $C\smallsetminus\{(0:1:1), (0:1:0)\}$ and $C'\smallsetminus\{(0:1:0)\}$. The singular point $(0:1:0)$ of $C'$ has two preimages in $C$, namely $(0:1:1)$ and  $(0:1:0)$.
 	The places of $F/k$ are in one-to-one correspondence with $C(k)$. 
 	The inclusion $k(x)\subseteq F$ of function fields corresponds to the morphism
 	$$C\to \mathbb{P}^1_k,\ (u:v:w)\mapsto (v+w: -2u).$$
 	
 	By the Hurwitz genus formula (\cite[Cor. 3.5.6]{Stichtenoth:AlgebraicFunctionFieldsAndCodes}), this morphism has exactly four ramification points. They are $\alpha_3,\alpha_4,\alpha_5,\alpha_6\in \mathbb{A}^1(k)\subseteq\mathbb{P}^1(k)$. Let $P_3,P_4,P_5,P_6$ denote the (unique) places of $F/k$ extending the places of $k(x)$ corresponding to  $\alpha_3,\alpha_4,\alpha_5,\alpha_6$. The point of $C(k)$ corresponding to $P_i$ is $(2\alpha_i^2:-4\alpha_i^3-1:1)$ for $i=3,\ldots,6$.

 	For a place $P$ of $F/k$ let $\nu_P$ denote the corresponding valuation on $F$. Similarly, every $\lambda\in k\cup\{\infty\}$ defines a valuation $\nu_\lambda$ on $k(x)$. 
 	For $a\in k(x)$ we have $\nu_{P_i}(a)=2\nu_{\alpha_i}(a)$ for $i=3,\ldots,6$. For any other place $P$ of $F/k$ (i.e., $P\notin\{ P_3,P_4,P_5,P_6\}$) we have $\nu_{P}(a)=\nu_\lambda(a)$, where $\lambda$ corresponds to the restriction of $P$. With this preparation, one sees immediately that the poles of $b=\frac{(x-\beta)^2}{x^4+x+\alpha}$ are $P_3,P_4,P_5,P_6$ and $\nu_{P_i}(b)=-2$ for $i=3,\ldots,6$.
 	Because $b=\eta^2$, it follows that the poles of $\eta$ are $P_3,P_4,P_5,P_6$ and $\nu_{P_i}(\eta)=-1$ for $i=3,\ldots,6$.

 	There are two places $P_1,P_2$ of $F/k$ that restrict to the infinite place of $k(x)$. They correspond to the points $(0:1:0)$ and $(0:1:1)$ of $C(k)$. Then the poles of $x$ are $P_1, P_2$ and $\nu_{P_i}(x)=-1$ for $i=1,2$.
 	
 	Set $\cP=\{P_1,\ldots,P_6\}$. Then the poles of $bdx$ and $\eta dx$ belong to $\cP$, because outside of $\cP$ neither $b,\eta$ or $x$ has a singularity. As $\nu_{P_i}(\frac{1}{x})=1$ for $i=1,2$ we can use $x^{-1}$ as a uniformizer at $P_1,P_2$. Then $bdx=-bx^2d(x^{-1})$ and so $\nu_{P_i}(bdx)=\nu_{P_i}(-bx^2)=\nu_\infty(-bx^2)=0$. Hence $P_1,P_2$ are not poles of $bdx$.
 	
 	As 
 	$$\nu_{P_i}(z^2)=\nu_{P_i}(x^4+x+\alpha)=2\nu_{\alpha_i}(x^4+x+\alpha)=2,$$
 	for $i=3,\ldots,6$, we see that $z$ is a uniformizer at $P_3,P_4,P_5,P_6$.
 	From $z^2=x^4+x+\alpha$ it follows that $\de(z)=\frac{4x^3+1}{2z}$ and so 
 	$$bdx=\tfrac{(x-\beta)^2 2z}{z^2(4x^3+1)}dz=\tfrac{2(x-\beta)^2}{(4x^3+1)z}dz.$$
 	Therefore $\nu_{P_i}(bdx)=-1$ for $i=3,\ldots,6$. Thus the poles of $bdx$ are $P_3,P_4,P_5,P_6$.
 	
 	For $i=1,2$ we have $\nu_{P_i}(b)=2$ and so $\nu_{P_i}(\eta)=1$. It follows that
 	$$\nu_{P_i}(\eta dx)=\nu_{P_i}(-\eta x^2d(x^{-1}))=\nu_{P_i}(-\eta x^2)=-1$$
 	for $i=1,2$. So $P_1,P_2$ are poles of $\eta dx$. On the other hand,
 	$\eta dx=\frac{2\eta z}{4x^3+1}dz$ and so $\nu_{P_i}(\eta dx)=0$ for $i=3,\ldots,6$.
 	So the poles of $\eta dx$ are $P_1,P_2$.
 	
 	In summary, we see that the poles of $f_1dx$ as well as of $f_2dx$ are exactly $\cP$.
 	
 	\medskip
 	
 	We next compute the residues. As $\nu_{P_i}(x-\alpha_i)=2$ for $i=3,\dots,6$, the $P_i$-adic expansion of $x-\alpha_i$ is of the form  $c_iz^2+\dots$ with $c_i\in k^\times$. Substituting $x=\alpha_i+c_iz^2+\dots$ we find
 	$$
 	bdx=\tfrac{2(x-\beta)^2}{(4x^3+1)z}dz=\tfrac{2(\alpha_i-\beta+c_iz^2+\dots)^2}{4(\alpha_i+c_iz^2+\ldots)^3+1}\tfrac{1}{z}dz=\left ( \tfrac{2(\alpha_i-\beta)^2}{4\alpha_i^3+1}\tfrac{1}{z}+\dots\right)dz.
 	$$
 	Therefore $\res_{P_i}(bdx)=\frac{2(\alpha_i-\beta)^2}{4\alpha_i^3+1}$ for $i=3,\dots,6$. We next compute $\res_{P_i}(\eta dx)$ for $i=1,2$.
 	
 	As $\nu_{P_i}(z^2)=\nu_{P_i}(x^4+x+\alpha)=-4$, we see that $\nu_{P_i}(z)=-2$ for $i=1,2$. So the $P_i$-adic expansion of $z$ is of the form $z=c_i(x^{-1})^{-2}+\ldots$ with $c_i\in K^\times$ for $i=1,2$. We have
 	$$\eta dx=-\eta x^2d(x^{-1})=\tfrac{-(x-\beta)x^2}{z}d(x^{-1})=\tfrac{-(x^{-1})^{-3}+\beta(x^{-1})^{-2}}{c_i(x^{-1})^{-2}+\ldots}d(x^{-1})$$
 	and so $\res_{P_i}(\eta dx)=-1/c_i$.
 	
 	Plugging $z=c_i(x^{-1})^{-2}+\ldots$ into $z^2=x^4+x+\alpha$ we find $(c_i(x^{-1})^{-2}+\ldots)^2=(x^{-1})^{-4}+\ldots$ and so $c_i\in\{1,-1\}$ for $i=1,2$. Suppose $c_1=c_2$. Then $\frac{z}{x^2}$ and $\frac{1}{x}$ are rational functions on $C$ that take the same value in $k$ on $P_1$ and $P_2$ and $F=k(\frac{z}{x^2},\frac{1}{x})$. Thus $P_1=P_2$; a contradiction. So $c_1\neq c_2$.
 	To find the exact value of $c_i$ we can evaluate $\frac{z}{x^2}=1-\frac{2u^3}{(v+1)^2}$ at the point $(0:1:1)$ corresponding to $P_2$. This yields $c_2=1$ and so $c_1=-1$. Thus $\res_{P_1}(\eta dx)=1$ and $\res_{P_2}(\eta dx)=-1$.
 	
 	With this information at hand we can now calculate
 	$$Z_1((f_1,f_2),\cP)=\left\{(d_1,d_2)\in \bZ^2\ \bigg|\ \sum_{i=1}^2 d_i \res_{P}(f_idx)\in \bZ\ \forall\, P\in \cP \,\mbox{and}\,
 	\sum_{P\in\cP}\sum_{i=1}^2 d_i \res_P(f_idx)=0\right\}.$$
 	Note that for $P_i$ ($i=3,\ldots,6$), the first condition is $(d_1+d_2)\frac{2(\alpha_i-\beta)^2}{4\alpha_i^3+1}\in\mathbb{Z}$, whereas for $P_1,P_2$ the first condition is vacuous, because $\res_{P_i}(f_j)\in\mathbb{Z}$ for $i=1,2$.
 	There indeed exist $\alpha,\beta\in k$ such that $\frac{2(\alpha_i-\beta)^2}{4\alpha_i^3+1}\in\mathbb{Q}$ for $i=3,\ldots,6$, for example, $\alpha=0$ and $\beta=\frac{1}{2}$.

 	Set 
 	$
 	r=\prod_{i=3}^6 \frac{(\alpha_i-\beta)^2}{4\alpha_i^3+1}.
 	$
 	A direct calculation shows that
 	$
 	r=\frac{(\beta^4+\beta+\alpha)^2}{256\alpha^3-27}.
 	$ Thus, the first condition mentioned above, implies that $r\in \bQ$ if $d_1+d_2\neq 0$.
 	
 	Assume that $r\notin\mathbb{Q}$. Then
 	$$Z_1((f_1,f_2),\cP)=\{(-d,d)\in\mathbb{Z}^2|\ d\in\mathbb{Z}\}$$
 	and for $(-d,d)\in Z_1((f_1,f_2),\cP)$ we have 
 	$$\sum_{P\in \cP} \left(\sum_{i=1}^2 d_i \res_{P}(f_idx)\right) {P}=-2d{P_1}+2d{P_2}=2d{(P_2-P_1)}$$
 	and so
 	$$Z_2((f_1,f_2),\cP)=\{(-d,d)\in\mathbb{Z}^2|\ 2d(P_2-P_1) \text{ is a principal divisor of } F/k\}.$$
 	
 	Recall (see e.g., \cite[Prop. 6.1.6]{Stichtenoth:AlgebraicFunctionFieldsAndCodes}) that the choice of a point of $C(k)$ defines a bijection between the degree zero divisor class group $\operatorname{Cl}^0(F/k)$ and $C(k)$. As usual, we choose the point to be $(0:1:0)\in C(k)$. Thus the point $(0:1:0)$ corresponding to $P_1$ is the neutral element of the group $C(k)$. Then $2d{(P_2-P_1)}$ is principal, i.e., equal to zero in $\operatorname{Cl}^0(F/k)$, if and only if $2d(0:1:1)=0$ in $C(k)$. 
 	Therefore $Z_2((f_1,f_2),\cP)$ is trivial if and only if the point $(0,1)$ is not a torsion point of the elliptic curve $v^2=u^3-4\alpha u+1$.
 	In this case $f_1,f_2$ are logarithmically independent over $F$ by Lemma~\ref{lemma:logarithmicderivative}. It is not easy to exactly determine the values of $\alpha$ for which $(0,1)$ is a torsion point, but this question will be further discussed in Example \ref{ex: need Jac-open}.
 	
 	If $(0,1)$ is a torsion point, then
 	a basis of $Z_2((f_1,f_2),\cP)$ is given by $e_1=(-e,e)$, where $2e$ is the order of $(0,1)$ in $C(k)$. One can then write $2e{(P_2-P_1)}=(h_1)$ with $h_1\in F^\times$. In this case, $f_1,f_2$ are logarithmically independent if and only if 
 	$$\tfrac{\de(h_1)}{h_1}-(-e)f_1-ef_2=\tfrac{\de(h_1)}{h_1}+2e\eta$$ is nonzero.
 	The latter questions will also be further discussed in Example \ref{ex: need Jac-open}.
 \end{ex}

\subsection{Exponential bounds for linear differential operators} \label{subsec: exponential bound}
 
\label{subsec: Exponential bounds for linear differential operators} 
  
%\mw{In this section $k'$ has been replaced with $k$.}
 
% By considering the vector $\bfv$ of all monomials of degree at most $m$ in the entries of a fundamental solution matrix $Y$, the study of the algebraic relations of degree at most $m$ among the entries of $Y$, reduces to the study of the linear relations satisfied by the entries of $\bfv$. To study the behaviour of the linear relations under specialization, the key idea is to find a useful bound on the degree of the rational functions occurring as the coefficients in a basis of the linear relations.
 
A solution $h$ of a homogeneous scalar linear differential equation over $k(x)$ is an \emph{exponential solution} if its logarithmic derivative $\de(h)/h$ lies in $k(x)$, i.e., $h$ satisfies $\de(h)=ah$, with $a\in k(x)$. The goal of this section is to describe a fairly explicit bound on the degrees of the logarithmic derivatives of the exponential solutions. 

%  
%  To describe this bound and its behaviour under specialization,
%   we need to bound the degree of the logarithmic derivative of exponential solutions of linear differential operators. This is discussed in the following section.
% 

 %Throughout Section \ref{subsec: exponential bound} we denote with $k'$ an algebraically closed field of characteristic zero.

 \medskip 
 
 We will work with the (non-commutative) ring $k(x)[\de]$ of linear differential operators with coefficients in $k(x)$ (as defined in \cite[Section 2.1]{SingerPut:differential}).
 Thus a (scalar) linear differential equation
 $$\de^n(y)+a_{n-1}\de^{n-1}(y)+\ldots+a_0y=0$$
 with $a_0,\ldots,a_n\in k(x)$, can be written more succinctly as $Ly=0$, where
 $$L=\de^n+a_{n-1}\de^{n-1}+\ldots+a_0\in k(x)[\de]$$
 is the corresponding linear differential operator.
 
 A nonzero solution $h$ of $Ly=0$ (in some Picard-Vessiot extension of $k(x)$) is called an \emph{exponential solution} of $Ly=0$ if $\de(h)/h\in k (x)$. The reader is referred to Section 4.1 of \cite{SingerPut:differential} for more background on exponential solutions and their computation.

 Recall that the \emph{degree} of a nonzero rational function $a=\frac{a_1}{a_2}\in k(x)$ with $a_1,a_2,\in k[x]$ relatively prime, is defined as $\deg(a)=\max\{\deg(a_1),\deg(a_2)\}$. By convention, $\deg(0)=-\infty$.

 \begin{defi}
 	\label{defi:exponentialbounds}
 	An integer $N$ is an \emph{exponential bound} for $L y=0$, if
 	$\deg(\de(h)/h)\leq N$ for any exponential solution $h$ of $L y=0$ (belonging to some Picard-Vessiot extension of $k(x)$).
 \end{defi}
 The main point of this section is to describe an exponential bound $N(L)$ for $Ly=0$ that is well-behaved under specialization. This bound is also described in the second paragraph on page 109 of \cite{SingerPut:differential}.

 The elements $u\in k(x)$ of the form $u=\delta(h)/h$ for some exponential solution $h$ of $L y=0$ are exactly the solutions of the Riccati equation associated with $L y=0$ (\cite[Def. 4.6]{SingerPut:differential}).
 Another description of these $u$'s is, as the set of elements of $k(x)$ such that $\de-u\in k(x)[\de]$ is a right-hand factor of $L$.
 
 To describe the bound $N(L)$ we need to introduce some notation from \cite{vanHoeij:Formalsolutionsandfactorization} and \cite{vanHoeij:Factorizationofdifferentialoperatos}.  Set $\bar{\de}=x\de$. Then $k(x)[\de]=k(x)[\bar{\de}]$, i.e., every operator in $\de$ can be rewritten as an operator in $\bar{\de}$. Set
 \[
 {\bE=\bigcup_{r\in \bZ_{>0}} k[x^{-1/r}]}
 \]
 %\rf{$E$ was changed into $\bE$ as $E$ denotes a field in the previous context.}
 and let $e\in \bE$. 
 The \emph{ramification index} $\ram(e)$ of $e$ is the minimal $r$ such that $e\in k[x^{-1/r}]$. 
 For $L\in k(x)[\bar{\de}]$, or more generally, for $L\in  k((x))[\bar{\de}]$ we set
 $S_e(L)=L(\bar{\de}+e)\in k((x^{-1/r}))[\bar{\de}]$ and
 %and we consider $S_e(L)$ as an operator in $k((x^{1/r}))[\bar{\de}]$, where $r=\ram(e)$.
  write
 $$
 S_e(L)=\sum_{i\geq \ell } x^{i/r} N_{i,e,L}(\bar{\de}),
 $$
 where $N_{i,e,L}\in k[T]$ (with $T$ a new variable) and $N_{\ell,e,L}\neq 0$. Following \cite[Section~3.4]{vanHoeij:Formalsolutionsandfactorization} $N_{\ell,e,L}$ is called the \emph{Newton polynomial} of $S_e(L)$ at $0$ for slope $0$. Let $m_{e,0}(L)$ be the multiplicity of $0$ in $N_{\ell,e,L}(T)$.  Following \cite[Def. 3.1]{vanHoeij:Factorizationofdifferentialoperatos}, we define the generalized exponents of $L$:
 
 \begin{defi}
 	\label{defi:Newton Polynomial at 0}
 	An element $e\in\bE$ is a \emph{generalized exponent} of $L$ at $0$ if   $m_{e,0}(L)>0$.
 \end{defi}
 The number $m_{e,0}(L)$ is called the \emph{multiplicity} of $e$ in $L$.
 By (3.3) on page 542 of \cite{vanHoeij:Factorizationofdifferentialoperatos}, we have
 \begin{equation}
 	\label{eq:sum of multiplicities}
 	\sum_{e\in\bE} m_{e,0}(L)=\ord(L) (=n).
 \end{equation}
 In particular, $L$ only has finitely many generalized exponents at $0$.
 
 \begin{ex} \label{ex: generalized exponents for Bessel}
 	Let us determine the generalized exponents for the linear differential operator $L=\de^2+\frac{1}{x}\de+1-(\frac{\alpha}{x})^2\in k(x)[\de]$ corresponding to Bessel's differential equation, where $\alpha$ is an arbitrary element of $k$. Rewriting $L$ in terms of $\bar{\de}$, we have
 	$L=\frac{1}{x^2}\bar{\de}^2+1-(\frac{\alpha}{x})^2$. So, after normalizing (which does not effect the generalized exponents), the relevant operator is $L'=\bar{\de}^2+x^2-\alpha^2$. For $e\in \bE$ we have
 	$$L'(\bar{\de}-e)=(\bar{\de}-e)^2+x^2-\alpha^2=\bar{\de}^2-2e\bar{\de}-x\de(e)+e^2+x^2-\alpha^2.$$ Thus, choosing $e\in \{\alpha,-\alpha\}$, we have $L'(\bar{\de}-e)=\bar{\de}^2-2e\bar{\de}+x^2$, which has Newton polynomial $T^2-2eT$. By (\ref{eq:sum of multiplicities}), the generalized exponents of $L$ are $\alpha$ and $-\alpha$. In case $\alpha=0$, the generalized exponent $e=0$ has multiplicity two.
 \end{ex}

 \begin{rem}
 	\label{rem:exponentialpart}
 	Due to Lemma 3.4 of \cite{vanHoeij:Factorizationofdifferentialoperatos}, if $\bar{\de}-u\in k((x))[\bar{\de}]$ is a right-hand factor of $L\in k(x)[\bar{\de}]\subseteq k((x))[\bar{\de}]$, then $u=e+u'$, where $e\in k[x^{-1}]$ is a generalized exponent of $L$ at $0$ with $\ram(e)=1$ and $u'\in xk[[x]]$. However, this can also be seen directly: If $L=L_1\cdot(\bar{\de}-u)$ is a factorization in $k((x))[\bar{\de}]$, and $u$ is written as $u=e+u'$ with $e\in k[x^{-1}]$ and $u'\in xk[[x]]$, then $S_e(L)=S_e(L_1)\cdot(\bar{\de}-u')$, which has a Newton polynomial vanishing at zero, because for every power of $x$ occurring in $S_e(L_1)u'$, there is a smaller power of $x$ occurring in $S_e(L_1)\bar{\de}$.
 \end{rem}

 In order to define the generalized exponents at a point $\lambda\in\mathbb{P}^1(k)=k\cup\{\infty\}$ different from $0$, we introduce (following \cite[Section 3.4]{vanHoeij:Factorizationofdifferentialoperatos}) the $k$-algebra automorphism
 $$l_\lambda: k(x)[\de] \rightarrow k(x)[\de]$$ as follows: if $\lambda\in k$, then $l_\lambda(x)=x+\lambda$ and $l_\lambda(\de)=\de$, otherwise $l_\infty(x)=1/x$ and $l_\infty(\de)=-x^2\de$. For $e\in\bE$ set $m_{e,\lambda}(L)=m_{e,0}(l_\lambda(L))$.
 \begin{defi}
 	\label{defi:exponential part at p}
 	If $m_{e,\lambda}(L)>0$, then $e$ is called a \emph{generalized exponent of $L$ at $\lambda$} and $m_{e,\lambda}(L)$ is called the multiplicity of $e$ in $L$ at $\lambda$.
 \end{defi}
 Since we are only interested in the factors of $L\in k(x)[\de]$ that are of the form $\de-u$ with $u\in k(x)$, we only need to consider the generalized exponents $e$ of $L$ with $\ram(e)=1$.
 However, generalized exponents with ramification index larger than $1$ are necessary for (\ref{eq:sum of multiplicities}) to hold and this formula is crucial for showing that our exponential bound is preserved under many specializations (Proposition \ref{prop: exponential bound}). We will next explain how Remark~\ref{rem:exponentialpart} can be used to obtain the possible principal parts of an exponential solution of $L$.

 Before proceeding, let us recall some terminology for rational function. The \emph{principal part} of an element $u\in k(x)$ at $\lambda\in k$ is $\sum_{i=\ell}^{-1}u_i(x-\lambda)^i$, when $u=\sum_{i=\ell}^\infty u_i(x-\lambda_i)^i$ is written as an element of $k((x-\lambda))$. The \emph{residue} $\res_\lambda(u)$ of $u$ at $\lambda$ is $\res_\lambda(u)=u_{-1}$.
 The \emph{order} $\ord_\lambda(u)$ of $u$ at $\lambda$ is $\ord_\lambda(u)=-\ell$, where $u_\ell\neq 0$. By convention, $\ord_\lambda(0)=0$.
 
 The residue at infinity is defined as $\res_\infty(u)=\res_0(-\frac{1}{x^2}u(\frac{1}{x}))$, while the order at infinity is $\ord_\infty(u)=\ord_0(u(\frac{1}{x}))$. These definitions are such that the sum of all residues of $u$ is zero (see e.g. \cite[Cor. 4.3.3]{Stichtenoth:AlgebraicFunctionFieldsAndCodes}).
 
 The principal part of $u$ at $\infty$ is $\sum_{i=\ell}^1u_ix^{-i}$, when $u=\sum_{i=\ell}^\infty u_ix^{-i}$ is written as an element of $k((x^{-1}))$. This definition is such that $\res_\infty(u)$ can be determined (as $-u_1$) from the principal part at infinity.
 
 Assume that $e\in\bE$ is a generalized exponent of $L$ at $\lambda$ with $\ram(e)=1$. We set
 \[
 \mpp(e,\lambda)=
 \begin{cases}
 	\frac{l_{-\lambda}(e)}{x-\lambda} & \text{ if } \lambda\in k\\
 	-\frac{l_\infty(e)}{x} & \text{ if } \lambda=\infty. 
 \end{cases}
 \]

 \begin{lemma} \label{lemma: get pp of u}
 	Let $u\in k(x)$ be such that $\de-u$ is a right-hand factor of $L$. Then, for every $\lambda\in \mathbb{P}^1(k)$, the principal part of $u$ at $\lambda$ is of the form $\mpp(e,\lambda)$ for some generalized exponent $e$ of $L$ at $\lambda$ with $\ram(e)=1$.
 \end{lemma}
 \begin{proof}
 	We write $L=L_1\cdot(\de-u)$, with $L_1\in k(x)[\de]$. First assume that $\lambda\in k$. Applying $l_\lambda$ yields
 	$l_\lambda(L)=l_\lambda(L_1)\cdot(\de-\l_\lambda(u))=L_2\cdot(\bar{\de}-xl_\lambda(u))$
 	for some $L_2\in k(x)[\bar{\de}]$.
 	By Remark~\ref{rem:exponentialpart}, $xl_\lambda(u)=e+u'$, where $e\in\bE$ is a generalized exponent of $L$ at $\lambda$ with $\ram(e)=1$ and $u'\in xk[[x]]$. %As $xl_\lambda(u)\in k(x)$, we must have $r=1$, i.e., $\ram(e)=1$.
 	Applying $l_{-\lambda}$ to $xl_\lambda(u)=e+u'$, we find $u=\frac{l_{-\lambda}(e)}{x-\lambda}+\frac{l_{-\lambda}(u')}{x-\lambda}$.
 	Thus $\frac{l_{-\lambda}(e)}{x-\lambda}$ is the principal part of $u$ at $\lambda$.
 	
 	Now let us consider the case $\lambda=\infty$.
 	We have $$l_\infty(L)=l_\infty(L_1)\cdot(-x^2\de-l_\infty(u))=l_\infty(L_1)\cdot(-x\bar{\de}-\l_\infty(u))=L_2\cdot\big(\bar{\de}-\tfrac{-l_\infty(u)}{x}\big).$$
 	Therefore, as above, by Remark \ref{rem:exponentialpart}, we have $\frac{-l_\infty(u)}{x}=e+u'$, where $e$ is a generalized exponent of $L$ at $\infty$ with $\ram(e)=1$ and $u'\in xk[[x]]$. Applying $(l_\infty)^{-1}=l_\infty$, we obtain $u=-\frac{l_\infty(e)+l_\infty(u')}{x}$. This implies the claim.
 \end{proof}
 For $\lambda\in\mathbb{P}^1(k)=k\cup\{\infty\}$ we set
 $$
 \pp(L,\lambda)=\left\{\mpp(e,\lambda) \mid \mbox{$e\in \bE$ is a generalized exponent of $L$ at $\lambda$ with $\ram(e)=1$}\right\}.
 $$
 With $\sing(L)\subseteq k$ we denote the set of all singular points of $L$, i.e., the elements of $k$ that are poles of one of the coefficients $a_i$ of $L$.
 Write $\sing(L)\cup \{\infty\}=\{\lambda_1,\dots,\lambda_m\}$ and define
 \[
 \begin{array}{cccc}
 	\Phi_L: \pp(L,\lambda_1)\times \dots \times\pp(L,\lambda_m) &\longrightarrow & k \\
 	(f_1,\dots,f_m) &\longmapsto & -\sum_{i=1}^m\res_{\lambda_i}(f_i)
 \end{array}
 \]
 and
 $$
 N(L)=\max \{\{0\}\cup (\im(\Phi_L)\cap \bZ)\}+\sum_{\lambda\in \sing(L)\cup\{\infty\}} \max_{f\in \pp(L,\lambda)} \ord_\lambda(f).
 $$

 \begin{lemma} \label{lemma: exponential bound}
 	The number $N(L)$ is an exponential bound for $Ly=0$.
 \end{lemma}
 \begin{proof}
 	Let $h$ be an exponential solution of $Ly=0$ and set $u=\de(h)/h\in k(x)$. We have to show that $\deg(u)\leq N(L)$. As $u$ is a solution of the Riccati equation associated with $Ly=0$, it follows from (4.3) on page 107 of \cite{SingerPut:differential}, that $u$ is of the form
 	\begin{equation}
 		\label{eq:formoffactor}
 		u=\frac{\delta(p_1)}{p_1}+q_1+\frac{p_2}{q_2},
 	\end{equation}
 	where $p_i,q_i\in k[x]$, $\deg(p_2)<\deg(q_2)$, the zeros of $q_2$ are in $\sing(L)$ and the zeros of $p_1$ are not in $\sing(L)$.
 %	\rf{Since $R$ and $S$ denote the rings, I changed $R$ and $S$ into $P_2,Q_2$ here.}
 %	\mw{I changed these to lower case, so that throughout, all polynomials are lower case.}
 	
 	Note that the partial fraction decomposition of $\frac{\delta(p_1)}{p_1}$ is of the form $\sum_{i=1}^d\frac{n_i}{x-\lambda'_i}$, with $n_i\in\mathbb{N}$, $n_1+\ldots+n_d=\deg(p_1)$ and $\lambda'_i\in k\smallsetminus\sing(L)$. The partial fraction decomposition of $\frac{p_2}{q_2}$ is a sum of terms of the form $\sum_{i=1}^d\frac{\lambda'_i}{(x-\lambda)^i}$ with $\lambda'_i\in k$ and $\lambda\in\sing(L)$. The polynomial $q_1$ can be thought of as the contribution at $\infty$.
 	
 	By Lemma \ref{lemma: get pp of u} we have
 	
 	$$\ord_\lambda(u)\leq \max_{f\in \pp(L,\lambda)}\{\ord_\lambda(f)\}$$ for any $\lambda\in\mathbb{P}(k)$.
 	As $\ord_\lambda(q_1+\frac{p_2}{q_2})\leq\ord_\lambda(u)$ for $\lambda\in \sing(L)\cup\{\infty\}$, this yields
 	$\ord_\lambda(q_1+\frac{p_2}{q_2})\leq \max_{f\in \pp(L,\lambda)}\{\ord_\lambda(f)\}$ for $\lambda\in \sing(L)\cup\{\infty\}$. Thus
 	\begin{equation} \label{eq: bound for QRS}
 		\deg(q_1+\tfrac{p_2}{q_2})\leq\sum_{\lambda\in \sing(L)\cup\{\infty\}}\ord_\lambda(q_1+\tfrac{p_2}{q_2})\leq \sum_{\lambda\in \sing(L)\cup\{\infty\}} \max_{f\in \pp(L,\lambda)}\ord_\lambda(f).
 	\end{equation}	
 	To also bound $\deg(p_1)$, let $B\subseteq k$ be the set of poles of $u$ that are not in $\sing(L)$. Then $B$ is the set of zeros of $p_1$. Moreover, $q_1$ and $\frac{p_2}{q_2}$ have residue zero at any point of $B$.
 	Therefore
 	$$\sum_{\lambda\in B}\res_\lambda(u)=\sum_{\lambda\in B}\res_\lambda(\tfrac{\de(p_1)}{p_1})=\deg(p_1).$$
	%\rf{$P_1$ is changed into $p_1$.}
 	As $\sing(L)\cup B\cup\{\infty\}$ contains all poles of $u$, 
 	we have
 	$$
 	0=\sum_{\lambda\in \sing(L)\cup B\cup\{\infty\}} \res_\lambda(u)=\deg(p_1)+\sum_{\lambda\in \sing(L)\cup \{\infty\}}\res_\lambda(u).
 	$$

 	On the other hand, $\res_\lambda(u)=\res_\lambda(f)$ for some $f\in \pp(L,\lambda)$.
 	Therefore $\deg(p_1)\in \im(\Phi_L)\cap\mathbb{Z}$. So $\deg(p_1)\leq \max\{\{0\}\cup(\im(\Phi_L)\cap\mathbb{Z})\}$.
 	Combing this with (\ref{eq:formoffactor}) and (\ref{eq: bound for QRS}) yields the claim of the lemma.
 \end{proof}

 \begin{ex} \label{ex: N(L) for Bessel}
 	Let us determine the exponential bound $N(L)$
 %	\rf{$L(N)$ was changed into $N(L)$.}
 	from Lemma \ref{lemma: exponential bound} for the linear differential operator $L=\de^2+\frac{1}{x}\de+1-(\frac{\alpha}{x})^2\in k(x)[\de]$ corresponding to Bessel's differential equation. Note that $\sing(L)=\{0\}$.
 	We already determined the generalized exponents of $L$ at $0$ in Example \ref{ex: generalized exponents for Bessel}. They are $\alpha$ and $-\alpha$. Thus $\pp(L,0)=\{\frac{\alpha}{x},\frac{-\alpha}{x}\}$.
 	
 	To compute $N(L)$, we also need to determine the generalized exponents at $\infty$. We have $l_\infty(L)=(-x^2\de)^2-x^3\de+1-(\alpha x)^2=x^2\bar{\de}^2+1-(\alpha x)^2$. So, after normalizing, the relevant operator is $L'=\bar{\de}^2+\frac{1}{x^2}-\alpha^2$. For $e\in \bE$, we have $L'(\bar{\de}-e)=\bar{\de}^2-2e\bar{\de}-x\de(e)+e^2+\frac{1}{x^2}-\alpha^2$. Let $i$ and $-i$ denote the roots of the polynomial $T^2+1$ in $k$. For $e=\frac{i}{x}-\frac{1}{2}$ we find
 	$L'(\bar{\de}-e)=\bar{\de}^2-(\frac{2i}{x}-1)\bar{\de}+\frac{1}{4}-\alpha^2$, which has Newton polynomial $2iT$. Similarly, for $e=\frac{-i}{x}-\frac{1}{2}$, one finds the Newton polynomial $-2iT$.
 	Thus, by (\ref{eq:sum of multiplicities}), the generalized exponents of $L$ at $p=\infty$ are $\frac{i}{x}-\frac{1}{2}$ and $\frac{-i}{x}-\frac{1}{2}$. Therefore $\pp(L,\infty)=\{i+\frac{1}{2x},\ -i+\frac{1}{2x}\}$. The image of $\Phi_L\colon \pp(L,0)\times\pp(L,\infty)\to k$ is $\{-\alpha-\frac{1}{2},\ \alpha-\frac{1}{2}\}$.
 	
 	Let us first assume that $\alpha-\frac{1}{2}$ is not an integer. Then $\im(\Phi_L)\cap\mathbb{Z}=\emptyset$ and it follows that $N(L)=1$. More precisely, if $u\in k(x)$ is a solution of the Riccati equation associated with $Ly=0$ and $u=\frac{\de(p_1)}{p_1}+q_1+\frac{p_2}{q_2}$ is written as in (\ref{eq:formoffactor}), then the proof of Lemma \ref{lemma: exponential bound} shows that $\deg(p_1)=0$, i.e., $\frac{\de(P_1)}{P_1}=0$. Moreover, the only possibilities for $q_1+\frac{p_2}{q_2}$ are $i+\frac{\alpha}{x},\ -i+\frac{\alpha}{x},\ i-\frac{\alpha}{x},\ -i-\frac{\alpha}{x}$. These are not solutions of the Riccati equation
 	$$ \de(u)+u^2+\tfrac{1}{x}u+1-(\tfrac{\alpha}{x})^2=0$$
 	associated with $Ly=0$. Thus, there are no exponential solutions in this case.
  	
 	Now assume that $\alpha-\frac{1}{2}\in\mathbb{Z}$. Then $\max\{\{0\}\cup (\im(\Phi_L)\cap \mathbb{Z})\}=|\alpha|-\frac{1}{2}$. Therefore $N(L)=|\alpha|+\frac{1}{2}$. 
 	In this case there are indeed exponential solutions. The interested reader can find their explicit form in the Appendix of \cite{Kolchin:AlgebraicGroupsAndAlgebraicDependence}.
 \end{ex}

 \subsection{Linear relations}
 \label{subsec: Linear relations}
 Let $F$ be a differential field and $R/F$ a Picard-Vessiot ring. For a vector $v\in R^\ell$ satisfying $\de(v)=A'v$ for some $A'\in F^{\ell\times\ell}$ we set
 $$
 \linrel(v,F)=\{p\in F[y_1,\dots,y_\ell] \mid \mbox{$p$ is linear homogeneous and $p(v)=0$}\}.
 $$
 Then $ \linrel(v,F)$ is an $F$-vector space of dimension at most $\ell$; the vector space of linear relations among the entries of $v$.	
 The main goal of this section is to describe, in the case $F=k(x)$, with the help of Section \ref{subsec: exponential bound}, a bound $N$ such that $\linrel(v,F)$ has a basis consisting of elements of the form $a_1y_1+\ldots+a_\ell y_\ell$, with $a_1,\ldots,a_\ell\in k(x)$ and $\deg(a_i)\leq N$ for $i=1,\ldots,\ell$.
 
 By the cyclic vector lemma, any linear differential system $\de(y)=Ay$, with $A\in k(x)^{n\times n}$ is equivalent to a scalar linear differential equation $Ly=0$ with $L=\de^n+a_{n-1}\de^{n-1}+\ldots+a_0\in k(x)[\de]$, i.e., there
 exists a matrix $T\in\Gl_n(k(x))$ such that $T^{-1}AT+\delta(T^{-1})T$ is of the form
 \[
 \begin{pmatrix}
 	0& 1& 0 &  &  \\
 	& \ddots & \ddots & \ddots &  \\
 	&  & \ 0 & 1 & 0 \\
 	& &  & 0 & 1\\
 	-a_0 & -a_1 & \dots & \dots & -a_{n-1}
 \end{pmatrix}\in k(x)^{n\times n}.
 \]
 We call such a $T$ a \emph{transformation matrix} from $\de(y)=Ay$ to $Ly=0$. Furthermore, $\deg(T)$ is defined as the maximum of all degrees of entries of $T$.

 To a linear differential system $\de(y)=Ay$, with $A\in k(x)^{n\times n}$, one can associate, for every $i=1,\ldots,n$, a new linear differential system $\de(y)=(\bigwedge^i A)y$ with $\bigwedge^i A\in k(x)^{{n\choose i}\times {n \choose i}}$. If one works with differential modules (as in \cite[Section 2.2]{SingerPut:differential}) and $M$ is the differential module associated with $\de(y)=Ay$, then $\bigwedge^iM$ is the differential module associated with $\de(y)=(\bigwedge^iA)y$.

 \begin{lemma} \label{lemma: linear relations}
 	Let $v\in R^\ell$ be a solution of $\de(y)=A'y$, where $A'\in k(x)^{\ell\times \ell}$ and $R/k(x)$ is a Picard-Vessiot ring (not necessarily for $\de(y)=A'y$). Let $d$ be the dimension of the $k(x)$-vector space generated by the entries of $v$ and let $T\in k(x)^{m\times m}$ $(m={ \ell \choose d})$ be a transformation matrix from $\de(y)=(\bigwedge^dA')y$ to $Ly=0$, where $L\in k(x)[\de]$. Furthermore, let $N$ be an exponential bound for $L$.
 	
 	Then $\linrel(v,k(x))$ has a $k(x)$-basis consisting of elements of the form $a_1y_1+\ldots+a_\ell y_\ell$, with $a_1,\ldots,a_\ell\in k(x)$ and
 	$\deg(a_i)\leq 2m\deg(T)+m(m-1) N$ for $i=1,\ldots,\ell$.
 \end{lemma}
 \begin{proof}
 	
 	Write $v=(v_1,\ldots,v_\ell)^t\in R^\ell$.
 	Without loss of generality, we may assume that $v_1,\dots,v_d$ are linearly independent over $k(x)$. For $i=d+1,\dots,\ell$, write
 	\begin{equation} \label{eq: vj}
 		v_i=\sum_{j=1}^d a_{i,j} v_j, \,\,a_{i,j}\in k(x).
 	\end{equation}
 	 Then $\big(y_i-\sum_{j=1}^d a_{i,j}y_j\big)_{i=d+1,\dots,\ell}$ is a $k(x)$-basis of $\linrel(v,k(x))$ and our goal is to bound $\deg(a_{i,j})$.
 	Let $V$ be the $k$-subspace of $R^\ell$ generated by $\{g(v)|\ g\in G(k)\}$, where $G$ is the differential Galois group of $R/k(x)$. We claim that $\dim_{k}V=d$. Set
 	$$\operatorname{Rel}(V)=\{(a_1,\ldots,a_\ell)^t\in k(x)^\ell|\ a_1u_1+\ldots+a_\ell u_\ell=0 \ \forall\ u=(u_1,\ldots, u_\ell)^t\in V\}.$$
 	From (\ref{eq: vj}) it follows that $u_i=\sum_{j=1}^d a_{i,j} u_i$ for every $(u_1,\ldots, u_\ell)^t\in V$ and $i=d+1,\ldots,\ell$. On the other hand, $v_1,\ldots,v_d$ are $k(x)$-linearly independent. Thus $\operatorname{Rel}(V)$ is a $k(x)$-vector space of dimension $\ell-d$.
 	
 	We will show (cf. \cite[Prop. 1.5]{BeukersBrownawellHeckman:SiegelNormality}) that also $\dim_{k(x)} \operatorname{Rel}(V)=\ell-\dim_{k}V$. Let $b _1,\ldots,b_r$ be a basis of $V$ and let $B\in R^{\ell\times r}$ be the matrix with columns $b_1,\ldots,b_r$.
 	As $b_1,\ldots,b_r$ are solutions of $\de(y)=A'y$, their $k$-linear independence, implies their $E$\=/linear independence, where $E$ is the field of fractions of $R$ (\cite[Lemma 1.7]{SingerPut:differential}).
 	Thus there exists an $r\times r$-submatrix $C\in\Gl_r(E)$ of $B$. As $g\in G(k)$ acts on $B$ and $C$ via right multiplication with a matrix in $\Gl_r(k)$, we see that $BC^{-1}$ is fixed by $g$. So $BC^{-1}\in k(x)^{\ell\times r}$ by the differential Galois correspondence. We have
 	$$\operatorname{Rel}(V)=\{a\in k(x)^\ell|\ aB=0\}=\{a\in k(x)^\ell|\ aBC^{-1}=0\}.$$
 	Since $BC^{-1}\in k(x)^{\ell\times r}$ has linearly independent columns, we see that $\dim_{k(x)}\operatorname{Rel}(V)=\ell-r=\ell-\dim_k(V)$. Therefore $\dim_{k}V=d$ as claimed.
 	
 	Let $b _1,\ldots,b_d$ be a $k$-basis of $V$.
 	For a vector or a matrix $w$, denote by $w^{(i)}$ the $i$-th row of $w$.
 	For $I=(n_1,\dots,n_d)$ with $1\leq n_1<n_2<\dots<n_d\leq \ell$, set $b_{I}=\det((b_j^{(n_i)})_{1\leq i,j\leq d}) \in R$. A calculation, that is probably best understood by using the exterior power of the associated differential module (\cite[Lemma 2.27]{SingerPut:differential}), shows that
 	\begin{equation} \label{eq: exterior}
 		\delta(b_I)=\sum_{\substack{J=(m_1,\dots,m_d)\\ 1\leq m_1 <m_2<\dots<m_d\leq \ell}} a_{I,J} b_J,
 	\end{equation}
 	where  $a_{I,J}\in k(x)$ and $\bigwedge^d A'=(a_{I,J})\in k(x)^{m\times m}$.
 	That is, the vector $(b_I)_I\in R^{m}$ is a solution of $\delta(y)=(\bigwedge^d A') y$.
 	We extend the definition of the $a_{i,j}$'s from (\ref{eq: vj}) to all values of $i\in\{1,\ldots,\ell\}$ by
 	$$a_{i,j}=\begin{cases}
 		1 \text{ if } i=j,\\
 		0 \text{ otherwise}
 	\end{cases}
 	$$
 	for $i=1,\ldots,d$ and $j=1,\ldots,d$. Then $(\ref{eq: vj})$ is true for all values of $i\in\{1,\ldots,\ell\} $ and so
 	$b_n^{(i)}=\sum_{j=1}^da_{i,j}b_n^{(j)}$ for $i=1,\ldots,\ell$ and $n=1,\ldots,d$, or in matrix form
 	$$\Big(b_j^{(i)}\Big)_{1\leq i\leq \ell \atop 1\leq j\leq d}=\Big(a_{i,j}\Big)_{1\leq i\leq \ell \atop 1\leq j\leq d}\Big(b_j^{(i)}\Big)_{1\leq i\leq d \atop 1\leq j\leq d}.$$
 	For $I=(n_1,\dots,n_d)$ with $1\leq n_1<n_2<\dots<n_d\leq \ell$, we thus have
 	\begin{equation} \label{eq: fo exterior}
 		b_I=c_I b_{(1,2,\dots,d)},
 	\end{equation}
 	with $c_I=\det((a_{n_i,j})_{1\leq i,j\leq d})\in k(x)$. In particular,
 	\begin{equation}
 		\label{eq:coefficients}
 		c_I=\begin{cases}
 			1 & \text{ for } I=(1,2,\dots,d), \\
 			(-1)^{d-j}a_{i,j} & \text{ for } I=(1,\dots,j-1,j+1,\dots,d,i),
 		\end{cases}
 	\end{equation}
 	where $1\leq j \leq d$ and $d<i\leq \ell$.
 	From (\ref{eq: exterior}) and (\ref{eq: fo exterior}) it follows that $b_{(1,2,\dots,d)}\in R$ is exponential over $k(x)$, i.e., $\de(b_{(1,2,\dots,d)})/b_{(1,2,\dots,d)}\in k(x)$.
 	
 	As $T$ transforms $\de(y)=(\bigwedge^d A')y$ into $Ly=0$, the solution $(b_I)_I\in R^m$ of $\de(y)=(\bigwedge^d A')y$ is transformed to a vector $T^{-1}(b_I)_I=(h,\de(h),\ldots,\de^{\ell-1}(h))^t\in R^\ell$ with  $L(h)=0$. Using 
 	(\ref{eq: fo exterior}), we see that $h=b_{(1,2,\dots,d)}c$ with $c\in k(x)$. It follows that $h$ is an exponential solution of $L(y)=0$. Write $\de(h)=a_1h$ with $a_1\in k(x)$. Then also $\de^i(h)=a_i h$ with $a_i\in k(x)$ for $i=1,\ldots,m-1$. So
 	$$
 	T^{-1}(b_I)_I=h\begin{pmatrix} 1 \\ a_1 \\ \vdots \\ a_{m-1} \end{pmatrix} \quad \text{ or } \quad  (b_I)_I=hT\begin{pmatrix} 1 \\ a_1 \\ \vdots \\ a_{m-1} \end{pmatrix} .
 	$$
 	The last equation implies $b_{I}=hT^{(I)}(1,a_1,\ldots,a_{m-1})^t$, where $T^{(I)}$ denotes the $I$-th row of $T$.
 	Combining this with (\ref{eq: fo exterior}) and (\ref{eq:coefficients}) yields
 	\begin{equation} \label{eq: formula for aij}
 		(-1)^{d-j} a_{i,j}=\frac{T^{(I)}\begin{pmatrix}1 \\ \vdots \\ a_{m-1} \end{pmatrix}}{T^{((1,2,\dots,d))}\begin{pmatrix}1 \\ \vdots \\ a_{m-1} \end{pmatrix}},
 	\end{equation}
 	where $I=(1,\dots,j-1,j+1,\dots,d,i)$.
 	
 	We claim that $\deg(a_i)\leq i\deg(a_1)$ for $i=1,\ldots,m-1$. Indeed, writing $a_1=\frac{p_1}{q}$ with $p_1,q\in k[x]$ relatively prime, we will show by induction on $i$, that $a_i=\frac{p_i}{q^i}$ with $p_i\in k[x]$ and $\deg(p_i)\leq i\deg(a_1)$. We have
 	$$a_{i+1}=\delta(a_i)+a_ia_1=\frac{\de(p_i)q^i-ip_iq^{i-1}\de(q)}{q^{2i}}+\frac{p_ip_1}{q^{i+1}}=\frac{\de(p_i)q-ip_i\de(q)+p_ip_1}{q^{i+1}}.$$
 	So $\deg(p_{i+1})=\deg(\de(p_i)q-ip_i\de(q)+p_ip_1)\leq (i+1)\deg(a_1).$

 	Recall that the degree of a rational function satisfies $\deg(ab)\leq \deg(a)+\deg(b)$ and $\deg(a+b)\leq \deg(a)+\deg(b)$ for $a,b\in k(x)$. 
 	From $\deg(a_i)\leq i\deg(a_1)$, we obtain
 	$$\deg(T^{(I)}(1,a_1,\ldots,a_{m-1})^t)\leq\sum_{i=0}^{m-1}(\deg(T)+i\deg(a_1))\leq m\deg(T)+\tfrac{m(m-1)}{2}\deg(a_1).$$
 	Using (\ref{eq: formula for aij}), we find
 	\begin{equation}
 		\label{eq:degreeboundforaij}
 		\deg(a_{i,j})\leq 2m \deg(T)+m(m-1) \deg(a_1).
 	\end{equation}
 	
 	As $a_1=\de(h)/h$, with $h$ an exponential solution of $Ly=0$, this implies the claim of the lemma.
 \end{proof}

\section{Specializations}
\label{sec: Specialization of differential torsors}

 In this section we prove our main specialization result (Theorem \ref{theo: main specialization}). To this end, we first show that various criteria and properties discussed in Section \ref{sec: Some topics in differential Galois theory} are preserved under many specializations.
Throughout Section \ref{sec: Specialization of differential torsors} we make the following assumptions:

\begin{itemize}
	\item $k$ is an algebraically closed field of characteristic zero;
	\item $\B$ is a finitely generated $k$-algebra that is an integral domain;
	\item $\X=\spec(\B)$;
	\item $K$ is the algebraic closure of the field of fractions of $\B$.
%	\item $F/K$ is a function field of one variable, i.e., a finitely generated field extension of transcendence degree one.
\end{itemize}	

\medskip

We usually identify a $c\in \X(k)$ with the corresponding $k$-algebra morphism $c:\B\rightarrow k$. We denote with $b^c$ the image of $b\in\B$ under $c$. For a polynomial or Laurent polynomial $p$ with coefficients in $\B$, we denote with $p^c$ the polynomial obtained from $p$ by applying $c$ to its coefficients. A similar notation applies to fractions of polynomials with coefficients in $\B$ (assuming that the denominator does not specialize to zero under $c$). Moreover, $M^c=\{p^c|\ p\in M\}$ for any set $M$ such that $p^c$ makes sense for any $p\in M$. In a similar spirit, if $v$ is a tuple or matrix such that the specialization $c$ can be applied to the entries of $v$, then $v^c$ is the tuple or matrix obtained from $v$ by applying $c$ to the entries of $v$.

\medskip

We now provide an outline of the proof of the main specialization result (respectively Theorem B from the introduction). Starting from an inclusion of algebraically closed fields $k\subseteq k'$ and a differential equation $\de(y)=Ay$, $A\in k'(x)^{n\times n}$ with differential Galois group $G$ and Picard-Vessiot ring $R/k'(x)$, we first spread out (using Lemma \ref{lemma: spread out PVring}) $R$ and $G$ into nice families $\R$ and $\G$ over a parameter space $\X$. Our goal then is to show that the fibre $\R^c$ over $c$ is Picard-Vessiot for all $c$ in an ad$\times$Jac-open subset of $\X(k)$. This is achieved via two main steps, corresponding to the two main steps of Hrushovski's algorithm. The first main step is to show, using the criterion of Lemma \ref{LM:criterionforprotoPV}, that $\R^c/k(x)$ is proto-Picard-Vessiot for all $c$ in an ad-open subset of $\X(k)$. The key ingredient for this is Theorem \ref{theo:basisunderspecialization}, which, roughly speaking, states that for a fixed degree $d$, a basis of the space of algebraic relations of degree at most $d$ among the entries of a solution matrix for $\de(y)=Ay$ specializes to a basis of the space of algebraic relations of degree at most $d$ among the entries of a solution matrix for $\de(y)=A^cy$ for all $c$ in an ad-open subset of $\X(k)$. Theorem \ref{theo:basisunderspecialization} is proved in Section \ref{subsec: Algebraic relations under specialization}, using Section \ref{subsec: Linear relations} (via a reduction to the study of linear relations) and Section~\ref{subsec: The exponential bound under specialization}, where it is shown that the exponential bound from Section \ref{subsec: Exponential bounds for linear differential operators} is preserved under specialization on an ad-open subset of $\X(k)$.

The second main step of the proof is to find, inside the ad-open subset of $\X(k)$ over which $\R^c/k(x)$ is proto-Picard-Vessiot, an ad$\times$Jac-open subset over which $\R^c/k(x)$ is Picard-Vessiot. The criterion of Lemma \ref{LM:criterion} allows us to decide if a proto-Picard-Vessiot ring is Picard-Vessiot by testing if certain algebraic functions are logarithmically independent. 
%\rf{ It seems that we have not specified what $\R_c$ (or $\mathcal{F}_c$) is. Do we need to replace it with $\R^c$?}
%\mw{You are right that this is not clearly defined. The notation $\mathcal{R}_c$ and $\G_c$ is also used (once) on page 6. The lower index is standard notation for fibers, so I think in the introduction it might be better to leave it as is. It might be a bit confusing to have $\mathcal{R}^c$ and $\mathcal{R}_c$ used here so I changed this and give more of a hint of a definition.}
To apply this criterion, we thus need to show that logarithmic independence of algebraic functions is preserved on an ad$\times$Jac-open subset of the parameter space (Theorem~\ref{theo: main Appendix B}. This is proved in Section \ref{subsec: Logarithmic independence is preserved}, where it is shown, building on specialization results for divisors from Section \ref{subsec: Divisors under specialization}, that the criterion of Section \ref{subsec: logarithmic independence via residues} to test the logarithmic independence of algebraic functions is well-behaved under specialization.

\subsection{The exponential bound under specialization}

\label{subsec: The exponential bound under specialization}
In this section we show that the exponential bound from Lemma \ref{lemma: exponential bound} is preserved on an ad-open subset of the parameter space.

We first note a remark regarding the possibility of extending $\B$ if necessary that will be used repeatedly in what follows.

\begin{rem} \label{rem: extend B basic}
	Assume we are given a family $\mathcal{F}$ over $\X$ and we would like to show that a certain property holds of the fibre $\mathcal{F}_c$ for all $c$ in a
	\begin{itemize}
	\item Zariski open,
	\item ad-open,
	\item Jac-open or
	\item ad$\times$Jac-open 
	\end{itemize}	
	subset of $\X(k)$. Then we can, without loss of generality, replace $\B$ with 
	a finitely generated $\B$-subalgebra of $K$. 
\end{rem}	
\begin{proof}
In the applications below it will always be clear what the family $\mathcal{F}$ is. Let $\B'$ be a finitely generated $\B$-subalgebra of $K$. We assume here that we can base change the family $\mathcal{F}$ over $\X$ to a family $\mathcal{F}'$ over $\X'=\spec(\B')$ such that $\mathcal{F}'_{c'}=\mathcal{F}_{\f(c')}$ for all $c'\in\X'(k)$, where $\f\colon\X'(k)\to \X(k)$ is the morphism corresponding to the inclusion $\B\subseteq\B'$. By Chevalley's theorem, if $\U'$ is a nonempty Zariski open subset of $\X'(k)$, then $\f(\U')$ contains a nonempty Zariski open subset of $\X(k)$. By Lemma \ref{lemma: lift opens}, a similar statement holds for ad-open, Jac-open, and ad$\times$Jac-open subsets of $\X(k)$.
\end{proof}	

%We return to the setup of Notation \ref{notation: setup}.

Let $f\in\B[x]$ be a monic polynomial and let $$\L=\de^n+a_{n-1}\de^{n-1}+\ldots+a_0$$ be a linear differential operator with $a_0,\ldots,a_{n-1}\in \B[x]_f\subseteq K(x)$. So $\L\in K(x)[\de]$
and for any $c\in \X(k)$ we have a linear differential operator
$$\L^c=\de^n+a_{n-1}^c\de^{n-1}+\ldots+a_0^c\in k(x)[\de].$$

\begin{ex}
	Let $\B=k[\alpha]$ (with $\alpha$ transcendental over $k$), $f=x$ and $\L=\de^2+\frac{1}{x}\de+1-(\frac{\alpha}{x})^2\in \B[x]_f[\de]$.
%	\rf{$\B_f$ was changed into $\B[x]_f$.}
	From Example \ref{ex: N(L) for Bessel} we know that $N(\L)=1$ and $N(\L^c)=1$ for any $c\in\X(k)=k$ with $c-\frac{1}{2}\notin \mathbb{Z}$. On the other hand, $N(\L^c)=|c|+\frac{1}{2}$ in case $c-\frac{1}{2}\in\mathbb{Z}$.
	So $N(\L)$ specializes well only on an ad-open subset of $\X(k)$.
\end{ex}

\begin{prop}
	\label{prop: exponential bound}
	There exists an ad-open subset $\U$ of $\X(k)$ such that $N(\L^c)=N(\L)$ for all $c\in \U$. In particular, $N(\L)$ is an exponential bound for $\L^c$ for all $c\in\U$.
\end{prop}
\begin{proof}
	Enlarging $\B$ if necessary (Remark \ref{rem: extend B basic}), we may assume that $\sing(\L)\subseteq \B$ and that all coefficients of all generalized exponents at points in $\sing(\L)\cup\{\infty\}$ belong to $\B$.
	There exists a nonempty Zariski open subset $\U_1$ of $\X(k)$ such that
	for any $c\in \U_1$,
	\begin{enumerate}
		\item $\sing(\L)^c=\sing(\L^c)$ and $\lambda_1^c\neq \lambda_2^c$ for $\lambda_1,\lambda_2\in \sing(\L)$ (one can use resultants for this step);
		\item $\ram(e^c)=\ram(e)$ and $\ord_0(e^c)=\ord_0(e)$ for every generalized exponent $e$ of $\L$ at every point of $\sing(\L)\cup\{\infty\}$;
		\item $m_{e^c,\lambda^c}(\L^c)=m_{e,\lambda}(\L)$ for every $\lambda\in \sing(\L)\cup \{\infty\}$ and every generalized exponent $e$ of $\L$ at $\lambda$.
	\end{enumerate}
	For (iii), note that the units of the ring $\B((x))$ of formal Laurent series over $\B$, are exactly those Laurent series whose lowest nonzero coefficient is a unit in $\B$. So $f$ may not be a unit in $\B((x))$. However, if $b\in\B$ is the lowest nonzero coefficient of $f$, then $f$ is a unit in $\B_{b}((x))$ and therefore $\B[x]_f\subseteq \B_{b}((x))$. It follows that for $\lambda\in\sing(\L)\cup\{\infty\}$ and $e\in \B[x^{-1/r}]$, the operator $S_e(l_\lambda(\L))\in K((x^{1/r}))[\bar{\de}]$ has coefficients in $\B_{b}((x^{1/r}))$. Thus, all that is required to guarantee (iii), is that $c$ does not specialize $b$ or any coefficient of the Newton polynomial of $S_e(l_\lambda(\L))$ for slope $0$ to zero.
	
	Assume that $c\in \U_1$. By (iii), $e^c$ is a generalized exponent of $\L^c$ at $\lambda^c$ with multiplicity $m_{e,\lambda}(\L)$, for every generalized exponent $e$ of $\L$ at $\lambda\in\sing(\L)\cup\{\infty\}$. From formula (\ref{eq:sum of multiplicities}), it follows that every generalized exponent of $\L^c$ at $p^c$ is equal to $e^c$ for some generalized exponent $e$ of $\L$ at $\lambda$. Hence $\pp(\L^c,\lambda^c)=\pp(\L,\lambda)^c$ for every $\lambda\in \sing(\L)\cup\{\infty\}$. Furthermore, $\im(\Phi_{\L^c})=\im(\Phi_{\L})^c$.
	
	In general, $\im(\Phi_\L)^c\cap\mathbb{Z}$ can be larger than $\im(\Phi_\L)\cap\mathbb{Z}$. However, if $\Gamma$ is the subgroup of $\Ga(\B)$ generated by $\{1\}\cup \im(\Phi_\L)$, then $\im(\Phi_\L)^c\cap \bZ=\im(\Phi_\L)\cap \bZ$  for $c\in W_\X(\Ga,\Gamma)$.
	Therefore, for $\U$ an ad-open subset of $\U_1\cap W_\X(\Ga,\Gamma)$ and $c\in \U$, we have
	$\im(\Phi_{\L^c})\cap\bZ=\im(\Phi_\L)\cap\bZ$ and consequently $N(\L^c)=N(\L)$.
\end{proof}

\subsection{Divisors under specialization}
\label{subsec: Divisors under specialization}

Our next goal is to show that logarithmic independence of algebraic functions is preserved under many specializations. To show that the criterion for the logarithmic independence from Lemma \ref{lemma:logarithmicderivative} is preserved under many specializations, we first need to know that principal divisors and residues are preserved under many specializations. %In addition to the assumptions of Section \ref{sec: Specialization of differential torsors} detailed above, in Section \ref{subsec: Divisors under specialization} we assume that

\medskip

%
%\begin{itemize}
%%	\item $k$ is an algebraically closed field of characteristic zero;
%%	\item $\B$ is a finitely generated $k$-algebra that is an integral domain;
%%	\item $\X=\spec(\B)$;
%%	\item $K$ is the algebraic closure of the field of fractions of $\B$;
%	\item $F/K$ is a function field of one variable, i.e., a finitely generated field extension of transcendence degree one.
%\end{itemize}	
%}
%

Let $F/K$ be a function field of one variable, i.e., a finitely generated field extension of transcendence degree one.
Let $C$ be a smooth projective model of the function field $F/K$, i.e., $C$ is a smooth projective (irreducible) curve over $K$ with function field $F$. We fix a closed embedding of $C$ into $\mathbb{P}^n_K=\operatorname{Proj}(K[y_0,\ldots,y_n]).$ So $C$ is defined by a homogeneous prime ideal $I$ of $K[y_0,\ldots,y_n]$. Enlarging $\B$ if necessary, we can assume that $I$ is generated by homogeneous polynomials with coefficients in $\B$, and so $\mathcal{I}=I\cap \B[y_0,\ldots,y_n]$ is a homogeneous prime ideal of $\B[y_0,\ldots,y_n]$ such that $\mathcal{I}\otimes_\B K=I$ in $\B[y_0,\ldots,y_n]\otimes_\B K=K[y_0,\ldots,y_n]$. 

Set $\mathcal{C}=\operatorname{Proj}(\B[y_0,\ldots,y_n]/\mathcal{I})$, i.e., $\mathcal{C}$ is the closed subscheme of $\mathbb{P}^n_\B$ defined by $\mathcal{I}$. Then $\mathcal{C}$ is an integral scheme of finite type over $k$, projective over $\X$. Since $(\B[y_0,\ldots,y_n]/\mathcal{I})\otimes_\B K=K[y_0,\ldots,y_n]/I$, we have $\mathcal{C}\times_\X\spec(K)=C$. Thus the generic fibre of $\mathcal{C}\to \X$ is smooth, geometrically irreducible and of dimension one. These three properties are generic properties. (For geometric irreducibly see \cite[\href{https://stacks.math.columbia.edu/tag/0559}{Tag 0559}]{stacks-project}, for fibre dimension see \cite[\href{https://stacks.math.columbia.edu/tag/05F7}{Tag 05F7}]{stacks-project} and for smoothness see \cite[Prop. 17.7.11]{Grothendieck:EGAIV4}.) Thus there exists a nonempty Zariski open subset $\U$ of $\X$ such that for every $u\in\U$, the fibre $\mathcal{C}_u$ over $u$ is a smooth geometrically irreducible curve (over the residue field at $u$). Hence, replacing $\B$ with a localization of $\B$ if necessary, we may (and henceforth do) assume that $\mathcal{C}_u$ is a smooth geometrically irreducible curve for all $u\in \X$. In line with our notation for specializations we set 
$C^c=\mathcal{C}_c$ for $c\in \X(k)\subseteq \X$. Note that, concretely, $C^c\subseteq \mathbb{P}^n_k$ is the closed subscheme defined by the homogeneous polynomials obtained by applying $c$ to a set of homogeneous generators of $\mathcal{I}$.

Throughout Section \ref{subsec: Divisors under specialization} we assume that $\mathcal{C}$ is as described above. In particular, $\mathcal{C}\subseteq\mathbb{P}_\B^n$ is an integral scheme of finite type over $k$, projective over $\X$ and all fibres $C^c\subseteq\mathbb{P}^n_k$ are smooth, irreducible curves. 

In the sequel, we will usually identify the places of $F/K$ with $C(K)$ (and similarly for $k(C^c)/k$ and $C^c(k)$). The discrete valuation on $F/K$ (or $k(C^c)/k$) corresponding to a point (respectively place) $P$ is denoted $\nu_P$.

For every $c\in\X(k)$ we have a specialization map 
$$\mathcal{C}(\B)\to C^c(k),\ P\mapsto P^c,$$
where $P^c\colon\spec(k)\to\mathcal{C}\times_\X \spec(k)$ is such that the composition of $P^c$ with the projection to $\mathcal{C}$ is the composition of $\spec(k)\xrightarrow{c}\X$ with $P\colon\X\to \mathcal{C}$. 

A rational function $h$ on $\mathcal{C}$ yields a rational function $h^\gen$ on $C$ via the projection $C=\mathcal{C}\times_\X\spec(K)\to \mathcal{C}$. Note that $h^\gen$ is well-defined because  
for a nonempty Zariski open subset $\U$ of $\mathcal{C}$, the inverse image $\U\times_\X\spec(K)$ of $\U$ under the projection $\mathcal{C}\times_\X\spec(K)\to \mathcal{C}$ is nonempty, since the generic point of $\U$ and the point of $\spec(K)$ map to the generic point of $\X$.

Roughly speaking, the following lemma shows that principal divisors are preserved under specialization on a nonempty Zariski open subset. We think this lemma is well-known to the experts, but we were not able to locate a suitable reference.

\begin{lemma} \label{lemma: principal divisors under specialization}
	Let $h$ be a nonzero rational function on $\mathcal{C}$ and let $Z_1,\ldots,Z_r$ and $P_1,\ldots,P_s$ be the zeros and poles of $h^\gen$ in $C(K)$. Assume that $Z_1,\ldots,Z_r,P_1,\ldots,P_s\in\mathcal{C}(\B)\subseteq \mathcal{C}(K)=C(K)$. Then there exists a nonempty Zariski open subset $\U$ of $\X(k)$ such that for every $c\in \U$
	\begin{itemize}
		\item $h$ yields a rational function $h^c$ on $C^c$ via the projection $C^c=\mathcal{C}\times_\X\spec(k)\to \mathcal{C}$; 
		\item $Z_1^c,\ldots,Z_r^c,P_1^c,\ldots,P_s^c$ are all distinct;
		\item $Z_1^c,\ldots,Z_r^c$ is the set of zeros of $h^c$ and $P_1^c,\ldots,P_s^c$ is the set of poles of $h^c$;
		\item $\nu_{Z_i^c}(h^c)=\nu_{Z_i}(h^\gen)$ and $\nu_{P_i^c}(h^c)=\nu_{P_i}(h^\gen)$ for all $i$.
	\end{itemize}
	In particular, if $(h^\gen)=\sum_{i=1}^rd_iZ_i+\sum_{j=1}^se_jP_j$, then $(h^c)=\sum_{i=1}^rd_iZ^c_i+\sum_{j=1}^se_jP^c_j$.
\end{lemma}
\begin{proof}
	Assume that $h$ is given by a regular function on a nonempty Zariski open subset $\VV$ of $\mathcal{C}$. For $h^c$ to be well-defined, we need to know that $\mathcal{V}\times_\X\spec(k)\subseteq \mathcal{C}\times_\X\spec(k)=C^c$ is nonempty.
	Since $\mathcal{V}\to \X$ is a dominant morphism of schemes of finite type over $k$, the image contains a nonempty Zariski open subset $\U_1$ by Chevalley's theorem. Then $\mathcal{V}\times_\X\spec(k)$ is nonempty for every $c\in\U_1(k)$ because there is an element of $\mathcal{V}$ mapping to $c\in \X$. This takes care of the first item in the lemma. 
	
	For the remaining three items, we choose a linear form $\ell=a_0y_0+\ldots+a_ny_n\in K[y_0,\ldots,y_n]$ such that the hyperplane $H\subseteq\mathbb{P}^n_K$ defined by $\ell=0$ does not contain any of the $Z_i$'s or $P_i$'s. Enlarging $\B$ if necessary, we may assume that $\ell\in \B[y_0,\ldots,y_n]$. In fact, after localizing $\B$ if necessary, we can assume that $(a_0,\ldots,a_n)\in \B^{n+1}$ is a column of an invertible matrix in $\B^{(n+1)\times(n+1)}$. Thus, after a linear change of coordinates on $\mathbb{P}_\B^n$, we can assume that $\ell$ is given by $\ell=y_0$. 
	
	Let $\U_0$ be the open subscheme of $\mathbb{P}^n_\B$ where $x_0$ does not vanish. So $\U_0$ can be identified with $\A^n_\B=\spec(\B[y_1,\ldots,y_n])$. Recall (see \cite[Chapter 7, Section d]{Milne:AlgebraicGroupsTheTheoryOfGroupSchemesOfFiniteTypeOverAField} or \cite[Prop.~3.4]{Strickland:FormalSchemesAndFormalGroups}) that a point $P\in\mathbb{P}^n_\B(\B)$, i.e., a morphism $\X\to \mathbb{P}^n_\B$ (over $\B$) can be identified with a submodule $L$ of $\B^{n+1}$ such that $L$ is a direct summand (i.e., there exists a submodule $M$ of $\B^{n+1}$ such that $\B^{n+1}=L\oplus M$) and $L$ is locally free of rank $1$. The point $P$ lies in $\U_0(\B)\subseteq \mathbb{P}^n_\B$ (i.e., $\X\to \mathbb{P}^n_\B$ factors through $\U_0\hookrightarrow \mathbb{P}^n_\B$) if and only if $L$ has a basis vector of the form $(b_0,\ldots,b_n)\in \B^n$ with $b_0=1$. As the $Z_i$'s and $P_i$'s (considered as elements of $\mathbb{P}^n_K(K)$) do not lie in $H(K)$, we see that, after replacing $\B$ with a localization of $\B$, we may assume that all the $Z_i$'s and $P_i$'s (considered as elements of $\mathbb{P}^n_\B(\B)$) lie in $\U_0(\B)$.
	
	Consider the open subscheme $\mathcal{C}_0=\mathcal{C}\cap \U_0$ of $\mathcal{C}$. Then $\mathcal{C}_0$ identifies with a closed subscheme of $\U_0=\A^n_\B$ and the $Z_i$'s and $P_i$'s identify with elements of $\A^n_\B(\B)=\B^n$. On the Zariski open subset of $\X(k)$ where the product of suitable differences of coordinates of the $Z_i$'s and $P_i$'s does not vanish, the elements $Z_1^c,\ldots,Z_r^c,P_1^c,\ldots,P_s^c$ are all distinct. 
	
	Let $Z=(z_1,\ldots,z_n)\in \B^n\subseteq K^n$ be one of the $Z_i$'s. Assume that $h^\gen$ has a zero of multiplicity $m\geq 1$ at $Z$, i.e., $m=\nu_Z(h^\gen)$. So $h^\gen$ is contained in the $m$-th power $\m_{C,Z}^m$ of the maximal ideal $\m_{C,Z}$ of the local ring $\mathcal{O}_{C,Z}.$ 
	
	Let $C_0$ be the open subscheme of $C$ where $y_0$ does not vanish, in other words, $C_0=\mathcal{C}_0\times_\X\spec(K)$. Then $K(C)=K(C_0)$ and $\mathcal{O}_{C,Z}=\mathcal{O}_{C_0,Z}$. As the maximal ideal $\m_{C_0,Z}$ of $\mathcal{O}_{C_0,Z}$ is generated by the images of $y_1-z_1,\ldots,y_n-z_n$, we see that $\m_{C_0,Z}^m$ is generated by the set $M$ consisting of the images of $m$-fold products of elements from $\{y_1-z_1,\ldots,y_n-z_n\}$. So we can write $h^\gen=f_1m_1+\ldots+f_dm_d$ with $f_1,\ldots,f_d\in \mathcal{O}_{C_0,Z}\subseteq K(C_0)$ and $m_1,\ldots,m_d\in M$. Enlarging $\B$ if necessary, we may assume that $f_1,\ldots,f_d$ come from rational functions on $\mathcal{C}$, i.e., $f_i=h_i^\gen$ for some rational function $h_i$ on $\mathcal{C}$ for $i=1,\ldots,d$. 
	
	Recall that $k(\B)$ denotes the field of fractions of $\B$. Note that the function field $k(\mathcal{C})$ of $\mathcal{C}$ (considered as a scheme over $k$) agrees with the function field $k(\B)(\mathcal{C}_{k(\B)})$ of $\mathcal{C}_{k(\B)}=\mathcal{C}\times_\X\spec(k(\B))$ (considered as a scheme over $k(\B)$) and the latter is contained in the function field $K(C)$ (considered as a scheme over $K$). In fact, $K(C)=k(\B)(\mathcal{C}_{k(\B)})\otimes_{k(\B)}K$. So the identity $h^\gen=f_1m_1+\ldots+f_dm_d$ that holds in $K(C)$, can be interpreted as an identity
	\begin{equation} \label{eq: specialize to get order}
		h=h_1m_1+\ldots+h_dm_d
	\end{equation} that holds in $k(\mathcal{C})$.
	
	We already observed in the beginning of the proof that a rational function $h$ on $\mathcal{C}$ can be specialized to a rational function $h^c$ on $C^c$ for all $c$ in some nonempty Zariski open subset of $\X(k)$. To specialize (\ref{eq: specialize to get order}), we would like to know that if $h^\gen\in\mathcal{O}_{C_0,\z}$, then $h^c\in\mathcal{O}_{C_0^c,Z^c}$, where $C_0^c$ is the open subscheme of $C^c$ where $y_0$ does not vanish, i.e., $C_0^c=\mathcal{C}_0\times_\X\spec(k)$. 
	
	But $h^\gen\in\mathcal{O}_{C_0,Z}$ means that $h$ can locally be described  by a fraction$\frac{p}{q}$ of two polynomials $p,q\in \B[y_1,\ldots,y_n]$ such that $q$ does not vanish at $Z$. Thus, on the Zariski open subset of $\X(k)$, where $q(Z)\in \B$ does not vanish, the rational function $h^c$ is locally described by the fraction $\frac{p^c}{q^c}$ and so $h^c\in\mathcal{O}_{C_0^c,Z^c}$.
	
%	?? Notation about specialization used above, has this been defied before ??
	
	Since the maximal ideal $\m_{C_0^c,Z^c}$ of $\mathcal{O}_{C_0^c,Z^c}$ is generated by the images of $y_1-c(z_1),\ldots,y_n-c(z_n)$, the specialization
	$h^c=h_1^cm_1^c+\ldots+h^c_dm^c_d$ of (\ref{eq: specialize to get order}) shows that $h^c\in\m_{C_0^c,Z^c}^m$. Therefore $\nu_{Z^c}(h^c)\geq m=\nu_{Z}(h^\gen)$ for all $c$ in a nonempty Zariski open subset of $\X(k)$.
	
	From this inequality for the zeros, we can immediately derive a similar inequality for the poles, because a point is a pole of multiplicity $m$ of a rational function if and only if it is a zero of multiplicity $m$ of the inverse of the rational function. So there exists a nonempty Zariski open subset $\U_2$ of $\X(k)$ such that for all $c\in\U_2$
	\begin{itemize}
		\item $h^c$ is well-defined;
		\item $Z_1^c,\ldots,Z_r^c,P_1^c,\ldots,P_s^c$ are all distinct;
		\item  $\nu_{{Z}_i^c}(h^c)\geq\nu_{{Z}_i}(h^\gen)$ for $i=1,\ldots,r$; in particular, $Z_1^c,\ldots,Z_r^c$ are zeros of $h^c$;
		\item $\nu_{P^c_i}(h^c)\leq\nu_{P_i}(h^\gen)$ for all $i=1,\ldots,s$; in particular, $P_1^c,\ldots,P_s^c$ are poles of $h^c$.
	\end{itemize}
	
	Note that the statement of the lemma is trivial in case $h$ is constant (i.e., $h$ lies in $K$) since in that case there are no zeros or poles. We may thus assume that $h$ is not constant. Under this assumption we will now show that
	\begin{equation} \label{eq: equality of degrees}
		[K(C):K(h^\gen)]=[k(C^c):k(h^c)]
	\end{equation}
	for all $c$ in a nonempty Zariski open subset of $\X(k)$. Since $h^\gen$ is transcendental over $K$, the field extension $K(C)/K(h^\gen)$ is finite and we can choose $\eta\in K(C)$ such that $K(C)=K(h^\gen,\eta)$. Let $q\in K(h^\gen)[y]$ be the minimal polynomial of $\eta$ over $K(h^\gen)$. Replacing $\eta$ with a $K(h^\gen)$-multiple of $\eta$ if necessary, we can assume that $q\in K[h^\gen][y]$. Enlarging $\B$ if necessary, we can indeed assume that $q\in \B[h^\gen][y]$. Let $p\in\B[x,y]$ be such that $p(h^\gen,y)=q$. By Gauss's lemma, $p\in K[x,y]$ is irreducible. In fact, since $p$ is monic as a polynomial in $y$, we see that $p$ is a prime element of $\B[x,y]$. So $\mathcal{D}=\spec(\B[x,y]/(p))$ is an integral scheme and $\mathcal{D}_K=\mathcal{D}\times_\X\spec(K)=\spec(K[x,y]/(p))$ has function field $K(h^\gen,\eta)=K(C)$.

	Let $\mathcal{V}=\spec(\B[a_1,\ldots,a_m])$ be an open affine subscheme of $\mathcal{C}$. Then $$\mathcal{V}_K=\mathcal{V}\times_\X\spec(K)=\spec(\B[a_1,\ldots,a_m]\otimes_\B K)$$ is a nonempty affine open subscheme of $C=\mathcal{C}\times_\X\spec(K)$. Moreover $\mathcal{D}_K$ and $\mathcal{V}_K$ are birationally equivalent (over $K$). In other words, the fields of fractions of
	$\B[\overline{x},\overline{y}]\otimes_\B K$ and $\B[a_1,\ldots,a_m]\otimes_\B K$ are isomorphic as field extensions of $K$, where $\overline{x}$ and $\overline{y}$ stand for the images of $x$ and $y$ in $\B[x,y]/(p)$ respectively.
	Such an isomorphism and its inverse is given by fractions of polynomials with coefficients in $K$. Enlarging $\B$ if necessary, we may assume that all coefficients belong to $\B$. Then there exist nonzero elements $b\in \B[\overline{x},\overline{y}]$ and $a\in \B[a_1,\ldots,a_m]$ such that $\B[\overline{x},\overline{y}]_b$ and $\B[a_1,\ldots,a_n]_a$ are isomorphic over $\B$. 
%	\rf{$\B[\overline{x},\overline{y}]_d$ is changed into $\B[\overline{x},\overline{y}]_b$.}
	In other words, there exist an open subscheme $\mathcal{D}'$ of $\mathcal{D}$ and an open subscheme $\mathcal{V}'$ of $\mathcal{C}$ such that $\mathcal{D}'$ and $\mathcal{V}'$ are isomorphic over $\B$. Note that under this isomorphism, the rational function $h^\gen$ on $\mathcal{V}'$ corresponds to $\overline{x}$.
	
	By Chevalley's theorem, there exists a nonempty Zariski open subset $\U_3$ of $\X(k)$ such that $\U_3$ is contained in the image of the vertical maps in the commutative diagram
	$$
	\xymatrix{
		\mathcal{D}' \ar_{\pi_1}[rd] \ar^-\simeq[rr] & & \mathcal{V'}. \ar^{\pi_2}[ld] \\
		& \X &
	}
	$$
	For $c\in \U_3\subseteq \X(k)\subseteq \X$, the fibre $\pi_2^{-1}(c)$ is a nonempty open subscheme of
	$C^c$ and therefore has function field $k(C^c)$. On the other hand, the fibre $\pi_2^{-1}(c)$ is an open subscheme of the spectrum of $(\B[x,y]/(p))\otimes_\B k=k[x,y]/(p^c)$. 
	The isomorphism $\mathcal{D}'\simeq\mathcal{V}'$ induces an isomorphism $\pi_1^{-1}(c)\simeq \pi_2^{-1}(c)$. Under the corresponding isomorphism of functions fields, $h^c\in k(C^c)$ corresponds to the image of $x$ in the field of fractions of $k[x,y]/(p^c)$. The degree of the field of fractions of $k[x,y]/(p^c)$ over the image of $x$ is the degree of $p$ as a polynomial in $y$, which equals $[K(C):K(h^\gen)]$. Thus (\ref{eq: equality of degrees}) holds for all $c\in \U_3$.
	
	Set $\U=\U_2\cap \U_3$. By \cite[Theorem 1.4.11]{Stichtenoth:AlgebraicFunctionFieldsAndCodes}, the degree of the zero divisor and the degree of the pole divisor of $h^\gen$ is equal to $[K(C):K(h^\gen)]$ and similarly for $h^c$ in place of $h^\gen$. Thus
	\begin{align*}
		[K(C):K(h^\gen)]=\sum_{i=1}^r\nu_{Z_i}(h^\gen)\leq \sum_{i=1}^r\nu_{Z^c_i}(h^c)\leq [k(C^c):k(h^c)]=[K(C):K(h^\gen)]
	\end{align*}
	for all $c\in \U$. This implies that $\nu_{Z_i^c}(h^c)=\nu_{Z_i}(h^\gen)$ for $i=1,\ldots,r$ and that $h^c$ has no zeros other than $Z_1^c,\ldots,Z_r^c$. A similar argument applies for the poles instead of the zeros. We conclude that $\U$ has the desired properties.	 
\end{proof}

Roughly speaking, the following lemma shows that residues are preserved under specialization on a nonempty Zariski open subset.

\begin{lemma} \label{lemma: residues under specialization}
	Let $h$ and $t$ be rational functions on $\mathcal{C}$ such that $t^\gen$ is a uniformizer at $P\in\mathcal{C}(\B)\subseteq\mathcal{C}(K)=C(K)$. Then there exists a nonempty Zariski open subset $\U$ of $\X(k)$ such that for all $c\in \U$
	\begin{itemize}
		\item $h^c, t^c$ and  $\res_{P,t^\gen}(h^\gen)^c$ are well-defined;
		\item $t^c$ is a uniformizer at $P^c$ and
		\item $\res_{P^c,t^c}(h^c)=\res_{P,t^\gen}(h^\gen)^c$.
	\end{itemize}
\end{lemma}
\begin{proof}
	Note that uniformizers are characterized by the property that their valutation is one. It thus follows from Lemma \ref{lemma: principal divisors under specialization} that there exists a nonempty Zariski open subset $\U_1$ of $\X$ such that $h^c$ and $t^c$ are well-defined and that $t^c$ is a uniformizer at $P^c$ for all $c\in \U_1$.
	
	If $\nu_{P}(h^\gen)\geq 0$, then $\res_{P,t^\gen}(h^\gen)=0$ and the claim follows from Lemma \ref{lemma: principal divisors under specialization}. We may thus assume that $\nu_P(h^\gen)=-m$, with $m\geq 1$.
	
	From the uniqueness of the $P$-adic representation, it follows that there exist uniquely determined elements $a_{-m},\ldots,a_{-1}\in K$ such that $\nu_P(h^\gen-\sum_{i=1}^ma_{-i}(t^\gen)^{-i})\geq 0$ and that in this situation $a_{-1}=\res_{P,t^\gen}(h^\gen)$. Enlarging $\B$ if necessary, we may assume that $a_{-m},\ldots,a_{-1}\in\B$. By Lemma \ref{lemma: principal divisors under specialization}, there exists a nonempty Zariski open subset $\U$ of $\X(k)$ such that for every $c\in \U$ the rational function $f=h-\sum_{i=1}^ma_{-i}t^{-i}$ on $\mathcal{C}$ has a well-defined specialization $f^c=h^c-\sum_{i=1}^ma^c_{-i}(t^c)^{-i}$ with $\nu_{P^c}(f^c)\geq 0$. It follows that $a_{-1}^c=\res_{P^c,t^c}(h^c)$.
\end{proof}

The following lemma explains the origin of the Jac-open subset that occurs in our main specialization result (Theorem \ref{theo: main specialization}).
We denote with $\Div^0(-)$ the group of divisors of degree zero.

\begin{lemma} \label{lemma: get ab open set}
	%Let $\mathcal{C}$ be a projective scheme over $\X$ such that $C=\mathcal{C}\times_\X\spec(\K)$ is an irreducible non-singular curve. 
	Let $P_1,\ldots,P_r\in \mathcal{C}(\B)\subseteq\mathcal{C}(K)=C(K)$. Then there exists an ad$\times$Jac-open subset $\U$ of $\X(k)$ such that for all $c\in \U$ and all $d_1,\ldots,d_r\in\mathbb{Z}$, if $d_1P_1+\ldots+d_rP_r\in\Div^0(C)$ is not principal, then $d_1P_1^c+\ldots+d_rP^c_r\in\Div^0(C^c)$ is not principal.	
\end{lemma}
\begin{proof}
	Let $J$ be the Jacobian of $C$ and let $\varphi\colon C\to J$ be a canonical embedding. For more background on Jacobian varieties see e.g., \cite{Milne:AbelianVarieties} or \cite[Chapter 9]{Bosch:NeronModels}. Enlarging $\B$ if necessary (Remark~\ref{rem: extend B basic}), there exists a commutative separated group scheme $\mathcal{J}$ of finite type over $\B$ and a closed embedding $\f\colon\mathcal{C}\to \mathcal{J}$ over $\X$ such that 
	\begin{itemize}
		\item $\f_K=\varphi$ (in particular, $\mathcal{J}$ is of Jacobian type)  and
		\item there exists a nonempty Zariski open subset $\U_1$ of $\X(k)$ such that the base change $\f^c\colon C^c\to \mathcal{J}^c$ of $\f$ via $c\colon \spec(k)\to \X$ is a canonical embedding of $C^c$ into its Jacobian $\mathcal{J}^c$ for every $c\in\U_1$. 
	\end{itemize}
	For every $c\in \U_1$ we have a commutative diagram 
	$$
	\xymatrix{
		\mathcal{C}(\B) \ar@{^{(}->}[r] \ar[d] & \mathcal{J}(\B) \ar[d] \\
		C^c(k) \ar@{^{(}->}[r]  & \mathcal{J}^c(k)	
	}
	$$
	where the vertical maps are the specialization maps.	
	Let $\Gamma$ be the subgroup of $\mathcal{J}(\B)$ generated by $P_1,\ldots,P_r\in \mathcal{C}(\B)\subseteq\mathcal{J}(\B)$ and set $\mathcal{V}=W_\X(\mathcal{J},\Gamma)$ (notation as in Section~\ref{subsec: Open subsets defined by commutative group schemes}). 
	
	Then $\U=\U_1\cap \mathcal{V}$ is an ad$\times$Jac-open subset of $\X(k)$. 
	Let $c\in \U$ and assume that $d_1P_1+\ldots+d_rP_r\in\Div^0(C)$ is not principal. Then $d_1P_1+\ldots+d_rP_r\in\mathcal{J}(\B)\subseteq\mathcal{J}(K)=J(K)$ is nonzero because $\Div^0(C)\to J(K)$ is surjective with kernel the subgroup of principal divisors. Since the specialization map $\mathcal{J}(\B)\to \mathcal{J}^c(k)$ is injective on $\Gamma$, we see that $d_1P^c_1+\ldots+d_rP^c_r\in\mathcal{J}^c(k)$ is nonzero.
	Therefore $d_1P^c_1+\ldots+d_rP^c_r\in \Div^0(C^c)$ is not principal.
\end{proof}

\subsection{Logarithmic independence under specialization}
\label{subsec: Logarithmic independence is preserved}

In this section we show that logarithmic independence is preserved on an Ad$\times$Jac-open subset of the parameter space.

\medskip

Let $F$ be a finite field extension of $K(x)$. So $F/K$ is a function field of one variable, considered as a differential field via $\delta=\frac{d}{dx}$. Let $\eta \in F$ be such that $F=K(x,\eta)$ and the minimal polynomial of $\eta$ over $K(x)$ has coefficients in $\B[x]$. Let $f\in \B[x]$ be a monic polynomial  such that $\B[x]_f[\eta]$ is a differential subring of $F$. 
%\rf{Should we mention that $\B$ needs to be enlarged if necessary here?}
%\mw{In the proof of Theorem 4.26, where Theorem 4.10 is applied, it is said that $\B$ can be extended to achieve this. So I don't think we need to say it here.}
%
% we prove the main result of Appendix B.
%
% Throughout Section \ref{subsec: Logarithmic independence is preserved} we make the following assumptions:
%\begin{itemize}
%	\item $k$ is an algebraically closed field of characteristic zero;
%	\item $\B$ is a finitely generated $k$-algebra that is an integral domain;
%	\item $\X=\spec(\B)$;
%	\item $K$ is the algebraic closure of the field of fractions of $\B$;
%	\item $f\in \B[x]$ is a monic polynomial;
%	\item $F$ is a finite field extension of $K(x)$, so $F/K$ is a function field of one variable, considered as a differential field via $\delta=\frac{d}{dx}$;
%	\item $\eta \in F$ is such that $F=K(x,\eta)$, the minimal polynomial of $\eta$ over $K(x)$ has coefficients in $\B[x]$ and $\B[x]_f[\eta]$ is a differential subring of $F$.
%\end{itemize}	

Let $p\in \B[x,y]\subseteq K(x)[y]$ be the minimal polynomial of $\eta$ over $K(x)$. 
By Gauss's lemma $p\in K[x,y]$ is irreducible and by 
the Bertini-Noether theorem (\cite[Prop.~9.4.3]{FriedJarden:FieldArithmetic}) there exists a nonempty Zariski open subset $\U$ of $\X(k)$ such that $p^c\in k[x,y]$ is irreducible for all $c\in \U$. Then
$$\B[x]_f[\eta]\otimes_{\B[x]_f}k(x)=(\B[x]_f[y]/(p))\otimes_{\B[x]_f}k(x)=k(x)[y]/(p^c)$$
is a finite field extension of $k(x)$ for all $c\in\U$. Replacing $\B$ with a localization of $\B$, we mays assume that $\U=\X(k)$. For $c\in \X(k)$ we set 
$$F^c=\B[x]_f[\eta]\otimes_{\B[x]_f}k(x)$$ and for $h\in \B[x]_f[\eta]$ we set $h^c=h\otimes 1\in F^c$. The following lemma explains that this notation is consistent with the notation from Section \ref{subsec: Divisors under specialization}. Let $C$ be a smooth projective model of $F/K$
and let $\mathcal{C}$ and $C^c$ be as described in Section \ref{subsec: Divisors under specialization}.

\begin{lemma} \label{lemma: comparison}
	After enlarging $\B$ if necessary, there exists a nonzero element $b\in \B[x]_f[\eta]$ such that $\spec((\B[x]_f[\eta])_b)$ is isomorphic (over $\B$) to an open subscheme of $\mathcal{C}$. Moreover, there exists a nonempty Zariski open subset $\U$ of $\X(k)$ such that for all $c\in \U$ we have $F^c\simeq k(C^c)$ and for $h\in \B[x]_f[\eta]$ the rational function $h^c\in k(C^c)$ as defined in Lemma~\ref{lemma: principal divisors under specialization} corresponds to $h^c=h\otimes 1$ as defined above.
\end{lemma}
\begin{proof}
	The argument is quite similar to the argument used towards the end of the proof of Lemma \ref{lemma: principal divisors under specialization}. We therefore only replicate the main points.
	
	Let $\mathcal{D}=\spec(\B[x,y]/(p))$ be the closed subscheme of $\mathbb{A}^2_\B$ defined by $p$.
	Then $\mathcal{D}$ is an integral scheme of finite type over $k$ such that the geometric generic fibre $\mathcal{D}_K=\mathcal{D}\times_\X\spec(K)=\spec(K[x,y]/(p))$ of $\mathcal{D}\to\X$ is birationally equivalent (over $K$) to $C$. 
	
	Let $\mathcal{V}=\spec(\B[a_1,\ldots,a_m])$ be an open affine subscheme of $\mathcal{C}$. Then $$\mathcal{V}_K=\mathcal{V}\times_\X\spec(K)=\spec(\B[a_1,\ldots,a_m]\otimes_\B K)$$ is a nonempty affine open subscheme of $C=\mathcal{C}\times_\X\spec(K)$ and $\mathcal{D}_K$ and $\mathcal{V}_K$ are birationally equivalent (over $K$). In other words, the fields of fractions of
	$\B[\overline{x},\overline{y}]\otimes_\B K$ and $\B[a_1,\ldots,a_m]\otimes_\B K$ are isomorphic as field extensions of $K$, where $\overline{x}$ and $\overline{y}$ stand for the images of $x$ and $y$ in $\B[x,y]/(p)$ respectively.
	Enlarging $\B$ if necessary, we can find elements  $b\in \B[\overline{x},\overline{y}]$ and $a\in \B[a_1,\ldots,a_m]$ such that
	$\B[\overline{x},\overline{y}]_b$ and $\B[a_1,\ldots,a_n]_a$ are isomorphic over $\B$. 	
	
	Localizing both sides one more time, we get an isomorphism between $\B[\overline{x},\overline{y}]_{bf}$ and $\B[a_1,\ldots,a_m]_{a'}$ for some $a'\in\B[a_1,\ldots,a_m]$. This establishes the first claim of the lemma since $\mathcal{V'}=\spec(\B[a_1,\ldots,a_m]_{a'})$ is an open subscheme of $\mathcal{C}$ and $$\B[\overline{x},\overline{y}]_{bf}=(\B[x,y]/(p))_{bf}=(\B[x][\eta])_{bf}=(\B[x]_f[\eta])_b.$$
	By Chevalley's theorem, there exists a nonempty Zariski open subset $\U$ of $\X(k)$ such that $\U$ is contained in the image of the vertical maps in the commutative diagram
	$$
	\xymatrix{
		\spec((\B[x]_f[\eta])_b) \ar_{\pi_1}[rd] \ar^-\simeq[rr] & & \mathcal{V'}. \ar^{\pi_2}[ld] \\
		& \X &
	}
	$$
	For $c\in \U\subseteq \X(k)\subseteq \X$ the fibre $\pi_1^{-1}(c)$ is the spectrum of $(\B[x]_f[\eta])_b\otimes_\B k=(k[x,y]/(p^c))_{(bf)^c}$; a ring with field of fractions $F^c=k(x)[y]/(p^c)=\B[x]_f[\eta]\otimes_{\B[x]_f}k(x)$.
	On the other hand, the fibre $\pi_2^{-1}(c)=\mathcal{V}'\times_\X\spec(k)$ is an open subscheme of $\mathcal{C}\times_\X\spec(k)=C^c$ and therefore has function field $k(C^c)$. Thus the function fields $F^c$ and $k(C^c)$ of the fibres are isomorphic. On each fibre, $h^c$ is the interpretation of $h$ as a rational function on the fibre.	
\end{proof}

In the sequel, we will consider elements of $\B[x]_f[\eta]$ as rational functions on $\mathcal{C}$ as explained in Lemma \ref{lemma: comparison}. Note that if $h\in \B[x]_f[\eta]$ is interpreted as a rational function on $\mathcal{C}$, then $h^\gen$ is simply $h$, considered as an element of $F$. In this situation, we will therefore usually write $h$ instead of $h^\gen$. 

\begin{lemma} \label{lemma: residues of differentials under specialization}
	Let $h\in\B[x]_f[\eta]$ and $P\in \mathcal{C}(\B)\subseteq C(K)$. Then there exists a nonempty Zariski open subset $\U$ of $\X(k)$ such that $\res_P(h dx)^c=\res_{P^c}(h^c dx)$ for all $c\in \U$.
\end{lemma}
\begin{proof}
	After enlarging $\B$ and replacing $f$ with a multiple of $f$ if necessary, we can write $h dx=h_1 dt$, where $h_1,t\in \B[x]_f[\eta]$ and $t$ is a uniformizer at $P$. We then have $h=h_1\de(t)$ and so $h^c=h_1^c\de(t^c)$, i.e., $h^cdx=h_1^cdt^c$ for all $c\in \X(k)$.
	
	By Lemma \ref{lemma: residues under specialization}, there exists a nonempty Zariski open subset $\U$ of $\X(k)$ such that $t^c$ is a uniformizer at $P^c$ and $\res_{P^c,t^c}(h_1^c)=\res_{P,t}(h_1)^c$ for all $c\in \U$. Therefore,
	$$\res_P(hdx)^c=\res_P(h_1dt)^c=\res_{P,t}(h_1)^c=\res_{P^c,t^c}(h_1^c)=\res_{P^c}(h_1^cdt^c)=\res_{P^c}(h^cdx).$$
\end{proof}

We next show that the $\mathbb{Z}$-modules $Z_1$ and $Z_2$ (as defined in Section \ref{subsec: logarithmic independence via residues}) are preserved under many specializations.

\begin{lemma} \label{lemma: Z1 and Z2 under specialization}
	Let $\bff=(f_1,\ldots,f_m)$ be a tuple of elements of $\B[x]_f[\eta]$ and let $\cP$ be the set of poles of the differential forms $f_1dx,\ldots,f_mdx$ on $C$. Assume that $\cP\subseteq\mathcal{C}(\B)\subseteq\mathcal{C}(K)=C(K)$.
	Then there exists an ad-open subset $\U$ of $\X(k)$ such that for all $c\in \U$
	\begin{itemize}
		\item $\cP^c$ is the set of poles of the differential forms $f_1^cdx,\ldots,f_m^c dx$ on $C^c$ and
		\item $Z_1(\bff^c,\cP^c)=Z_1(\bff,\cP)$.
	\end{itemize}
	Moreover, there exists an ad$\times$Jac-open subset $\mathcal{V}$ of $\X(k)$ such that, in addition to the above two properties, we also have $Z_2(\bff^c,\cP^c)=Z_2(\bff,\cP)$ for all $c\in\mathcal{V}$.
\end{lemma}
\begin{proof}
	To resolve the first item, let us first discuss how the poles of a differential form $\omega=h dx$ are determined by $h$ and $x$. To see if $\omega$ has a pole at a place $P$, one can write $\omega=h_1 dt$ with $t=t_P$ a uniformizer at $P$ and then see if $\nu_{P}(\omega)=\nu_{P}(h_1)$ is negative. But $hdx=h_1dt$ entails $h_1=h\de_t(x)$, where $\de_t$ is the derivation determined by $\de_t(t)=1$. So $\nu_{P}(h_1)=\nu_{P}(h)+\nu_{P}(\de_t(x))$. 
	Using \cite[Prop. 4.2.7]{Stichtenoth:AlgebraicFunctionFieldsAndCodes}, we see that 
	$$\nu_P(\de_t(x))=\begin{cases} \nu_P(x)-1\ \text{ if } \ \nu_P(x)\neq 0, \\
		\nu_{P}(x-a_P)-1\ \text{ if } \ \nu_P(x)= 0, \end{cases}$$
	where $a_P$ denotes the value of $x$ at $P$, i.e., the image of $x$ in the residue field $\mathcal{O}_P/\m_P$. 
	Thus $\omega=hdx$ has a pole at $P$ if and only if ($\nu_P(x)\neq 0$ and $\nu_P(h)+\nu_P(x)-1\leq -1$) or ($\nu_P(x)=0$ and $\nu_P(h)+\nu_P(x-a_P)-1\leq -1$).
	
	Now let $h$ be one of the $f_i$'s. After enlarging $\B$ if necessary (Remark \ref{rem: extend B basic}), we may assume that all zeros and poles of $x$ and $h$ lie in $\mathcal{C}(\B)$. Then, by Lemma \ref{lemma: principal divisors under specialization}, there exists a nonempty Zariski open subset $\U_1$ of $\X(k)$ such that the principal divisors defined by $x$ and $h$ are preserved under all specializations belonging to $\U_1$. Similarly, by the same lemma, and after enlarging $\B$ if necessary, we can also assume that for every $c\in \U_1$ and every pole $P$ of $h$ we have $\nu_P(x-a_P)=\nu_{P^c}((x-a_P)^c)=\nu_{P^c}(x^c-a_{P^c})$.

	In particular, for $c\in \U_1$, every zero or pole of $x^c\in k(C^c)$ is the specialization of a zero or pole of $x\in F/K$. For a place $Q$ of $k(C^c)/k$ that is not the specialization of a zero or pole of $x$ or of a pole of $h$, 
	we then have 
	$\nu_Q(x^c)=0$ and $\nu_Q(h^c)+\nu_Q(x^c-a_Q)-1\geq 0+0=0$.
	So, by the above consideration, $Q$ is not a pole $h^cdx$.
	
	By construction, for those places $Q$ of $k(C^c)/k$ that come from zeros or poles of $x$ or from poles of $h$, all the relevant valuations are preserved under specialization and so we see that for $c\in \U_1$ the poles of $h^cdx$ are exactly the specializations of the poles of $hdx$.
	Thus we can find a nonempty Zariski open subset $\U_2$ of $\X(k)$ such that the first item of the lemma is satisfied for all $c\in\U_2$.
	
	By Lemma \ref{lemma: residues of differentials under specialization}, there exists a nonempty Zariski open subset $\U_3$ of $\X(k)$ such that $\res_{P}(f_idx)^c=\res_{P^c}(f_i^cdx)$ for all $P\in\cP$ and $i=1,\ldots,m$. Thus $Z_1(\bff,\cP)\subseteq Z_1(\bff^c,\cP^c)$ for all $c\in\U_2\cap\U_3$. For the reverse inclusion, let $\Gamma$ be the finitely generated subgroup of $(\B,+)$ generated by $1$ and $\res_P(f_idx)$ $(P\in\cP,\ i\in\{1,\ldots,m\})$. Let $\U$ be a nonempty ad-open subset contained in $\U_2\cap\U_3\cap W_\X(\Ga,\Gamma)$.
	
	Let $c\in \U$ and  $(d_1,\ldots,d_m)\in Z_1(\bff^c,\cP^c)$, i.e.,  $d_{P,c}=\sum_{i=1}^md_i\res_{P^c}(f_i^c dx)\in \mathbb{Z}$ for all $P\in\cP$ and $\sum_{P\in\cP}d_{P,c}=0$.
	Then $$\left(\sum_{i=1}^md_i\res_{P}(f_i dx)-d_{P,c}\right)^c=\sum_{i=1}^md_i\res_{P}(f_i dx)^c-d_{P,c}=\sum_{i=1}^md_i\res_{P^c}(f_i^c dx)-d_{P,c}=0$$ for all $P\in\cP$.
	As $c\in W_\X(\Ga,\Gamma)$, we can conclude that $\sum_{i=1}^md_i\res_{P}(f_i dx)-d_{P,c}=0$. In particular, $\sum_{i=1}^md_i\res_{P}(f_i dx)\in\mathbb{Z}$. Moreover, $\sum_{P\in\cP}\sum_{i=1}^md_i\res_{P}(f_i dx)=\sum_{P\in\cP}d_{P,c}=0$. Thus $(d_1,\ldots,d_m)\in Z_1(\bff,\cP)$ and so $Z_1(\bff,\cP)=Z_1(\bff^c,\cP^c)$ for all $c\in \U$.
	
	\medskip
	
	Our next goal is to find a nonempty Zariski open subset $\U_4$ of $\X(k)$ such that $Z_2(\bff,\cP)\subseteq Z_2(\bff^c,\cP^c)$ for all $c\in \U_2\cap\U_3\cap W_\X(\Ga,\Gamma)\cap\U_4$.
	
	Let $(e_{1,\ell},\ldots,e_{m,\ell})_{1\leq \ell\leq n}$ be a basis of the $\mathbb{Z}$-module $Z_2(\bff,\cP)$ and for $\ell=1,\ldots,n$ write 
	$\sum_{P\in\cP}e_{\ell,P}P=(h_\ell)$, with $h_\ell\in F^\times$ and $e_{\ell,P}=\sum_{i=1}^me_{i,\ell}\res_{P}(f_idx)\in\mathbb{Z}$. Replacing $f$ with a mulitple and enlarging $\B$ if necessary, we may assume that $h_1,\ldots,h_n,h_1^{-1},\ldots,h_n^{-1}\in \B[x]_f[\eta]$. By Lemma \ref{lemma: principal divisors under specialization}, there exists a nonempty Zariski open subset $\U_4$ of $\X(k)$ such that $(h_\ell^c)=\sum_{P\in \cP}e_{\ell,P}P^c$ for $\ell=1,\ldots,n$ and all $c\in \U_4$.
	
	Let $c\in\U_2\cap\U_3\cap W_\X(\Ga,\Gamma)\cap\U_4$ and  $(d_1,\ldots,d_m)\in Z_2(\bff,\cP)$. Then there exist $a_1,\ldots,a_n\in\mathbb{Z}$ such that $d_i=\sum_{\ell=1}^na_\ell e_{i,\ell}$ for $i=1,\ldots,m$.
	Moreover, since $\sum_{i=1}^md_i\res_{P}(f_idx)\in \mathbb{Z}$, we have 
	$$\sum_{i=1}^md_i\res_{P}(f_idx)=\left(\sum_{i=1}^md_i\res_{P}(f_idx)\right)^c=\sum_{i=1}^md_i\res_P(f_idx)^c=\sum_{i=1}^md_i\res_{P^c}(f_i^cdx)$$
	for all $P\in\cP$. Therefore
	\begin{align*}
		((h_1^{a_1}\ldots h_n^{a_n})^c)& = ((h_1^c)^{a_1}\ldots((h_n)^c)^{a_n})=\sum_{\ell=1}^na_\ell(h_\ell^c)=\sum_{\ell=1}^n\sum_{P\in  \cP}a_\ell e_{\ell,P}P^c=\\
		&=\sum_{\ell=1}^n\sum_{P\in \cP}\sum_{i=1}^ma_\ell e_{i,\ell}\res_{P}(f_idx)P^c=\sum_{P\in\cP}\sum_{i=1}^md_i\res_P(f_idx)P^c=\\
		&=\sum_{P\in\cP}\sum_{i=1}^md_i\res_{P^c}(f^c_idx)P^c.
	\end{align*}
	This shows that $(d_1,\ldots,d_m)\in Z_2(\bff^c,\cP^c)$. So $Z_2(\bff,\cP)\subseteq Z_2(\bff^c,\cP^c)$ for all $c\in \U_2\cap\U_3\cap W_\X(\Ga,\Gamma)\cap\U_4$.
	
	For the reverse inclusion, we can use Lemma \ref{lemma: get ab open set} to find an ad$\times$Jac-open subset $\mathcal{V}_1$ of $\X(k)$ such that for all $c\in \mathcal{V}_1$ and all $d_P\in\mathbb{Z},\ (P\in\cP)$, if $\sum_{P\in\cP}d_PP\in\Div^0(C)$ is not principal, then $\sum_{P\in\cP}d_{P}{P}^c \in\Div^0(C^c)$ is not principal.
	For all $c\in \U_2\cap\U_3\cap W_\X(\Ga,\Gamma)\cap\U_4\cap\mathcal{V}_1$ we then have 
	$Z_2(\bff,\cP)=Z_2(\bff^c,\cP^c)$.
\end{proof}

We are now prepared to show that logarithmic independence is preserved on an Ad$\times$Jac-open subset.

\begin{theo} \label{theo: main Appendix B}
	Let $f_1,\ldots,f_m\in \B[x]_f[\eta]$ be logarithmically independent over $F$. Then there exists an ad$\times$Jac-open subset $\U$ of $\X(k)$ such that $f_1^c,\ldots,f_m^c$ are logarithmically independent over $F^c$ for all $c\in \U$.
\end{theo}
\begin{proof}
	Set $\bff=(f_1,\ldots,f_m)$ and 
	let $\cP$ be the set of poles of $f_1,\ldots,f_m\in F$. Fix a basis  $\{(e_{1,\ell},\ldots,e_{m,\ell})|\ \ell=1,\ldots,n\}$ of $Z_2(\bff,\cP)$ and for $\ell=1,\ldots,n$ write 
	\begin{equation} \label{eq: for hl}
		\sum_{P\in \cP} \left(\sum_{i=1}^m e_{i,\ell}\res_P(f_idx)\right)P=(h_\ell) \ \text{ with } \ h_\ell\in F^\times.
	\end{equation}
	After enlarging $\B$ and replacing $f$ with a multiple if necessary, we can assume that
	\begin{itemize}
		\item the set $\cP$ belongs to $\mathcal{C}(\B)\subseteq\mathcal{C}(K)=C(K)$ and
		\item $h_1,\ldots,h_n, h_1^{-1},\ldots,h_n^{-1}\in \B[x]_f[\eta]$.
	\end{itemize}	
	Let $d$ be the degree of $\eta$ over $K(x)$. As $a_\ell=\frac{\delta(h_\ell)}{h_\ell}-\sum_{i=1}^m e_{i,\ell}f_i\in \B[x]_f[\eta]$ for $\ell=1,\ldots,n$, we can find $r,s\in\nn$ and $b_{i,j,\ell}\in\B$ such that
	$$a_\ell=\frac{1}{f^r}\left(\sum_{i=1}^s \sum_{j=1}^{d} b_{i,j,\ell} x^i \eta^j\right)$$
	for $\ell=1,\ldots,n$. Since $f_1,\ldots,f_m$ are logarithmically independent over $F$, it follows from Lemma \ref{lemma:logarithmicderivative} that the differential forms $a_1dx,\ldots,a_ndx$ are $\mathbb{Z}$-linearly independent. Note that the latter condition is equivalent to: If $d_1,\ldots,d_n\in\mathbb{Z}$ such that $d_1b_{i,j,1}+\ldots+d_nb_{i,j,n}=0$ for all $1\leq i\leq s,\ 1\leq j\leq d$, then $d_1=\ldots=d_n=0$.
	
	Let $\Gamma$ be the subgroup of $\Ga(\B)=(\B,+)$ generated by all $b_{i,j,\ell}$'s and let $\U_1=W_\X(\Ga,\Gamma)$ be the ad-open subset of $\X(k)$ defined by $\Gamma$. By Lemma \ref{lemma: Z1 and Z2 under specialization} there exists an ad$\times$Jac-open subset $\U_2$ of of $\X(k)$ such that $\cP^c$ is the set of poles of $\bff^c$ and $Z_2(\bff^c,\cP^c)=Z_2(\bff,\cP)$ for all $c\in \U_2$. Moreover, by Lemmas \ref{lemma: principal divisors under specialization} and \ref{lemma: residues of differentials under specialization}, we can find a nonempty Zariski open subset $\U_3$ of $\X(k)$ such that the principal divisors defined by $h_\ell,\ (\ell=1,\ldots,n)$ are preserved under specialization and $\res_P(f_idx)^c=\res_{P^c}(f_i^cdx) \ (i=1,\ldots,m,\ P\in\cP)$ for all $c\in\U_3$. 
	
	Let $\U$ be an ad$\times$Jac-open subset of $\X(k)$ contained in $\U_1\cap\U_2\cap \U_3$. We claim that $f_1^c,\ldots,f_m^c$ are logarithmically independent over $F^c$ for all $c\in\U$. 
	
	Let $c\in\U$. Using that $c\in \U_3$, we can specialize (\ref{eq: for hl}) to
	$$\sum_{P\in \cP} \left(\sum_{i=1}^m e_{i,\ell}\res_{P^c}(f^c_idx)\right)=(h^c_\ell)$$
	for $\ell=1,\ldots,n$. Furthermore, $a_\ell^c=\left(\frac{\delta(h_\ell)}{h_\ell}-\sum_{i=1}^m e_{i,\ell}f_i\right)^c=\frac{\delta(h^c_\ell)}{h^c_\ell}-\sum_{i=1}^m e_{i,\ell}f^c_i$.
	Using that $Z_2(\bff^c,\cP^c)=Z_2(\bff,\cP)$, the criterion of Lemma \ref{lemma:logarithmicderivative} states that $f_1^c,\ldots,f_m^c$ are logarithmically independent over $F^c$ if and only if $a_1^c,\ldots,a_n^c$ are $\mathbb{Z}$-linearly independent.
	Note that the degree of $\eta^c$ over $k(x)$ equals $d$, the degree of $\eta$ over $K(x)$ (cf. the discussion before Lemma~\ref{lemma: comparison}). As 
	$$a_\ell^c=\frac{1}{(f^c)^r}\left(\sum_{i=1}^s \sum_{j=1}^{d} b_{i,j,\ell}^c x^i (\eta^c)^j\right),$$ we see that $a_1^c,\ldots,a_n^c$ are $\mathbb{Z}$-linearly independent if and only if $d_1b^c_{i,j,1}+\ldots+d_nb^c_{i,j,n}=0$ for $d_1,\ldots,d_n\in\mathbb{Z}$, $1\leq i\leq s,\ 1\leq j\leq d$ implies $d_1=\ldots=d_n=0$. But $d_1b^c_{i,j,1}+\ldots+d_nb^c_{i,j,m}=0$ implies, by the injectivity of $\Gamma\to k,\ \gamma\mapsto \gamma^c$, that $d_1b_{i,j,1}+\ldots+d_nb_{i,j,n}=0$.
	Thus $a_1^c,\ldots,a_n^c$ are $\mathbb{Z}$-linearly independent and $f_1^c,\ldots,f_m^c$ are logarithmically independent as desired.
\end{proof}

\begin{rem} \label{rem: eta=1 logarithmic independence}
	In case $\eta=1$, i.e., $F=K(x)$, the Jac-open subset is not needed, i.e., the set $\U$ in Theorem \ref{theo: main Appendix B} can be chosen to be ad-open.
\end{rem}
\begin{proof}
	In this case we can choose $\mathcal{C}=\mathbb{P}^1_\B$ and so $C^c=\mathbb{P}^1_k$ for all $c\in \X(k)$. In this case Lemma \ref{lemma: get ab open set} is trivial because all degree zero divisors are principal. Thus Lemma \ref{lemma: get ab open set} does not create a Jac-open set and also the subsequent results (Lemma \ref{lemma: Z1 and Z2 under specialization} and Theorem \ref{theo: main Appendix B}) can make do without the Jac-open set.
\end{proof}

\subsection{The setup for specializing Picard-Vessiot rings}

Let $k\subseteq k'$ be an inclusion of algebraically closed fields. In Lemma \ref{lemma: spread out PVring} we have seen how a Picard-Vessiot ring $R/k'(x)$ can be spread out to a differential torsor $\R/\B[x]_f$. In this section we discuss some first properties of $\R$ and its specializations $R^c$.
We now describe our framework for specializing Picard-Vessiot rings.

\begin{notation}
	\label{notation: setup}
	\mbox{}%We are mainly interested in the following setup.	
	\begin{itemize}
		\item $k$ is an algebraically closed field of characteristic zero;
		\item $\B$ is a finitely generated $k$-algebra and an integral domain;
		\item $\X=\spec(\B)$;
		\item $K$ is the algebraic closure of the field of fractions $k(\B)$ of $\B$;
		\item $\G$ is an affine group scheme of finite type over $\B$;
		\item $f\in \B[x]$ is a monic polynomial and $\B[x]_f$ is a differential ring with respect to $\de=\frac{d}{dx}$;
		\item  $\cA\in\B[x]_f^{n\times n}$;
		\item $\R/\B[x]_f$ is a differential $\G$-torsor for $\de(y)=\cA y$ with fundamental matrix $\Y\in\Gl_n(\R)$ such that $\R$ is flat over $\B[x]_f$;
		
		\item $R^c=\R\otimes_{\B[x]_f}k(x)$, where $c\in \X(k)=\Hom_k(\B,k)$ and the tensor product is formed using the extension of $c\colon \B\to k$ to $c\colon \B[x]_f\to k(x)$ determined by $c(x)=x$;
		\item $Y^c=\Y\otimes 1\in\Gl_n(R^c)$ is the image of $\Y$ in $R^c$;
		\item $A^c\in k(x)^{n\times n}$ is obtained from $\cA$ by applying $c\colon \B[x]_f\to k(x)$ to the coefficients of $\cA$;
		\item $G^c=\G_c$ is the algebraic group over $k$ obtained from $\G$ by base change via $c\colon \B\to k$;
	
		\item $R^\gen=\R\otimes_{\B[x]_f}K(x)$ is $\de$-simple;
		\item $Y^\gen=\Y\otimes 1\in \Gl_n(R^\gen)$ is the image of $\Y$ in $R^\gen$;
		\item $G^\gen=\G_K$ is the algebraic group over $K$ obtained from $\G$ by base change via $\B\to K$.
		\end{itemize}
	Furthermore, for $h\in \R$, denote with $h^c=h\otimes 1\in R^c$ the image of $h$ in $R^c$.
	Similarly, $h^\gen=h\otimes 1\in R^\gen$ denotes the image of $h$ in $R^\gen$.
\end{notation}

\noindent {\bf For the remainder of Section \ref{sec: Specialization of differential torsors} we assume Notation \ref{notation: setup}.}

\medskip
Notation \ref{notation: setup} gives a precise meaning to the idea of a family of potential Picard-Vessiot rings for a family of linear differential equations with prescribed potential differential Galois groups.  (Lemmas \ref{lemma: generic} and \ref{lemma: R^c} below elaborate this point.) We think of $\R$ as defining the family $(R^c)_{c\in\X(k)}$ and we would like to compare the individual $R^c$'s with the ring $R^\gen$ at the generic fibre.

We now explore some first consequences of Notation \ref{notation: setup}.

\begin{lemma} \label{lemma: injective}
	The ring $\R$ is an integral domain and the morphisms $\B[x]_f\to \R$,\ $\B\to \B[\G]$ and $\R\to R^\gen$ are injective.
\end{lemma}
\begin{proof}
	Suppose that $h\in \B[x]_f$ is a nonzero element of the kernel of  $\B[x]_f\to \R$. Then the image $h^\gen$ of $h$ in $R^\gen$ is zero but also a unit. Thus $R^\gen$ must be the zero ring. This contradicts the assumption that $R^\gen$ is $\de$-simple.
	
	Assume that $b\in\B$ lies in the kernel of $\B\to\B[\G]$. Then the image of $b$ in $\R\otimes_\B \B[\G]$ is zero. But then the second diagram of Lemma \ref{lemma: equivalence with coaction} implies that $b=0$.

As $\R$ is flat over $\B[x]_f$ it is clear that $\R\to R^\gen$ is injective.	From this it follows that $\R$ is an integral domain because differentially simple rings are integral domains.
\end{proof}

Because $\R\to R^\gen$ is injective, it is not really necessary to distinguish between an element $h\in\R$ and its image $h^\gen$ in $R^\gen$.
However, we will sometimes use the notation $h^\gen$ to emphasize that we are considering $h$ as an element of $R^\gen$.

At the generic fibre we have a Picard-Vessiot ring:

\begin{lemma} \label{lemma: generic}
	The $K(x)$-$\de$-algebra $R^\gen$ is a Picard-Vessiot ring for $\de(y)=\cA y$ with fundamental matrix $Y^\gen$ and differential Galois group $G^\gen$.
\end{lemma}
\begin{proof}
	As $\R/\B[x]_f$ is a differential $\G$\=/torsor, $\R\otimes_\B K/\B[x]_f\otimes_\B K$ is a differential $G^\gen$\=/torsor by Lemma \ref{lemma: actions} (i). But $\R\otimes_\B K=\R\otimes_{\B[x]_f}K[x]_f$ and $\B[x]_f\otimes_\B K=K[x]_f$. Thus $R^\gen/K(x)$ is a differential $G^\gen$-torsor by Lemma \ref{lemma: actions} (ii).
	
	As $\R=\B[x]_f[\Y,\frac{1}{\det(\Y)}]$ and $\de(\Y)=\cA\Y$, we have $R^\gen=K(x)[Y^\gen,\frac{1}{\det(Y^\gen)}]$ 
	and $\de(Y^\gen)=\cA Y^\gen$. Because $R^\gen$ is assumed to be $\de$-simple, we see that $R^\gen/K(x)$ is a Picard-Vessiot ring for $\de(y)=\cA y$.
	
	It follows from \cite[Prop. 1.12]{BachmayrHarbaterHartmannWibmer:DifferentialEmbeddingProblems} that the differential Galois group of $R^\gen/K(x)$ is $G^\gen$.
\end{proof}

So, reversing the idea of Lemma \ref{lemma: spread out PVring}, we can think of $\R/\B[x]_f$ as having been obtained by ``spreading out'' the Picard-Vessiot ring $R^\gen/K(x)$. At the special fibres we either have nothing or a ring that ``looks like'' a Picard-Vessiot ring:

\begin{lemma} \label{lemma: R^c}
	If $c\in \X(k)$ is such that $R^c$ is not the zero ring, then $R^c/k(x)$ is a differential $G^c$-torsor for $\de(y)=A^cy$ with fundamental matrix $Y^c$. Moreover, there exists a nonempty Zariski open subset $\U$ of $\X(k)$ such that $R^c$ is not the zero ring for every $c\in \U$.
\end{lemma}
\begin{proof}
	Assume that $R^c$ is not the zero ring. As in Lemma \ref{lemma: generic}, using Lemma \ref{lemma: actions}, one sees that $R^c/k(x)$ is a differential $G^c$\=/torsor. Furthermore, $R^c=k(x)[Y^c,\frac{1}{\det(Y^c)}]$ and $\de(Y^c)=A^cY^c$.
	Let $T$ be a $k$-algebra and $g\in G^c(T)$. When interpreted as an element of $\G(T)$, $g$ acts on $\Y\otimes1\in\Gl_n(\R\otimes_\B T)$ by multiplication with a matrix $[g]\in \Gl_n(T)$. Thus $g$ also acts on $Y^c\otimes 1\in\Gl_n(R^c\otimes_kT)$ by multiplication with $[g]$. This proves the first claim of the lemma.
	
	For the second claim, note that $\B[x]_f\subseteq \R$ by Lemma \ref{lemma: injective} and that $\R$ is a finitely generated $\B[x]_f$-algebra. By \cite[Cor. 3, Chapter V, \S 3.1]{Bourbaki:commutativealgebra} (applied with $b=1$) there exists a nonzero $h\in \B[x]_f$ such that any morphism $\psi$ (of rings) from $\B[x]_f$ into an algebraically closed field with $\psi(h)\neq 0$, extends to a morphism on $\R$. Without loss of generality we may assume $h\in\B[x]$. Let $a\in \B$ be the leading coefficient of $h$ and let $\U$ be the Zariski open subset of $\X(k)$ defined by $a$, i.e., $\U$ is the complement of the vanishing locus of $a$.
	
	Let $c\in \U$. We will show that $R^c$ is not the zero ring. Taking $\psi$ as $\psi\colon \B[x]_f\xrightarrow{c} k(x)\to \overline{k(x)}$, we see that $\psi(h)\neq 0$ and so $\psi$ extends to a morphism $\psi\colon \R\to \overline{k(x)}$. Then $\R\otimes_{\B[x]_f}k(x)\to\overline{k(x)},\ r\otimes s\mapsto \psi(r)s$ is a morphism of rings. Thus $\R\otimes_{\B[x]_f}k(x)$ cannot be the zero ring.
\end{proof}

We can do slightly better:

\begin{lemma}
	\label{lemma:domain}
	There exists a nonempty Zariski open subset $\U$ of $\X(k)$ such that $R^c$ is an integral domain for all $c\in\U$.
\end{lemma}
\begin{proof}
The key fact from algebraic geometry that we will use is that being geometrically integral ``spreads out'' from the generic fibre (\cite[Theorem 3.2.1 (ii)]{Poonen:RationalPointsOnVarieties}). More concretely, if $\B\subseteq \R$ is an inclusion of finitely generated $k$-algebras, such that $\R\otimes_\B K$ is an integral domain (where, as usual, $K$ is the algebraic closure of the field of fractions $k(\B)$ of $\B$), then there exists a nonempty Zariski open subset $\U$ of $\spec(\B)$ such that $\R\otimes_\B \overline{k(\p)}$ is an integral domain for every $\p\in\U$.

	To apply this to our situation, note that $\R\otimes_\B K$ is an integral domain because $\R\otimes_\B K=\R\otimes_{\B[x]_f}K[x]_f\subseteq \R\otimes_{\B[x]_f}K(x)=R^\gen$ and the latter is an integral domain. 
	
	By the first paragraph, there exists a nonempty Zariski open subset $\U$ of $\X(k)$ such that $\R\otimes_\B k$ is an integral domain for all $c\in \U$. As
	$$R^c=\R\otimes_{\B[x]_f}k(x)=(\R\otimes_{\B[x]_f}k[x]_{f^c})\otimes_{k[x]_{f^c}}k(x)=(\R\otimes_\B k)\otimes_{k[x]_{f^c}}k(x),$$
	we see that $R^c$ is a localization of the integral domain $R\otimes_\B k$ for all $c\in \U$. So $R^c$ is either the zero ring or an integral domain. Thus the claim follows from Lemma \ref{lemma: R^c}.
\end{proof}

The following example illustrates Lemma \ref{lemma:domain}.

\begin{ex} \label{ex: integral domain}
	Let $a$ be a unit of $\B[x]_f$ such that $a$ is not a square in $k(\B)(x)$. Then $\B[x]_f[y]/(y^2-a)=\B[x]_f[\eta]$ is an integral domain and naturally a differential ring (with $\de(\eta)=\frac{\delta(a)}{2a}\eta$). Let $b\in\B[x]_f$. As in Example \ref{ex: torsor for monomial matrices} set $$\R=\B[x]_f[\eta][y_1,y_2,y_1^{-1},y_2^{-1}]=\B[x]_f[X,\tfrac{1}{\det(X)}]/(p_1,p_2)=\B[x]_f[\Y,\tfrac{1}{\det(\Y)}],$$ 
	where $\de(y_1)=(b+\eta)y_1$, $\de(y_2)=(b-\eta)y_2$,	
		$$
	p_1=X_{21}X_{22}-(b^2-a)X_{11}X_{12}, \quad  \quad  p_2=X_{21}X_{12}+X_{22}X_{11}-2bX_{11}X_{12,} $$
	
	 $$\cA=\begin{pmatrix}  0 & 1 \\ a+\de(b)-b^2-\frac{\delta(a)b}{2a} & 2b+\frac{\delta(a)}{2a}\end{pmatrix}\in\B[x]_f^{2\times 2},$$
	 and $\de(X)=\cA X$. Then $\R$ is an integral domain and $\R/\B[x]_f$ is a differential $\G$-torsor for $\de(y)=\cA y$, where $\G$ is the group scheme of $2\times 2$ monomial matrices over $\B$ (as seen in Example \ref{ex: torsor for monomial matrices}).
	 
	  We claim that for $c\in \X(k)$
	 the $k(x)$-$\de$-algebra $R^c=\R\otimes_{\B[x]_f}k(x)$ is an integral domain if and only if $a^c$ is not a square in $k(x)$.
	 Note that $\R$ is a Laurent-polynomial ring in $y_1,y_2$ over $\B[x]_f[\eta]$. Therefore, $\R\otimes_{\B[x]_f}k(x)$ is a Laurent polynomial ring over $\B[x]_f[\eta]\otimes_{\B[x]_f}k(x)$. The latter is an integral domain, if and only if  $\B[x]_f[\eta]\otimes_{\B[x]_f}k(x)=k(x)[y]/(y^2-a^c)$ is an integral domain and this is the case if and only if $a^c$ is not a square in $k(x)$.
	 
	 Thus, for example, for $\B=k[\alpha,\beta]$, with $\alpha,\beta$ algebraically independent over $k$ and $a=x^2+\alpha x+\beta$, we see that $R^c$ is an integral domain if and only if $c\in \U$, where $\U\subseteq \X(k)$ is the Zariski open subset where the discriminant $\alpha^2-4\beta$ does not vanish.
\end{ex}

By Lemma \ref{lemma: R^c} the $R^c$'s ``look like'' Picard-Vessiot rings for $\de(y)=A^cy$ with differential Galois group $G^c$. The only missing piece of information is whether or not $R^c$ is $\de$-simple. If $R^c$ is $\de$-simple, then $R^c/k(x)$ is a Picard-Vessiot ring with differential Galois group $G^c$.

The main question is: Are there ``many'' $c$'s in $\X(k)$ such that $R^c/k(x)$ is Picard-Vessiot?
(In Theorem \ref{theo: main specialization} we will see that the answer is yes.)

The following example illustrates that the set of all $c\in \X(k)$ such that $R^c$ is Picard-Vessiot is in general not constructible.

\begin{ex}
	Let $\B=k[\alpha]$, with $\alpha$ a variable over $k$ and set $f=x$. Let
	$$\mathcal{A}=\begin{pmatrix}
	0 & 1 \\
	(\frac{\alpha}{x})^2-1 & -\frac{1}{x}
\end{pmatrix}\in \B[x]_f^{2\times 2}=k[\alpha,x]_x^{2\times 2}
$$
be the companion matrix of Bessel's differential equation $\de^2(y)+\frac{1}{x}\de(y)+(1-(\frac{\alpha}{x})^2)y=0$. Let $X$ be a $2\times 2$ matrix of indeterminates over $\B[x]_f$. Consider $\B[x]_f[X,\frac{1}{\det(X)}]$ as a $\B[x]_f$-$\de$-algebra via $\de(X)=\mathcal{A}{X}$. Then
$$\de(\det(X)-\tfrac{1}{x})=\de(\det(X))+\tfrac{1}{x^2}=-\tfrac{1}{x}\det(X)+\tfrac{1}{x^2}=-\tfrac{1}{x}(\det(X)-\tfrac{1}{x})$$
and so $(\det(X)-\frac{1}{x})$ is a $\de$-ideal of $\B[x]_f[X,\frac{1}{\det(X)}]$ and $\R=\B[x]_f[X,\frac{1}{\det(X)}]/(\det(X)-\frac{1}{x})$ is a $\B[x]_f$-$\de$-algebra. Let $\Y\in \Gl_n(\R)$ denote the image of $X$ in $\R$. Let $\G=\Sl_2=\Sl_{2,\B}$ be the group scheme $\Sl_2$ over $\B$. Then $\G$ acts on $\R/\B[x]_f$: For a $\B$-algebra $\T$, a $g\in\Sl_2(\T)$ defines a $\B[x]_f\otimes_\B \T$-$\de$-automorphism $g\colon \R\otimes_\B\T\to \R\otimes_\B\T$ by $g(\Y\otimes 1)=(\Y\otimes 1)(1\otimes g)$.
The corresponding morphism $\rho\colon \R\to \R\otimes_\B \B[\Sl_2]$ of $\B[x]_f$-$\de$-algebras (as in Lemma \ref{lemma: equivalence with coaction}) is given by $\rho(\Y)=(\Y\otimes 1)(1\otimes T)$, where $T\in\B[\Sl_2]^{2\times 2 }$ is the matrix of coordinate functions on $\Sl_2$. The morphism $\R\otimes_{\B[x]_f}\R\to \R\otimes_\B \B[\G],\ a\otimes b\mapsto (a\otimes 1)\cdot\rho(b)$ is an isomorphism, its inverse is determined by $1\otimes T\mapsto (\Y\otimes 1)^{-1}(1\otimes \Y)$. So $\R/\B[x]_f$ is a differential $\G$-torsor for $\de(y)=\mathcal{A} y$.

From a more geometric perspective, if $\Q$ is any ring (e.g., $\B[x]_f$) and $a\in\Q^\times$, it is clear that the closed subscheme $\Z$ of $\Gl_{2,\Q}$ defined by $\det(X)-a$ is an $\Sl_{2,\Q}$-torsor: For any $\Q$-algebra $\Q'$, the map
$\Z(\Q')\times\Sl_2(\Q')\to \Z(\Q')\times \Z(\Q'),\ (z,g)\mapsto (z,zg)$ is bijective with inverse $(z_1,z_2)\mapsto (z_1, z_1^{-1}z_2)$.

The differential Galois group of Bessel's equation is $\Sl_{2}$ if $\alpha-\frac{1}{2}\notin\mathbb{Z}$ and the multiplicative group otherwise (Appendix of \cite{Kolchin:AlgebraicGroupsAndAlgebraicDependence}). The significance of the condition $\alpha-\frac{1}{2}\notin\mathbb{Z}$ is also explained in Example \ref{ex: N(L) for Bessel}. It follows that $R^\gen=\R\otimes_{\B[x]_f}K(x)/K(x)$ is Picard-Vessiot. Moreover, $R^c$ is Picard-Vessiot if and only if $c\in k\smallsetminus\{\frac{1}{2}+m|\ m\in\mathbb{Z}\}$.
\end{ex}
The following three lemmas will be useful later on.
\begin{lemma} \label{lemma: nonzero extends}
	Let $h\in \R$ be nonzero. Then there exists a nonempty Zariski open subset $\U$ of $\X(k)$ such that $h^c\in R^c$ is nonzero for every $c\in \U$.
\end{lemma}
\begin{proof}
Applying \cite[Cor. 3, Chapter V, \S 3.1]{Bourbaki:commutativealgebra} to the inclusion $\B[x]_f\subseteq \R$, we find a nonzero element $h'\in\B[x]_f$ such that every morphism
$\psi$ (of rings) from $\B[x]_f$ into an algebraically closed field with $\psi(h')\neq 0$, extends to a morphism $\psi$ on $\R$ with $\psi(h)\neq 0$. Without loss of generality we may assume that $h'\in \B[x]$. Let $a\in \B$ be the leading coefficient of $h'$ and let $\U$ be the Zariski open subset of $\X(k)$ defined by $a$.

For $c\in \U$ let $\psi$ be defined by $\psi\colon \B[x]_f\xrightarrow{c}k(x)\to \overline{k(x)}$. Then $\psi(h')\neq 0$ and $\psi$ extends to $\R$ with $\psi(h)\neq 0$. We can thus define a morphism $\psi'\colon R^c=\R\otimes_{\B[x]_f}k(x)\to \overline{k(x)},\ r\otimes s\mapsto \psi(r)s$. Then $\psi'(h^c)=\psi(h)\neq 0$. Thus $h^c\neq 0$.
\end{proof}

Recall (Definition \ref{defi: differential Gtorsor for de(y)=Ay}) that the fundamental matrix $\Y\in\Gl_n(\R)$ defines a morphism $\G\to \Gl_{n,\B}$ of group schemes over $\B$.

\begin{lemma} \label{lemma: closed embedding}
There exists a nonzero element $b\in\B$ such that the base change of $\G\to \Gl_{n,\B}$ via $\B\to\B_b$ is a closed embedding. 
\end{lemma}
\begin{proof}
Set $\R'=\R\otimes_{\B[x]_f}k(\B)(x)$. As in the proof of Lemma \ref{lemma: generic} we see that $\R'/k(\B)(x)$ is a differential $\G_{k(\B)}$-torsor for $\de(y)=\cA y$ with fundamental matrix $\Y\otimes 1\in \Gl_n(\R')$.
From this we deduce that the morphism $\G_{k(\B)}\to\Gl_{n,k(\B)}$ determined by $\Y\otimes 1$ is a closed embedding as follows: Assume $\T$ is a $k(\B)$\=/algebra and $g\in\G_{k(\B)}(\T)=\Hom_{k(\B)}(k(\B)[\G],\T)$ is in the kernel of $\G_{k(\B)}(\T)\to\Gl_n(\T)$, i.e., $g$ acts trivially on $\R'\otimes_{k(\B)}\T$. 
Let $\rho'\colon \R'\to \R'\otimes_{k(\B)}k(\B)[\G]$ denote the coaction and $e\in\Hom(k(\B)[\G],\T)=\G(\T)$ the identity element. As $g$ acts trivially, the maps
 $$\R'\xrightarrow{\rho'} \R'\otimes_{k(\B)}k(\B)[\G]\xrightarrow{\R'\otimes g} \R'\otimes_{k(\B)}\T  $$

and
 $$\R'\xrightarrow{\rho'} \R'\otimes_{k(\B)}k(\B)[\G]\xrightarrow{\R'\otimes e} \R'\otimes_{k(\B)}\T  $$
 agree (with $a\mapsto a\otimes 1$). As the $\R'$-linear extension of $\rho'$ is an isomorphism, it follows that $\R'\otimes g=\R'\otimes e$ and therefore $g=e$. This shows that $\G_{k(\B)}(\T)\to \Gl_n(\T)$ is injective for every $k(\B)$-algebra $\T$. Therefore $\G_{k(\B)}\to \Gl_{n,k(\B)}$ is a closed embedding (\cite[Section~15.3]{Waterhouse:IntroductiontoAffineGroupSchemes}).

The morphism $\G\to \Gl_{n,\B}$ determined by $\Y$, after base change via $\B\to k(\B)$, yields the morphism $\G_{k(\B)}\to\Gl_{n,k(\B)}$ determined by $\Y\otimes 1$ (and, as we have just seen, the latter is a closed embedding). So the map $\psi\colon \B[X,\frac{1}{\det(X)}]\to \B[\G]$, dual to the morphism $\G\to \Gl_{n,\B}$, becomes surjective after the base change $\B\to k(\B)$. In other words, the map $S^{-1}\psi\colon S^{-1}\B[X,\frac{1}{\det(X)}]\to S^{-1}\B[\G]$ is surjective, where $S$ is the multiplicatively closed subset of all nonzero elements of $\B$.

Let $a_1,\ldots,a_m\in \B[\G]$ be generators of $\B[\G]$ as a $\B$-algebra and let $\frac{f_1}{s_1},\ldots,\frac{f_m}{s_m}\in S^{-1}\B[X,\frac{1}{\det(X)}]$ be such that $(S^{-1}\psi)(\frac{f_i}{s_i})=\frac{a_i}{1}$ for $i=1,\ldots,m$.
%\rf{$f_1$ is changed into $f_i$.}
Then there exist $s_1',\ldots,s_m'\in S$ such that $s_i'(\psi(f_i)-s_ia_i)=0$ for $i=1,\ldots,m$. Thus $b=s_1\ldots s_ms_1'\ldots s_m'$ has the desired property.
\end{proof}

%\begin{lemma}?? 
%	 Let $\B[x]_f'$ denote the integral closure of $\B[x]_f$ in $\R$ and let $K(x)'$ denote the integral closure of $K(x)$ in $R^\gen$. Then the canonical map $\B[x]'_f\otimes_{\B[x]_f}K(x)\to K(x)'$ is an isomorphism.
%\end{lemma}	
%\begin{proof}
%	The inclusion $k(\B)(x)=\B[x]_f\otimes_{\B[x]_f}k(\B)(x)\subseteq \R\otimes_{\B[x]_f}k(\B)(x)$ can be identified with the localization of the inclusion $\B[x]_f\subseteq\R$ at the multiplicatively closed subset of all nonzero elements of $\B[x]_f$. As integral closure commutes with localization (\cite[\href{https://stacks.math.columbia.edu/tag/0307}{Tag 0307}]{stacks-project}), the integral closure of $k(\B)(x)$ in $\R\otimes_{\B[x]_f}k(\B)(x)$ is $\B[x]'_f\otimes_{\B[x]_f}k(\B)(x)$. As integral closure commutes with separable algebraic base change (\cite[\href{https://stacks.math.columbia.edu/tag/0CBF}{Tag 0CBF}]{stacks-project}), it follows that the integral closure of $K(x)$ in $$(\R\otimes_{\B[x]_f}k(\B)(x))\otimes_{k(\B)(x)}K(x)=\R\otimes_{\B[x]_f}K(x)=R^\gen$$ is $(\B[x]'_f\otimes_{\B[x]_f}k(\B)(x))\otimes_{k(\B)(x)}K(x)=\B[x]'_f\otimes_{\B[x]_f}K(x)$
%\end{proof}

\begin{lemma} \label{lemma: integral closure specializes}
	Let $\B[x]_f'$ denote the integral closure of $\B[x]_f$ in $\R$ and for $c\in\X(k)$ let $k(x)'$ denote the integral closure of $k(x)$ in $R^c$. Then there exists a nonempty Zariski open subset $\U$ of $\X(k)$ such that $\B[x]_f'\otimes_{\B[x]_f}k(x)\to k(x)'\subseteq R^c=\R\otimes_{\B[x]_f}k(x)$ is an isomorphism for every $c\in\U$.
\end{lemma}
\begin{proof}
	Let $m$ denote the geometric number of irreducible components of $R^\gen/K(x)$. As $\overline{k(\B)(x)}=\overline{K(x)}$, it is clear that the geometric number of irreducible components of $\mathcal{R}\otimes_{\B[x]_f}k(\B)(x)/k(\B)(x)$ is also $m$.
	
    We claim that the (vector space) dimension $\dim_{k(\B)(x)}\B[x]'_f\otimes_{\B[x]_f}k(\B)(x)$ equals $m$.
	The inclusion $k(\B)(x)=\B[x]_f\otimes_{\B[x]_f}k(\B)(x)\subseteq \R\otimes_{\B[x]_f}k(\B)(x)$ can be identified with the localization of the inclusion $\B[x]_f\subseteq\R$ at the multiplicatively closed subset of all nonzero elements of $\B[x]_f$. As integral closure commutes with localization (\cite[\href{https://stacks.math.columbia.edu/tag/0307}{Tag 0307}]{stacks-project}), the integral closure $k(\B)(x)'$ of $k(\B)(x)$ in $\R\otimes_{\B[x]_f}k(\B)(x)$ is $k(\B)(x)'=\B[x]'_f\otimes_{\B[x]_f}k(\B)(x)$. As integral closure commutes with separable algebraic base change (\cite[\href{https://stacks.math.columbia.edu/tag/0CBF}{Tag 0CBF}]{stacks-project}), it follows that the integral closure $K(x)'$ of $K(x)$ in $$(\R\otimes_{\B[x]_f}k(\B)(x))\otimes_{k(\B)(x)}K(x)=\R\otimes_{\B[x]_f}K(x)=R^\gen$$ is $K(x)'=(\B[x]'_f\otimes_{\B[x]_f}k(\B)(x))\otimes_{k(\B)(x)}K(x)$. %=\B[x]'_f\otimes_{\B[x]_f}K(x)$ 
	As $[K(x)': K(x)]=m$ by Corollary \ref{cor: geometric number}, it follows that  $\dim_{k(\B)(x)}\B[x]'_f\otimes_{\B[x]_f}k(\B)(x)=m$ as claimed.
	
	It is known that $\B[x]'_f$ is a finitely generated $\B[x]_f$-module (\cite[Cor. 13.13]{Eisenbud:view}). Thus, there exists a nonempty Zariski open subset $\mathcal{V}_1$ of $\spec(\B[x]_f)$ such that $\dim_{k(\p)}\B[x]_f'\otimes_{\B[x]_f}k(\p)$ is constant on $\mathcal{V}_1$ (e.g., using \cite[Theorem 14.4]{Eisenbud:view}). As $\dim_{k(\B)(x)}\B[x]'_f\otimes_{\B[x]_f}k(\B)(x)=m$, we must have $\dim_{k(\p)}\B[x]_f'\otimes_{\B[x]_f}k(\p)=m$ for all $\p\in\mathcal{V}_1$.

	For a morphism of schemes of finite type the number of geometrically irreducible components of the fibres is constant on a nonempty Zariski open subset (\cite[\href{https://stacks.math.columbia.edu/tag/055A}{Tag 055A}]{stacks-project}). 
	%\rf{It seems that ``type" is missing after ``finite".}
	Applying this to the dual of the inclusion $\B[x]_f\subseteq \R$ shows that there exists a non-empty Zariski open subset $\mathcal{V}_2$ of $\spec(\B[x]_f)$ such that the number of geometrically irreducible components of $\R\otimes_{\B[x]_f}k(\p)$ is constant on $\mathcal{V}_2$. As the number of  geometrically irreducible components of $\R\otimes_{\B[x]_f}k(\B)(x)$ is $m$, we see that the number of geometrically irreducible components of $\R\otimes_{\B[x]_f}k(\p)$ is $m$ for all $\p$ in $\mathcal{V}_2$.
	
	The inclusion $\B\subseteq K$, being a composition of a localization with a field extension, makes $K$ a flat $\B$-module. Thus the map $\B[x]'_f\otimes_\B K\to \R\otimes_\B K$ is injective. We already noted in the proof of Lemma \ref{lemma:domain} that $\R\otimes_\B K$ is an integral domain and so  $\B[x]'_f\otimes_\B K$ is an integral domain. As in loc. cit. it follows that there exists a nonempty Zariski open subset $\U_1$ of $\X(k)$ such that $\B[x]'_f\otimes_{\B[x]_f}k(x)$ is an integral domain for all $c\in\U_1$.
	
	Let $h\in\B[x]$ be a nonzero polynomial such that $D(h)\subseteq\mathcal{V}_1\cap\mathcal{V}_2$ and let $\U$ be the Zariski open subset of $\U_1$ where the leading coefficient of $h$ does not vanish. We will show that 
	$\B[x]_f'\otimes_{\B[x]_f}k(x)\to k(x)'$ is an isomorphism for every $c\in \U$. Let $c\in\U$ and let $\p$ be the kernel of $c\colon \B[x]_f\to k(x)$. Then $\B[x]_f'\otimes_{\B[x]_f}k(\p)=\B[x]_f'\otimes_{\B[x]_f}k(x)$ and $\R\otimes_{\B[x]_f}k(\p)=\R\otimes_{\B[x]_f}k(x)=R^c$. As $c$ does not vanish on the leading coefficient of $h$ we have $\p\in\mathcal{V}_1\cap \mathcal{V}_2$. So $\dim_{k(x)}\B[x]_f'\otimes_{\B[x]_f}k(x)=m$ and the geometric number of connected components of $R^c$ is $m$. It thus follows from Corollary \ref{cor: geometric number} that $[k(x)':k(x)]=m$. Therefore, $\B[x]_f'\otimes_{\B[x]_f}k(x)$ and $k(x)'$ have the same degree $m$ over $k(x)$. It thus suffices to show that $\B[x]_f'\otimes_{\B[x]_f}k(x)\to k(x)'$ is injective. But as $c\in \U_1$, the ring $\B[x]_f'\otimes_{\B[x]_f}k(x)$ is an integral domain and therefore a field, since it is integral over $k(x)$.
\end{proof}

%As explained in the following remark, Lemma \ref{lemma: lift opens} allows us, in the context of Notation~\ref{notation: setup}, to enlarge $\B$ if necessary.

The following remark, which is important from a technical perspective and will be used repeatedly in what follows, specializes Remark \ref{rem: extend B basic} to the context of Notation \ref{notation: setup}.
 
\begin{rem} \label{rem: extend B}
	Assume that in the setup of Notation \ref{notation: setup} we would like to show that there exists
	\begin{itemize}
		\item an ad-open,
		\item a Jac-open or
		\item an ad$\times$Jac-open
	\end{itemize}
	subset $\U$ of $\X(k)$ such that 
	
	\begin{itemize}
		\item the differential $k(x)$-algebra $R^c$ has a certain property for every $c\in \U$ or
		\item that for given $f_1,\ldots,f_m\in \R$, the elements $f_1^c,\ldots,f_m^c\in R^c$ have a certain property,
	\end{itemize}
	then we can, without loss of generality, replace $\B$ with an integral domain $\B'$ such that $\B\subseteq \B'$ and $\B'$ is finitely generated and algebraic over $\B$. 	
\end{rem}
\begin{proof}	
	Note that $k(\B)\subseteq k(\B')\subseteq \overline{k(\B)}=K$.	
	Set $$\R'=\R\otimes_\B\B'=\R\otimes_{\B[x]_f}\B'[x]_f\subseteq\R\otimes_{\B[x]_f}K(x)=R^\gen.$$ Then $\R'$ is not the zero ring and so $\R'$ is a differential   $\G'=\G_{\B'}$\=/torsor over $\B[x]_f\otimes_\B\B'=\B'[x]_f$ by Lemma \ref{lemma: actions}. 
	We next verify that $\R'/\B'[x]_f$ has all the properties listed in Notation \ref{notation: setup}. The matrix $\Y'=\Y\otimes1\in \Gl_n(\R')$ is such that $\de(\Y')=\cA\Y'$ and $\R'=\B'[x]_f[\Y',\frac{1}{\det(\Y')}]$. Since $\G$ acts on $\Y$ through matrix multiplication, also $\G'$ acts on $\Y'$ through matrix multiplication. Thus $\R'/\B'[x]_f$ is a differential $\G'$-torsor for $\de(y)=\cA y$ with fundamental matrix $\Y'$.
	As $\R$ is a flat over $\B[x]_f$, we see that $\R'=\R\otimes_{\B[x]_f}\B'[x]_f$ is flat over $\B'[x]_f$. Moreover, $$R'^\gen=\R'\otimes_{\B'[x]_f}K(x)=(\R\otimes_{\B[x]_f}\B'[x]_f)\otimes_{\B'[x]_f}K(x)=\R\otimes_{\B[x]_f}K(x)=R^\gen$$
	is $\de$-simple.
%	\rf{$\otimes$ was added.}
	
	Similarly, if $\f\colon\X'(k)\to \X(k)$ is the morphism induced by the inclusion $\B\subseteq \B'$ and $c'\in\X(k)$ is such that $\f(c')=c$, then $${R'}^{c'}=\R'\otimes_{\B'[x]_f}k(x)= (\R\otimes_{\B[x]_f}\B'[x]_f)\otimes_{\B'[x]_f}k(x)=\R\otimes_{\B[x]_f}k(x)=R^c.$$ 
	Furthermore, for $i=1,\ldots,m$, the image $f_i^{c'}$ of $f_i$ in $R'^{c'}$ agrees with the image $f_i^{c}$ of $f_i$ in $R^c$.
	
	Thus, if there exists a subset $\U'$ of $\X'(k)$ of one of the three kinds listed in the remark such that one of the two properties of the remark is true for all $c'\in \U'$, then Lemma~\ref{lemma: lift opens} yields a subset $\U$ of $\X(k)$ of the corresponding kind such that $\U\subseteq \f(\U')$ and the corresponding property holds for all $c\in \U$.
\end{proof}

%\subsection{Specializations of differential torsors: proto-Picard-Vessiot rings} \label{subsec:proto-Picard-Vessiot}

\subsection{Algebraic relations under specialization}

\label{subsec: Algebraic relations under specialization}

Recall that we assume Notation \ref{notation: setup} for the remainder of Section \ref{sec: Specialization of differential torsors}. The main result of this section (Theorem \ref{theo:basisunderspecialization}) is that a basis of the vector space of algebraic relations of degree at most $d$ among the entries of a fundamental solution matrix for $\de(y)=\cA y$ specializes to a basis of the vector space of algebraic relations of degree at most $d$ among the entries of a fundamental solution matrix for $\de(y)=A^c y$ for all $c$ in an ad-open subset of the parameter space. This is in spirit similar to Theorem B from the introduction. However, the present result is weaker because we have to fix the degree of the algebraic relations. 

By considering the vector of all monomials of degree at most $d$ in the entries of a fundamental solution matrix, the study of the algebraic relations of degree at most $d$ among the entries of a fundamental solution matrix can be reduced to the study of the linear relations satisfied by the entries of one solution vector. We therefore, first study the behavior of  linear relations under specialization.
% the key idea is to find a useful bound on the degree of the rational functions occurring as the coefficients in a basis of the linear relations. To describe this bound and its behavior under specialization, we need to bound the degree of the logarithmic derivative of exponential solutions of linear differential operators. This is discussed in the following section.
 Combining the main result of Section \ref{subsec: Linear relations} with the main result of Section \ref{subsec: The exponential bound under specialization} we obtain:

\begin{prop}
	\label{prop:uniform bound for vector spaces}
	Let $v\in \R^\ell$ be such that $v^\gen\in (R^\gen)^\ell$ is a solution of $\delta(y)=\mathcal{A'}y$, with $\cA'\in K(x)^{\ell\times \ell}$. Then there exists a positive integer $N$ and an ad-open subset $\U$ of $\X(k)$ such that for every $c\in \U$ the following holds: If $\m_c$ is a maximal $\de$\=/ideal of $R^c$ and $\overline{v^c}\in (R^c/\m_c)^\ell$ is the image of $v^c$ in $R^c/\m_c$, then
	the $k(x)$-vector space $\linrel(\overline{v^c},k(x))$ has a basis consisting of elements of the form $a_1y_1+\ldots+a_\ell y_\ell$, with $a_1,\ldots,a_\ell\in k(x)$ and $\deg(a_i)\leq N$ for $i=1,\ldots,\ell$.
\end{prop}
\begin{proof}
	For $i=1,\dots,\ell$, let $\L_i \in K(x)[\de]$ be a monic differential operator, equivalent to $\delta(y)=(\bigwedge^i\cA')y$ and let $T_i\in \Gl_{{\ell\choose i}}(K(x))$ be a corresponding transformation matrix. (See Section \ref{subsec: Linear relations}.) Enlarging $\B$ if necessary (Remark \ref{rem: extend B}), we can assume that $\cA'$, $T_i$ and $\L_i$ $(1\leq i\leq \ell)$ have coefficients in $\B[x]_{ff_1}$, where $f_1\in \B[x]$ is a monic polynomial. So, for a nonempty Zariski open subset $\U'$ of $\X(k)$, the matrix $T_i^c$ lies in $\Gl_{{\ell\choose i}}(k(x))$ and is a transformation matrix from $\de(y)=(\bigwedge^i \cA')^cy=(\bigwedge^i\cA'^c)y$ to $\L_i^cy=0$. Note that  $v^c$ is a solution of $\de(y)=\cA'^cy$ for every $c\in\X(k)$.
	
	By Proposition~\ref{prop: exponential bound}, for each $i=1,\dots,\ell$, there exists an ad-open subset $\U_i$ of $\X(k)$ such that $N(\L_i)$ is an exponential bound for $\L_i^c y=0$ for any $c\in \U_i$.
	Note that $\dim_{k(x)}\linrel(\overline{v^c},k(x))$  lies between $1$ and $\ell$ for any $c\in\X(k)$. It thus follows from Lemma \ref{lemma: linear relations} that
	$$
	N=\max_{1\leq i\leq \ell} \left\{ 2{\ell\choose i}\deg(T_i)+{\ell\choose i}\left({\ell\choose i}-1\right) N(\L_i) \right\}
	$$ and $\U$ any ad-open subset of $\U'\cap \U_1\cap\ldots\cap\U_\ell$ have the required property.
\end{proof}

Let $F$ be a differential field, $A\in F^{n\times n}$ and $R/F$ a Picard-Vessiot ring for $\de(y)=Ay$ with fundamental matrix $Y\in\Gl_n(R)$. For a fixed positive integer $d$, we set
$$
\operatorname{AlgRel}(Y,d,F)=\{p\in F[X]\mid p(Y)=0,\ \deg(p)\leq d\}.
$$
Then $\algrel(Y,d,F)$ is a finite dimensional $F$-vector space; the vector space of algebraic relations of degree at most $d$ among the entries of $Y$.

%
%If $c\in\X(k)$ and $\m$ is a maximal $\de$-ideal of $R^c$, then $R^c/\m$ is a Picard-Vessiot ring over $k(x)$ for $\de(y)=A^cy$ with fundamental solution matrix $\overline{Y^c}$, the image of $Y^c$ in $R^c/\m$. As any two maximal $\de$-ideals of $R^c$, only differ by multiplication with a matrix in $\Gl_n(k)$,
%we see that $\dim_{k(x)}\algrel(\overline{Y^c},m,k(x))$ does not depend on the choice of $\m$.
%The goal of Appendix~A is to investigate the relation between $\algrel(\overline{Y^c},m,k(x))$ and $\algrel(Y^{\gen},m,K(x))$. In particular, we will show (Theorem \ref{theo:basisunderspecialization}) that there exists an ad\=/open subset $\U$ of $\X(k)$ such that $$\dim_{k(x)}\algrel(\overline{Y^c},m,k(x))=\dim_{K(x)}\algrel(Y^{\gen},m,K(x))$$ for any $c\in \U$.
%We are now prepared to prove the main result of Appendix A.

The following theorem shows that a basis of the vector space of algebraic relations of degree at most $d$ among the entries of a fundamental solution matrix is preserved on an ad-open subset of the parameter space.

\begin{theo}
	\label{theo:basisunderspecialization}
	
	For any $c\in\X(k)$, let $\m_c$ be a maximal $\de$-ideal of $R^c$ and let $\overline{Y^c}\in \Gl_n(R^c/\m_c)$ denote the image of $Y^c$ in $R^c/\m_c$.
	Let $p_1,\ldots,p_m\in \B[x][X]$ be a $K(x)$\=/basis of $\algrel(Y^{\gen},d,K(x))$.
	Then there exists an ad-open subset $\U$ of $\X(k)$ such that $p_1^c,\ldots,p_m^c$ is a $k(x)$\=/basis of $\algrel(\overline{Y^c},d,k(x))$
	for any $c\in \U$.
\end{theo}
\begin{proof}
	The main step of the proof is to show that there exists an ad-open subset $\U_1$ of $\X(k)$ such that
	\begin{equation} \label{eq: bound dim}
		\dim_{k(x)}\algrel(\overline{Y^c},d,k(x))\leq  \dim_{K(x)}\algrel(Y^{\gen},d,K(x))
	\end{equation}
	for every $c\in\U_1$.
	
	Let $\bfm_1,\ldots,\bfm_r$ denote all monomials in $X$ of degree at most $d$ (so $r={n^2+d \choose d}$)
	and let $\ell$ be the dimension of the $K(x)$-vector space generated by $\bfm_1(Y^\gen),\ldots,\bfm_r(Y^\gen)$ (inside $R^\gen$).
	We may assume, without loss of generality, that $\bfm_1(Y^\gen),\ldots,\bfm_\ell(Y^\gen)$ are $K(x)$-linearly independent. We see that $m=\dim_{K(x)}\algrel(Y^{\gen},d,K(x))=r-\ell$. 
%	\rf{In order to use the argument in the proof of Lemma 5.11, I exchanged the symbols $\nu$ and $s$.}
%	\mw{I don't understand why we need to reference a Lemma here, is it not obvious that $\dim_{K(x)}\algrel(Y^{\gen},d,K(x))=r-\ell$?}
%	\rf{I think you are right. We do not need to reference the Lemma.}
%	
	
	As $\de(Y^\gen)=\cA Y^\gen$, the vector $(\bfm_1(Y^\gen),\ldots,\bfm_r(Y^\gen))^t\in (R^\gen)^r$ is a solution of $\de(y)=\cA_1y$ for some $\cA_1\in K(x)^{r\times r}$. Since $\bfm_i(Y^\gen)$ $(\ell+1\leq i \leq r)$ can be expressed as a $K(x)$-linear combination of $\bfm_1(Y^\gen),\ldots,\bfm_\ell(Y^\gen)$, it follows that the vector $(\bfm_1(Y^\gen),\ldots,\bfm_\ell(Y^\gen))^t$ is a solution of $\de(y)=\cA'y$ for some $\cA'\in K(x)^{\ell\times \ell}$.
	Set $v=(\bfm_1(\Y),\ldots,\bfm_\ell(\Y))^t\in \R^{\ell}$.
	
	By Proposition \ref{prop:uniform bound for vector spaces}, there exists a positive integer $N$ and an ad-open subset $\U'$ of $\X(k)$ such that for every $c\in \U'$ the following  holds: If $\widetilde{\m_c}$ is a maximal $\de$-ideal of $R^c$ and $\widetilde{v^c}\in (R^c/\widetilde{\m_c})^\ell$ is the image of $v^c$ in $R^c/\widetilde{\m_c}$, then $\linrel(\widetilde{v^c},k(x))$ has a $k(x)$-basis of the form
	$a_1y_1+\ldots+a_\ell y_\ell$ with $a_i\in k(x)$ and $\deg(a_i)\leq N$ for $i=1,\ldots,\ell$.

	Recall that the \emph{wronskian} $\operatorname{wr}(h_1,\ldots,h_s)$ of elements $h_1,\ldots,h_s$ of some differential ring $R$, is defined as the determinant of the wronskian matrix
	
	$$\begin{pmatrix}
		h_1 & \ldots & h_s \\
		\de(h_1) & \ldots & \de(h_s) \\
		\vdots & & \vdots \\
		\de^{s-1}(h_1) & \ldots & \de^{s-1}(h_s)
	\end{pmatrix}.
	$$
	The crucial property of the wronskian is that, if $R$ is a field, then $h_1,\ldots,h_s$ are $R^\de$-linearly independent if and only if $\operatorname{wr}(h_1,\ldots,h_s)$ is nonzero (\cite[Lemma 1.12]{SingerPut:differential}).
	
	Because $\bfm_1(Y^\gen),\dots,\bfm_{\ell}(Y^\gen)\in R^\gen$ are $K(x)$-linearly independent, the elements
	$x^j\bfm_i(Y^\gen)\ (0\leq j \leq \ell N,\ 1\leq i\leq \ell)$ of $R^\gen$ are $K$-linearly independent. Thus their wronskian is nonzero.
	Let $w\in \R$ be the wronskian of $x^j\bfm_i(\Y)\ (0\leq j\leq \ell N,\ 1\leq i\leq \ell)$.
	By Lemma \ref{lemma: nonzero extends}, there exists a nonempty Zariski open subset $\U''$ of $\X(k)$ such that $w^c\in R^c$ is nonzero for every $c\in \U'$. Therefore, by Lemma \ref{lemma: max de ideals dense}, there exists, for every $c\in \U''$, a maximal $\de$-ideal $\widetilde{\m_c}$ of $R^c$ such that $w^c\notin\widetilde{\m_c}$. Let $\widetilde{Y^c}\in\Gl_n(R^c/\widetilde{\m_c})$ denote the image of $Y^c$ in $R^c/\widetilde{\m_c}$. As the image of $w$ in $R^c/\widetilde{\m_c}$ is nonzero for every $c\in\U''$, the family  $x^j\bfm_i(\widetilde{Y^c})\ (0\leq j\leq \ell N,\ 1\leq i\leq\ell)$ is $k$-linearly independent.
	
	We claim that $\bfm_1(\widetilde{Y^c}),\ldots,\bfm_{\ell}(\widetilde{Y^c})$ are $k(x)$-linearly independent for every $c\in \U'\cap \U''$. Suppose this is not the case. Then $\linrel(\widetilde{v^c},k(x))$ is nonzero for some $c\in\U'\cap\U''$. By the choice of $N$ and $\U'$ above, there exist $a_1,\ldots,a_\ell\in k(x)$ with $\deg(a_i)\leq N$ and $a_1\bfm_1(\widetilde{Y^c})+\ldots+a_\ell \bfm_\ell(\widetilde{Y^c})=0$. Clearing denominators, we find that the family $x^j\bfm_i(\widetilde{Y^c})\ (0\leq j\leq \ell N,\ 1\leq i\leq \ell)$ is $k$-linearly dependent; a contradiction.

	Thus $\bfm_1(\widetilde{Y^c}),\ldots,\bfm_{\ell}(\widetilde{Y^c})\in R^c/\widetilde{\m_c}$ are $k(x)$-linearly independent. This implies that
		$$
		\dim_{k(x)}\algrel(\widetilde{Y^c},d,k(x))\leq r-\ell=\dim_{K(x)}\algrel(Y^{\gen},d,K(x)).
		$$
	Since $\dim_{k(x)}\algrel(\overline{Y^c},d,k(x))$ does not depend on the choice of the maximal $\de$-ideal $\m_c$ of $R^c$, we obtain (\ref{eq: bound dim}) with $\U_1$ an ad-open subset of $\U'\cap \U''$.
	
	As $p_i(Y^\gen)=0$, also $p_i^c(Y^c)=(p_i(\Y))^c=0$ for $i=1,\ldots,m$ and $c\in\X(k)$. So $p_1^c,\ldots,p_m^c\in \algrel(\overline{Y^c},d,k(x))$ for any $c\in\X(k)$.

	The $K(x)$-linear independence of $p_1,\ldots,p_m\in \B[x][X]$ can be expressed through the non-vanishing of determinants. Thus, there exists a nonempty Zariski open $\U_2$ of $\X(k)$ such that $p_1^c,\ldots,p_m^c\in k(x)[X]$ are $k(x)$-linearly independent for any $c\in \U_2$. Combining this with (\ref{eq: bound dim}), we see that, for $c\in \U_1\cap\U_2$, equality holds in (\ref{eq: bound dim}) and that $p_1^c,\ldots,p_m^c$ is a $k(x)$-basis of $\algrel(\overline{Y^c},d,k(x))$.
	We can thus choose $\U$ as any ad-open subset of $\U_1\cap\U_2$.
\end{proof}

\subsection{Picard-Vessiot rings under specialization}

Recall that we assume Notation \ref{notation: setup} for the remainder of Section \ref{sec: Specialization of differential torsors}. In this section we prove our main specialization result, namely, that $R^c/k(x)$ is Picard-Vessiot for all $c$ in an ad$\times$Jac-open subset of $\X(k)$.

%\subsubsection{Specializations yielding proto-Picard-Vessiot rings}

Combining the criterion of Lemma \ref{LM:criterionforprotoPV} with Theorem \ref{theo:basisunderspecialization}, we first show that $R^c/k(x)$ is proto-Picard-Vessiot for all $c$ in an ad-open subset of $\X(k)$.
% We assume the setup of Notation \ref{notation: setup}.

\begin{prop}
	\label{PROP:protoPV}
	There exists an ad-open subset $\U$ of $\X(k)$ such that $R^c/k(x)$ is proto-Picard-Vessiot for all $c\in\X(k)$.
\end{prop}
\begin{proof}
	By Lemma~\ref{lemma:domain}, there exists a nonempty Zariski open subset $\U_1$ of $\X(k)$ such that $R^c$ is an integral domain for every $c\in\U_1$.
%	As in Appendix A, for a differential field $F$, a Picard-Vessiot ring $R/F$ for $\de(y)=Ay$ with fundamental matrix $Y\in\Gl_n(R)$ and $m\in\nn$,  we consider the finite dimensional $F$-vector space
%	$$
%	\algrel(Y,m,F)=\{h\in F[X]\mid h(Y)=0,\ \deg(h)\leq m\}.
%	$$
%	
	
	Let $p_1,\dots,p_m\in K[x][X]$ be a $K(x)$-basis of $\algrel(Y^\gen,\bfd(n),K(x))$. Extending $\B$ if necessary (Remark \ref{rem: extend B}), we may assume that $p_1,\ldots,p_m\in \B[x][X]$.
	As $p_i(\Y)=0$, also $p_i^c(Y^c)=(p_i(\Y))^c=0$ for $i=1,\ldots,m$ and all $c\in\X(k)$.

	For $c\in \X(k)$ let $\m_c$ be a maximal $\de$-ideal of $R^c$ and let $\overline{Y^c}\in\Gl_n(R^c/\m_c)$ denote the image of $Y^c$ in $R^c/\m_c$. By Theorem~\ref{theo:basisunderspecialization}, there exists an ad-open subset $\U_2$ of $\X(k)$ such that $p_1^c,\dots,p_m^c$ is a $k(x)$-basis of $\algrel(\overline{Y^c},\bfd(n),k(x))$ for any $c\in\U_2$.
	
	We will show that $R^c/k(x)$ is proto-Picard-Vessiot for any $c\in \U_1\cap\U_2$. Let $p\in k(x)[X]$ with $\deg(p)\leq \bfd(n)$ and $p(Y^c)\in\m_c$. According to Lemma \ref{LM:criterionforprotoPV}, it suffices to show that $p(Y^c)=0$.

	As $p(\overline{Y^c})=0$, i.e., $p\in\algrel(\overline{Y^c},\bfd(n),k(x))$,  we have $p=\sum_{i=1}^m a_i p_i^c$ for appropriate $a_1,\ldots,a_m\in k(x)$.
	Therefore, $p(Y^c)=\sum_{i=1}^m a_i p_i^c(Y^c)=0$.
	Thus, we can choose $\U$ as any ad-open subset of $\U_1\cap\U_2$.
\end{proof}

%\subsection{Specializations of differential torsors: the toric part}
%\label{subsec: Specializations of differential torsors: the toric part}
%
%In the previous subsection we have seen that, in the context of Notation \ref{notation: setup}, there exists a large set of specializations $c$ such that $R^c/k(x)$ is proto-Picard-Vessiot. 

We next make the step from proto-Picard-Vessiot to Picard-Vessiot. Roughly, this corresponds to the second main step of Hrushovski's algorithm, which is more or less the Compoint-Singer algorithm from \cite[Section 2.5]{CompointSinger:ComputingGaloisGroupsOfCompletelyReducibleDifferentialEquations}. 
 Besides Proposition \ref{PROP:protoPV}, key ingredients for the proof are the characterization of Picard-Vessiot rings among proto-Picard-Vessiot rings from Lemma \ref{LM:criterion} and the preservation of logarithmic independence under specialization from Theorem \ref{theo: main Appendix B}.
%We first develop a practical criterion to decide if a proto-Picard-Vessiot ring is Picard-Vessiot.
%\label{sec:specializationsofalgebraicfunctions}
%In this appendix, we investigate the behavior of logarithmically independent algebraic functions under specialization.
%Our main result is Theorem \ref{theo: main Appendix B}, which, roughly speaking, states that a finite family of logarithmically independent algebraic functions, remains logarithmically independent under many specializations. 
%
%
% A key ingredient for the proof is the main result of Appendix B, stating that logarithmic independence of algebraic functions is preserved under many specializations.
The following theorem is our main specialization result.
 We assume the setup of Notation \ref{notation: setup}.

\begin{theo} \label{theo: main specialization}
There exists an ad$\times$Jac-open subset $\U$ of $\X(k)$ such that $R^c/k(x)$ is Picard-Vessiot for all $c\in\U$. In particular, the set of all $c\in \X(k)$ such that $R^c/k(x)$ is Picard-Vessiot, is Zariski dense in $\X(k)$.
\end{theo}
\begin{proof}
		Extending $\B$ if necessary (Remark \ref{rem: extend B}), we can assume that the morphism $\G\to \Gl_{n,\B}$ defined by $\Y$ is a closed embedding (Lemma \ref{lemma: closed embedding}). We can thus consider $\G$ as a closed subgroup of $\Gl_{n,\B}$. Note that the equations defining $\G$ as a closed subgroup of $\Gl_{n,\B}$, also define $G^\gen$ as a closed subgroup of $\Gl_{n,K}$. %Let ??$p_1,\ldots,p_d\in K[X,\frac{1}{\det(X)}]$ be such that $\I((G^\gen)^\circ)=(p_1,\ldots,p_d)$, where $\I((G^\gen)^\circ)\subseteq K[X,\frac{1}{\det(X)}]$ is the defining ideal of $(G^\gen)^\circ$.		
	 Let $\{\chi_1,\dots,\chi_m\}$ be a basis of $X((G^\gen)^\circ)$ and let $q_1,\ldots,q_m\in K[X,\frac{1}{\det(X)}]$ be such that the image of $q_i$ in %$K[X,\frac{1}{\det(X)}]/\I((G^\gen)^\circ)
	 $K[(G^\gen)^\circ]$ agrees with $\chi_i$ for $i=1,\ldots,m$. 
%	 \mw{We no longer consider $p_1,\ldots,p_d\in K[X,\frac{1}{\det(X)}]$ such that $\I((G^\gen)^\circ)=(p_1,\ldots,p_d)$}
	 Let $K(x)'$ denote the integral closure of $K(x)$ in $R^\gen$ and let $\vartheta\in \Hom_{K(x)'}(R^\gen, \overline{K(x)})$. Note that $K(x)'$ is canonically embedded into $\overline{K(x)}$ by forming $\overline{K(x)}$ inside the algebraic closure of the field of fractions of $R^\gen$. As $\vartheta(R^\gen)$ is a finite field extension of $K(x)$, there exists an $\eta\in \overline{K(x)}$ such that $\vartheta(R^\gen)=K(x)[\eta]=K(x,\eta)$.  In particular, $B=\vartheta(Y^\gen)^{-1}\in \Gl_n(\overline{K(x)})$ has coefficients in $K(x)[\eta]$.
%??	 Let $\vartheta\in \Hom_{K(x)}(R^\gen, \overline{K(x)})$ be such that $\vartheta(Y^\gen)^{-1}Y^\gen\in (G^\gen)^\circ(E)$, where $E$ is the field compositum of $\overline{K(x)}$ and the field of fractions of $R^\gen$ (inside the algebraic closure of the field of fractions of $R^\gen$). (Such a $\vartheta$ exists by Lemma \ref{lemma: generic} and Corollary \ref{cor: existence of h}.)
	  %Let $\eta\in \overline{K(x)}$ be such that $B=\vartheta(Y^\gen)^{-1}\in \Gl_n(\overline{K(x)})$ has coefficients in $K(x)[\eta]=K(x,\eta)$. 
	  Replacing $\eta$ with a $K(x)$\=/multiple of $\eta$ if necessary, we can assume that the minimal polynomial $p\in K(x)[y]$ of $\eta$ over $K(x)$ has coefficients in $K[x]$. 
	 
	 Extending $\B$ and replacing $f$ by a multiple of $f$ if necessary (Remark \ref{rem: extend B}), we may assume that
	 \begin{itemize}
	 	\item  the coefficients of $p,q_1,\ldots,q_m$ are all in $\B$,
	 	\item $\delta(\eta)\in \B[x]_f[\eta]$, so that $\B[x]_f[\eta]$ is a differential subring of $K(x,\eta)$ and
	 	\item $B,B^{-1}\in \B[x]_f[\eta]^{n\times n}$. %\mw{I deleted $\de(B)$ here, as  $\B[x]_f[\eta]$ is a differential subring this does not seem to be needed.}
	 	\end{itemize}
	
	 By \cite[Prop. 3.5]{Feng:DifferenceGaloisGrousUnderSpecialization}, there exists a nonempty Zariski open subset $\U_1$ of $\X(k)$ such that 
	 %{\color{red} reduced subscheme} 
	 %\mw{We cannot say variety here because varieties are assumed to be irreducible throughout now.}
	 % of $\Gl_{n,k}$ defined by $p_1^c,\ldots,p_d^c$ agrees with $(G^c)^\circ$ and
	   the images $\chi_1^c,\ldots,\chi_m^c$ of $q_1^c,\ldots,q_m^c$ in $k[(G^c)^\circ]$ are a basis of $X((G^c)^\circ)$ for every $c\in\U_1$.
	 
 As $p\in\B[x,y]$ is irreducible in $K(x)[y]$ and monic in $y$, it follows from Gauss's lemma that $p$ is irreducible in $K[x,y]$. By the Bertini-Noether theorem (\cite[Prop.~9.4.3]{FriedJarden:FieldArithmetic}), there exits a nonempty Zariski open subset $\U_2$ of $\U_1$ such that $p^c(x,y)$ is irreducible in $k[x,y]$, and hence also irreducible in $k(x)[y]$, for any $c\in \U_2$. So, for $c\in \U_2$, the ring $\B[x]_f[\eta]\otimes_{\B[x]_f}k(x)=k(x)[y]/(p^c)$ is a field, in fact, a finite field extension of $k(x)$.
	 
	 For $c\in \X(k)$, let $k(x)'$ denote the integral closure of $k(x)$ in $R^c$.
	 Our next goal is to construct, using $\vartheta$, elements $\vartheta^c\in \Hom_{k(x)'}(R^c,\overline{k(x)})$.
%	  such that $\vartheta^c(Y^c)^{-1}Y^c\in (G^c)^\circ(E^c)$, where $E^c$ is the field compositum of $\overline{k(x)}$ and the field of fractions of $R^c$. We know that $p_i(\vartheta(Y^\gen)^{-1}Y^\gen)=0$ for $i=1,\ldots,d$. However, to specialize these equations, we first need to discuss extensions of $c\colon \R\to R^c,\ r\mapsto r^c$.
	 Recall (Lemma \ref{lemma: injective}) that we may regard $\R$ as a $\de$-subring of $R^\gen$. So the morphism $\vartheta\colon R^\gen\to K(x,\eta)$ of $K(x)'$\=/algebras restricts to a morphism $\vartheta\colon \R\to \B[x]_f[\eta]$ of $\B[x]_f'$-algebras. Thus for $c\in\U_2$ we can base change $\vartheta\colon \R\to \B[x]_f[\eta]$ via $c\colon \B[x]_f\to k(x)$ to a morphism $\theta^c\colon R^c=\R\otimes_{\B[x]_f}k(x)\to \B[x]_f[\eta]\otimes_{\B[x]_f}k(x)$ of $\B[x]'_f\otimes_{\B[x]_f}k(x)$-algebras. By Lemma \ref{lemma: integral closure specializes}, there exists a nonempty Zariski open subset $\U_3$ of $\U_2$ such that the map $\B[x]'_f\otimes_{\B[x]_f}k(x)\to k(x)'\subseteq R^c$ is an isomorphism. So composing $\theta^c$ with a $\B[x]'_f\otimes_{\B[x]_f}k(x)=k(x)'$ embedding of $\B[x]_f[\eta]\otimes_{\B[x]_f}k(x)$ into $\overline{k(x)}$, we obtain a morphism $\vartheta^c\colon R^c\to \overline{k(x)}$ of $k(x)'$-algebras. Note that $B^c=\vartheta^c(Y^c)^{-1}\in\Gl_n(\overline{k(x)})$ agrees with the image of $B$ under $\B[x]_f[\eta]\to \B[x]_f[\eta]\otimes_{\B[x]_f}k(x)\to \overline{k(x)}$ for all $c\in\U_3$.

	 For $\ell=1,\ldots,m$ we set
	 $$
	 f_\ell=\sum_{i,j=1}^n\frac{\partial q_\ell}{\partial X_{ij}}(I_n)(B\mathcal{A}B^{-1}+\delta(B)B^{-1})_{ij}\in \B[x]_f[\eta]\subseteq K(x,\eta)\subseteq \overline{K(x)}.
	 $$
	 As $R^\gen/K(x)$ is Picard-Vessiot, it follows from Lemma \ref{LM:criterion} and Remark \ref{rem: logarithmic independence preserved under algebraic extension} that $f_1,\ldots,f_m$ are logarithmically independent over $K(x,\eta)$.
	 
	 For an element $a$ of $\B[x]_f[\eta]$ and $c\in\U_3$, let us write $a^c=a\otimes 1\in \B[x]_f[\eta]\otimes_{\B[x]}k(x)$ for the image of $a$ in $\B[x]_f[\eta]\otimes_{\B[x]}k(x)\hookrightarrow \overline{k(x)}$.
	 %Note that then $\psi_c(a^c)=c(a)$.
	 	 For $\ell=1,\ldots,m$, we then have
	%\begin{equation} \label{eq: formula for fl}
$$
	f_\ell^c=\sum_{i,j=1}^n\frac{\partial q_{\ell}^c}{\partial X_{ij}}(I_n)(B^cA^c(B^c)^{-1}+\delta(B^c)(B^c)^{-1})_{ij}\in \B[x]_f[\eta]\otimes_{\B[x]_f}k(x)\hookrightarrow\overline{k(x)}.
	$$
	%	\end{equation}
  By Theorem~\ref{theo: main Appendix B}, there exists an ad$\times$Jac-open subset $\U'$ of $\X(k)$ contained in $\U_3$ such that $f_1^c,\dots,f_m^c$ are logarithmically independent over $\B[x]_f[\eta]\otimes_{\B[x]_f}k(x)$ for all $c\in \U'$. Then $\U'$ is of the form $\U'=\U_4\cap\mathcal{V}$ with $\U_4$ ad-open and $\mathcal{V}$ Jac-open.
 
  By Proposition~\ref{PROP:protoPV}, there exists an ad-open subset $\U_5$ of $\X(k)$ such that $R^c/k(x)$ is proto-Picard-Vessiot for any $c\in \U_5$. 
  Let $\U_6$ be an ad-open subset of $\U_4\cap\U_5$. Then $\U=\U_6\cap\mathcal{V}$ is ad$\times$Jac-open. 
    %Using the embedding $\psi_c\colon \B[x]_f[\eta]\otimes_{\B[x]_f}k(x)\to \overline{k(x)}$, identity (\ref{eq: formula for fl}) can be interpreted as an identity in $\overline{k(x)}$.
 Now Lemma~\ref{LM:criterion} implies that $R^c/k(x)$ is Picard-Vessiot for any $c\in \U$.
 
 The last statement of the theorem follows from Theorem \ref{theo: ad meets ab is dense}. 
\end{proof}

The proof of Theorem \ref{theo: main specialization} shows that if the algebraic group $G^\gen$ over $K$ is such that $(G^\gen)^t=G^\gen$ (e.g., $G^\gen$ is semisimple or unipotent), then the Jac-open subset is not needed, i.e., the set of good specializations contains an ad-open subset.
However, we can do much better than that:

\begin{cor} \label{cor: no Jac if connected}
	Assume that in the setup of Notation \ref{notation: setup} the generic differential Galois group $G^\gen$ is connected. Then there exists an ad-open subset $\U$ of $\X(k)$ such that $R^c/k(x)$ is Picard-Vessiot for every $c\in\U$.
\end{cor}
\begin{proof}
	If $G^\gen$ is connected, then the $\vartheta\in\Hom_{K(x)}(R^\gen,\overline{K(x)})$ in the proof of Theorem~\ref{theo: main specialization} can be chosen such that $B=\vartheta(Y^\gen)^{-1}\in \Gl_n(K(x))$ by Remark \ref{rem: C1}. We can therefore choose $\eta=1$ and so $f_1,\ldots,f_m\in \B[x]_f\subseteq K(x)$. By Remark \ref{rem: eta=1 logarithmic independence}, the Jac-open subset $\mathcal{V}$ is then not needed so that $\U'=\U_5$ is ad-open and therefore also $\U$ is ad-open.
\end{proof}

While Jac-open subsets do not occur when the generic differential Galois is connected, they do indeed occur in the non-connected case. This is illustrated in the following example.  

\begin{ex} \label{ex: need Jac-open}
		 Let $\alpha$ and $\beta$ be variables over $k$ and set $\B=k[\alpha,\beta]$. In line with Notation~\ref{notation: setup} we also set $K=\overline{k(\alpha,\beta)}$. 

Let $a=\frac{(x-\beta)^2}{x^4+x+\alpha}$, $f=(x^4+x+\alpha)(x-\beta)$ and consider the differential equation
\begin{equation}
\label{eq: dgl for elleptic curve}
\de^2(y)-\big(2a+\tfrac{\de(a)}{2a}\big)\de(y)+\big(a^2-a-\tfrac{\de(a)}{2}\big)y=0
\end{equation}
over $\B[x]_f$. Note that the companion matrix $\cA\in \B[x]_f^{2\times 2}$ of (\ref{eq: dgl for elleptic curve}) is of the form described in Examples \ref{ex: torsor for monomial matrices} and \ref{ex: integral domain} with $b=a$.

As $a$ is not a square in $k(\alpha,\beta)(x)$, we see that $\B[x]_f[y]/(y^2-a)=\B[x]_f[\eta]$ is an integral domain and a differential ring.
As in the previous examples we set $$\R=\B[x]_f[\eta][y_1,y_2,y_1^{-1},y_2^{-1}]=\B[x]_f[X,\tfrac{1}{\det(X)}]/(p_1,p_2)=\B[x]_f[\Y,\tfrac{1}{\det(\Y)}],$$ where
	$$
p_1=X_{21}X_{22}-(a^2-a)X_{11}X_{12} \quad  \quad  p_2=X_{21}X_{12}+X_{22}X_{11}-2aX_{11}X_{12}, $$
$\de(y_1)=(a+\eta)y_1$, $\de(y_2)=(a+\eta)y_2$ and
$\de(\Y)=\cA \Y$.

 We know from Example \ref{ex: torsor for monomial matrices} that $\R/\B[x]_f$ is a differential  $\G$-torsor for $\de(y)=\cA y$ where $\G$ is the group scheme of $2\times 2$ monomial matrices over $\B$.
Therefore, $$R^\gen=\R\otimes_{\B[x]_f}K(x)=K(x,\eta)[y_1,y_2,y_1^{-1},y_{2}^{-1}]=K(x)[Y^\gen,\tfrac{1}{Y^\gen}]$$ 
is a differential $G^\gen=\G_K$-torsor.
Example \ref{ex: elements for logarithmic test} shows that $R^\gen/K(x)$ is Picard-Vessiot if and only if $f_1=a+\eta$ and $f_2=a-\eta$ are logarithmically independent over $K(x,\eta)$.
From Example \ref{ex: get elleptic curve} we know that $f_1,f_2$ are logarithmically independent over $K(x,\eta)$ if the point $(0,1)$ is not a torsion point on the elliptic curve (over $K$) defined by $v^2=u^3-4\alpha u+1$.

Combining Example 2.4 (i) of \cite{Zannier:ICM2014} and Proposition 12.1 of \cite{Zannier:unlikelyintersectionsandpellequationsinpolynomials} we see that the latter is the case and so $R^\gen/K(x)$ is Picard-Vessiot. (Note that in the notation of these references $\infty_+=(0,1)$ and $\infty_-$ is the neutral element.)

Set $\X=\spec(\B)=\mathbb{A}^2_k$. 
We would like to determine an ad$\times$Jac-open subset $\U$ of $\X(k)=k^2$ such that $R^c=\R\otimes_{\B[x]_f}k(x)$ is Picard-Vessiot for every $c\in \U$. To this end, we follow the same strategy as at the generic point.
Example \ref{ex: integral domain} shows that $R^c$ is an integral domain provided that $256(\alpha^c)^3-27\neq 0$. For these $c$'s, Example \ref{ex: elements for logarithmic test} shows that $R^c$ is Picard-Vessiot if and only if $f_1^c=a^c+\eta^c$ and $f_2^c=a^c-\eta^c$ are logarithmically independent over $F^c=k(x,\eta^c)=k(x)[y]/(y^2-a^c)$. With $z^c=\frac{x-\beta^c}{\eta^c}$, we have
$F^c=k(x,z^c)$ and $(z^c)^2=x^4+x+\alpha^c$.

Set $\U=W_{\X}(\Ga, \Gamma) \cap \VV $, where $\Gamma$ is the subgroup of $(\B,+)$ generated by $(\beta^4+\beta+\alpha)^2$ and $256\alpha^3-27$ and  $\VV$ is the set of all $c=(\alpha^c,\beta^c)\in \X(k)$ such that $256(\alpha^c)^3-27\neq 0$ and $(0:1:1)$ is not a torsion point of the elliptic curve $\cE^c$ with affine equation: $u^2=v^3-4\alpha^c v+1$. 
%\rf{Since $E$ is used to denote a field in the previous context, I used $\cE$ to denote the elliptic curve here.}
So $\U$ is an ad$\times$Jac-open subset of $\X(k)$. We will show that $R^c/k(x)$ is Picard-Vessiot for every $c\in \U$. 

Let $c\in \U$. Because $c\in W_{\X}(\Ga, \Gamma)$, we have 
$$
 (\beta^c)^4+\beta^c+\alpha^c\neq 0, \,\,256(\alpha^c)^3-27\neq 0,\,\,\mbox{and}\,\,\frac{((\beta^c)^4+\beta^c+\alpha^c)^2}{256(\alpha^c)^3-27}\notin \bQ.
$$ From Example~\ref{ex: get elleptic curve} with $\alpha,\beta$ being replaced by $\alpha^c,\beta^c$ respectively, we see that $Z_1((f_1^c,f_2^c),\cP)=\{(d,-d)\mid d\in \bZ\}$, where $\cP=\{P_1,P_2,\dots,P_6\}$ are all poles of $f_1^c dx$ and $f_2^c dx$.
%\rf{Since $P_1,P_2$ are used to denote polynomials in the previous context, I changed $P_i$ into $\p_i$ here.}
 In particular, $P_1,P_2$ are places of the function field $F^c=k(x,z^c)$ with $(z^c)^2=x^4+x+\alpha^c$ and $P_1$ is the place corresponding to the neutral element $(0:1:0)\in E^c(k)$, while $P_2$ corresponds to the point $(0:1:1)\in E^c(k)$. For each $(d,-d)\in Z_1((f_1^c, f_2^c), \cP)$,
$$
  \sum_{P\in \cP} \left(\sum_{i=1}^2 d \res_{P} (f_i^c dx)\right)P=2d(P_2-P_1)
 $$
 Furthermore, as $c\in \VV$, $2d(P_2-P_1)$ is not a torsion point of $\cE^c(k)$ if $d\neq 0$. Thus $Z_2((f_1^c,f_2^c), \cP)$ is trivial and therefore $f_1^c, f_2^c$ are logarithmically independent (see Example~\ref{ex: get elleptic curve} for details). So $R^c/k(x)$ is Picard-Vessiot for every $c\in\U$.

In the following, we also show that there are infinitely many $c=(\alpha^c,\beta^c)\in k^2=\X(k)$ with $256(\alpha^c)^3-27\neq 0$ such that $R^c$ is not Picard-Vessiot, i.e., $f_1^c,f_2^c$ are logarithmically dependent. 
By Example 2.4 (i) of \cite{Zannier:ICM2014}, there exist infinitely many $\alpha^c\in k$ with $256(\alpha^c)^3-27\neq 0$ such that $(0,1)$ is a torsion point on the elliptic curve $\cE^c$. For every such $\alpha^c\in k$, let $m_c\geq 2$ denote the order of $(0,1)$. Now $m_c(0,1)$ is the neutral element of $\cE^c$, if and only if the divisor $m_c(P_2-P_1)$ is principal. Thus, there exist an $h_c\in F^c$ such that $(h_c)=m_c(P_2-P_1)$. Note that $\frac{dx}{z^c}$ is a regular differential of $F^c$. The only poles of $\frac{dh^c}{h_c}=\frac{\de(h_c)}{h_c}dx$ are $P_1$ and $P_2$. Moreover, these poles are simple with residues $-m_c$ and $m_c$ respectively. On the other hand, $\frac{m_c x}{z^c}dx=\frac{m_cx^2}{z^c}d(x^{-1})$ and the $P_i$-adic expansion of $z^c$ is $z^c=c_i(x^{-1})^{-1}+\ldots$ for $i=1,2$ with $c_1=-1$ and $c_2=1$ (Example~\ref{ex: get elleptic curve}). It follows that also the differential $\frac{m_cx}{z^c}dx-\frac{\de(h_c)}{h_c}dx$ is regular. 
As $F^c/k$ has genus one, the space of regular differentials is one dimensional. Therefore, there exists an $s\in k$ such that $\frac{m_cx}{z^c}-\frac{\de(h_c)}{h_c}=\frac{s}{z^c}$. If we set $\beta^c=\frac{s}{m_c}\in k$, then
$$m_c(f_1^c-f_2^c)=m_c\eta^c=\tfrac{m_c(x-\beta^c)}{z^c}=\tfrac{\de(h_c)}{h_c},$$
so that indeed $f_1^c,f_2^c$ are logarithmically dependent for $c=(\alpha^c,\beta^c)\in\X(k)$.

As the above $c$'s are not really given explicitly, let us present at least one explicit $c\in W_\X(\Ga,\Gamma)$ such that $R^c/k(x)$ is not Picard-Vessiot. (In this sense the set $\mathcal{V}$ is necessary.)

Set $c=(\alpha^c, \beta^c)=(-\frac{1+\sqrt{-3}}{4}, -\frac{1+\sqrt{-3}}{8})$. Then 
$$
   256\alpha^c-27=5, \quad ((\beta^c)^4+\beta^c+\alpha^c)^2=\frac{37249 (\sqrt{-3}+1)^2}{262144}.
$$
So $c=(\alpha^c,\beta^c)\in W_\X(\Ga,\Gamma)$ but $f_1^c, f_2^c$ are logarithmically dependent. In fact, one has that
$$
  2(f_1^c-f_2^c)dx=4\frac{(x-\frac{1+\sqrt{-3}}{8}) dx}{\sqrt{x^4+x-\frac{1+\sqrt{-3}}{4}}}=\frac{\delta(h)}{h},
$$
where 
$h=p+q\sqrt{x^4+x-\frac{1+\sqrt{-3}}{4}}$ with
\begin{align*}
    p&=(2x^4-x^2+x)\sqrt{3}+\sqrt{-1}(2x^4+4x^3+x^2+x+1),\\
    q&=(2x^2-1)\sqrt{3}+\sqrt{-1}(2x^2+4x+1).
\end{align*}
\end{ex}

\section{Applications of the specialization theorem}

\label{sec: Applications of the specialization theorem}

In this section, we present some applications of our specialization theorem (Theorem~\ref{theo: main specialization}). The main application, which in fact motivated this entire paper, is the proof of Matzat's conjecture. Besides Matzat's conjecture, we also present the proofs for the results announced in the introduction and a very short proof of the solution of the inverse problem.

\subsection{Exceptional parameter values}
In this section we deduce the results announced in the introduction from Theorem~\ref{theo: main specialization}.

\subsubsection{Exceptional parameter values for the algebraic relations among the solutions in a family of linear differential equations} \label{subsubsec: Exception parameter values for relations}

Let $k\subseteq k'$ be an inclusion of algebraically closed fields of characteristic zero and let $R/k'(x)$ be a Picard-Vessiot ring for $\de(y)=
Ay$ with $A\in k'(x)^{n\times n}$. Then $R$ is of the form $R=k'(x)[X,\frac{1}{\det(X)}]/\m$ for some maximal differential ideal $\m$ of $k'(x)[X,\frac{1}{\det(X)}]$, where $\de(X)=AX$. Let $p_1,\ldots,p_m$ be a generating set of the ideal $\m$. Fix a finitely generated $k$-subalgebra $\B_0$ of $k'$ and $f_0\in\B_0[x]$ a monic polynomial such that $A\in \B_0[x]_{f_0}^{n\times n}$ and $p_1,\ldots,p_m\in \B_0[x]_{f_0}[X,\frac{1}{\det(X)}]$. Set $\X_0=\spec(\B_0)$ and for $c_0\in\X_0(k)$ let $\m^{c_0}$ denote the ideal of $k(x)[X,\frac{1}{\det(X)}]$ generated by $p_1^{c_0},\ldots,p_m^{c_0}$.

%\mw{The notation here is not consistent with the notation used in the introduction, but I made it so that it is consistent with Lemma \ref{lemma: spread out PVring}, which I think is more important/useful. Of course we could adapt the notation in the introduction but that would not look so nice, so I think it is best to leave as is.}

\begin{cor} \label{cor: algebraic relations preserved}
	There exists an ad$\times$Jac-open subset $\U_0$ of $\X_0(k)$ such that $\m^{c_0}$ is a maximal differential ideal of $k(x)[X,\frac{1}{\det(X)}]$ for every $c_0\in\U_0$, where $k(x)[X,\frac{1}{\det(X)}]$ as a differential ring with respect to the derivation determined by $\de(X)=A^{c_0}X$.	
\end{cor}
\begin{proof}
	By Lemma \ref{lemma: spread out PVring}, there exists
	\begin{itemize}
		\item a finitely generated $k$-subalgebra $\B$ of $k'$ containing $\B_0$ and contained in the algebraic closure of the field of fractions of $\B_0$;
		\item a monic polynomial $f\in\B[x]$ such that $f_0$ divides $f$ in $\B[x]$;
		\item an affine group scheme $\G$ of finite type over $\B$;
		\item a differential $\G$-torsor $\R/\B[x]_{f}$ for $\de(y)=Ay$
such that 
$\R$ is a flat $\B[x]_{f}$-module,
 $\R\otimes_{\B[x]_{f}}K(x)$ is $\de$-simple, where $K\subseteq k'$ is the algebraic closure of the field of fractions of $\B$ and $\R\otimes_{\B[x]_{f}}k'(x)\simeq R$. 
	\end{itemize}
%	\rf{$\B$ was changed into $\B'$.}
Let $Y$ denote the image of $X$ in $R=k'(x)[X,\frac{1}{\det(X)}]/\m$ and let $\Y\in \Gl_n(\R)$ be the fundamental matrix corresponding to $Y$ under the isomorphism $\R\otimes_{\B[x]_{f}}k'(x)\simeq R$. Then $$\m=\left\{p\in k'(x)[X,\tfrac{1}{\det(X)}]\ \big|\ p(Y)=0\right\}$$ and for
$$\mathcal{I}=\{p\in \B[x]_{f}[X,\tfrac{1}{\det(X)}]\ \big|\ p(\Y )=0\}$$ we have $\mathcal{I}\otimes_{\B[x]_{f}}k'(x)=\m=(p_1,\ldots,p_m)\subseteq k'(x)[X,\frac{1}{\det(X)}]$. 

Set $\X=\spec(\B)$ and let $\f\colon \X(k)\to \X_0(k)$ be the map induced by the inclusion $\B_0\subseteq \B$.
 %For $c'\in\X'(k)$ let $c\in \X(k)$ denote the image of $c'$ under the map

 We claim that there exists a nonempty Zariski open subset 
$\U$ of $\X(k)$
such that $\mathcal{I}\otimes_{\B[x]_{f}}k(x)=\m^{\f(c)}\subseteq k(x)[X,\frac{1}{\det(X)}]$ for all $c\in \U$,
%\rf{$\X'(k)$ was changed into $\U'$.}
 where the tensor product $\mathcal{I}\otimes_{\B[x]_{f}}k(x)$ is formed using $c\colon \B[x]_{f}\to k(x)$. As $p_1,\ldots,p_m\in \mathcal{I}$, we have $\m^{\f(c)}\subseteq \mathcal{I}\otimes_{\B[x]_{f}}k(x)$ for all $c\in \X(k)$.

Set
$$\mathcal{I}'=\{p\in k(\B)(x)[X,\tfrac{1}{\det(X)}]\ \big|\ p(\Y)=0\}.$$
As $\mathcal{I}'\otimes_{k(\B)(x)}k'(x)=\m=(p_1,\ldots,p_m)\subseteq k'(x)[X,\frac{1}{\det(X)}],$
we must have $\mathcal{I}'=(p_1,\ldots,p_m)\subseteq k(\B)(x)[X,\tfrac{1}{\det(X)}]$. 

Let $q_1,\ldots,q_d\in \B[x]_{f}[X,\frac{1}{\det(X)}]$ be such that $\mathcal{I}=(q_1,\ldots,q_d)$. As $q _i\in \mathcal{I}\subseteq \mathcal{I}'=(p_1,\ldots,p_m)$, we can find a nonzero $b\in\B$ and a monic $f'\in \B[x]$ such that $bf'q_i\in (p_1,\ldots,p_m)\subseteq \B[x][X,\frac{1}{\det(X)}]$ for $i=1,\ldots,d$. Let $\U$ be the Zariski open subset of $\X(k)$ where $b$ does not vanish. Then $q_i^{c}\in(p_1^{c},\ldots,p_m^{c})\subseteq k(x)[X,\frac{1}{\det(X)}]$ for $i=1,\ldots,d$ for every $c\in\U$. Therefore
$$\mathcal{I}\otimes_{\B[x]_{f}}k(x)=(q_1,\ldots,q_d)\otimes_{\B[x]_{f}}k(x)=(q_1^{c},\ldots,q_d^{c})\subseteq (p_1^{c},\ldots,p_m^{c})=\m^{\f(c)}$$
and so  $\mathcal{I}\otimes_{\B[x]_{f}}k(x)=\m^{\f(c)}$ for all $c\in\U$ as claimed. In particular, since $\mathcal{I}$ is a differential ideal, also $\m^{\f(c)}$ is a differential ideal for all $c\in\U$.

Theorem \ref{theo: main specialization} applied to the differential torsor $\R/\B[x]_{f}$ yields an ad$\times$Jac-open subset $\mathcal{V}$ of $\X(k)$ such that $\R\otimes_{\B[x]_{f}}k(x)$ is $\de$-simple for every $c\in \mathcal{V}$. For $c\in\U\cap\mathcal{V}$ we then have that
\begin{align*}
\R\otimes_{\B[x]_{f}}k(x)=(\B[x]_{f}[X,\tfrac{1}{\det(X)}]/\mathcal{I})\otimes_{\B[x]_{f}}k(x) & =     k(x)[X,\tfrac{1}{\det(X)}]/\mathcal{I}\otimes_{\B[x]_{f}}k(x)= \\
& =k(x)[X,\tfrac{1}{\det(X)}]/\m^{\f(c)}
\end{align*}
is $\de$-simple, i.e., $\m^{\f(c)}$ is a maximal $\de$-ideal. By Lemma \ref{lemma: lift opens}, there exists an ad$\times$Jac-open subset $\U_0$ of $\X_0(k)$ such that $\U_0\subseteq \f(\U\cap\mathcal{V})$. Then $\m^{c_0}$ is a maximal $\de$-ideal for all $c_0\in \U_0$.
\end{proof}

\begin{rem} \label{rem: no Jac for connected in relations}
	If, in the context of Corollary \ref{cor: algebraic relations preserved}, the differential Galois group of $\de(y)=Ay$ (over $k'(x)$) is connected, then the set $\U_0$ can be chosen to be ad-open.
\end{rem}
\begin{proof}
	If the differential Galois group $G$ of $\de(y)=Ay$ is connected, also the differential Galois group $G^\gen$ of $\R^\gen=\R\otimes_{\B[x]_{f}}K(x)$ is connected because $G=(G^\gen)_{k'}$ by Lemma~\ref{lemma: base change of PV ring over constants}. Thus, by Corollary \ref{cor: no Jac if connected}, the ad$\times$Jac-open subset $\mathcal{V}$ of $\X(k)$ in the proof of Corollary \ref{cor: algebraic relations preserved} can be chosen to be ad-open. Then by Lemma \ref{lemma: lift opens} also $\U_0$ can be chosen to be ad-open.
\end{proof}

\subsubsection{Families of not solvable algebraic groups} \label{subsubsec: Families of non-solvable algebraic groups}

We next work towards a proof of Corollary B1 from the introduction. A linear differential equation is solvable by Liouvillian functions if and only if the identity component of the differential Galois group is solvable (\cite[Theorem 1.43]{SingerPut:differential}). For the proof of Corollary~B1 we need to know that the contrapositive of this property spreads out from the generic fibre to an open subset of the parameter space.

Throughout Section \ref{subsubsec: Families of non-solvable algebraic groups} we make the following assumptions:
\begin{itemize}
	\item $k$ is an algebraically closed field of characteristic zero;
	\item $\B$ is a finitely generated $k$-algebra that is an integral domain;
	\item $K=\overline{k(\B)}$ is the algebraic closure of the field of fractions $k(\B)$ of $\B$;
	\item $\X=\spec(\B)$;
	\item $\G$ is an affine group scheme of finite type over $\X$ such that $\G\to \X$ is dominant, i.e., the dual map $\B\to \B[\G]$ is injective.
	\item $\xi$ is the generic point of $\X$ and $\G_\xi=\G\times_\X\spec(k(\xi))$ is the generic fibre (an algebraic group over $k(\xi)=k(\B)$) whereas $\G_{\overline{\xi}}=\G\times_\X\spec(K)$ is the geometric generic fibre (an algebraic group over $K$). Moreover, $\G_c$ is the fibre over $c\in\X$ (an algebraic group over the residue field of $c$).
%	\rf{As $x$ is used to denote the differential variable, I changed $x$ into $\p$ that is used to denote a point of $\X$ in the proof of Lemma 2.67.}
%	\mw{I was not happy with the notation $\p$ here. For me $\p$ should only be used for a prime ideal. From a formal point of view the elements of $\X$ are indeed prime ideals but one should not think of them as prime ideals, rather as points of a topological space. So I used $c$ instead, this is not ideal because usually we only use $c$ for $c\in\X(k)$ and not for $c\in \X$ but I hope this is okay.}
\end{itemize}	

The main goal of this subsection is to prove the following:

\begin{prop} \label{prop: not solvable spreads out}
	If $(\G_{\overline{\xi}})^\circ$ is not solvable, then there exists a nonempty Zariski open subset $\U$ of $\X(k)$ such that $(\G_{c})^\circ$ is not solvable for every $c\in \U$.
\end{prop}

The proof of Proposition \ref{prop: not solvable spreads out} is given at the end of Section \ref{subsubsec: Families of non-solvable algebraic groups}. We will need some properties of the derived subgroup of an algebraic group. For more background see \cite[Section 6 d]{Milne:AlgebraicGroupsTheTheoryOfGroupSchemesOfFiniteTypeOverAField} or \cite[Section 10.1]{Waterhouse:IntroductiontoAffineGroupSchemes}.
Let $k'$ be a field of characteristic zero and let $G$ be an algebraic group over $k'$. The \emph{derived group} $\mathcal{D}(G)$ of $G$ can be defined as the smallest closed normal subgroup of $G$ such that $G/\mathcal{D}(G)$ is abelian.  Alternatively, it can also be described as the closed subgroup of $G$ generated by the commutator map $G\times G\to G,\ (g,h)\mapsto ghg^{-1}h^{-1}$ (\cite[Prop. 6.18]{Milne:AlgebraicGroupsTheTheoryOfGroupSchemesOfFiniteTypeOverAField}). More generally, for every $n\geq 1$, consider the map $$\f_n\colon G^{2n}\to G,\ (g_1,h_1,\ldots,g_n,h_n)\mapsto g_1h_1g_1^{-1}h_1^{-1}\ldots g_nh_ng_n^{-1}h_n^{-1}$$ and its dual $\f_n^*\colon k'[G]\to \otimes^{2n}k'[G]$. 
So $\ker(\f_n^*)$ is the defining ideal of the closure of the image of $\f_n$.
One has $\ker(\f_1^*)\supseteq \ker(\f_2^*)\supseteq\ldots$ and the defining ideal $\I(\mathcal{D}(G))\subseteq k'[G]$ of $\mathcal{D}(G)$ is $\I(\mathcal{D}(G))=\bigcap_{n\geq 1}\ker(\f_n^*)$. In fact, $\I(\mathcal{D}(G))=\ker(\f_n^*)$ for some $n$ (\cite[Prop.~6.20]{Milne:AlgebraicGroupsTheTheoryOfGroupSchemesOfFiniteTypeOverAField}).

The algebraic group $G$ is \emph{perfect} if $\mathcal{D}(G)=G$. Thus $G$ is perfect if and only of $\f_n^*$ is injective for some $n\geq 1$. In other words, $G$ is perfect if and only if $\f_n$ is dominant for some $n\geq 1$.

The higher derived groups $\mathcal{D}^i(G)$ are defined recursively by $\mathcal{D}^i(G)=\mathcal{D}(\mathcal{D}^{i-1}(G))$. Then $G\supseteq \mathcal{D}(G)\supseteq\mathcal{D}^2(G)\supseteq\ldots$ is a descending sequence of closed subgroups (the derived series) that must eventually terminate, say at $\mathcal{D}^n(G)$. If $\mathcal{D}^n(G)=1$ is the trivial group, then $G$ is called \emph{solvable}. Otherwise, $\mathcal{D}^n(G)$ is a nontrivial perfect closed subgroup of $G$.

\begin{lemma} \label{lemma: perfect spreads out}
 If $\G_\xi$ is perfect and nontrivial, then there exists a nonempty Zariski open subset $\U$ of $\X$ such that $\G_{c}$ is perfect and nontrivial for every $c\in \U$.
% \rf{Should $\X$ be replaced with $\X(k)$ here?}
% \mw{We could replace $\X$ by $X(k)$ but it is also correct as it stands.}
\end{lemma}	
\begin{proof}
	As above, for $n\geq 1$ consider the scheme morphism $$\f_n\colon\G^{2n}\to\G,\ (g_1,h_1,\ldots,g_n,h_n)\mapsto g_1h_1g_1^{-1}h_1^{-1}\ldots g_nh_ng_n^{-1}h_n^{-1}.$$ 
	Since $\G_\xi$ is perfect $(\f_n)_\xi\colon \G_\xi^n\to \G_\xi$ is dominant for some $n\geq 1$. By  \cite[Theorem~9.6.1~(ii)]{Grothendieck:EGAIV3}, the set of all $c\in \X$ such that $(\f_n)_c\colon \G_c^n\to \G_c$ is dominant is constructible. As it contains the generic point $\xi$, this set thus contains a nonempty Zariski open subset $\U'$ of $\X$. So for every $c\in \U'$, the algebraic group $\G_c$ is perfect.
	
	By generic freeness (\cite[\href{https://stacks.math.columbia.edu/tag/051S}{Tag 051S}]{stacks-project}), there exists nonzero $b\in \B$ such that $\B[\G]_b$ is a free $\B_b$-module, say of rank $\kappa$.
		Then, $\B[\G]\otimes_\B k(\p)$ is a $k(\p)$-vector space of dimension $\kappa$ for every prime ideal $\p\in D(b)$. The assumption that $\G_\xi$ is nontrivial means that $\B[\G]\otimes_\B k(\B)$ is not reduced to $k(\B)$. So $\kappa > 1$ and $\B[\G]\otimes_\B k(\p)$ is not reduced to $k(\p)$ for any $\p\in D(b)$, i.e., $\G_c$ is nontrivial for $c\in D(b)$. Thus $\U=\U'\cap D(b)$ has the required property.
\end{proof}

\begin{lemma} \label{lemma: nonsolvable spreads out}
 If $\G_\xi$ is not solvable, there exists a nonempty Zariski open subset $\U$ of $\X$ such that $\G_c$ is not solvable for every $c\in \U$.
\end{lemma}
\begin{proof}
	Since $\G_{\xi}$ is not solvable, there exists an $n\geq 1$ such that $H=\mathcal{D}^n(\G_\xi)$ is nontrivial and perfect. The closed subgroup $H$ of $\G_{\xi}$ spreads out to a closed subgroup $\H$ over a nonempty Zariski open subset of $\X$. In detail, there exists a nonempty affine Zariski open $\U'$ of $\X$ and a closed subgroup scheme $\mathcal{H}$ of $\G_{\U'}=\G\times_\X \U'$ such that $\mathcal{H}_\xi=H$. This follows, for example, from \cite[Theorems 8.8.2 and 8.10.5]{Grothendieck:EGAIV3}.	
	
	 Lemma \ref{lemma: perfect spreads out} applied to $\H$, yields a nonempty Zariski open subset $\U$ of $\U'$ such that $\mathcal{H}_{c}$ is nontrivial and perfect for every $c\in \U$. Thus, for every $c\in \U$, the algebraic group $\G_c$ contains the closed nontrivial perfect subgroup $\mathcal{H}_c$ and can therefore not be solvable.
\end{proof}
We are now prepared to prove the main result of this subsection.

\begin{proof}[Proof of Proposition \ref{prop: not solvable spreads out}]
Because solvability (\cite[Cor. 6.31]{Milne:AlgebraicGroupsTheTheoryOfGroupSchemesOfFiniteTypeOverAField}) and the formation of the identity component (\cite[Prop. 1.34]{Milne:AlgebraicGroupsTheTheoryOfGroupSchemesOfFiniteTypeOverAField}) is compatible with base change, we see that $(\G_\xi)^\circ$ is not solvable. As in the proof of Lemma \ref{lemma: nonsolvable spreads out}, the closed subgroup $H=(\G_\xi)^\circ$ of $\G_\xi$ spreads out to a closed subgroup scheme over a nonempty Zariski open subset of $\X$, i.e., there exists a nonempty Zariski open subset $\U_1$ of $\X$ and a closed subgroup $\H$ of $\G_{\U_1}=\G\times_\X\U_1$ such that $\H_\xi=(\G_\xi)^\circ$. Because $\H_c$ is (geometrically) connected, there exists a nonempty Zariski open subset $\U_2$ of $\U_1$ such that $\H_{c}$ is (geometrically) connected for every $c\in \U_2$ (\cite[Theorem 9.7.7]{Grothendieck:EGAIV3}). Moreover, there exists a nonempty Zariski open subset $\U_3$ of $\U_1$ such that $\dim(\G_{c})=\dim(\G_\xi)$ and $\dim(\H_c)=\dim(H_\xi)(=\dim(\G_\xi))$ for all $c\in\U_3$ (\cite[\href{https://stacks.math.columbia.edu/tag/05F7}{Tag~05F7}]{stacks-project}). In summary, this shows that we can find a nonempty Zariski open subset $\U'$ of $\U_1$ such that $\H_{c}=(\G_{c})^\circ$ for all ${c}\in \U'$.

Applying Lemma \ref{lemma: nonsolvable spreads out} to the affine group scheme $\H_{\U'}=\H\times_{\U_1}\U'$ over $\U'$ yields a nonempty Zariski open susbset $\U''$ of $\U'$ such that $\H_{c}=(\G_{c})^\circ$ is not solvable for all ${c}\in \U''$. In particular, $(\G_{c})^\circ$ is not solvable for every $c\in \U=\U''\cap \X(k)$.
\end{proof}

\subsubsection{Exceptional parameter values for solving in a family of linear differential equations}
\label{subsubsec:Exceptional parameter values for solving }

Throughout Subsection \ref{subsubsec:Exceptional parameter values for solving } we make the following assumptions: 

\begin{itemize}
	\item $k\subseteq k'$ is an inclusion of algebraically closed fields of characteristic zero;
	\item $\de(y)=Ay$ is a linear differential equation over $k'(x)$, i.e., $A\in k'(x)^{n\times n}$;
	\item  $\B_0\subseteq k'$ is a finitely generated $k$-algebra and $f_0\in\B_0[x]$ is a monic polynomial such that $A\in\B_0[x]_{f_0}^{n\times n}$;
	\item $\X_0=\spec(\B_0)$;
	\item for $c_0\in\X_0(k)$ we denote with $A^{c_0}\in k(x)^{n\times n}$ the matrix obtained from $A$ by applying $c_0\colon \B_0\to k$ to the coefficients of the entries of $A$.
	
\end{itemize}

We first treat the case of solving in Liouvillian extensions. Recall that a subset of $\X_0(k)$ is called ad$\times$Jac-closed (or ad-closed) if its complement is ad$\times$Jac-open (or ad-open).

\begin{cor} \label{cor: Liouvillian solutions preserved}
	Assume that the differential equation $\de(y)=A y$ (over $k'(x)$) does not have a basis of solutions consisting of Liouvillian functions, then the set of all $c_0\in\X_0(k)$ such that the differential equation $\de(y)=A^{c_0}y$ (over $k(x)$) has a basis of solutions consisting of Liouvillian functions, is contained in an ad$\times$Jac-closed subset of $\X_0(k)$.
\end{cor}
\begin{proof}
		Let $R/k'(x)$ be a Picard-Vessiot ring for $\de(y)=Ay$. By Lemma \ref{lemma: spread out PVring}, there exists
	\begin{itemize}
		\item a finitely generated $k$-subalgebra $\B$ of $k'$ containing $\B_0$ and contained in the algebraic closure of the field of fractions of $\B_0$;
		\item a monic polynomial $f\in\B[x]$ such that $f_0$ divides $f$ in $\B[x]$;
		\item an affine group scheme $\G$ of finite type over $\B$;
		\item a differential $\G$-torsor $\R/\B[x]_{f}$ for $\de(y)=Ay$ such that 
		$\R$ is a flat $\B[x]_{f}$-module, $\R\otimes_{\B[x]_{f}}K(x)$ is $\de$-simple, where $K$ is the algebraic closure of the field of fractions of $\B$ and $\R\otimes_{\B[x]_{f}}k'(x)\simeq R$.
	\end{itemize}

	Note that $\R\otimes_{\B[x]_{f}}K(x)$ is a Picard-Vessiot ring over $K(x)$ with differential Galois group $\G_{K}$ (Lemma \ref{lemma: generic}). So it follows from Lemma \ref{lemma: base change of PV ring over constants} that $\R\otimes_{\B[x]_{f}}k'(x)\simeq R$ is a Picard-Vessiot ring over $k'(x)$ with differential Galois group $\G_{k'}$. Because $\de(y)=Ay$ (over $k'(x)$) does not have a basis consisting of Liouvillian functions, we know that $(\G_{k'})^\circ$ is not solvable (\cite[Theorem 1.43]{SingerPut:differential}). As the formation of the identity component (\cite[Prop. 1.34]{Milne:AlgebraicGroupsTheTheoryOfGroupSchemesOfFiniteTypeOverAField}) and  solvability (\cite[Cor. 6.31]{Milne:AlgebraicGroupsTheTheoryOfGroupSchemesOfFiniteTypeOverAField}) is compatible with base change, this implies that $(\G_{K})^\circ$ is not solvable. 
	
	Set $\X=\spec(\B)$. By Proposition \ref{prop: not solvable spreads out}, there exists a nonempty Zariski open subset $\U_1$ of $\X(k)$ such that $(\G_{c})^\circ$ is not solvable for every $c\in \U_1$.

Theorem \ref{theo: main specialization} applied to the differential torsor $\R/\B[x]_{f}$ for $\de(y)=Ay$ yields an ad$\times$Jac-open subset $\U_2$ of $\X(k)$ such that $R^{c}=\R\otimes_{\B[x]_{f}}k(x)$ is Picard-Vessiot for all $c\in \U_2$. By Lemma \ref{lemma: lift opens}, there exists an ad$\times$Jac-open subset $\U_0$ of $\X_0(k)$ such that $\U_0$ is contained in the image $\U_1\cap\U_2$ under $\X(k)\to \X_0(k)$. For any $c_0\in \U_0$, there thus exists a $c\in\U_1\cap\U_2$ mapping to $c_0$. We then have $A^{c_0}=A^{c}$ and $R^{c}/k(x)$ is a Picard-Vessiot ring for $\de(y)=A^cy$ with differential Galois group $G^{c}=\G_{c}$. As $(\G_{c})^\circ$ is not solvable, we see that $\de(y)=A^{c_0}y$ does not have a basis of solutions consisting of Liouvillian functions for every $c_0\in \U_0$. Thus the set of all $c_0\in\X_0(k)$ such that $\de(y)=A^{c_0}y$ has a basis of solutions consisting of Liouvillian functions, is contained in the complement of $\U_0$.
\end{proof}

\begin{rem}
	If, in the context of Corollary \ref{cor: Liouvillian solutions preserved}, the differential Galois group of $\de(y)=Ay$ (over $k'(x)$) is connected, we can make do with an ad-closed subset of $\X_0(k)$.
\end{rem}
\begin{proof}
	As in Remark \ref{rem: no Jac for connected in relations}, we see that the differential Galois group $G^\gen$ of $\R^\gen$ is connected and so $\U_2$ can be chosen to be ad-open by Corollary \ref{cor: no Jac if connected}.
\end{proof}

We next treat the case of solving in algebraic extensions.

\begin{cor} \label{cor: algebraic solutions preserved}
	Assume that the differential equation $\de(y)=A y$ (over $k'(x)$) does not have a basis of solutions consisting of algebraic functions, then the set of all $c_0\in\X_0(k)$ such that the differential equation $\de(y)=A^{c_0}y$ (over $k(x)$) has a basis consisting of algebraic functions, is contained in an ad$\times$Jac-closed subset of $\X_0(k)$.
\end{cor}
\begin{proof}
	This could be deduced directly from Corollary \ref{cor: algebraic relations preserved}. Alternatively, noting that a differential equation has a basis of algebraic solutions if and only if its differential Galois group is finite, this can be proved exactly like Corollary \ref{cor: Liouvillian solutions preserved} but replacing Proposition~\ref{prop: not solvable spreads out} with the following statement: With the notation of Subsection \ref{subsubsec: Families of non-solvable algebraic groups}, if $\G_{\overline{\xi}}$ is finite, then there exists a nonempty Zariski open subset $\U$ of $\X(k)$ such that $\G_c$ is finite for every $c\in \U$. 
\end{proof}
%\rf{Except for Corollary \ref{cor: algebraic relations preserved} do we also need an argument on dimension is preserved on a nonempty Zariski open subset?}
%\mw{Yes, you are right, to deduce Cor 5.8 from Cor 5.1 one would need to use that dimension is preserved on a nonempty Zariski open subset. }
\begin{rem}
	Again, if in the context of Corollary \ref{cor: algebraic solutions preserved}, the differential Galois group of $\de(y)=Ay$ (over $K(x)$) is connected, then we can make do with an ad-closed subset of $\X_0(k)$.
\end{rem}

\subsection{A short solution of the inverse problem}
\label{subsec: inverse problem}

The \emph{solution of the inverse problem} states that every linear algebraic group over $k$ is a differential Galois group over $k(x)$. Using the Riemann-Hilbert correspondence, this was proved for $k=\C$ already in 1979 (\cite{TretkoffTretkoff:SolutionOfTheInverseProblem}) but it took more than 25 years and contributions of many authors to finally solve the inverse problem for an arbitrary algebraically closed field $k$ of characteristic zero (\cite{Hartmann:OnTheInverseProblemInDifferentialGaloisTheory}).

In this short section, we explain how Theorem \ref{theo: main specialization} can be used to deduce the solution of the inverse problem over $k$, from the solution of the inverse problem over $\C$. Of course, our solution is only short if one accepts Theorem \ref{theo: main specialization} and the solution of the inverse problem over $\C$, as given.

\begin{theo}
	Let $k$ be an algebraically closed field of characteristic zero. Then every linear algebraic group over $k$ is a differential Galois group over $k(x)$.
\end{theo}
\begin{proof}
	Let $G$ be a linear algebraic group over $k$. Then $G$ descends to a finitely generated field, i.e., there exists a subfield $k_0$ of $k$ finitely generated over $\mathbb{Q}$ and a linear algebraic group $G_0$ over $k_0$ such that $(G_0)_k=G$. 
	Passing to the algebraic closure $k_1\subseteq k$ of $k_0$, we obtain a countable algebraically closed field $k_1$ and a linear algebraic group $G_1$ over $k_1$ such that $(G_1)_k=G$.
	
	By Lemma \ref{lemma: base change of PV ring over constants}, it suffices to show that $G_1$ is a differential Galois group over $k_1(x)$. Since $k_1$ is countable, there exists an embedding of $k_1$ into $\mathbb{C}$ that we now fix.
	
	By \cite{TretkoffTretkoff:SolutionOfTheInverseProblem}, there exists a Picard-Vessiot ring $R/\C(x)$ for some equation $\de(y)=Ay$, ($A\in\C(x)^{n\times n}$) with differential Galois group $(G_1)_\C$. The Picard-Vessiot ring $R/\C(x)$ can now be spread out by Lemma \ref{lemma: spread out PVring}, i.e., there exists
	\begin{itemize}
		\item a finitely generated $k_1$-subalgebra $\B$ of $\C$,
		\item a monic polynomial $f\in \B[x]$ with $A\in\B[x]_f^{n\times n}$
		\item and a differential $(G_1)_\B$-torsor $\R/\B[x]_f$ for $\de(y)=Ay$ such that $\R$ is flat over $\B[x]_f$ and $\R\otimes_{\B[x]_f}K(x)$ is $\de$-simple, where $K$ is the algebraic closure of the field of fractions of $\B$.
	\end{itemize}
	Set $\X=\spec(\B)$.
	By the specialization theorem (Theorem \ref{theo: main specialization}), there exists a $c\in \X(k_1)$ such that $R^c=\R\otimes_{\B[x]_f}k_1(x)$ is Picard-Vessiot. So $R^c/k_1(x)$ is a Picard-Vessiot ring for $\de(y)=A^cy$ with differential Galois group $((G_1)_\B)^c=((G_1)_\B)_{k_1}=G_1$. In particular, $G_1$ is a differential Galois group over $k_1(x)$ as desired.	 
\end{proof}

\subsection{Proof of Matzat's conjecture}
\label{subsec: Proof of Matzat's conjecture}

In this section we show that every differential embedding problem of finite type over $(k(x),\frac{d}{dx})$ has a solution (Theorem \ref{theo: solve embedding problems}) and deduce Matzat's conjecture from this.
 The idea for the proof is to first solve the differential embedding problem over a sufficiently large field of constants $k'\supseteq k$ and then to go back down to $k$ using Theorem~\ref{theo: main specialization}.
 
 \medskip
 
 Throughout Section \ref{subsec: Proof of Matzat's conjecture} we assume that $k$ is an algebraically closed field of characteristic zero. Let us recall some background on differential embedding problems. Let $F$ be a differential field with $F^\de=k$. If $R\subseteq S$ is an inclusion of Picard-Vessiot rings over $F$, then there is a restriction map $G(S/F)\twoheadrightarrow G(R/F)$ on the differential Galois groups which is a quotient map, i.e., the dual map is injective.
 
\begin{defi} \label{defi: differential embedding problem}
	A \emph{differential embedding problem} of finite type over $F$ is a pair $(\f\colon G\twoheadrightarrow H,\ R)$, consisting of a quotient map $\f\colon G\twoheadrightarrow H$ of linear algebraic groups over $k$ and a Picard-Vessiot ring $R/F$ with differential Galois group $H$. A \emph{solution} is a Picard-Vessiot ring $S/F$ containing $R$ together with an isomorphism $G\simeq G(S/F)$ such that
	\begin{equation} \label{eq: embedding problem}
	\xymatrix{
		G \ar^-\simeq[rr] \ar@{->>}[rd] & & G(S/F) \ar@{->>}[ld] \\
		& H=G(R/F)	
	}
	\end{equation}
	commutes. In other words, the restriction map $G(S/F)\twoheadrightarrow G(R/F)$ realizes the given morphism $\f\colon G\twoheadrightarrow H$.
\end{defi}
The attribute ``of finite type'' in Definition \ref{defi: differential embedding problem} refers to the assumption that $G$ (and hence $H$) is a group scheme of finite type (over $k$). In general, to deal with Matzat's conjecture, it is important to allow more general differential embedding problems (not necessarily of finite type). See \cite{BachmayrHarbaterHartmannWibmer:FreeDifferentialGaloisGroups}. However, for our purpose differential embedding problems of finite type are sufficient. On the ring side, the commutativity of (\ref{eq: embedding problem})
is expressed by the commutativity of
$$
\xymatrix{
	S \ar[r] & S\otimes_k k[G] \\
	R \ar[r] \ar[u] & R\otimes_k k[H] \ar[u]	
}
$$
where $R\otimes_k k[H]\to S\otimes_k k[G],\ r\otimes a\mapsto r\otimes\f^*(a)$ with $\f^*\colon k[H]\to k[G]$ the dual of $\f$.

As detailed in the following remark, a differential embedding problem over $k(x)$ can always be extended to a differential embedding problem over $k'(x)$.

\begin{rem} \label{rem: embedding problem}
Let $(\f\colon G\twoheadrightarrow H,\ R)$ be a differential embedding problem of finite type over $k(x)$ and let $k'$ be an algebraically closed field extension of $k$. Then $R'=R\otimes_{k(x)}{k'(x)}$ is a Picard-Vessiot ring over $k'(x)$ with differential Galois group $H_{k'}$ (Lemma \ref{lemma: base change of PV ring over constants}).
%\rf{$k(x)$ was changed into $k'(x)$.}
 Therefore, $(\f_{k'}\colon G_{k'}\twoheadrightarrow H_{k'}, R')$ is a differential embedding problem of finite type over $k'(x)$.
\end{rem}

To apply Theorem \ref{theo: main specialization} we first need to ``spread out'' a solution of the extended differential embedding problem into a family with nice properties. This is achieved in the following lemma. Recall that actions of affine group schemes on differential algebras were defined in Section \ref{subsec: differential torsors}.

\begin{lemma} \label{lemma: spread out solution to differential embedding problem}
	Let $(\f\colon G\twoheadrightarrow H,\ R)$ be a differential embedding problem of finite type over $k(x)$. Let $k'$ be an algebraically closed field extension of $k$ and let $S'/k'(x)$ be a solution of the induced differential embedding problem $(\f_{k'}\colon G_{k'}\twoheadrightarrow H_{k'},\ R')$ over $k'(x)$, where $R'=R\otimes_{k(x)}k'(x)$. Then there exist
	\begin{enumerate}
		\item a monic polynomial $h\in k[x]$, and a $k[x]_h$-$\de$-subalgebra $R_0$ of $R$ with an action of $H$ on $R_0/k[x]_h$
		such that $R_0\otimes_{k[x]_h}k(x)=R$  (equivalently, $U^{-1}R_0=R$, where $U=k[x]_h\smallsetminus \{0\}$) and the induced action of $H$ on $U^{-1}R_0/U^{-1}k[x]_h$ is the natural action of $H$ on $R/k(x)$,
%		\rf{Since $S'$ and $S^{\gen}$ are used to denote the rings, to avoid the confusion, I changed $S$ into $U$ here. }
		\item a finitely generated $k$-subalgebra $\mathcal{B}$ of $k'$ and a monic polynomial $f\in \mathcal{B}[x]$ such that $h\in k[x]\subseteq \mathcal{B}[x]$ divides $f$ in $\mathcal{B}[x]$,
		\item a matrix $\mathcal{A}\in \B[x]_f^{n\times n}$ and a differential $G_\B$-torsor $\mathcal{S}/\B[x]_f$ for $\de(y)=\cA y$ such that $\mathcal{S}$ is flat over $\B[x]_f$
		and $S^\gen=\mathcal{S}\otimes_{\B[x]_f}K(x)$ is $\de$-simple, where $K=\overline{k(\mathcal{B})}\subseteq k'$ is the algebraic closure of the field of fractions of $\B$,
	
		\item a morphism $\mathcal{R}\to \mathcal{S}$ of $\mathcal{B}[x]_f$-$\de$-algebras that is equivariant with respect to the morphism $\f_{\mathcal{B}}\colon G_{\mathcal{B}}\to H_{\mathcal{B}}$, where $\mathcal{R}=R_0\otimes_{k[x]_h}\mathcal{B}[x]_f=(R_0\otimes_k\B)_f$ and $\mathcal{R}/\B[x]_f$ has an $H_\B$-action induced from \rm{(i)} above. (The equivariance means that for every $\mathcal{B}$-algebra $\mathcal{T}$ and $g\in G_{\B}(\T)$ the diagram
		$$
		\xymatrix{
			\SSS\otimes_\B\T \ar^g[r] & \SSS\otimes_\B\T \\
			\R\otimes_\B\T \ar[u] \ar^-{\f(g)}[r]& \R\otimes_\B \T \ar[u]
		}
		$$
		commutes.)
	\end{enumerate}
\end{lemma}
\begin{proof}
	We begin with (i). 	Choose $A\in k(x)^{m\times m}$ such that $R/k(x)$ is a Picard-Vessiot ring for $\de(y)=Ay$. Let $h\in k[x]$ be a monic polynomial such that $A\in k[x]_h^{m\times m}$ and let $Y\in \Gl_m(R)$ satisfy $\de(Y)=AY$. Then $R_0=k[x]_h[Y,\frac{1}{\det(Y)}]$ is a $\de$-subring of $R$ because $\de(Y)=AY\in R_0^{m\times m}$.
	For a $k$-algebra $T$ an element $g\in H(T)$ acts on $R\otimes_k T$ by $Y\otimes 1\mapsto Y\otimes [g]$ for some matrix $[g]\in \Gl_m(T)$. Thus $g\colon R\otimes_k T \to R\otimes_k T$ restricts to a $k[x]_h\otimes_k T$\=/$\de$\=/automorphism of $R_0\otimes_k T$. 
%	\rf{$k[x]_k$ was changed into $k[x]_h$.}
	This defines an action of $H$ on $R_0/k[x]_h$ that clearly induces the given action on $R/k(x)$.
	
	To address the remaining points, fix $A'\in k'(x)^{n\times n}$ and $Y'\in \Gl_n(S')$ such that $S'/k'(x)$ is a Picard-Vessiot ring for $\de(y)=A'y$ and $\de(Y')=A'Y'$. As
	 $R\subseteq R'\subseteq S'=k'(x)[Y',\frac{1}{\det(Y')}]$, there exists a finitely generated $k$-subalgebra $\B_0$ of $k'$ and a monic polynomial $f'_0\in \B_0[x]$ such that all entries of $Y$ and $\frac{1}{\det(Y)}$ lie in $\B_0[x]_{f'_0}[Y',\frac{1}{\det(Y')}]\subseteq S'$. Set $f_0=hf_0'\in\B_0[x]$. Then $R_0=k[x]_h[Y,\frac{1}{\det(Y)}]\subseteq \B_0[x]_{f_0}[Y',\frac{1}{\det(Y')}]\subseteq S'$.
		
	Now Lemma \ref{lemma: spread out PVring} yields a finitely generated $k$-subalgebra $\B$ of $k'$ with $\B_0\subseteq \B$, a monic polynomial $f\in \B[x]$ such that $f_0$ divides $f$ in $\B[x]$, a differential $G_\B$-torsor $\mathcal{S}/\B[x]_f$ for $\de(y)=A'y$ with fundamental solution matrix $\Y$ such that $\mathcal{S}$ is flat over $\B[x]_f$, $S^\gen=\mathcal{S}\otimes_{\B[x]_f}K(x)$ is $\de$-simple and $\mathcal{S}\otimes_{\B[x]_f}k'(x)\simeq S'$ via $\Y\mapsto Y'$. So, $\mathcal{S}$ can be identified with $\B[x]_f[Y',\frac{1}{\det(Y')}]\subseteq S'$ and $R_0\subseteq \mathcal{S}$.  
	
	 By now we have established (i), (ii) and (iii) of the lemma.
%	 \rf{``proposition" was changed into ``lemma".}
	  Let us move on to (iv).
	 Setting $\R=R_0\otimes_{k[x]_h}\B[x]_f$, the inclusion $R_0\subseteq \mathcal{S}$ yields a morphism $\R\to \SSS$ of $\B[x]_f$-$\de$-algebras. Since $S'/k'(x)$ is a solution of the differential embedding problem $(\f_{k'}\colon G_{k'}\twoheadrightarrow H_{k'},\ R')$ the diagram
	 $$
	 \xymatrix{
	 S' \ar[r]  & S'\otimes_{k'}k'[G_{k'}] \\
	 R' \ar[r] \ar[u] & R'\otimes_{k'}k'[H_{k'}] \ar[u]	
	 }	
	 $$
	 commutes. Therefore also
	  $$
	 \xymatrix{
	 	\SSS \ar[r]  & \SSS\otimes_\B \B[G_\B]\ar@{=}[r] & \SSS\otimes_{k}k[G] \ar@{^(->}[r]& S'\otimes_{k}k[G]=S'\otimes_{k'}k'[G_{k'}] \\
	 	R_0 \ar[rr] \ar[u] & & R_0\otimes_{k}k[H] \ar[u]  \ar@{^(->}[r] & R'\otimes_k k[H]=R'\otimes_{k'}k'[H_{k'}]	\ar[u]
	 }	
	 $$
	 and
	 $$
	 \xymatrix{
	 \SSS \ar[r] & \SSS\otimes_\B \B[G_\B] \\
	 \R \ar[r] \ar[u] & \R\otimes_\B \B[H_\B] \ar[u]	
	 	}
	 $$
	 commutes. This establishes (iv).
\end{proof}

The following theorem was proved in \cite{BachmayrHartmannHarbaterPop:Large} under the assumption that $k$ has infinite transcendence degree over $\mathbb{Q}$. Thanks to Theorem \ref{theo: main specialization} we are able to remove this unnecessary assumption.

\begin{theo} \label{theo: solve embedding problems}
	Let $k$ be an algebraically closed field of characteristic zero. Then every differential embedding problem of finite type over $(k(x),\frac{d}{dx})$ is solvable.
\end{theo}
\begin{proof}
	Let $(\f\colon G\twoheadrightarrow H,\ R)$ be a differential embedding problem of finite typer over $k(x)$. Let $k'$ be an algebraically closed field extension of $k$ such that $k'$ has infinite transcendence degree over $\mathbb{Q}$. Consider the induced differential embedding problem $(\f_{k'}\colon G_{k'}\twoheadrightarrow H_{k'},\ R')$, where $R'=R\otimes_{k(x)}k'(x)$. According to \cite[Cor. 4.6]{BachmayrHartmannHarbaterPop:Large} it has a solution $S'/k'(x)$. Let $h,R_0,\B,f,\mathcal{A},\R$ and $\SSS$ be as in Lemma \ref{lemma: spread out solution to differential embedding problem} and set $\X=\spec(\B)$. 
	Then $\SSS/\B[x]_f$ satisfies all the assumptions of Notation \ref{notation: setup} and it follows from Theorem~\ref{theo: main specialization} that there exists a $c\in\X(k)$ such that $S^c=\SSS\otimes_{\B[x]_f}k(x)$ is a Picard-Vessiot ring. In fact, $S^c/k(x)$ is a Picard-Vessiot ring with differential Galois group $G^c=(G_{\B})_c=G$.
	
	The morphism $\R\to \SSS$ of $\B[x]_f$-$\de$-algebras yields a morphism
	$\R\otimes_{\B[x]_f}k(x)\to \SSS\otimes_{\B[x]_f}k(x)$ of $k(x)$-$\de$-algebras. But $$\R\otimes_{\B[x]_f}k(x)=(R_0\otimes_{k[x]_h}\B[x]_f)\otimes_{\B[x]_f}k(x)=R_0\otimes_{k[x]_h}k(x)=R$$ and $\SSS\otimes_{\B[x]_f}k(x)=S^c$.
	We thus obtain an embedding $R\to S^c$ of Picard-Vessiot rings over $k(x)$. By (iv) of Lemma \ref{lemma: spread out solution to differential embedding problem} the diagram
	$$
	\xymatrix{	
	\SSS \ar[r] & \SSS\otimes_\B \B[G_\B] \\
	\R \ar[r] \ar[u]& \R\otimes_\B\B[H_\B] \ar[u]
}
	$$
	commutes. Therefore also its base change
		$$
	\xymatrix{	
		S^c\ar[r] & S^c\otimes_k k[G] \\
		R \ar[r] \ar[u]& R\otimes_k k[H] \ar[u]
	}
	$$
via $\B[x]_f\to k(x)$ commutes. Thus $S^c/k(x)$ is a solution of $(\f\colon G\twoheadrightarrow H,\ R)$.
%\rf{$\f\colon$ was added.}
\end{proof}
We are now prepared to prove Matzat's conjecture in full generality.

\begin{theo} \label{theo: Matzat's conjecture}
	Let $k$ be an algebraically closed field of characteristic zero and let $F$ be a one-variable function field over $k$, equipped with a non-trivial $k$-derivation. Then the absolute differential Galois group of $F$ is the free proalgebraic group on a set of cardinality $|F|$.
\end{theo}
\begin{proof}
	The case when $k$ has infinite transcendence degree (over $\mathbb{Q}$) is proved in \cite[Theorem 2.9]{Wibmer:SubgroupsOfFreeProalgebraicGroupsAndMatzatsConjecture}. So we can assume that $k$ has finite transcendence degree. In particular, $k$ is countable.
In \cite[Prop. 2.8]{Wibmer:SubgroupsOfFreeProalgebraicGroupsAndMatzatsConjecture} it is shown that if Matzat's conjecture is true for $(k(x),\frac{d}{dx})$, then it is true for $F$. So we have reduced to proving the theorem for $(k(x),\frac{d}{dx})$ with $k$ countable. According to \cite[Cor. 3.9]{BachmayrHarbaterHartmannWibmer:FreeDifferentialGaloisGroups}, the absolute differential Galois group of a countable differential field $F$ is free on a countably infinite set if and only if every differential embedding problem over $F$ of finite type is solvable. It therefore suffices to show that every differential embedding problem of finite type over $(k(x),\frac{d}{dx})$ is solvable. This is exactly Theorem \ref{theo: solve embedding problems}.
\end{proof}

\bibliographystyle{alpha}
\bibliography{bibdata}

	Ruyong Feng, KLMM, Academy of Mathematics and Systems Science, Chinese Academy of Sciences, and School of Mathematics, University of Chinese Academy of Sciences, No.55 Zhongguancun East Road, 100190, Beijing, China, \texttt{ryfeng@amss.ac.cn} 

\medskip

Michael Wibmer, School of Mathematics, University of Leeds, LS2 9JT, Leeds, United Kingdom, \texttt{m.wibmer@leeds.ac.uk}

\end{document}